\documentclass[12pt]{UOthesis}
\usepackage{mathrsfs}
\usepackage{amssymb,latexsym,amscd,verbatim,color}
\usepackage[all]{xy}
\usepackage{bbm}
\usepackage{tikz-cd}
\usepackage{tikz}
\usepackage[capitalize]{cleveref} 
\usepackage{verbatim}
\usepackage{microtype,fixltx2e,lmodern} 
\usepackage{enumerate,comment,braket,xspace,csquotes} 
\usepackage{multirow}
\usepackage{stmaryrd}

\newtheorem{thm}{Theorem}[section]
\newtheorem{prop}[thm]{Proposition}
\newtheorem{lemma}[thm]{Lemma}
\newtheorem{cor}[thm]{Corollary}
\newtheorem{hyp}[thm]{Hypothesis}
\theoremstyle{definition}
\newtheorem{defn}[thm]{Definition}
\newtheorem{notation}[thm]{Notation}
\newtheorem{remark}[thm]{Remark}
\newtheorem{remarks}[thm]{Remarks}
\newtheorem{example}[thm]{Example}
\newtheorem{examples}[thm]{Examples}
\newtheorem{mynote}{Comment[mynote]}
\newtheorem{para}[thm]{}
\newtheorem{thm*}{Theorem}[chapter]
\numberwithin{equation}{section}
\newcounter{spec}
{\end{list}}

\renewcommand{\hat}{\widehat}
\newcommand{\Spec}{\operatorname{Spec}} 
\newcommand{\Spf}{\operatorname{Spf}} 
\newcommand{\Hom}{\operatorname{Hom}}      
\newcommand{\Aut}{\operatorname{Aut}}      
\newcommand{\Ext}{\operatorname{Ext}}      
\newcommand{\iext}{\operatorname{\underline{Ext}}}
\newcommand{\Hp}{\operatorname{\cH_{p}^{\varphi}}}
\newcommand{\Hpp}{\operatorname{\cH_{p}^{\varpi}}}
\newcommand{\hp}{\operatorname{h_p}}   
\newcommand{\W}{\operatorname{W}}
\newcommand{\D}{\operatorname{\mathbb{D}}}
\newcommand{\K}{\operatorname{\mathbb{K}}}
          
\newcommand{\Mi}{\mathcal{M}_1}   
\newcommand{\gMi}{\mathcal{M}^{\text{gr}}_1} 

\newcommand{\tMi}{{}^t\mathcal{M}_1} 
\newcommand{\Tcrys}{\operatorname{T_{crys}}} 

\newcommand{\TdR}{\operatorname{T_{dR}}} 
\newcommand{\Tsing}{\operatorname{T_{sing}}} 
\newcommand{\Tp}{\operatorname{T_{p}}} 
\newcommand{\Tl}{\operatorname{T_{\ell}}} 

\newcommand{\TsingQ}{\operatorname{T_{sing}^{\Q}}} 
\newcommand{\mup}{\mu_{p^{\infty}}} 
\newcommand{\ocp}{\cO_{\C_p}}

\newcommand{\Sch}{\operatorname{Sch}}

\newcommand{\Mod}{{\operatorname{Mod\text{}}}}
\newcommand{\Vect}{{\operatorname{Vect\text{}}}}

\newcommand{\Nat}{\operatorname{Nat}}

\newcommand{\Sp}{\operatorname{sp}}
\newcommand{\Zp}{\operatorname{Z_p}}
\newcommand{\Vp}{\operatorname{V_p}}
\newcommand{\Ainf}{\operatorname{A_{inf}}}
\newcommand{\BdR}{\operatorname{B_{dR}}}
\newcommand{\BdRp}{\operatorname{B^+_{dR}}}
\newcommand{\Acris}{\operatorname{A_{cris}}}
\newcommand{\Bcris}{\operatorname{B_{cris}}}
\newcommand{\Bcrisp}{\operatorname{B^+_{cris}}}
\newcommand{\At}{\operatorname{A_2}}
\newcommand{\Bt}{\operatorname{B_2}}
\newcommand{\Bct}{\operatorname{B_{2,cris}}}
\newcommand{\BHT}{\operatorname{B_{HT}}}
\newcommand{\DBHT}{\operatorname{D_{B_{HT}}}}
\newcommand{\DdR}{\operatorname{D_{dR}}}
\newcommand{\Dcris}{\operatorname{D_{cris}}}
\newcommand{\DHT}{\operatorname{D_{HT}}}
\newcommand{\Aocris}{\operatorname{A^0_{cris}}}
\newcommand{\Ann}{\operatorname{Ann}}  

\newcommand{\End}{\operatorname{End}}      

\newcommand{\Tr}{\operatorname{Tr}}
\newcommand{\Div}{\operatorname{Div}}

\newcommand{\ihom}{{\rm\underline{Hom}}}  

\newcommand{\E}{\mathbb{E}}      

\newcommand{\C}{\mathbb{C}}     
\newcommand{\F}{\mathbb{F}}     
\newcommand{\Q}{\mathbb{Q}}     
\newcommand{\Z}{\mathbb{Z}}     

\newcommand{\N}{\mathbb{N}}
\newcommand{\G}{\mathbb{G}}     
\newcommand{\R}{\mathbb{R}}     

\newcommand{\ve}{^{\vee}} 

\newcommand{\im}{\operatorname{Im}}        
\renewcommand{\ker}{\operatorname{Ker}}  
\newcommand{\coker}{\operatorname{Coker}} 
\newcommand{\gr}{\operatorname{gr}}        
\newcommand{\Pic}{\operatorname{Pic}}     
\newcommand{\Alb}{\operatorname{Alb}}     

\newcommand{\rank}{\operatorname{rank}}    
\newcommand{\tors}{{\operatorname{tors}}} 

\renewcommand{\tilde}{\widetilde}

\newcommand{\ie}{{\it i.e., }}
\newcommand{\st}{{\,\mid\, }}

\newcommand{\Rep}{{\operatorname{Rep_{\Q_p}}}}

\newcommand{\Lie}{{\operatorname{Lie}}} 
\newcommand{\Gal}{{\operatorname{Gal}}} 
\newcommand{\coLie}{{\operatorname{coLie}}} 
\newcommand{\Der}{{\operatorname{Der}}} 
\newcommand{\Tang}{{\operatorname{T}}} 
\newcommand{\Fil}{{\operatorname{Fil}}} 
\newcommand{\MF}{{\operatorname{MF^{\varphi}_{K}}}} 
\newcommand{\MFw}{\operatorname{MF^{\varphi,w.a.}_{K}}} 
\newcommand{\Ex}{{\operatorname{E}}}
\newcommand{\nat}{^{\natural}}
\newcommand{\logcris}{\operatorname{\log_{cris}}}

\renewcommand{\bar}{\overline}
\newcommand{\barQ}{\overline{\Q}} 
\newcommand{\barQp}{\overline{\Q}_p} 
\newcommand{\into}{\hookrightarrow}

\renewcommand{\lim}{\varprojlim}
\newcommand{\colim}{\varinjlim}
\newcommand{\onto}{\mbox{$\to\!\!\!\!\to$}}

\newcommand{\cA}{\mathcal{A}}
\newcommand{\cB}{\mathcal{B}}
\newcommand{\cC}{\mathcal{C}}
\newcommand{\cD}{\mathcal{D}}

\newcommand{\cF}{\mathcal{F}}

\newcommand{\cH}{\mathcal{H}}
\newcommand{\cI}{\mathcal{I}}

\newcommand{\cK}{\mathcal{K}}

\newcommand{\cM}{\mathcal{M}}

\newcommand{\cO}{\mathcal{O}}
\newcommand{\cP}{\mathcal{P}}

\newcommand{\cS}{\mathcal{S}}
\newcommand{\cT}{\mathcal{T}}

\newcommand{\sA}{\mathscr{A}}

\newcommand{\sD}{\mathscr{D}}

\newcommand{\sF}{\mathscr{F}}
\newcommand{\sG}{\mathscr{G}}
\newcommand{\sI}{\mathscr{I}}

\newcommand{\sV}{\mathscr{V}}

\newcommand{\sX}{\mathscr{X}}
\newcommand{\sY}{\mathscr{Y}}

\newcommand{\fg}{\mathfrak{g}}

\newcommand{\fp}{\mathfrak{p}}
\newcommand{\fh}{\mathfrak{h}}

\newcommand{\fm}{\mathfrak{m}}
\newcommand{\fb}{\mathfrak{b}}
\newcommand{\fa}{\mathfrak{a}}

\newcommand{\fppf}{_{\text{fppf}}} 
\newcommand{\etale}{_{\text{\'etale}}} 

\renewcommand{\phi}{\varphi}
\renewcommand{\epsilon}{\varepsilon}
\newcommand{\boxtensor}{\def\boxtimesten{\Box\kern-7.59pt\raise1.2pt
\hbox{$\times$} }}                                  

\newcounter{elno}                   

\allowdisplaybreaks

\NoLogo
\NoSignatures

\setUOname{Mohammadreza Mohajer}         
\setUOcpryear{2024}               
\setUOtitle{Linear relations of p-adic periods of 1-motives}  

 \setUOdatedef{11/11/2015}           
 \setUOrefone{Dr. Antonio Lei}            
 \setUOreftwo{Dr. Kiril Zaynulin}            
 \setUOrefthree{Dr. Colin Ingals}          
 \setUOextrefone{Dr. Jan Minac} 
 \setUOsupone{Dr. Abdellah Sebbar}      

\setUOdegree{}
\phd                
\setUOlogoloc{UOlogoBW}{width=1in}

{
  \newtheoremstyle{remarkstyle}{\topsep}{\topsep}{\rm}{}{\bfseries}{.}{.5em}{}
  \theoremstyle{remarkstyle}
}

\begin{document}


\notachapter{Abstract}
This work aims to develop p-adic analogs of known results for classical periods, focusing specifically on 1-motives. We establish an integration theory for 1-motives with good reductions, which generalizes the Colmez-Fontaine-Messing p-adic integration for abelian varieties with good reductions. We also compare the integration pairing with other pairings such as those induced by crystalline theory. Additionally, we introduce a formalism for periods and formulate p-adic period conjectures related to p-adic periods arising from this integration pairing. Broadly, our p-adic period conjecture operates at different depths, with each depth revealing distinct relations among the p-adic periods. Notably, the classical period conjecture (Kontsevich-Zagier conjecture over $\barQ$) for 1-periods fits within our framework, and, according to the classical subgroup theorem of Huber-Wüstholz for 1-motives, the conjecture for classical periods of 1-motives holds true at depth 1. Finally, we identify three $\Q$-structures arising from $\barQ$-rational points of the formal p-divisible group associated with a 1-motive $M$ with a good reduction at $p$, and we prove p-adic period conjectures at depths 2 and 1, relative to periods induced by the p-adic integration of $M$ and these $\Q$-structures. Our proof involves a p-adic version of the subgroup theorem that we obtain for 1-motives with good reductions.

\clearpage

\notachapter{R\'esum\'e}
Ce travail vise à développer des analogues p-adiques des résultats connus pour les périodes classiques, en se concentrant spécifiquement sur les 1-motifs. Nous établissons une théorie de l'intégration pour les 1-motifs avec bonne réduction, qui généralise l'intégration p-adique de Colmez-Fontaine-Messing pour les variétés abéliennes avec bonne réduction. Nous comparons également cette int\'egration avec d'autres int\'egrations, telless que celles induites par la théorie cristalline. De plus, nous introduisons un formalisme pour les périodes et formulons des conjectures sur les périodes p-adiques liées aux périodes p-adiques provenant de ce couplement d'intégration. De manière générale, notre conjecture sur les périodes p-adiques opère à différents niveaux de profondeur, chaque profondeur révélant des relations distinctes entre les périodes p-adiques. Notamment, la conjecture des périodes classiques (conjecture de Kontsevich-Zagier sur $\barQ$) pour les 1-périodes apparaît comme un cas particulier dans notre cadre, où il est connu que la conjecture des périodes est vérifiée au niveau 1, comme le montre le théorème classique de sous-groupe de Huber-Wüstholz pour les 1-motifs. Enfin, nous identifions trois $\Q$-structures provenant des points $\barQ$-rationnels du groupe formel p-divisible associé à un 1-motif $M$ avec bonne réduction en p, et nous prouvons les conjectures sur les périodes p-adiques aux profondeurs 2 et 1, en relation avec les périodes induites par l'intégration 
p-adique de $M$ et ces $\Q$-structures. Notre preuve implique une version p-adique du théorème du sous-groupe que nous obtenons pour les 1-motifs avec bonne réduction.

\clearpage

\clearpage

\notachapter{Acknowledgement}   
 I would like to express my deepest gratitude to my supervisor, Prof. Abdellah Sebbar, for his unwavering support, kindness, patience, and empathy throughout my doctoral journey. He was the first to introduce me to the topics of classical periods of elliptic curves, Grothendieck's period conjecture, and Andr\'e's related work, which has profoundly shaped the direction of my research. His belief in my potential has been a constant source of motivation and personal growth, for which I am deeply grateful.


\clearpage

\tableofcontents
\newpage
\thispagestyle{plain}
\begin{flushright}
{\large\fontfamily{ppl}\selectfont \emph{Nor you nor I can read the etern decree,\\
To that enigma we can find no key;\\
You and I talk only this side of the veil;\\
But, if that veil be lifted, neither you nor I will remain.
}}
\vspace{0.5cm}

\footnotesize{---Khayyam, Omar. Rubaiyat of Omar Khayyam.\\ Translated by Edward FitzGerald, edited version}
\end{flushright}
\clearpage
\nonumchapter{Introduction}
The integration pairing, by Stoke's theorem, induces a pairing
\[
\operatorname{H^{sing}_n}(X^{an},\Q)\times \operatorname{H_{dR}^n}(X)\to\C,\,\, (\sigma,\omega)\mapsto\int_\sigma\omega
\]
where $X$ is a smooth variety over $K\subseteq\C$ and $X^{\text{an}}$ is the analytification of $X$. The above pairing is called a (classical\footnote{We use the term "classical" to distinguish these periods from p-adic periods, which are the main focus of this research.}) period pairing. A classical period of $X$ is a complex number in the image of this pairing. The classical period pairing is perfect (by the Grothendieck-Deligne's theorem), i.e., after extending the scalars to $\C$, we have a functorial isomorphism
\[
\operatorname{H^n_{sing}}(X,\Q)\otimes_{\Q}\C\cong\operatorname{H^n_{dR}(X)}\otimes_K\C
.\]
We have two types of obvious relations among these periods:
\begin{enumerate}
    \item Bilinearity: $$\displaystyle\int_{a_1\sigma_1+a_2\sigma_2}b_1\omega_1+b_2\omega_2=a_1b_1\int_{\sigma_1}\omega_1+a_1b_2\int_{\sigma_1}\omega_2+a_2b_1\int_{\sigma_2}\omega_1+a_2b_2\int_{\sigma_2}\omega_2$$
for any $\sigma_1,\sigma_2\in\operatorname{H^n_{sing}}(X,\Q), \omega_1,\omega_2\in\operatorname{H^n_{dR}(X)}$, $a_1,a_2\in\Q$, and $b_1,b_2\in K$.
    \item Functoriality: $\displaystyle\int_{f_*\sigma}\omega=\int_{\sigma}f^*\omega$
    for any morphism $f:X\to Y$, where $\omega\in\operatorname{H^n_{dR}(Y)}$ and $\sigma\in\operatorname{H^n_{sing}}(X,\Q)$.
\end{enumerate}
The classical period conjecture states that there are no relations among periods except those induced by bilinearity and functoriality.

The transcendence properties
of periods make it natural to ask questions about linear and algebraic relations
among them. The study of linear relations of periods was originally formulated in \cite{MR1829883} and \cite{MR1718044} which is a special case of algebraic relations among these periods. Periods are given a cohomological interpretation and all relations among them should be induced by relations between motives. This idea was initiated in \cite{grothendieck1966rham} to study period algebras. A neutral Tannakian category is required to construct the period algebras. In this framework, the period conjecture for a full abelian rigid tensor subcategory generated by a motive $M$ is equivalent to Grothendieck’s period conjecture which states that, the transcendental degree of the period algebra of $M$ is equal to the dimension of the motivic Galois group $G(M)$. Grothendieck did not publish his conjecture himself. For the history and the formulation of the conjecture, see \cite{And04} and \cite{And17}. For further details on equivalent constructions of classical periods and the conjecture in the context of Nori motives, we refer the reader to \cite{huber2017periods} or \cite{huber2018galois}. In the general case, the Grothendieck's period conjecture is widely open, however, the linear space of periods and the relations among them are now better known for certain cases specifically in the category of 1-motives.

A Deligne 1-motive (\cref{definition of 1-motive}) over a field $K$ is a two-term complex $M=[L\to G]$ of group schemes over $K$, where $L$ is a lattice and $G$ is a semi-abelian group, i.e., an extension of an abelian scheme $A$ by a torus $T$. A 1-motive $M$ over a subfield $K\subset\C$ is equipped with the singular realization functor $\Tsing(M)$ and the de Rham realization functor $\TdR(M)$ yielding a comparison isomorphism 
\[
\Tsing(M)\otimes_{\Z}\C\cong\TdR(M)\otimes_K\C.
\]
Let $\TsingQ(M):=\Tsing(M)\otimes_\Z\Q$. One can write the pairing 
\begin{equation}\label{pairing1}
\TsingQ(M)\times\TdR\ve(M)\to \C
\end{equation}
induced by the above isomorphism which indeed arises from integration of 1-forms. A classical period of $M$ is a complex number which lies in the image of the above pairing. Huber and Wüstholz in \cite{huber2022transcendence} showed that all relations among classical periods of 1-motives over $\barQ$ are induced by bilinearity and functoriality. One of the main tools in their proof was the use of the celebrated analytic subgroup theorem (\cite{wustholz1989algebraische}).

\begin{center}
    \textit{What about p-adic analogues of this tale?}
\end{center}
This research is mainly motivated by this question.
The first step towards addressing this question is to find a suitable perfect pairing in the p-adic setting, known as the p-adic integration pairing. For a 1-motive $M$ with a good reduction over a p-adic local field $K$, we develop a p-adic integration theory which generalizes Colmez's construction of p-adic integration for abelian varieties with good reductions (\cite{colmez1992periodes}).
\begin{thm*}[\cref{theorem p-adic integration pairing for M}]
There exists a pairing $$\displaystyle\int^{\varpi}:\Tp(M)\times\TdR\ve(M)\to\Bt$$ which is bilinear, perfect, and Galois-equivariant in the first argument. Moreover, it respects the Hodge filtration.
\end{thm*}
In the above, $M=[L\to G]$ is a 1-motive with a good reduction over the p-adic local field $K$, and $\Tp(M)$ is the Tate module of $M$. Moreover, $\Bt:=\BdRp/t^2\BdRp$, where $\BdRp$ is the de Rham local ring with a uniformizer $t$ and the residue field $\C_p$. The map $\Tp(M)\to\TdR(M)\otimes_K\Bt$ corresponding to the pairing above is denoted by $\varpi_M$. As part of our work to develop a p-adic integration pairing for 1-motives with good reductions, we also introduce the following pairing: 
\[
\int^{\varphi}:\Tp(M)\times\coLie(G)\to\C_p(1).
\]
This pairing is induced by the map $\phi_M:\Tp(M)\to\Lie(G)\otimes_K\C_p(1)$, which coincides with Fontaine's map (\cite{fontaine1982formes}) when $M=[0\to A]$ is an abelian variety. We still refer to the map $\phi_M$ as Fontaine's map and the pairing $\displaystyle\int^{\phi}$ as Fontaine's pairing for the 1-motive $M$. Fontaine's map $\phi_M:\Tp(M)\to \coLie(G)\otimes\C_p(1)$ is surjective when tensored with $\C_p$, though the resulting map has a large kernel. Without tensoring the Fontaine's map $\phi_M$ and the integration map $\varpi_M$ with $\C_p$ and $\Bt$, we propose Theorems \ref{kernel Fontaine map phi} and \ref{kernel p-adic integration map varpi} to describe their kernels.

We call an element in $\Bt$ a Fontaine-Messing period of $M$ if it lies in the image of $$\displaystyle\int^{\varpi}:(\Tp(M)\otimes_{\Z_p}\Bt)\times(\TdR\ve(M)\otimes_K\Bt)\to\Bt.$$ Note that we have bilinearity and functoriality relations among the Fontaine-Messing periods of $M$. Comparing to the archimedean setting where natural $\Q$-structures exist on both sides of the comparison isomorphism 
\[\TsingQ(M)\otimes_{\Q}\C\cong\TdR(M)\otimes_K\C,\]
it is natural to ask the following question:

\textbf{Question. }\textit{Are there non-trivial natural $\Q$-structures $F(M)\subset\Tp(M)\otimes_{\Z_p}\Bt$ and $G(M)\subset\TdR\ve(M)\otimes_{K}\Bt$ such that all relations among the Fontaine-Messing periods of $M$ relative to $(F,G)$ are induced by bilinearity and functoriality? If there are any additional relations beyond bilinearity and functoriality, what are they?}

    Here, a Fontaine-Messing period of $M$ relative to $(F,G)$ is a Fontaine-Messing period of $M$ which is in the image of 
    \[
    \left.\int^{\varpi}\right|_{F(M)\times G(M)}:F(M)\times G(M)\to\Bt.
    \]
    When $M$ is a 1-motive over a number field $\K$, a natural choice for $G(M)$ is to take elements from $\TdR\ve(M_{\barQ})$. The main challenge is to find a suitable choice for $F(M)$ such that the Fontaine-Messing periods of $M$ relative to $(F,G)$ capture some arithmetic information of $M$.
    
    To address this problem, in Chapter 4, we first introduce a formalism of periods for an abelian category, which allows us to define and formulate the concept of period spaces and period conjectures in the p-adic setting. Roughly speaking, for a pairing $\cH$, an $\cH$-period is a number arising from the pairing $\cH$. We define the formal space of periods at depth $i$ (\cref{space of formal periods at depth i}), along with a period conjecture at depth $i$ (\cref{definition period conjecture at depth i}), inspired by \cite{hormann2021note}. Each depth reveals a different level of relations among $\cH$-periods, and the validity of the period conjecture at depth $1$ implies that all relations among $\cH$-periods are induced by bilinearity and functoriality. This framework also recovers the classical periods. In the case of 1-motives over $\barQ$ with $\cH$ defined by the integration pairing (\ref{pairing1}), it is known, according to \cite{huber2022transcendence}, that all relations among classical periods are induced by bilinearity and functoriality. Thus, in our setting, this means that the period conjecture holds at depth 1 for the classical periods of 1-motives over $\barQ$.

    Now that we have a suitable formalism of periods and a p-adic integration pairing for 1-motives, the main challenge is to find the $\Q$-structure $F(M)$. We identify three $\Q$-structures: $\hp(M)$, $\Hp(M)$, and $\Hpp(M)$. Broadly speaking, $\hp(M)$ arises from the image of the logarithm map over $\barQ$-rational points of the formal p-divisible group associated with $M$ (\cref{def tilde exp}), and the other two $\Q$-structures, $\Hp(M)$ and $\Hpp(M)$, are induced by pull-backs of $\hp(M)$ via the diagrams induced by Fontaine's map $\phi_M$ and the p-adic integration map $\varpi_M$ (\cref{definition Hp} and \cref{definition Hpp}). Inspired by \cite{cartier_l-functions_2007}, we provide different descriptions of elements in $\hp(M)$ using the notion of $\exp$-map for Dieudonn\'e modules in \cref{section p-adic logarithm}. In addition, we offer an interpretation of these elements from the perspective of p-adic Galois representations (\cref{theorem classification for algebraic poin in image of logarithm}).
    
 Finally, we define three pairings $\displaystyle\int^{\hp}$, $\displaystyle\int^{\Hp}$, and $\displaystyle\int^{\Hpp}$, obtained by restricting the p-adic integration pairing to these $\Q$-structures\footnote{For $\int^{\hp}$ and $\int^{\Hp}$, we set $G(M)=\Fil^1\TdR\ve(M_{\barQ})$. However, for $\int^{\Hpp}$, we choose $G(M)$ differently, as explained in \cref{three pairings hp Hp Hpp}.} (\cref{three pairings hp Hp Hpp}). We then prove the period conjectures at depths $1$, $2$, and $2$ for these pairings, respectively, for the category of 1-motives over a number field $\K$ with good reductions at $p$ (\cref{level 2 period cinjecture for Hpp}, \cref{depth 2 and 1 period cinjecture for Hp and hp}). We can summarize these results as follows:
 \begin{thm*}[\cref{level 2 period cinjecture for Hpp} and \cref{depth 2 and 1 period cinjecture for Hp and hp}]\label{theorem B}
     Let $M=[L\to G]$ be a 1-motive over a number field $\K$ with a good reduction at $p$ and $\K^u$ denote the algebraic extension of $\K$ obtained by taking the compositum of all finite extensions of $\K$ in which the prime $p$ does not ramify.
     \begin{enumerate}
         \item There exists a canonical $\Q$-structure $\tilde{\Hpp}(M)\subset\Tp(M)\otimes_{\Z_p}\Bt$ and a $\K^u$-linear subspace $\tilde{N}(M)$ of $\TdR\ve(M_{\bar{\K}})$ such that if $\displaystyle\alpha=\int^{\Hpp}_x\omega=0$ for $x\in\tilde{\Hpp}(M)$ and $\omega\in\tilde{N}(M)$, then there exists an exact sequence
     \[
     0\to M_1\to M^n\to M_2\to 0
     \]
     of 1-motives over a finite extension of $\K$ with good reductions at $p$, where $n\in\{1,2\}$, $x\in\tilde{\Hpp}(M_1)$, and $\omega\in\tilde{N}(M_2)$.
    \item Consider the $\Q$-structure $\Hp(M)\subset\Tp(M)\otimes_{\Z_p}\C_p$ ($\hp(M)$, resp.) and $\barQ$-linear subspace $\coLie(G_{\barQ})$ of $\TdR\ve(M_{\barQ})$. If $\displaystyle\alpha=\int^{\Hp}_x\omega=0$ for $x\in\tilde{\Hpp}(M)$ ($x\in\hp(M)$, resp.) and $\omega\in\coLie(G_{\barQ})$, then there exists an exact sequence
     \[
     0\to M_1\to M^n\to M_2\to 0
     \]
     of 1-motives over a finite extension of $\K$ with good reductions at $p$, where $n\in\{1,2\}$ ($n=1$, resp.), $x\in\tilde{\Hpp}(M_1)$ ($x\in\hp(M_1)$, resp.), and $\omega\in\Fil^1(\TdR\ve(M_{2}))_{\barQ}=\coLie(G_2)_{\barQ}$.
     \end{enumerate}
 \end{thm*}
    These theorems provide a precise description of the vanishing behaviour of our p-adic periods. Moreover, our results provide insights into the $\barQ$-linear (or $\K^u$-linear) independence of our p-adic periods; see, for example, \cref{prop 4.4.1}. One of the key components of the proofs is the \cref{p-adic subgroup theorem for Fontaine pairing}, which we refer to as the p-adic subgroup theorem for 1-motives. We derive this theorem using the p-adic analytic subgroup theorem for commutative algebraic groups (\cite{fuchs2015p}). Recall that, in the proof of the period conjecture for the classical periods of 1-motives (in \cite{huber2022transcendence}), the main ingredient was the (classical) analytic subgroup theorem (\cite{wustholz1989algebraische}). While the classical analytic subgroup theorem fostered a productive interplay between transcendence theory and algebraic groups via the exponential map for Lie groups, the p-adic subgroup theorem establishes a similar connection in the p-adic setting through the logarithm map.

    It is important to note that in the proof of these results, we do not rely on the main results from \cite{hormann2021note}. We only use Hormann's work in \cref{remark Hpp period does not hold at depth 1}, where we show that for $\Hpp$-periods (or $\Hp$-periods) of 1-motives with good reductions at $p$, there are some relations among periods beyond bilinearity and functoriality. We conclude this dissertation by some examples and computations in \cref{section: examples}. 

    What remains unknown is whether the $\Q$-structure $\hp(M)$ is a finite dimensional $\Q$-vector space. If so, then \cref{theorem B} implies that the dimension of the space of $\Hpp$-periods (respectively $\Hp$-periods or $\hp$-periods) of $M$ is equal to the dimension of its formal space at depth $2$ (respectively $2$ or $1$). Unlike the case of classical periods of 1-motives, where we could relate the dimension of the formal space of classical periods to the dimension of $\End(M)$ and $\End(\Tsing(M))$, this approach cannot be directly applied here due to the increased complexity of these $\Q$-structures, and we did not pursue further investigation into this matter.

    \subsection*{Related work}
    Several research efforts aimed at discovering p-adic analogs of classical periods have been carried out in \cite{ancona2022algebraic}, \cite{Ayoub2020NouvellesCD}, \cite{brown_notes_2017}, and \cite{furusho_p_2004}. For a smooth projective variety $X$ over $\barQ$ with a good reduction at $p$, Berthelot's comparison theorem \cite[Theorem V.2.3.2]{Berthelot} gives the isomorphism
    \begin{equation}\label{2}
        \operatorname{H^n_{dR}}(X/\barQ)\otimes\barQp\cong\operatorname{H^n_{crys}}(\bar{X}
,\barQp).   
\end{equation}
    In \cite{ancona2022algebraic}, the authors construct their p-adic periods, which they refer to as Andr\'e's p-adic periods, as the entries of the matrix $cl_p(X)\into \operatorname{H_{dR}}(X)\otimes\barQp$ via Berthelot's comparison isomorphism with respect to all choices of $\barQ$-bases, where $cl_p(X)$ denotes the algebraic classes modulo $p$ in crystalline cohomology. It is still unknown whether Andr\'e's p-adic periods arise from Coleman-Berkovich-Vologodsky p-adic integration (\cite{besser2000generalization}, \cite{berkovich2007integration}, \cite{vologodsky2001hodge}). However, using the motivic framework developed in \cite{ancona2022algebraic}, they establish an upper bound for the transcendental degree of Andr\'e's p-adic period algebra of a motive $M$ with good reduction at $p$:
    \begin{equation*}
     \dim\cP_p(\langle M\rangle)\leq \dim G_{dR}(M)-\dim G_{crys}(M),
    \end{equation*}
    where $G_{dR}(M)$ and $G_{crys}(M)$ are de Rham and crystalline Galois groups (dual groups), respectively.
    Similar to the classical Grothendieck period conjecture, the p-adic Grothendieck period conjecture holds for $M$ if the equality holds in the above inequality. This conjecture is widely open. As noted in \cite[Example 9.6]{ancona2022algebraic}, and based on \cite[Examples 2 and 3]{Yamashita2010}, the only non-trivial case currently known where this conjecture holds is for the Kummer motive $M(\alpha)$, where $\alpha$ is a rational number.

    The approach of \cite{ancona2022algebraic} does not provide insight into linear relations or vanishing behaviour of p-adic periods, nor can it be applied to construct p-adic periods of curves, as algebraic classes within crystalline cohomology cannot be identified in this case. Furthermore, this interpretation has no meaningful application to the category of 1-motives. In contrast, our approach addresses these issues and describes all possible relations among our p-adic periods. Additionally, their method assumes that the standard conjecture, stating that numerical equivalence is the same as homological equivalence  ( $\sim_{num}=\sim_{hom}$), holds true for the category of motives with good reduction, while our results are obtained without assuming the validity of any conjecture.

    In the setting of mixed Tate motives analogous constructions to Andr\'e's p-adic periods already exist in \cite{furusho_p_2004} and \cite{brown_notes_2017}. These works considered the action of the Frobenius on the crystalline realization and defined the p-adic periods of the motive as the coefficients of this action with respect to all choices of bases induced by the isomorphism (\ref{2}). Andr\'e's p-adic periods for mixed Tate motives come indeed from these coefficients (\cite{ancona2022algebraic}).

    The algebra of abstract p-adic periods constructed in \cite{Ayoub2020NouvellesCD} is related to the algebra of Anr\'e's p-adic periods, as it is pointed out in \cite{ancona2022algebraic}. Their construction of p-adic periods fits within the Formalism of periods we introduce in Chapter 4. However the relation between our p-adic periods ($\hp$-periods, $\Hp$-periods, and $\Hpp$-periods) and Andr\'e's p-adic periods remains unclear. On the one hand, our p-adic periods arise from p-adic integration (which we develop in Chapter 3), while it is still unknown yet whether Andr\'e's p-adic period originates from a p-adic integration theory. Furthermore, even for an abelian variety $A$, it is unclear whether elements in our $\Q$-structure $\hp(A)$, $\Hp(A)$, or $\Hpp(A)$ correspond to algebraic classes within $\operatorname{H^2_{crys}}(\bar{A})$. We have not explored their potential connections further, leaving them as a topic for future research.
    \subsection*{Why 1-motives?}
   In \cite{Murre-introductionmotives}, Murre associates to a smooth $n$-dimensional projective variety $X$ over an algebraic closed field $K$, a Chow cohomological Picard motive $M^1(X)$ along with the Albanese motive $M^{2n-1}(X)$. The projector $\pi_1\in CH^n(X\times X)_{\Q}$ defining $M^1(X)$ is obtained via the isogeny $\Pic^0(X)\to \Alb(X)$ between the Picard and Albanese variety, given by the restriction to a smooth curve $C$ on $X$ since $\Alb(C) = \Pic^0(C)$. In the case of curves, $M^1(X)$ is the Chow motive of $X$ (\cite{jannsen_classical_1994}) and we have 
    \begin{equation}\label{1}
    \Hom(M^1(X),M^1(Y))\cong\Hom(\Pic^0(X),\Pic^0(Y))  
    \end{equation}
    by \cite[Theorem 22]{weil_courbes_1971}. Moreover, the semi-simple abelian category of abelian varieties up to isogeny is equivalent to the pseudo-abelian envelope of the category of Jacobians up to isogeny. As Grothendieck observed, the theory of pure motives for smooth projective curves is, in fact, equivalent to the theory of abelian varieties up to isogeny, since one-dimensional pure motives are precisely abelian varieties.

    The identification (\ref{1}) suggests that we may take objects represented by Picard functors as models for larger categories of mixed motives of any varieties over arbitrary base schemes $S$. Let $X$ be a variety and $\bar{X}$ a closure of $X$ with divisor at infinity $X_{\infty}$, i.e., $X=\bar{X}-X_{\infty}$. When $\bar{X}$ is smooth, we have that $\Pic(X)$ is the cokernel of the canonical map $\Div_{\infty}(\bar{X})\to\Pic(\bar{X})$ associating $D\mapsto\cO(\D)$, for divisors $D$ on $\bar{X}$ supported at infinity. Thus, we may set our models following \cite{SCC_1958-1959__4__A10_0} and \cite{deligne1974theorie} as
    \[
    [\Div^0_{\infty}(\bar{X})\to\Pic^0(\bar{X})]
    \]
    induced by mapping algebraically equivalent zero divisors at infinity to line bundles.
    As discussed in the end of Chapter 4, this two-term complex is called the 1-motive associated with $X$ (or the homological Picard 1-motive of $X$). For further details on this topic and related discussions, we refer the reader to \cite{Barbieri-ontheory1-motive}. The primary goal of studying 1-motives is to recover the information about $\operatorname{H}^1(X)$ for various known Weil cohomology theories--such as de Rham cohomology, Betti cohomology, $\ell$-adic cohomology, and crystalline cohomology--by 1-motives and their realization functors. Specifically, the subcategory of $\operatorname{\cM\cM^{eff}_{Nori}}$ generated by $H^i(X,Y)$ for $i\leq 1$ is equivalent to Deligne's category of 1-motives (\cite{ABV15}). 

    1-motives lie at the intersection of several key objects in algebraic geometry, such as abelian varieties, tori, and finite groups. They generalize both these structures and encode important geometric and arithmetic information. In practical terms, 1-motives provide a framework that allows for explicit computations, such as in the study of extensions of algebraic groups, cohomology theories, and period integrals. The structure of 1-motives facilitates the development of comparison isomorphism theorems between their realization functors, providing deep insights into the arithmetic of varieties and allowing for the exploration of both classical periods and p-adic periods associated with them.

\section*{Reading guide}
This dissertation is composed of 4 chapters.
\subsection*{Chapter 1} This chapter serves as preparation for the core elements of the thesis. It will gather and summarize key results in the theory of p-divisible groups and p-adic Hodge theory, 
which will be used throughout the rest of the dissertation and provide a deeper understanding of the topics covered in this research.
For further details, we direct readers to the primary references listed at the beginning of the chapter. This chapter is challenging to read due to the depth and complexity of the topics it covers. To aid understanding, we briefly outline its main structure here. Chapter 1 is broadly divided into three main parts:
\begin{itemize}
    \item \textbf{Theory  of p-divisible groups (Barsotti-Tate groups)}: We begin by introducing the basic definitions and properties of p-divisible groups over a general base scheme $S$. We bring the definition of the p-adic logarithm associated with a p-divisible group over the valuation ring of a p-adic local field, and we state the Hodge-Tate decomposition theorem for p-divisible groups (\cref{Theorem: Hodge-Tate decomposition}).
    
    \item \textbf{Dieudonn\'e theory}: We review contravariant Dieudonn\'e theory over a prefect field $k$ of characteristic $p$, which associates a $\sigma$-semilinear $\W(k)[\cF]$-module $\D(G)$ of finite type to any p-divisible group $G$ over $k$. Here, $\W(k)$ denotes the ring of Witt vectors, and the action of $\cF$ is induced by the relative Frobenius on $G$. Inspired by Dieudonn\'e-Manin classification (\cref{Dieudoone Manin classification}), we introduce the concept of slope decomposition for isocrystals over $K_0=\W(k)[1/p]$. To generalize the Dieudonn\'e functor over more general base schemes, we review Grothendieck-Messing theory, which allows us to define the covariant Dieudonn\'e functor as an $\cF$-crystal on the crystalline site of $S$. Using this theory, one can obtain a canonical identification between the de Rham cohomology of an abelian scheme over $\W(k)$ and the contravariant Dieudonn\'e module associated with its reduction modulo $p$.
   
    \item \textbf{Period rings}: The Grothendieck’s comparison isomorphism between the de Rham cohomology of a smooth variety $X$ over subfield $K\into \C$ and the singular cohomology of its analytification $X^{an}/\C$ holds when the coefficients of both cohomology groups are extended to $\C$. Now, let $K$ be a p-adic local field. In search of a ring $B$ that provides a comparison isomorphism
    \begin{equation}\label{0.0.1}
    \displaystyle \operatorname{H^n_{\'et}}(X_{\bar{K}},\Q_p)\otimes_{\Q_p}B\cong\operatorname{H^n_{dR}}(X)\otimes_K B
    \end{equation}
    which is Galois equivariant, Fontaine defined the de Rham period ring $\BdRp$ and $\BdR$, where $\BdRp$ is a complete local ring with residue field $\C_p$, and $\BdR$ is its field of fractions. For more on why $B=\C_p$  does not work, and further details see \cite[\S 2, \S 3]{niziol_hodge_2021}. Faltings (\cite{faltings_crystalline_1988}) showed that for any proper smooth variety $X$ over a p-adic local field $K$, the choice $B=\BdR$ provides a canonical comparison isomorphism \ref{0.0.1}, compatible with the Galois action and filtration. In this section, in addition to introducing the de Rham period ring, we review the crystalline period ring $\Bcrisp$ and its field of fractions, $\Bcris$. The ring $\Bcrisp$ is indeed a local subring of $\BdRp$ that  is sufficiently large to which the Frobenius action naturally extends. As we will see at the end of this chapter, this period ring gives rise to a comparison isomorphism between the Tate module of a p-divisible group and its associated covariant Dieudonn\'e module.
\end{itemize}

\subsection*{Chapter 2} In this chapter, we begin by introducing 1-motives over a general base scheme $S$ and $S$-motivic points. We also define the category $\gMi(K)$ of 1-motives with good reduction at $p$ over a field $K$. For any $M\in\gMi(K)$, we consider three realization functors: the de Rham realization $\TdR(M)$, the Tate module $\Tp(M)$, and the crystalline realization $\Tcrys(\bar{M})$ of the reduction $\bar{M}$ modulo $p$. We compute the Hodge-Tate weights of the Tate module of $M$ and review the crystalline-de Rham comparison isomorphism as described in \cite{andreatta2005crystalline}.

\subsection*{Chapter 3} In this chapter, we construct the p-adic integration pairing for 1-motives with good reductions. By interpreting integrals as Riemann sums, we are led to P. Colmez’s approach (\cite{colmez1992periodes}) to p-adic integration (in the sense of Fontaine-Messing) for abelian varieties with good reduction over a p-adic local field $K$. The resulting periods from this integration relate the Tate module to the first de Rham cohomology of the abelian variety. These periods do not live in $\C_p$, but in $\Bt:=\frac{\BdRp}{t^2\BdRp}$, where $t$ is a uniformizer of the de Rham local ring $\BdRp$. We extend Colmez's construction of p-adic integration for abelian varieties with good reduction to 1-motives with good reduction. We introduce Fontaine's pairing for 1-motives and provide a detailed comparison with both the p-adic integration pairing for 1-motives and the crystalline integration pairing for 1-motives. Furthermore, we show that our p-adic integration pairing for 1-motives with good reduction is perfect and respects the Hodge filtration.

In the remainder of this chapter, we aim to recover the logarithm of semi-abelian varieties through the logarithm of their associated p-divisible groups. Inspired by \cite{cartier_l-functions_2007}, we obtain a local inverse for this logarithm and further explore its image using Galois cohomology.

\subsection*{Chapter 4} 
In this chapter, we first develop a formalism of periods that enables us to define the concept of periods and period conjectures relative to a pairing within an abelian category equipped with fibre functors. Next, we leverage the results from Chapter 3, to introduce our three $\Q$-structures: $\hp(M)$, $\Hp(M)$, and $\Hpp(M)$. These structures are associated with the pairings $\displaystyle\int^{\hp}$, $\displaystyle\int^{\Hp}$, and $\displaystyle\int^{\Hpp}$, yielding specific Fontaine-Messing p-adic periods of $M$, which we refer as $\hp$-periods, $\Hp$-periods, and $\Hpp$-periods of $M$, respectively. Finally, we identify all possible relations among these periods by proving the period conjectures relative to these pairings at depths 1, 2 and 2, respectively. We conclude Chapter 4 with several illustrative examples.

\subsection*{Appendices} This dissertation also includes Appendix A and Appendix B, which provide essential definitions and key properties from Algebraic Geometry and Number Theory that are relevant to the thesis.

\section*{Conventions}
\begin{itemize}
    \item We fix a prime number $p$.
    \item Unless otherwise specified, $K$ is typically a non-archimedean local field of mixed characteristic $(0,p)$ with perfect residue field $k$ and valuation ring $\cO_K$. The algebraic closure of $K$ is denoted by $\bar{K}$, and $\C_p$ is the completion of $\bar{K}$. Let us denote by $\Gamma_K$ the absolute Galois group of $K$.
    \item We denote the maximal unramified extension of $K$ by $K^{ur}$, with its residue field denoted by $\bar{\F_p}$.
    \item $\K$ denotes a number field, and $\bar{\K}=\barQ$ is its algebraic closure. We denote by $\Gamma_{\K}$ the absolute Galois group of $\K$.
    \item The scheme $S$ is always a locally Noetherian, integral, regular scheme.
    \item By $S$-group schemes, we always mean the commutative ones.
    \item If $\cA$ is one of the category of abelian varieties, connected commutative groups, or 1-motives over $K$, we always consider it as the isogeny category $\cA\otimes\Q$, which has the same objects as $\cA$, but morphisms are tensored with $\Q$. For the categories that we are interested in, the isogeny category is always an abelian category.
\end{itemize}

\clearpage


\pagenumbering{arabic}
\chapter{Comparison theorems for Barsotti-Tate groups}
Barsotti-Tate groups, also known as p-divisible groups, are fundamental structures in the intersection of algebraic geometry and number theory. They capture rich arithmetic and geometric information, particularly in characteristic $p$, playing a crucial role in understanding the properties of abelian varieties.

Comparison isomorphisms of p-divisible groups are pivotal in the theory of p-divisible groups as they bridge different cohomological frameworks, connecting Dieudonné theory, crystalline cohomology, and \'etale and de Rham cohomology. These isomorphisms are essential for understanding the local behavior of p-divisible groups and their deformation properties.

The study of Barsotti-Tate groups is central to advanced topics in arithmetic geometry such as the Mordell-Weil theorem over function fields, the formulation of Fontaine-Messing theory, deformation theory and the study of p-adic Hodge theory.

In this chapter, we review the theory of Barsotti-Tate groups, and Dieudonn\'e theory, and we conclude the chapter by outlining the fundamental concepts of p-adic Hodge theory. We mainly use the following references:\\
\begin{itemize}
    \item \textbf{For the theory of Barsotti-Tate groups:} \cite{demazure2006lectures}, \cite{tate1967pdivisble}, \cite{grothendieck1974groupes}, \cite{stix2009course}, \cite{shatz1986group}, \cite{messing_crystals_1972}
    \item \textbf{For the theory of Dieudonn\'e modules:} \cite{manin1961theory}, \cite{messing_crystals_1972}, \cite{mazur2006universal}, \cite{demazure2006lectures}, \cite{fontaine1982formes}, \cite{Dieudonne_1958}, \cite{AST_1977__47-48__1_0}
    \item \textbf{For p-adic Hodge Theory:} \cite{fontaine1994corps}, \cite{brinon2009cmi}, \cite{szamuely2016p-adicHodgeAccordingtoBeilinson}, \cite{fontaine2008theory}, \cite{bhatt_integral_2018}.
\end{itemize}

Let $S$ be a scheme and recall the notations of $S\fppf$, and $S\etale$ for the (small) fppf site\footnote{fppf stands for faithfully flat and locally of finite presentation. This site over $S$ consists of fppf schemes over $S$ with fppf coverings.} on $S$ and the (small) \'etale site on $S$ (see \cref{Appendix: def fppf and etale}). By Yoneda's lemma, we have a fully faithful embedding $G\mapsto \underline{G}:=\Hom_S(.,G)$ from the category of commutative group schemes over $S$ to the category of abelian representable sheaves on the $S\fppf$, the fppf site of $S$. We always view $G$ as an fppf representable sheaf on $S\fppf$. The category of representable fppf sheaves is an abelian category with enough injectives. Given any two $S$-group schemes $G$ and $H$, we use the notation $\iext^i_S(G,H)$ to denote the $i$-th Yoneda extension group of $G$ by $H$ in the abelian category of fppf representable sheaves over $S$. We write $\ihom_S(G,H)$ as the group of $S$-homomorphisms in $S\fppf$. If $S=\Spec(K)$, and both $G$ and $H$ are algebraic varieties over $K$\footnote{An algebraic variety over $K$ is a reduced, separated scheme of finite type over $K$.}, we write the $i$-th Yoneda extension group of $G$ by $H$ as $\Ext_K(G,H)$, and the group of $K$-homomorphisms from $G$ to $H$ as $\Hom_K(G,H)$.

We have the following commutative diagram 
\[
\begin{tikzcd}
\text{CommGrp/S} \arrow[r, hook] \arrow[d, hook]   & \text{AbSh}(S\fppf) \arrow[d, hook] \\
C^b(\text{CommGrp/S)} \arrow[r, hook]                     & C^b(\text{AbSh}(S\fppf))           
\end{tikzcd}
\]
where $C^b(\cA)$ is the category of bounded complexes for any abelian category $\cA$. The fully faithful embedding $\cA\into C^b(\cA)$ associates to an object $A$ in $\cA$ the complex \[\dots\to 0\to A\to 0\to\cdots .\]
Hence, we can, and will, identify a commutative group scheme $G$ over $S$ with the element in the derived category of abelian fppf sheaves representing $\underline{G}[0]$.

The basic facts about commutative group schemes and finite flat group schemes are summarized  in \cref{Appendix: groups}.

\section{P-divisible groups (Barsotti-Tate groups)}\label{sec: p-dvisible groups}

Recall the \cref{def: finite flat group scheme} of finite flat group schemes.
\begin{defn}
Let $p$ be a prime, $h\geq 0$ an integer, and $S$ a scheme.
\begin{enumerate}
    \item A p-divisible group (or a Barsotti-Tate group) of height $h$ over $S$ is a direct system $G=\{G_n\}_{n\in \N}$ of finite flat commutative group schemes over $S$ such that each $G_n$ is of order $p^{nh}$ and we have the exact sequence
    $$0\to G_n\xrightarrow{i_{n}}G_{n+1}\xrightarrow{[p^n]}G_{n+1} $$
    for each $n$. We often write $G_n[p^m]:=\ker([p^m]_{G_n})$.
    \item Let $G$ and $H$ be p-divisible groups over $S$. A homomorphism from $G$ to $H$ is a system $\{f_n\colon G_n\to H_n\}$ of group scheme homomorphisms that are compatible with the transition maps.
    \item Let $f\colon G\to H$ be a homomorphism of p-divisible groups. We define the kernel of $f$ to be $\ker f=\colim\ker f_n$.
    \item For a p-divisible group $G=\colim G_n$ over an affine scheme $S=\Spec R$ of height $h$, we can define the Cartier dual of $G$ to be $G\ve:=\colim G_n\ve$ with transition maps $[p]\ve\colon G_n\ve\to G_{n+1}\ve$. The group scheme $G\ve$ is a p-divisible group of height $h$.
\end{enumerate}
\end{defn}
\begin{remark}
    \begin{enumerate}
        \item Let $R$ be a henselian local ring with residue field $k$. Recall the connected-\'etale exact sequence for finite flat group schemes (\cref{connected-etale exact sequence for finite flat}). Assume that $G=\colim G_n$ is a p-divisible group over $R$, $G^0_n$ denotes the connected component of $G_n$, and $G^{\'et}_n=G_n/G^0_n$ its \'etale part. Then $G^0:=\colim G^0_n$ and $G^{\'et}:=\colim G^{\'et}_n$ are p-divisible groups over $R$ so that we have an exact sequence
        $$0\to G^0\to G\to G^{\'et}\to 0$$
        This is true since the functors $G_n\mapsto G_n^{0}$ and $G_n\mapsto G^{\'et}_n$ are exact. We say that $G$ is connected (resp. \'etale) if each $G_{n}$ is connected (resp. \'etale).
        \item Let $G=\colim G_n$ be a p-divisible group over scheme $S$ in characteristic $p$ of height $h$. Recall the Frobenius twist $G^{(p)}_n$ (\cref{Frobenius twist}). The Frobenius twist of $G$ is $G^{(p)}:=\colim G^{(p)}_n$ where the transition maps are induced by the transition maps for $G$. The group scheme $G^{(p)}$ is a p-divisible group over $S$ of height $h$. We define the relative Frobenius of $G$ by $F_G:=(F_{G_n})$ and the Verschiebung of $G$ by $V_G:=(V_{G_n})$.
    \end{enumerate}
\end{remark}

Constant group schemes can be viewed as abstract groups. When $M$ is an abstract group, its associated constant group scheme is denoted by $\underline{M}_S$. See \cref{examples group schemes in appendix}(5), for details.
\begin{example}
\begin{enumerate}
    \item \textbf{The constant p-divisible group.} The group scheme $\underline{\Q_p/\Z_p}=\colim \underline{\Z/p^n\Z}$ with the natural inclusions is an \'etale p-divisible group of height $1$.
    \item \textbf{The p-power roots of unity.} The group scheme $\mu_{p^{\infty}}:=\colim\mu_{p^n}$ with natural inclusions is a connected p-divisible group of height $1$. Indeed, $\mu_{p^{\infty}}=\colim \G_m[p^n]:=\G_m[p^{\infty}]$.
    \item Given an abelian scheme $A$ over $S$, we can define its p-divisible group (or Barsotti-Tate group) by $A[p^{\infty}]=\colim A[p^n]$ with the natural inclusions. The height of $A[p^{\infty}]$ is $2g$ where $g$ is the dimension of $A$.
    \item We have $(\underline{\Q_p/\Z_p})\ve\cong \mu_{p^{\infty}}$, and, by the duality theorem for abelian schemes (\cref{duality theorem for abelian schemes}), $A[p^{\infty}]\ve\cong A\ve[p^{\infty}]$. 
\end{enumerate}
\end{example}
\begin{prop}
Let $G=\colim G_n$ be a p-divisible group over $S$.
\begin{enumerate}
    \item $[p]:G\to G$ is an epimorphism (surjective as fppf sheaves) and $\ker([p^n]_G)=G_n$.
    \item $G\ve=\colim \Hom(G_n,\mu_{p^n})$, where $\mu_{p^n}$ is $\G_m[p^n]$.
\end{enumerate}
\end{prop}
\begin{example}\label{example: p-divisible group elliptic curve connected-\'etale}
    Let $E$ be an elliptic curve over $\bar{\F}_p$. The finite flat group scheme $E[p]$ admits a connected-\'etale exact sequence 
    \begin{equation}
        0\to E[p]^0\to E[p]\to E[p]^{\'et}\to 0
    \end{equation}
which is split over $\bar{\F}_p$. The \'etale part $E[p]^{\'et}$ is of order $1$ or $p$, since $E[p]^{\'et}(\bar{\F}_p)\cong E[p](\bar{\F}_p)$. Therefore, it is of order $1$ when $E$ is supersingular and it is of order $p$ when $E$ is ordinary.

Assume that $E$ is ordinary. We obtain
\[
\mu_p\cong(\underline{\Z/p\Z})\ve\into E[p]\ve\cong E\ve[p]\cong E[p]
\]
where,
\begin{itemize}
    \item The embedding is obtained by taking duality from $E[p]\twoheadrightarrow E[p]^{\'et}$;
    \item The first isomorphism is obtained by the duality theorem for abelian schemes;
    \item The last isomorphism is due to the fact that $E$ is self-dual.
\end{itemize}
We obtain an embedding $\mu_p\into E[p]^0$ which must be an isomorphism because $\mu_p$ is connected and its order is equal to $E[p]$. Hence, we have the following splitting exact sequence
\begin{equation}
    0\to \mu_p\to E[p]\to \underline{\Z/p\Z}\to 0.
\end{equation}
As $E[p^{\infty}]$ has height $2$, we can say  that both the connected and the \'etale parts of $E[p^{\infty}]$ have height $1$, i.e., they are $\mup$ and $\underline{\Q_p/\Z_p}$ respectively. Hence, the connected-\'etale exact sequence of $E[p^{\infty}]$ is
\begin{equation}
    0\to \mu_{p^{\infty}}\to E[p^{\infty}]\to \underline{\Q_p/\Z_p}\to 0.
\end{equation}
\end{example}

\begin{defn}
A quasi-isogeny of $p$-divisible groups $G$ and $H$ is a global section $\rho$ of the Zariski sheaf $\ihom_S(G,H)\otimes_{\Z}\Q$ such that locally there exists an integer $n$ for which $p^n\rho$ is an isogeny.
\end{defn}
\begin{thm}[Rigidity theorem for p-divisible groups, \cite{drinfel1976coverings}]
Let $\cO$ be a complete discrete valued ring of mixed characteristic $(0,p)$, and let $S$ be a scheme over $\cO$ on which $p$ is locally nilpotent. Suppose $S_0\to S$ is a closed immersion whose ideal of definition is locally nilpotent. Then every homomorphism $\bar{\varphi}\colon G\times_S S_0 \to H\times_S S_0$ of $p$-divisible groups admits a unique lifting $\varphi:G\to H$, i.e., the special fibre functor $G\mapsto G\times_S S_0$ is fully faithful. Moreover, $\varphi$ is a quasi-isogeny if and only if $\bar{\varphi}$ is.
\end{thm}

\begin{defn}
Let $G=\colim G_n$ be a p-divisible group over field $K$. We define the Tate module of $G$ by 
$$\Tp(G)=\lim G_n(\bar{K}),$$
where the transition maps are induced by $[p]\colon G_{n+1}\to G_n$. We denote the Tate module of $\mu_{p^{\infty}}$ by $\Z_p(1)$. The Tate comodule of $G$ is 
$$\Phi_p(G)=\colim G_n(\bar{K}).$$
The Tate comodule is $G(\bar{K})$, where $G$ is considered as an fppf sheaf. 
\end{defn}
\begin{prop}\cite{stix2009course}
Assume that $G$ is a p-divisible group over a field $K$ of height $h$. We have the following:
\begin{enumerate}
    \item $\Tp(G)$ is a free $\Z_p$-module of rank $h$.
    \item $\Tp (G)\cong\Hom_{\Z_p}(\Q_p/\Z_p,G(\bar{K}))$.
    \item $\Phi_p(G)\cong \Tp(G)\otimes_{\Z_p}\Q_p/\Z_p$.
    \item $\Tp(G)\cong \Hom_{\Z_p}(\Tp(G\ve),\Z_p(1))$.
    \item $\Phi_p(G)\cong\Hom_{\Z_p}(\Tp(G\ve),\mu_{p^{\infty}}(\bar{K}))$
\end{enumerate}
\end{prop}
\section{Formal schemes and formal Lie groups}\label{sec: formal shcemes and formal lie groups}
Let $A$ be an admissible topological ring and $I_{\lambda}$ a fundamental system of ideals of definition for ring $A$ (see \cite[\href{https://stacks.math.columbia.edu/tag/07E8}{Definition 07E8}]{stacks-project}). The affine formal scheme, denoted $\sX=\Spf(A)$, is the subspace of $\Spec A$ consisting of all open prime ideals of $A$ endowed with the structure sheaf $\cO_{\sX}=\lim\cO_{\Spec(A/I_{\lambda})}$. Every morphism between two affine formal schemes corresponds to a continuous homomorphism of topological rings. If $A$ is a ring with discrete topology, then $\Spf(A)$ is equal to $\Spec(A)$ (see \cite[\href{https://stacks.math.columbia.edu/tag/0AHY}{Section 0AHY}]{stacks-project} for more details).

Let $\sX=\Spf(A)$ be an affine formal scheme and $h_{\sX}$ be its functor of points. Then $h_{\sX}=\colim h_{\Spec(A/I)}$ where the colimit is over the collection of ideals of definition of the admissible topological ring $A$. Indeed, $h_{\sX}$ is a (classical) affine formal algebraic space. In fact, the category of affine formal schemes is equivalent to the category of classical affine formal algebraic spaces (\cite[\href{https://stacks.math.columbia.edu/tag/0AI6}{Section 0AI6}]{stacks-project}).

Let $S$ be a scheme, $X$ an affine scheme over $S$, and $T$ a closed subset of $X$. The functor 
$$(\Sch/S)_{\text{fppf}}\to \text{Sets},\, U\mapsto \{f:U\to X\mid f(U)\subseteq T\}$$
is a formal algebraic space and it is called the formal completion of $X$ along $T$ which we denote by $\hat{X}_{T}$, and $T=\mathbb{V}(I)$ for some finitely generated ideal $I\subset A$. Then $\hat{X}_{T}=\Spf(\hat{A})$ where $\hat{A}$ is the $I$-adic completion of $A$ (\cite[\href{https://stacks.math.columbia.edu/tag/0AIX}{Section 0AIX}]{stacks-project}). In fact, the category of smooth affine formal schemes over a noetherian complete local ring $R$ is equivalent to the category of complete noetherian local $R$-algebras.

\begin{remark}[Formal completion of an abelian sheaf along its zero section]
    A closed immersion $T\into U$ is called a nil-immersion of order $\leq k$ if the ideal $\sI$ defining it verifies $\sI^{k+1}=0$.

    Let $\cF$ be an abelian fppf sheaf on $S$. For any positive integer $k$ we denote by $\text{Inf}^k\cF$  the sheaf associated to the sub-presheaf of $\cF$
\[
U\mapsto \{s\in \cF(U)\st \text{there exists a nil--immersion }T\into U \text{ of order }\leq k \text{ with  } s\mid_T=0\}.\]
We denote $\hat{\cF}:=\colim \text{Inf}^k\cF$ and it is called the formal completion of $\cF$ along its zero section.
\end{remark}
A formal group over $S$ is a group object in the category of formal schemes over $S$.
\begin{example}\label{formal completion of G along its zero section}
If $G$ is a group scheme over $S$ with augmentation ideal $\sI$, then $\text{Inf}^kG=\Spec(\cO_G/\sI^{k+1})$ and $\hat{G}=\colim \Spec(\cO_G/\sI^{k+1})$ is the completion of $G$ along its identity. If $G$ is a formal group scheme, then $\hat{G}=G$. 
\end{example}

Let $S$ be a scheme and $\sX=\colim\Spec(A/I_{\lambda}), \sY=\colim\Spec(B/J_{\mu})$ two affine formal schemes over $S$ corresponding to $A$ and $B$ respectively. The fibre product $X\times_SY$ is also an affine formal scheme corresponding to $A\hat{\otimes}_CB$, where $\hat{\otimes}$ is the completed tensor product (\cite[\href{https://stacks.math.columbia.edu/tag/0AN3}{Lemma 0AN3}]{stacks-project}).

Assume that $R$ is a complete noetherian local ring with residue field $k$ of characteristic $p$. A formal group (or formal group scheme) over $R$ is a group object in the category of formal schemes over $R$. A connected smooth formal group over $R$ is called a formal Lie group. For a formal Lie group $\sG=\Spf(\sA)$ over $R$ we have an isomorphism $A\cong R[[t_1,\dots,t_d]]$ of profinite $R$-algebras. The number $d$ is called the dimension of the formal Lie group $\sA$ which is uniquely defined as the $R$-rank of the tangent space $\sI/\sI^2$, where $\sI=(t_1,\dots,t_d)$ is the augmentation ideal of $\sA$.

The group structure on $\Spf(\sA)$ is given by a continuous ring homomorphism $$\mu\colon\sA\to\sA\hat{\otimes}\sA=R[[t_1,\dots,t_d]]\hat{\otimes}R[[u_1,\dots,u_d]]=R[[t_1,\dots,t_d,u_1,\dots,u_d]]$$ such that the power series $\Phi_i(T,U):=\mu(t_i)$ form a family $\Phi(T,U):=(\Phi_i(T,U))$ that satisfies the following axioms:
\begin{enumerate}
    \item associativity: $\Phi(T,\Phi(U,V))=\Phi(\Phi(T,U),V)$
    \item unit: $\Phi(T,0)=T=\Phi(0,T)$
    \item If the formal Lie group is commutative, then we also have $\Phi(T,U)=\Phi(U,T)$.
\end{enumerate}
Let $[p]^*\colon\sA\to\sA$ be the map induced by multiplication by $p$ on the formal Lie group $\sG=\Spf(\sA)$. We say that $\sG$ (or $\sA$) is a p-divisible formal Lie group if $\sG$ is commutative and $[p]^*$ is finite flat.

\begin{example}
The formal Lie group $\hat{\G}_a$ and $\hat{\G}_m$ are given by $R[[t_1]]$ with comultiplications $t_1\mapsto t_1+u_1 $ and $t_1+u_1+t_1u_1$ respectively.
\end{example}

\begin{prop}[Serre-Tate]\cite{demazure2006lectures}\label{prop 1.2.1}
Let $R$ be a complete noetherian local with residue field $k$ of characteristic $p$.
\begin{enumerate}
    \item Let $G=\colim G_n$ be a connected p-divisible group over $R$ with $G_n=\Spec(A_n)$ for each $n$. We have a continuous isomorphism $$\lim (A_n\otimes_R k)\cong k[[t_1,\dots,t_d]]$$
    for some positive integer $d$. Also, this isomorphism can be lifted to the continuous isomorphism $R[[t_1,\dots,t_d]]\cong \lim A_n$.
    \item Let $\sG=\Spf(\sA)$ be a p-divisible formal Lie group over $R$ with augmentation ideal $\sI$. Define $A_n:=\sA/[p^n]^*(\sI)$ and $G_n:=\Spec(A_n)$ for each $n$. Then, each $G_n$ is a finite flat group scheme and $\sG[p^{\infty}]:=\colim G_n$ is a connected p-divisible group scheme over $R$.
    \item Let $\sA$ be a p-divisible formal Lie group over $R$ with the augmentation ideal $\sI$ and let $A_n:=\sA/[p^n]^*(\sI)$. Then we have a natural continuous isomorphism $\sA\cong\lim A_n$.
    \item There is an equivalence of categories
    $$\Big\{\text{p-divisible formal Lie groups over }R \Big\} \xrightarrow{\cong}\Big\{\text{connected p-divisible groups over }R\Big\}$$
\end{enumerate}
\end{prop}
\begin{prop}
Let $G$ be a p-divisible group of height $h$ over $R$. Let $d$ and $d^{\vee}$ denote the dimensions of $G$ and $G^{\vee}$ respectively. Then $h=d+d\ve$.
\end{prop}
\begin{proof}
    \cite[Theorem 72]{stix2009course}
\end{proof}
\begin{cor}\label{we have only two kind p-divisible group over algebraic closed charac p}
Let $k$ be an algebraic closed field of characteristic $p$. Every p-divisible group of height $h=1$ over $k$ is isomorphic to either $\underline{\Q_p/\Z_p}$ or $\mu_{p^{\infty}}$.
\end{cor}
\begin{proof}
As $G$ is of height 1, $G$ is either \'etale or connected. If $G$ is \'etale, then each $G_n=G_{n+1}[p^n]$ is finite \'etale. The group $G_n(\bar{k})$ is the $p^n$-torsion subgroup of $G_{n+1}(\bar{k})$ of order $p^n$. By induction, we can show that $G_n(\bar{k})=\Z/p^n\Z$ which implies that $G_n\cong \underline{\Z/p^n\Z}$. Because the functor $T\mapsto T(\bar{k})$ is an equivalence from the category of finite \'etale groups to the category of finite $\Gamma_k$-modules.

Now assume that $G$ is connected. As $G$ has dimension $1$, the above proposition implies that $G\ve$ has dimension $0$, and therefore it is \'etale. As we showed in the previous part, $G\ve=\underline{\Q_p/\Z_p}$ which implies that $G\cong \mup$.
\end{proof}
\section{Formal points on p-divisible groups}
For the rest of the section, we fix the base $R=\cO_K$ where $K$ is a local field extension of $\Q_p$ with the prefect residue field $k$ of characteristic $p$. We also let $L$ be the p-adic completion of an algebraic extension of $K$, and denote by $\fm_L$ and $k_L$ its maximal ideal and its residue field respectively. The completion of the algebraic closure of $K$ is denoted by $\C_p:=\hat{\bar{K}}$. Let us denote by $\Gamma_K$ the absolute Galois group of $K$.
\begin{defn}
Let $G=\colim G_n$ be a p-divisible group over $\cO_K$. We define the group of $\cO_L$-valued formal points on $G$ by 
$$G(\cO_L):=\lim_i G(\cO_L/\fm^i_L)=\lim_i\colim_n G_n(\cO_L/\fm^i_L).$$
\end{defn}
\begin{example}
By the definition, $\mu_{p^{\infty}}(\cO_L)=\lim_i\mu_{p^{\infty}}(\cO_L/\fm^i_L)=1+\fm_L$. Because the group $\lim_i\mu_{p^{\infty}}(\cO_L/\fm^i_L)$ contains all elements $x$ in $\cO_L^{\times}$ such that $\nu(x^{p^n}-1)$ can be arbitrary large. Moreover, we have $x^{p^n}-1=(x-1)^{p^n}$ mod $\fm_L$. Hence, $x\in 1+\fm_L$, when $x\in \mup(\cO_L)$.

On the other hand, the group of ordinary $\cO_L$-points on $\mu_{p^{\infty}}$ is given by $$\colim_n\mu_{p^n}(\cO_L)=\{\text{p-power torsion elements in }\cO^{\times}_L\}.$$
\end{example}
\begin{prop}
Let $G=\colim G_n$ be a p-divisible group over $\cO_K$ with $G_n=\Spec(A_n)$ for each $n$. Then
\begin{enumerate}
    \item We have a continuous isomorphism $G(\cO_L)\cong\Hom_{\cO_K-\text{cont}}(\sA,\cO_L)$, where $G(\cO_L)$ is the group of $\cO_L$-valued formal points equipped with its natural topology arising from that on $\cO_L$, $\sA=\lim_n A_n$, and the topology on the right side is the $\fm$-adic topology. When $G$ is connected, $\sA$ is a p-divisible formal Lie group, $\sA=\cO_K[[t_1,\dots,t_d]]$, and the continuous $R$-algebra homomorphisms are exactly the local $R$-algebra homomorphism. 
    
    \item Consider the natural $\Z_p$-module structure on $G(\cO_L)$ arising from that on $G(\cO_L/\fm^i_L)=\colim_n G_n(\cO_L/\fm^i_L)$. The torsion part of $G(\cO_L)$ is given by
    $$G(\cO_L)_{\tors}\cong \colim_n\lim_i G_n(\cO_L/\fm^i_L).$$

    \item If $G$ is \'etale, then $G(\cO_L)\cong G(k_L)$ where $k_L$ is the residue field of $\cO_L$.

    \item We have canonical isomorphisms $G_n(\bar{K})\cong G_n(\C_p)\cong G_n(\cO_{\C_p})$, and $G(\cO_{\C_p})^{\Gamma_K}=G(\cO_K)$.
    \item  We have an exact sequence 
    $$0\to G^0(\cO_L)\to G(\cO_L)\to G^{\'et}(\cO_L)\to 0.$$

    \item If $L$ is algebraically closed, then $G(\cO_L)$ is p-divisible i.e. the multiplication by $p$ on $G(\cO_L)$ is surjective.

\end{enumerate}
\end{prop}
\begin{proof}
    \begin{enumerate}
        \item We can make the following identifications by using the fact that $\cO_L$ is complete, and thus $A_n$ is complete for each $n$, since it is finite free over $\cO_K$ (see \href{https://stacks.math.columbia.edu/tag/031B}{Tag 031B}).  
        \begin{gather*}
        G(\cO_L)=\lim_i\colim_n G_n(\cO_L/\fm^i_L)=
        \lim_i\colim_n \Hom_{\cO_K}(A_n,\cO_L/\fm^i\cO_L)\\=\lim_i\colim_n \Hom_{\cO_K}(A_n/\fm^iA_n,\cO_L/\fm^i\cO_L)=\lim_i \Hom_{\cO_K}(\lim_{n}A_n/\fm^iA_n,\cO_L/\fm^i\cO_L)\\=\Hom_{\cO_K-cont}(\lim_{i,n}A_n/\fm^iA_n,\lim_i\cO_L/\fm^i\cO_L)\\=\Hom_{\cO_K-cont}(\lim_{n}A_n,\cO_L)=\Hom_{\cO_K-cont}(\sA,\cO_L).
        \end{gather*}
When $G$ is connected, it follows from \cref{prop 1.2.1}(1) that  $\sA\cong\cO_K[[t_1,\dots,t_d]]$.
        \item The group $\G(\cO_L)_{\tors}$ contains only $p$-torsions as $G(\cO_L/\fm^i\cO_L)=\colim_n G(\cO_L/\fm^i\cO_L)$ is a $\Z_p$-module. The exact sequence 
        \[
        0\to G_n(\cO_L/\fm^i\cO_L)\to G(\cO_L/\fm^i\cO_L)\xrightarrow{[p^n]}G(\cO_L/\fm^i\cO_L)\to 0
        \]
        yields an exact sequence by taking limit
        \[
        0\to \lim_i G_n(\cO_L/\fm^i\cO_L)\to G(\cO_L)\xrightarrow{[p^n]}G(\cO_L)\to0.
        \]
        Hence, the $p^n$-torsion points in $G(\cO_L)$ is $\lim_iG_n(\cO_L/\fm^i\cO_L)$. Thus, 
        \[
        G(\cO_L)_{\tors}=\colim_{n}\lim_i G_n(\cO_L/\fm^i\cO_L).
        \]
        \item Each $G_n$ is \'etale. Every \'etale group scheme is formally \'etale. By the infinitesimal lifting criterion of \'etaleness, we have $G_n=(\cO_L/\fm^i\cO_L)\cong G_n(\cO_L/\fm^{i+1}\cO_L)$. Thus
        \[
        G(\cO_L)=\lim_i\colim_n G_n(\cO_L/\fm^i\cO_L)\cong \lim_i\colim_nG_n(k_L)\cong G(k_L).
        \]
        \item By (1) we have $G(\ocp)^{\Gamma_K}=G(\cO_K)$. The identification $G_n(\bar{K})=G_n(\C_p)$ follows from the fact that the generic fibre of $G_n$ is \'etale. The isomorphism $G_n(\C_p)=G_n(\ocp)$ is a consequence of the valuation criterion of properness.
        \item \cite[Proposition 4]{tate1967pdivisble}.
        \item \cite[Corollary 1]{tate1967pdivisble}.
    \end{enumerate}
\end{proof}

Indeed, the identification $G(\cO_L)\cong\Hom_{\cO_K-\text{cont}}(\sA,\cO_L)$ states that for a connected p-divisible group $G$, the $\cO_L$-formal points on $G$ are exactly the $\cO_L$-points on the formal Lie group associated to $G$ from the Serre-Tate equivalence.

\section{Hodge-Tate decomposition}
Let $G$ be a p-divisible group over $\cO_K$ with the p-divisible formal Lie group $\sA^0=\cO_K[[t_1,\dots,t_d]]$ associated to $G^0$. Let $\sI$ be the augmentation ideal of $\sA^0$.
\begin{defn}
For $G$  as above:
\begin{enumerate}
    \item Given an $\cO_K$-module $M$, the tangent space of $G$ with values in $M$ is 
    $$t_G(M):=\Lie(G)(M):=\Hom_{\cO_k}(\sI/\sI^2,M),$$
    and the contangent space of $G$ with values in $M$ is
    $$t\ve_G(M):=\Lie\ve(G)(M):=\sI/\sI^2\otimes_{\cO_K}M.$$
    \item The valuation filtration of $G^0(\cO_L)$ is 
    $$\Fil^{\lambda}G^0(\cO_L):=\{f\in G^0(\cO_L):\nu(f(x))\geq\lambda,\, \forall x\in\sI\}$$
    for any real number $\lambda>0$, where we have the identification $G^0(\cO_L)\cong\Hom_{\cO_K-cont}(\sA^0,\cO_L)$.
\end{enumerate}
\end{defn}
\begin{lemma}
Let $f\in G(\cO_L)$ and $x\in\sI$. Then $\operatorname{lim}_{n\to\infty}\frac{(p^nf)(x)}{p^n}$ exists in $L$, and it is  zero if $x\in\sI^2$.
\end{lemma}
\begin{proof}
    \cite[11.3.2]{stix2009course}.
\end{proof}
The above lemma leads us to the following definition.
\begin{defn}\label{def logarithm p-divisible group}
Let $G$ be a p-divisible group over $\cO_K$ with the augmentation ideal $\sI$. We define the logarithm of $G$ to be the map 
$$\log_G\colon G(\cO_L)\to \Lie(G)(L)=t_G(L),\, \log_G(f)(x)=\operatorname{lim}_{n\to \infty}\frac{(p^nf)(\tilde{x})}{p^n}$$
such that $f\in G(\cO_L)$ and $x\in\sI/\sI^2$ with $\tilde{x}$ is any lift of $x$ to $\sI$.
\end{defn}
\begin{example}\label{logarith of p-divisible multiplicative}
When $G=\mu_{p^{\infty}}$ we have 
$$\mu_{p^{\infty}}(\cO_L)=\Hom_{\cO_K-cont}(\cO_L[[t]],\cO_L)\cong\fm_L\cong 1+\fm_L ,$$ where the last isomorphism is given by $x\mapsto x+1$. Let $\sI:=(t)$ be the augmentation ideal of the formal p-divisible group $\mup$, we get
$$t_{\mu_{p^{\infty}}}(L)=\Hom_{\cO_K}(\sI/\sI^2,L)\cong L.$$
Thus, we have the commutative diagram 
\begin{equation*}
\begin{tikzcd}
\mup(\cO_L) \arrow[r, "\log_{\mup}"] \arrow[d, "\cong"] & t_{\mup}(L) \arrow[d, "\cong"] \\
1+\fm_L \arrow[r]                                       & L                             
\end{tikzcd}
\end{equation*}
where the left vertical arrow and right vertical arrow are given by $f\mapsto 1+f(t)$ and $g\mapsto g(t)$ respectively. Let $x\in\fm_L$ and $f\in \mup(\cO_L)$. By considering the fact that the group law on formal p-divisible group $\Spf(\cO_K[[t]])$ is induced by $\cO_K[[t]]\to\cO_K[[t,t']],\, t\mapsto (1+t)(1+t')-1$, we can write
\[
(p^nf)(t)=f([p^n]_{\mup}(t))=f((1+t)^{p^n}-1)=(1+f(t))^{p^n}-1
\]
and therefore 
\begin{equation}\label{logarithm series 1 for mup Gm}
\log_{\mup}(1+x)=\colim_{n\to\infty}\frac{(1+x)^{p^n}-1}{p^n}=\colim_{n\to\infty}\sum^{p^n}_{i=1}\frac{1}{p^n}\binom{p^n}{i}x^i.
\end{equation}
For each $n$ and $i\leq p^n$,  we obtain 
\[
\frac{1}{p^n}\binom{p^n}{i}x^i-\frac{(-1)^{i-1}x^i}{i}=\frac{(p^n-1)\cdots (p^n-i+1)-(-1)^{i-1}(i-1)!}{i!}x^i.
\]
As the numerator of the right side is divisible by $p^n$, we have
\[
\nu\Big(\frac{1}{p^n}\binom{p^n}{i}x^i-\frac{(-1)^{i-1}x^i}{i}\Big)\geq n+i\nu(x)-\nu(i!)\geq n+i\nu(x)-\frac{i}{p-1}.
\]
This means that both series $\sum^{\infty}_{i=1}\frac{(-1)^{i-1}}{i}x^i$ and the series in the right side of \cref{logarithm series 1 for mup Gm} p-adically converge to the same number. Thus,
\[
\log_{\mup}(1+x)=\sum^{\infty}_{i=1}\frac{(-1)^{i-1}x^i}{i},
\]
which coincides with the usual p-adic logarithm $\log_p:1+\fm\to L$.
\end{example}
\begin{prop}
Let $G$ be a p-divisible group over $\cO_K$ with the augmentation ideal $\sI$. Then
\begin{enumerate}
    \item $\log_G$ is a $\Z_p$-homomorphism.
    \item $\log_G$ induces an isomorphism
    $$\Fil^{\lambda}G^0(\cO_L)\xrightarrow{\cong}\{\tau\in t_G(L)\colon \nu(\tau(x))\geq \lambda, \forall x\in\sI/\sI^2\}.$$
    \item $\ker(\log_G)=G(\cO_L)_{\tors}$ and $\log_G$ induces an isomorphism $G(\cO_L)\otimes_{\Z_p}\Q_p\cong t_G(L)$.
    \item We have an exact sequence 
    $$0\to \Phi_p(G)\to G(\cO_{\C_p})\xrightarrow{\log_G}\Lie G(\C_p)\to 0.$$
\end{enumerate}
\end{prop}
\begin{proof}
    \begin{enumerate}
        \item \cite[11.3.2]{stix2009course}.
        \item \cite[Lemma 86]{stix2009course}.
        \item \cite[Crollary 87]{stix2009course}.
        \item It is a direct consequence of (3). 
    \end{enumerate}
\end{proof}
\begin{defn}
Let $G=\colim G_n$ be a p-divisible group over $\cO_K$. We define the Tate-module and Tate-comodule of $G$ to be $\Tp(G):=\lim G_n(\bar{K})$ and $\Phi_p(G):=\colim G_n(\bar{K})$. We also denote $\Vp(G):=\Tp(G)\otimes_{\Z_p}\Q_p$.
\end{defn}
\begin{thm}[Hodge-Tate decomposition for p-divisible groups]\label{Theorem: Hodge-Tate decomposition}
We have a commutative diagram
\begin{equation*}
\begin{tikzcd}[column sep=tiny]
0 \arrow[r] & \Phi_p(G) \arrow[r] \arrow[d, "\cong"]                      & G(\cO_{\C_p}) \arrow[r, "\log_G"] \arrow[d, "\alpha"] & \Lie G(\C_p) \arrow[d, "d\alpha"] \arrow[r] & 0 \\
0 \arrow[r] & {\Hom_{\Z_p}(\Tp(G\ve),\mu_{p^{\infty}}(\bar{K}))} \arrow[r] & {\Hom_{\Z_p}(\Tp(G\ve),1+\fm_{\C_p})} \arrow[r]  & {\Hom_{\Z_p}(\Tp(G\ve),\C_p)} \arrow[r]  & 0
\end{tikzcd}    
\end{equation*}
where $\alpha$ and $d\alpha$ are $\Gamma_K$-equivariant and injective. Their restrictions to $\Gamma_K$-invariant elements yields the isomorphisms
\begin{align*}
G(\cO_K)\to\Hom_{\Z_p[\Gamma_K]}(\Tp(G\ve),\mu_{p^{\infty}}(\cO_{\C_p})),\\
t_G(K)\to \Hom_{\Z_p[\Gamma_K]}(\Tp(G\ve),\C_p).   
\end{align*}
Moreover, there is a canonical $\C_p[\Gamma_K]$-isomorphism
$$\Hom_{\Z_p}(\Tp(G),\C_p)\cong \Lie\, G\ve(\C_p)\oplus \coLie\,G(\C_p)(-1)$$
which is called the Hodge-Tate decomposition for $G$.
\end{thm}
\begin{proof}
A complete proof can be found in \cite[\S 4]{tate1967pdivisble}.
\end{proof}
Let us construct the maps $\alpha$ and $d\alpha$ as we will use them later. We write
\begin{align*}
    \Tp(G\ve)=\lim G\ve_n(\bar{K})=\lim G\ve_n(\ocp)=\lim\Hom((G_n)_{\ocp},(\mu_{p^n})_{\ocp})
\end{align*}
\begin{equation}\label{1.4.2}
    =\Hom(G\times_{\cO_K}\ocp,(\mup)_{\ocp}).
\end{equation}
Let $u\in \Tp(G\ve)$. Under the above identification of (\ref{1.4.2}) and functoriality on points, $u$ corresponds to the map $u(\ocp):G(\ocp)\to \mup(\ocp)=1+\fm_{\C_p}$. We set 
\[
\alpha(g)(u):=u(\ocp)(g),\, \text{for any }g\in G(\ocp) \text{ and } u\in \Tp(G\ve).
\]
We also define $d\alpha: t_G(\C_p)\to \Hom_{\Z_p}(\Tp(G\ve),\C_p)$ by setting
\[
d\alpha(v)(u):=du_{\C_p}(v),\, \text{for any }v\in t_G(\C_p) \text{ and } u\in \Tp(G\ve),
\]
where $du_{\C_p}:t_{G}(\C_p)\to t_{\mup}(\C_p)\cong\C_p$ is the map corresponding to $u$ by functoriality on tangent spaces and under the identification \ref{1.4.2}.

The bottom row of \cref{Theorem: Hodge-Tate decomposition} is obtained by applying $\Hom_{\Z_p}(\Tp(G\ve), \, .)$ on the exact sequence
\[
0\to \mup(\bar{K})\to \mup(\ocp)\to t_{\mup}(\C_p)\to 0.
\]

\subsection{Generic fibre functor}
\textbf{References:} \cite{hainesnotes}, \cite{tate1967pdivisble}\\
Again, in this section, we consider the category of p-divisible groups over $\cO_K$, where $K$ is a local discrete-valued field with prefect residue field $k$.

Recall the definition of the discriminant ideal:
\begin{defn}
    Let $D$ be a Dedekind domain with field of fraction $K$. Let $L/K$ be a field extension and $B$ the integral closure of $D$ in $L$. The discriminant ideal is the ideal in $D$ generated by all discriminants of the form 
    \[
    \Delta(x_1,\dots,x_n):=\det[\Tr(x_ix_j)],
    \]
    where $\{x_1,\dots,x_n\}$ is a basis of $L$ over $K$. 
    
    The discriminant ideal is a principal ideal, as $L/K$ is finite. 
\end{defn}
\begin{prop}
Let $f\colon G\to H$ be a homomorphism of $p$-divisible groups over $\cO_K$. If the map $f$ induces an isomorphism on the generic fibres $G\times_{\Spec\cO_K}\Spec K\to H\times_{\Spec\cO_K}\Spec K$, then $f$ is an isomorphism. 
\end{prop}
\begin{proof}
    Let $G=\colim G_n$ and $H=\colim H_n$. Assume that $G_n=\Spec(A_n)$ and $H_n=\Spec(B_n)$ for each $n$. Let $g_n:B_n\to A_n$ be the map induced by $f$. We want to show that $g_n$ is an isomorphism. As $B_n$ is flat over $\cO_K$ and $g_n\otimes 1_K:B_n\otimes_{\cO_K}K\to A_n\otimes_{\cO_K}K$ is an isomorphism, $g_n$ must be injective\footnote{Here, we actually use the fact that $B_n$ is faithfully flat. By \cite[\href{https://stacks.math.columbia.edu/tag/00HP}{Lemma 00HP}]{stacks-project}, $B_n$ is faithfully flat since $B_n/\fm_KB_n\neq 0$ due to Nakayama's Lemma}. Both $A_n$ and $B_n$ have the same discriminant ideal generated by $p^{ndp^{nh}}$, where $h$ is the height and $d$ is the dimension of $G$ (see \cite[Proposition 2]{tate1967pdivisble}). Thus, $A_n\cong B_n$.   
\end{proof}
\begin{prop}[\cite{tate1967pdivisble}]
Let $G$ be a p-divisible group over $\cO_K$, and let $M$ be a $\Z_p$-direct summand of $\Tp(G)$ that is stable under the action of $\Gamma_K$. There exists a p-divisible group $H$ over $\cO_K$ with a homomorphism $\iota\colon H\to G$ which induces an isomorphism $\Tp(H)\cong M$.
\end{prop}
\begin{thm}[\cite{tate1967pdivisble}]
The generic fibre functor $G\mapsto G\times_{\Spec\cO_K}\Spec K$ is fully faithful. That is, for any p-divisible groups $G$ and $H$ over $\cO_K$, the natural map $$\Hom(G,H)\to\Hom(G\times_{\Spec\cO_K}\Spec K,H\times_{\Spec\cO_K}\Spec K)$$
is bijective.
\end{thm}
\begin{cor}\label{theorem: fully faithfullness of tate module for p-divisible group}
The functor $G\mapsto\Tp(G)$ is fully faithful, i.e. for any $p$-divisible groups $G$ and $H$ over $\cO_K$, the natural map 
$$\Hom(G,H)\to \Hom_{\Z_p[\Gamma_K]}(\Tp(G),\Tp(H))$$
is bijective.
\end{cor}
\begin{proof}
By the above theorem, we know that the generic fibre functor is fully faithful from the category of p-divisible groups over $\cO_K$ to the category of p-divisible groups over $K$. As $K$ is a field of characteristic $0$, every p-divisible group over $K$ is \'etale. Therefore, the assignment $G_{K}\mapsto\Tp(G)$ is fully faithful. If $M$ is a finite free $\Z_p[\Gamma_K]$-module, $G_n:=M/p^nM$ gives a finite flat group scheme. Hence, the assignment $G_K\mapsto\Tp(G)$ is also essentially surjective. Thus, we have an equivalence of categories from the category of p-divisible groups over $K$ to the category of finite free $\Z_p[\Gamma_K]$-modules. Therefore, the functor $G_{\cO_K}\mapsto G_K \mapsto \Tp(G)$ is fully faithful.
\end{proof}
\begin{prop}[\cite{chai2013complex} Proposition 1.4.4.3]
The special fibre functor $G\mapsto G_k$ is faithful.
\end{prop}

\section{Dieudonn\'e functor}\label{sec: Dieudonne}
The \cref{theorem: fully faithfullness of tate module for p-divisible group}, along with its proof technique, signifies the true inception of the p-adic Hodge theory. Tate's proof of this result led to his discovery of the Hodge–Tate decomposition for abelian varieties 
$A$ over $K$ with good reduction. This subsequently prompted him to question whether a similar decomposition might apply universally to the p-adic étale cohomology of all smooth proper K-schemes. This leads us to the theory of Dieudonn\'e modules.

\subsection{Witt vectors}
Witt vectors were first proposed by \cite{zbMATH03025104}. For any associative, commutative ring $R$ with unity, the ring of Witt vectors over $R$ is denoted by $\W(R)$. For any natural number $n$, the ring $\W_n(R)$ represents the truncated Witt vectors of length $n$. The assignment $R\mapsto\W(R)$ defines a covariant functor from the category of commutative rings with unity into the category of rings. We can write any element of $\W(R)$ as a sequence $(a_0,a_1,\dots)$, where $a_i\in R$. The map $[.]:R\to\W(R), \, a\mapsto (a,0,0,\dots)$ is called the Teichm$\ddot{\text{u}}$ler lift. For the detailed construction we refer the readers to \cite[Chapter VI, page 330]{lang_algebra_2002}, \cite{rabinoff2014theory}, and \cite{encyclopediaWittvector}.

One of the motivations behind the definition of Witt vectors was to provide a systematic construction of the ring of p-adic integers. In particular, We have $\W(\F_p)=\Z_p$, and $\W_n(\F_p)$ represents the set of $(u_i)_{i\geq 0}\in\Z_p$ such that $u_k=0$ for all $k\geq n$, so $\W_n(\F_p)\cong\Z_p/p^n$.

Another motivation for the definition of Witt vectors was to develop a machinery to describe the unramified extensions of p-adic fields. The theory of Witt vectors enables the recovery of the maximal unramified subfield of $K$ by providing a canonical construction of local fields with residue fields. Let $K$ be a p-adic local field with prefect residue field $k$. There is a unique ring map $\W(k)\to\cO_K$ lifting the identification $\W(k)/p=k=\cO_K/\fm$. The map $\W(k)\to\cO_K$ is local and injective, since the image of $p$ is nonzero in $\fm$. Moreover, $\cO_K/p$ is a vector space over $\W(k)/p=k$ with basis $\{\pi^i\mid i=0,1\dots,e-1\}$, where $\pi$ is the uniformizer and $e$ is the ramification index. This implies that $\{\pi^i\mid i=0,1\dots,e-1\}$ is a $\W(k)$-basis for $\cO_K$. In particular, $\cO_K$ is a finite free module over $\W(k)$ of rank $e$. Thus, $K$ is a finite extension of $K_0=\W(k)[\frac{1}{p}]$ of degree $e$. This shows that $K_0$ is the maximal unramified subfield $K$ and we have $\hat{\bar{K_0}}=\hat{\bar{K}}=\C_p$.

Let $\bar{k}$ be the algebraic closure of $k$ which is the residue field of $\cO_{\bar{K}}$. We cannot embed $\W(\bar{k})$ into $\cO_{\bar{K}}$, as $\cO_{\bar{K}}$ is not complete. However, there is a canonical embedding $\W(\bar{k})\to \cO_{\hat{\bar{K}}}=\cO_{\C_p}=\cO_{\hat{\bar{K_0}}}$ and $\W(\bar{k})$ is the valuation ring of the completion $\hat{K^{un}_0}$ of the maximal unramified extension of $K_0$ with residue field $\bar{k}$. In particular, $\W(\bar{\F}_p)$ is the valuation ring of the completion of the maximal unramified extension of $\Q_p$.

Witt vectors were also motivated by the need to connect arithmetic in characteristic $p$ with characteristic $0$, providing a powerful framework for studying p-divisible groups, and their deformations in number theory and algebraic geometry. Moreover, as discussed in this section, Witt vectors play a crucial role in Dieudonné theory, a cohomology theory that generalizes the de Rham cohomology and offers a framework for studying cohomological invariants in characteristic $p$.

The ring of Witt vectors over a prefect $\F_p$-algebra satisfies the following universal property:
\begin{prop}{\cite[lemma 1.1.6]{Kedlaya_2015}}
Let $A$ be a prefect $\F_p$-algebra and let $R$ be a p-adic complete ring. Let $\W(A)$ be the ring of Witt vectors over $A$ and $\bar{\pi}:A/pA\to R$ a multiplicative map such that $A/pA\xrightarrow{\bar{\pi}} R\to R/pR$ is a ring homomorphism. Then $\bar{\pi}$ has a unique lift to a multiplicative map $\hat{\pi}:A\to R$ and a ring homomorphism $\pi:\W(A)\to R$ such that
\[
\pi(\sum^{\infty}_{n=0}[a_n]p^n)=\sum^{\infty}_{n=0}\hat{\pi}(a_n)p^n,\,\, \text{for any } a_n\in A
\]
where $[.]$ denotes the Teichm$\ddot{\text{u}}$ler lift of $a_n$ in $\W(A)$.
\end{prop}
\begin{example}
Let $\W(A)$ be the Witt vectors over a $\F_p$-algebra $A$. Then by the above lemma, the $p$-th power map on $A$ uniquely lifts to a map \[\sigma_{\W(A)}:\W(A)\to\W(A)\]
called the Frobenius automorphism of $\W(A)$, which satisfies
\[
\sigma_{\W(A)}(\sum^{\infty}_{n=0}[a_n]p^n)=\sum^{\infty}_{n=0}[a^p_n]p^n,\,\, \text{for any } a_n\in A
\]
for all $a_n\in \W(A)$. The perfectness of $A$ implies that $\sigma_{\W(A)}$ is indeed an automorphism.
\end{example}
\begin{example}
    Let $A=\F_q$, where $q=p^n$. Then we have the identification $\W(\F_q)=\Z_p[\zeta_{q-1}]$ where $\zeta_{q-1}$ is a primitive $(q-1)$-th root of unity. The Forbenius automorphism $\sigma_{\W(\F_q)}$ is a map defined by $\zeta_{q-1}\mapsto\zeta^p_{q-1}$ which is invariant on $\Z_p$. The field of fractions of $\W(\F_q)$ is the unramified extension of $\Q_p$ with residue field $k=\F_q$. 
\end{example}
Let $k$ be a perfect field of characteristic $p>0$, $\W(k)$ the ring of its Witt vectors, and $\sigma$ the Forbenius automorphism over $\W(k)$.
\begin{defn}
    The Dieudonn\'e ring of $k$ is the associative ring \[\sD_k:=\W(k)[\sF,\sV]/(\sF\sV-p,\sV\sF-p,Fc-\sigma(c)\sF,c\sV-\sV\sigma(c) ,\,\, \forall\sigma\in\W(k)).\]
\end{defn}
The Dieudonn\'e ring $\sD_k$ is a non-commutative ring when $k\neq\F_p$. For $\F_p$, we have $\sD_k=\Z_p[X,Y]/(XY-p)$.
\begin{remark}
    A left $\sD_k$-module is the same as a $\W(k)$-module $D$ equipped with a $\sigma$-semilinear endomorphism $F:D\to D$ and a $\sigma^{-1}$-semilinear endomorphism $V:D\to D$ such that $FV=VF=[p]_D$.
\end{remark}
\begin{defn}\label{Def: Dieudonn\'e mdoule}
    A Dieudonn\'e module over $\W(k)$ is a finite free $\W(k)$-module equipped with a Frobenius semilinear endomorphism $F:D\to D$ such that $pD\subseteq F(D)$.
\end{defn}
We have the following theorem that expresses the basic theory of Dieudonn\'e modules. The proof and more details can be found in \cite{AST_1977__47-48__1_0} and \cite{Dieudonne_1958}.
\begin{thm}\label{Theorem anti-equivalence dieudonne module and p-divisbible group}
    \begin{itemize}
        \item There is an anti-equivalence $\D$ of categories from the category of finite flat group schemes $G$ of $p$-power order over $k$ to the category of $\sD_k$-modules of finite $W(k)$-length with the following properties:
        \begin{enumerate}
            \item The order of $G$ is $p^{\ell(\D(G))}$.
            \item The functor $\D$ is compatible with perfect extensions i.e. if $k'/k$ is an extension of prefect fields then we have a natural isomorphism $\W(k')\otimes_{\W(k)}\D(G)\cong\D(k')$ as left $\sD_{k'}$-modules. In particular, if we take $\sigma:k\cong k$, we have the identification $\sigma^*(\D(G))\cong\D(G^{(p)})$ as $\W(k)$-modules, where $G^{(p)}$ is the base change of $G$ along $\sigma$.
            \item The relative Frobenius $F_G:G\to G^{(p)}$ induces the $\W(k)$-linear map
            \[
            \sigma^*(\D(G))\cong\D(G^{(p)})\xrightarrow{\D(F_G)}\D(G)
            \]
            which is the action of $F$ on $\D(G)$. Moreover, $G$ is connected (\'etale resp.) if and only if $F$ is nilpotent (isomorphism resp.)\footnote{Recall \cref{frobenius and connectedness and etaleness} for relative Frobenius and Verschiebung}.
            \item The $k$-vector space $\D(G)/F\D(G)$ is canonically isomorphic to the tangent space $t\ve_G$.
        \end{enumerate}
        \item The functor $G\mapsto \D(G):=\lim\D(G_n)$ is an anti-equivalence of categories
        \[
        \D:\Big\{\text{p-divisible groups over } k \Big\}\xrightarrow{\cong}\Big\{\text{Dieudonn\'e modules over }\W(k)\Big\}
        \] and we have 
        \begin{enumerate}
            \item This equivalence is compatible with any extension $k'/k$ of prefect fields.
            \item We have a canonical isomorphism $\D(G_n)\cong \D(G)/p^n\D(G)$ for any p-divisible group $G=\lim G_n$ over $k$.
            \item The rank of $\D(G)$ is equal to the height of $G$.
            \item There exists a canonical identification $\D(G\ve)[1/p]\cong \D(G)[1/p]\ve$.
        \end{enumerate}
    \end{itemize}
\end{thm}
Recall that a map $f:G\to H$ between two p-divisible groups is an isogeny, if it is finite and faithfully flat (epimorphism as fppf sheaves) with finite and flat kernel. Two p-divisible groups are called isogeneous if there exists such an $f$.
\begin{prop}
Let $f:G\to H$ be a homomorphism of p-divisible groups. The following are equivalent:
\begin{enumerate}
    \item $f $ is an isogeny.
    \item $\D(f): \D(H)\to \D(G)$ is injective.
    \item The induced map $\D(G)\otimes\Q_p\to \D(H)\otimes\Q_p$ is an isomorphism.
\end{enumerate}
\end{prop}
\begin{proof}
Over connected or quasi-compact base scheme $S$, a morphism $f:G\to H$ is an isogeny if and only if there exists a morphism $g:H\to G$ such that $g\circ f=p^n1_{G}$ and $f\circ g=p^n1_{H}$ for some integer $n\geq 0$ (see \cite[lemma 9]{fargues2022lubin-tate}). Now, the statement is straightforward from \cref{Theorem anti-equivalence dieudonne module and p-divisbible group}.
\end{proof}

\begin{example}
    \begin{enumerate}
        \item $\D(\underline{\Q_p/\Z_p})\cong\W(k)$ together with $F=\sigma$.
        \item $\D(\mup)\cong\W(k)$ together with $F=p\sigma$.
        \item If $k$ is algebraically closed, then \cref{we have only two kind p-divisible group over algebraic closed charac p} implies that the two previous Dieudonn\'e modules are the only Dieudonn\'e modules of dimension 1 over $\W(k)$. 
        \item By \cref{example: p-divisible group elliptic curve connected-\'etale}, we can say that if $E$ is an ordinary elliptic curve over $\bar{k}$, then we have $\D(E[p^{\infty}])\cong\W(k)\oplus\W(k)$ together with $F=\sigma\oplus p\sigma$.
    \end{enumerate}
\end{example}

For any such Dieudonn\'e module $D$ (in \cref{Def: Dieudonn\'e mdoule}), $F$ is bijective on $D[1/p]:=D\otimes_{\Z_p}\Q_p$. Hence, $D[1/p]$ is an example of the following sort of structure:
\begin{defn}\label{Def: isocrystal}
Let $K_0$ be the field of fractions of $\W(k)$. An isocrystal over $K_0$ is a finite dimensional $K_0$-vector space $N$ equipped with a bijective Frobenius semilinear endomorphism $F:N\to N$.
\end{defn}
\begin{remark}
Given an isocrystal $N$ over $K_0$, the dual space $N\ve=\Hom_{K_0}(N,K_0)$ is naturally an isocrystal over $K_0$ where the Frobenius $F_{N\ve}$ is give by 
\[
F_{N\ve}(f)(n)=\sigma(f(F^{-1}(n))),\,\, \text{for all }f\in\N\ve \text{ and } n\in N.
\]

Given two isocrystals $N_1$ and $N_2$ over $K_0$, the $K_0$-vector space $N_1\otimes_{K_0}N_2$ is an isocrystal over $K_0$ with the Frobenius automorphism $F_{N_1}\otimes F_{N_2}$.\\

If $M$ is a $\W(k)$-lattice in an isocrystal $N$, then $M$ is a Dieudonn\'e module if and only if $M$ is stable under $F$ and $V:=pF^{-1}$.
\end{remark}

\begin{defn}\label{simple isocrystal with slope}
Let $K_0[\sF]=\sD_k[1/p]$ be the polynomial ring satisfying $\sF c=\sigma(c)\sF$ for any $c\in K_0$. Let $r$ and $d$ be two integers with $r>0$. The quotient 
\[
N_{r,d}:=K_0[\sF]/(K_0[\sF](\sF^r-p^d))
\]
is an isocrystal, where the Frobenius structure on $N_{r,d}$ is defined by left action by $\sF$.

The isocrystal $N_{r,d}$ is isomorphic to the isocrystal $\bigoplus^rK_0$ together with $\sigma$-semilinear automorphism $F_{r,d}$ given by 
\[
F_{r,d}(e_1)=e_2,\cdots,F_{r,d}(e_{r-1})=e_r, F_{r,d}(e_r)=p^de_1;
\]
where $e_1,\cdots,e_r$ are standard basis vectors.

Let $\lambda=d/r$, where $gcd(r,d)=1$ and $r>0$. We define the simple isocrystal of slope $\lambda$ to be $N_{\lambda}:=N_{r,d}$.
\end{defn}
\begin{defn}
A filtered module over a commutative ring $R$ is an $R$-module $M$ endowed with a collection $\{\Fil^i(M)\}_{i\in\Z}$ of submodules that is decreasing i.e. $\Fil^{i+1}(M)\subseteq\Fil^i(M)$ for all $i\in\Z$ and satisfies the following properties:
\begin{enumerate}
    \item[(I)] It is separated i.e. $\bigcap_i\Fil^i(M)=0$,
    \item[(II)] It is exhaustive i.e. $\bigcup_i\Fil^i(M)=M$.
\end{enumerate}
For any filtered $R$-module $M$, the associated graded module is $$\gr(M)=\bigoplus_i\Fil^i(M)/\Fil^{i+1}(M)$$ and we write $\gr^i(M)=\Fil^i/\Fil^{i+1}(M)$ for every $i\in\Z$.

An $R$-linear map between two filtered modules $M$ and $N$ over $R$ is called a morphism of filtered modules if it maps each $\Fil^i(M)$ into $\Fil^{i}(N)$.

We denote by $\Fil_R$ the category of finitely generated filtered modules over $R$.
\end{defn}
\begin{defn}
A filtered ring is a ring $R$ equipped with an exhaustive and separated filtration $\{R^i\}$ as $\Z$-modules such that $1\in R^0$ and $R^i.R^j\subseteq R^{i+j}$ for all $i,j\in\Z$. The associated graded ring is $\gr(R)=\bigoplus_i R^i/R^{i+1}$.

The filtered $R$-algebra is an $R$-algebra $A$ with a structure of filtered ring $\{A^i\}$ such that each $A^i$ is an $R$-submodule of $A$. The associated graded $R$-algebra is $\gr(A)=\bigoplus_i A^i/A^{i+1}$.
\end{defn}
\begin{defn}\label{filtered Dieudonn\'e module}{\cite{Fontaine_Laffaille_1982}}
A filtered Dieudonn\'e module over $\W(k)$ is an $\W(k)$-module $D$ of finite type
endowed with:
\begin{enumerate}
    \item a decreasing filtration $(D^i)$ where the $D^i$ are direct factors of $D$;
    \item a family of Frobenius linear maps $F_i:D^i\to D$ such that
\end{enumerate}
\begin{gather*}
    D^i=D, D^j=0; \text{ for }i\ll 0, \text{ and }j\gg 0\\
    F_i\mid_{D^{i+1}}=p.F_{i+1}, \text{ for any }i;\\
    D=\sum_i F_i(D^i).
\end{gather*}
\end{defn}
\begin{prop}{\cite{Fontaine_Laffaille_1982}}
    The category of filtered Dieudonn\'e modules over $\W(k)$ is an abelian category.
\end{prop}
\begin{defn}\label{definition of filtered isocrystal}
    A filtered isocrystal (or a filtered $\varphi$-modules) over $K$ is an isocrystal $N$ over $K_0$ with a collection $\{\Fil^i(N_K)\}_{i\in\Z}$ which yields a structure of a decreasing, separated, and exhaustive filtered vector space over $K$ on $N_K:=N\otimes_{K_0}K$.
    
    A morphism of filtered isocrystals is a morphism $f:N_1\to N_2$ of isocrystals such that the induced $K$-linear map $f_K:N_1\otimes_{K_0} K\to N_2\otimes_{K_0} K$ is a morphism of filtered $K$-vector spaces.
    
    We denote by $\MF$ the category of filtered isocrystals over $K$.
\end{defn}
\begin{remark}
    The category $\MF$ is an additive category. Moreover, the tensor product $N_1\otimes N_2$ is a filtered isocrystal with natural structure of isocrystal on $N_1\otimes_{K_0}N_2$ and filtration 
    \[
    \Fil^i(N_{1,K}\otimes_K N_{2,K})=\sum_{n+m=i}\Fil^nN_{1,K}\otimes_K\Fil^{m}N_{2,K}.\
    \]
    The dual object $N\ve$ is a filtered isocrystal where its isocrystal structure is induced by $N\ve=\Hom_{K_0}(N,K_0)$ and the filtration is given by 
    \[
    \Fil^i(N\ve)_K=(\Fil^{-i+1}N_K)\ve.
    \]
\end{remark}
\begin{example}\label{unit filtered Dieudonn\'e module}
    A unit filtered Dieudonn\'e module is the filtered Dieudonn\'e module $1_{FD}:=\W(k)$ with $F=\sigma$ and together with the filtration
    \[
    \Fil^i\W(k)=\begin{cases}
    \W(k), & i\leq 0\\ 0,& i>0,
    \end{cases}
    \]
    and $F_i$ on $\Fil^i\W(k)$ for $i\leq 0$ is $p^{-i}$ times the usual Frobenius.
    
    Similarly, $K_0$ can be viewed as a filtered isocrystal with $F=\sigma$ and together with the filtration 
    \[
    \Fil^i K=\begin{cases}
        K,& i\leq 0\\ 0,& i> 0.
    \end{cases}
    \]
\end{example}

    A filtered isocrystal is not actually a filtered object in the category of isocrystals, as the action of the Frobenius is not necessarily compatible with the filtration.

\begin{example}
Let $A$ be an abelian variety over $\cO_K$ and let $\bar{A}$ be its special fibre over a perfect field $k$. By the work of Berthelot-Breen-Messing \cite{Berthelot_Breen_Messing_1982}, we have a comparison isomorphism between crystalline and de Rham cohomology 
\[
\D(\bar{A}[p^{\infty}])[1/p]\otimes_{K_0}K\cong \operatorname{H^1_{cris}}(\bar{A}/\W(k))\otimes_{\W(k)}K\cong \operatorname{H^1_{dR}}(A/\cO_K)\otimes_{\cO_K}K.
\]
in which the first one is a filtered isocrystal over $K$, where the filtration of $K$-vector space is induced by Hodge-filtration on the right side.
\end{example}

In the view of \cref{simple isocrystal with slope}, where $F$ looks as it acts with eigenvalues of p-adic ordinal $d/r$, we shall write $\Delta_{\alpha}=N_{r,d}$ for each $\alpha=d/r$ in reduced form with $r>0$. One can bring the following definition:

We say that $\alpha_1\leq \cdots\leq \alpha_n$ are slopes of an isocrystal $N$ if we can have 
\[
N=\bigoplus^n_{i=1}\Delta_{\alpha_i}
\]

However, this is not well-defined. Fix a basis $\{e_i\}$ of $N$ and consider the matrix $(a_{ij})$ associated to the action $F_N$, where we have $F_N(e_j)=\sum a_{ij}e_i$. This matrix transforms in a semilinear manner under a change of basis, and so its set of eigenvalues $\{\lambda_i\}$ depends on the choice of the basis. Moreover, the following example of Katz shows that the set $\{ord_p(\lambda_i)\}$ is dependent on the choice of the basis as well.

\begin{example}[Katz-\cite{brinon2009cmi}]\label{example slopes in different basis}
 Let $K_0=\W(\F_{p^2})[1/p]=\Q_p(\zeta_{p^2-1})$ with $p\equiv 3\pmod{4}$. As $p-1$ is even, $K_0$ contains $i=\sqrt{-1}$. Let $N=K_0e_1\oplus K_0e_2$ and we define $F$ by 
 \[
 F(e_1)=(p-1)e_1+(p+1)ie_2,\,\, F(e_2)=(p+1)ie_1-(p-1)e_2.
 \]
 $F$ corresponds to the matrix 
 \[
 \begin{pmatrix}
p-1 & (p+1)i\\ (p+1)i & -(p-1)
 \end{pmatrix}.
 \]
 We can extend $F$ uniquely by Frobenius-semi-linearity to form an isocrystal $N$. With respect to the given basis, the above matrix has the characteristic polynomial $\lambda^2-4p$, so its roots are $\pm 2\sqrt{p}$ which have p-adic ordinal $1/2$.

 Now we consider a new basis as follows
 \[
 e'_1=e_1+ie_2,\,\, e'_2=ie_1+e_2.
 \]
 We have \[F(e'_1)=F(e_1)+\sigma(i)F(e_2)=(p-1)e_1+(p+1)ie_2+(p+1)e_1+(p-1)ie_2=2p(e_1+ie_2)=2pe'_1\] and 
 \[
 F(e'_2)=\sigma(i)F(e_1)+F(e_2)=-i(p-1)e_1+(p+1)e_2+(p+1)ie_1-(p-1)e_2=2e'_2.
 \]
 Hence, with respect to the new basis the matrix 
 \[\begin{pmatrix}
     2p & 0\\
     0 & 2
 \end{pmatrix}\]
 has the eigenvalues $2$ and $2p$ with p-adic ordinal $0$ and $1$ respectively.
\end{example}
The above example shows that we need to find another way to define the concept of slope for an isocrystal $N$ over $K_0$.
\begin{thm}{(Dieudonn\'e-Manin, \cite[II]{manin1961theory})}\label{Dieudoone Manin classification}
For an algebraic closed field $k$ of characteristic $p>0$, the category of isocrystals over $K_0=\W(k)[1/p]$ is semisimple (i.e. all objects are finite direct sums of copies of simple objects and all short exact sequences split). Furthermore, the simple objects are given up to isomorphism by the isocrystal $N_{r,d}=\Delta_{\alpha}$, with $\alpha=d/r$ and $gcd(r,d)=1$. 
\end{thm}

    Let $N$ be an isocrystal over $K_0$. The Dieudonn\'e-Manin classification gives a unique decomposition of $\hat{N}:=\hat{K^{un}_0}\otimes_{K_0}N$ 
    \begin{equation}\label{eq: Dieudonne-Manin decompostion}
\hat{N}=\bigoplus_{\alpha\in\Q}\hat{\Delta}^{e_{\alpha}}_{\alpha}
    \end{equation}
The integer $\mu_{\alpha}=\dim_{\hat{K^{un}_0}}\hat{\Delta}=re_{\alpha}$ is called the number (with multiplicity) of eigenvalues of $F_N$ with slope $\alpha$, where $\alpha=r/d$ (with $r>0$, $gcd(r,d)=1$). We say that $\alpha\in\Q$ is a slope of $N$ if $\hat{\Delta}_{\alpha}$ is non-trivial direct summand in decomposition \ref{eq: Dieudonne-Manin decompostion}. 
\begin{defn}\label{def: isoclinic}
    We say that an isocrystal $N$ is isoclinic with slope $\alpha$ if $N\neq 0$ and $\hat{N}\cong \hat{\Delta}^e_{\alpha}$ for some $e\geq 1$.
\end{defn}
\begin{example}
    Back to \cref{example slopes in different basis}, we want to show that $N$ has slopes $0$ and $1$. Considering the basis $\{e'_1,e'_2\}$ we saw that $F$ acts via multiplication by $2p$ and $2$ with respect to this basis. We want to show that $\hat{N}=\hat{\Q^{un}_{p}}\otimes_{K_0}N\cong \hat{\Delta}_0\oplus\hat{\Delta}_1$. As the map $\W(\bar{\F}_p)^{\times}\to\W(\bar{\F}_p)^{\times},\,\,a\mapsto \sigma(a)/a$ is surjective, we can find some $b\in\W(\bar{\F}_p)^{\times}$ such that $\sigma(b)/b=1/2$. Now we change the basis to $\{be'_1,be'_2\}$. We have 
    \[
    F(be'_1)=\sigma(b)F(e'_1)=(1/2)b(2pe'_1)=p(be'_1),\,\, F(be'_2)=\sigma(b)F(e'_2)=(1/2)b(2e')=be'_2.
    \]
    Thus, $\hat{\Q^{un}_{p}}e_1$ represents $\Delta_1$ and $\hat{\Q^{un}_{p}}e_2$ represents $\Delta_0$. Hence, $\hat{N}=\hat{\Q^{un}_{p}}\otimes_{K_0}N\cong \hat{\Delta}_0\oplus\hat{\Delta}_1$. 
\end{example}
\begin{remark}
    If $N_1$ and $N_2$ are isocrystals over $K_0$ that are isoclinic with slopes $\alpha_1$ and $\alpha_2$ respectively, then $N_1\otimes N_2$ is isoclinic with slope $\alpha_1+\alpha_2$. Because $\hat{\Delta}_{\alpha_1}\otimes\hat{\Delta}_{\alpha_2}=\hat{\Delta}_{\alpha_1+\alpha_2}$.
\end{remark}
The Dieudonné–Manin classification does not hold when the base field is not algebraically closed; however, the slope decomposition into isoclinic parts still uniquely descends.
\begin{lemma}\label{isoclinic decomposition lemma}
For any nonzero isocrystal $N$ over $K_0$ with slopes $\alpha_1<\cdots<\alpha_n$, there is a unique decomposition $N=\bigoplus N(\alpha_i)$ into direct sum of nonzero subobjects that are isoclinic with respective slopes $\alpha_1<\cdots<\alpha_n$.
\end{lemma}
\begin{proof}
    \cite[Lemma 8.1.11]{brinon2009cmi}.
\end{proof}
\begin{defn}\label{Newton number}
Let $N$ be a nonzero isocrystal over $K_0$ with slopes $\{\alpha_1<\cdots<\alpha_n\}$ having multiplicities $\{\mu_1,\cdots,\mu_n\}$. The Newton polygon $\operatorname{Newt}(N)$ is the lower convex hull with leftmost point $(0,0)$ and having $\mu_i$ consecutive segments of horizontal distance $1$ and slope $\alpha_i$. In other words, it matches the points $$(0,0), (\mu_1,\mu_1\alpha_1),\cdots, (\mu_1+\cdots+\mu_n,\mu_1\alpha_1+\cdots+\mu_n\alpha_n)$$ consecutively. As $\mu_i\alpha_i\in\Z$, all the second components are integers.

The Newton number is the $y$-coordinate of the rightmost endpoint which is \[t_N(N)=\sum_i\mu_i\alpha_i.\] 
\end{defn}
\begin{example}
Let $E$ be an ordinary elliptic curve over $k$. \cref{example: p-divisible group elliptic curve connected-\'etale} shows that $E[p^{\infty}]\cong \mup\oplus \underline{\Q_p/\Z_p}$ over $\bar{k}$. Passing to isocrystals we have $\D(E[p^{\infty}])[1/p]\cong\Delta_0\oplus \Delta_1$. Thus, the Newton polygon of $\D(E[p^{\infty}])[1/p]$ connects the points $(0,0), (1,0),$ and $(2,1)$.
\begin{center}
    \includegraphics[scale=1.5]{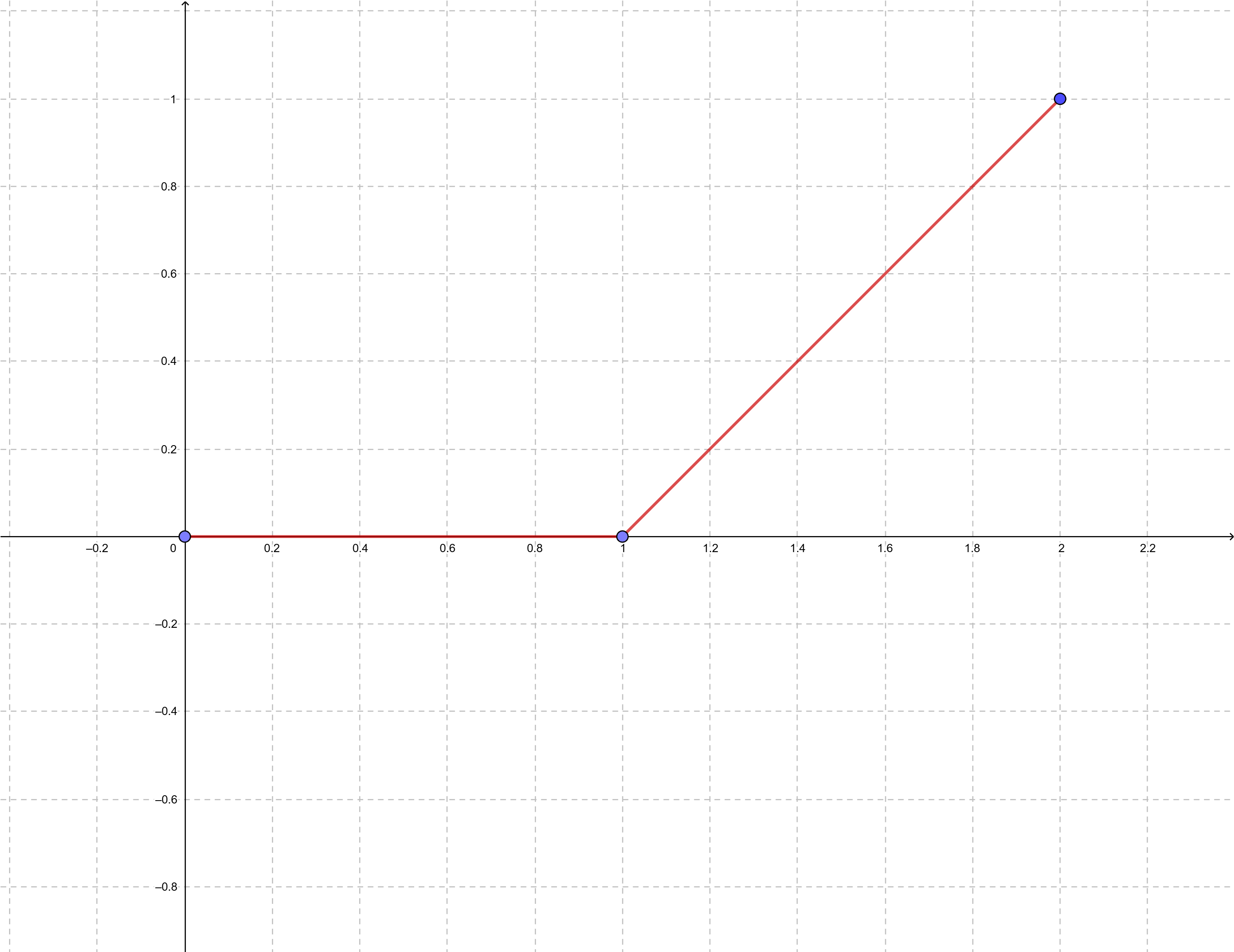}
\end{center}

\end{example}
\begin{cor}{\cite[page 88]{demazure2006lectures}}
    If $N$ is an isocrystal over $K_0$ arising form a p-divisible group over $k$ then we have the slope decomposition $N=\bigoplus N(\alpha_i)$, where each $N(\alpha_i)$ is an isoclinic with slope $\alpha_i$ in $[0,1]$. 
\end{cor}
If $N$ has the slope sequence $\alpha_1<\cdots<\alpha_n$, then the slope sequence for $N\ve $ is $1-\alpha_n<\cdots<1-\alpha_1$.
\begin{remark}
Since the p-divisible group of an abelian variety $A[p^{\infty}]$ is in the same isogeny set as $A[p^{\infty}]\ve$, the isogeny classes of isocrystals coming from abelian varieties correspond to Newton polygons with the following information:
\begin{itemize}
    \item All slopes are between $0$ and $1$.
    \item The Newton polygon starts with $(0,0)$ and ends at $(2g,g)$ where $g$ is dimension of abelian variety $A$
    \item The Newton polygon is symmetric i.e. $\alpha_i=1-\alpha_{2g-i+1}$ for all $1\leq i\leq 2g$.
\end{itemize}
We have two cases:
\begin{enumerate}
    \item All slopes of $A[p^{\infty}]$ are equal to $1/2$ and $A$ is called supersingular.
    \item The Newton polygon connects $(0,0)$ to $(2g,g)$ and $A$ is called ordinary.
\end{enumerate}
\end{remark}
The category $\MF$ of filtered isocrystals over $K$ is not an abelian category. But there is a full subcategory of $\MF$ that is abelian, closed under extensions and dual (see \cite[\S 4]{fontaine1982formes}). It is called the category of weakly admissible filtered isocrystals.
\begin{defn}\label{definition weakly addmissible isocrystal}
A filtered isocrystal $N$ over $K$ is weakly admissible if
\[
t_N(N')\geq t_H(N'):=\sum_{i\in\Z}i.\dim_K\gr^i(N'_K), \,\text{ for all sub-filtered isocrystal } N'\subset N
\] and equality holds when $N'=N$. $T_H(N')$ is referred to as the Hodge number of $N'$, and its corresponding polygon is called the Hodge polygon.

The full subcategory of $\MF$ consisting of weakly admissible filtered isocrystals is denoted by $\MFw$. 
\end{defn}
Similarly, we can define the notions of Hodge and Newton numbers, as well as weakly admissible isocrystals, for any isocrystals equipped with a filtration over $K_0$, such as those filtered isocrystals over $K_0$ that arise from a filtered Dieudonn\'e module (\cref{filtered Dieudonn\'e module}).
\section{Crystalline nature of p-divisible groups}\label{sec: crystalline nature of p-divisible groups}
The problem of generalizing the Dieudonn\'e theory to p-divisible groups over a more general base scheme $S$ over which $p$ is locally nilpotent has been advertised and tackled by Grothendieck \cite{grothendieck1974groupes}. Grothendieck’s proposal was to define $\D(G)$ as an $\cF$-crystal on the crystalline site of $S$ (see \cref{def: crystalline site}). This gives a direct application to the theory of infinitesimal extension of p-divisible groups and clears the connection to the classical Dieudonn\'e theory. 

We summarize the key results from \cite{mazur2006universal}, \cite{messing_crystals_1972}, and \cite{Grothendieck_1971}.

\begin{defn}[Divided powers]
Let $R$ be a ring and $I$ an ideal of $R$. We say that $I$ is equipped with divided powers if we are given a family of maps $\gamma_n:I\to I$ for $n\geq 0$ which satisfy the following
conditions:
\begin{enumerate}
    \item $\gamma_0(x)=1$, for all $x\in I$,
    \item $\gamma_1(x)=x$, for all $x\in I$,
    \item $\gamma_n(x+y)=\sum_{0\leq i\leq n}\gamma_i(x)\gamma_{n-i}(y)$, for all $x,y\in I$,
    \item $\gamma_i(ax)=a^i\gamma_i(x)$, for all $a\in R, x\in I, i\geq 1$,
    \item $\gamma_i(x)\gamma_j(x)=\frac{(i+j)!}{i!j!}\gamma_{i+j}(x)$, for all $i,j\geq 0, x\in I$,
    \item $\gamma_i(\gamma_j(x))=\frac{(ij)!}{i!(j!)^i}\gamma_{ij}(x)$, for all $i,j\geq 1, x\in I$.
\end{enumerate}
\end{defn}
\begin{remark}
    The basic idea for a divided power (DP) structure is that $\gamma_i(x)$ behaves like $x^i/i!$. Although dividing by $i!$ may not make sense in general, we can read $\gamma_i(x)$ as $x^i/i!$.

    If $\Q\subseteq R$, we have a unique divided power structure on $(R,I)$ that is $\gamma_i(x)=x^i/i!$.
\end{remark}
\begin{defn}
    Given a DP structure $(R,I,\gamma)$, we say that the divided powers are nilpotent if there is an integer $N$ such that the ideal generated by elements of the form 
    \[
    \gamma_{i_1}(x_1)\cdots\gamma_{i_k}(x_k),\,\, i_1+\cdots+i_k\geq N
    \]
    is zero. This means that $I^N=0$.
\end{defn}
Assume that $(R,I,\gamma)$ has a divided power (DP) structure. We can define an exponential group homomorphism $I\to 1+I$ by
\[
\exp(x):=1+\sum_{n\geq 1}\gamma_n(x)
\] and a logarithmic map $1+I\to I$ by
\[
\log(1+x)=\sum_{n\geq 1}(n-1)!(-1)^{n-1}\gamma_n(x).
\]
These homomorphisms are inverse to each other and yields an isomorphism $(1+I)\cong I$.
 \begin{example}
         The map $\gamma_n(p)=p^n/n!$ uniquely extends a DP structure on $(\W(k),p\W(k))$. For any integer $n\geq 1$ we can write
         \[
         n=a_0+a_1p+\cdots+a_mp^m,
         \]
         where $0\leq a_j\leq p-1, \, j=0,\cdots m$. Let $s=a_0+a_1+\cdots a_m$, then 
         \[
         \nu(n!)=\frac{n-s}{p-1}\leq n-1
         \]
         where $\nu$ is the p-adic valuation. This implies that $\gamma_n(p)=p^n/n!\in p\W(k)$.

         We can replace $\W(k)$ by any separated and complete Noetherian adic ring of characteristic zero.
 \end{example}
\begin{defn}
    Let $(R,I,\gamma)$ and $(R',I',\gamma')$ be two DP structures. A DP homomorphism is a ring homomorphism $\phi:R\to R'$ which is compatible with DP structures and $\phi(I)\subseteq I'$.
\end{defn}
\begin{defn}
    Let $(R,I,\gamma)$ be a DP structure and let $\phi:R\to R'$ a ring homomorphism. We say that $\gamma$ extends to $R'$ if there exists a DP structure $\gamma'$ on $IR'$ such that the map $(R,I,\gamma)\to (R',IR',\gamma')$ is a DP homomorphism.
\end{defn}
\begin{prop}{\cite{grothendieck1974groupes}}
    Let $(R,I,\gamma)$ be a DP structure and $\phi:R\to R'$ be a ring homomorphism. Then $\gamma$ extends to $R'$ if one of the following conditions hold:
    \begin{enumerate}
        \item $I$ is a principal ideal.
        \item $R\to R'$ is flat.
    \end{enumerate}
\end{prop}
Let $S$ be a scheme and $\cI$ a quasi-coherent ideal sheaf of $\cO_{S}$. A DP structure on $(S,\cI)$ is given by assigning to each open subset $U$ a DP structure on $(\Gamma(U,S),\Gamma(U,\cI))$ commuting with the restriction maps. A morphism of DP structures between $(S,\cI,\gamma)$ and $(S',\cI',\gamma')$ is a morphism $f:S\to S'$ such that $f^{-1}(\cI')$ maps into $\cI$ under the map $f^{-1}\cO_{S'}\to \cO_{S}$ and the DP structure induced on the image of $f^{-1}(\cI')$ coincide with the one defined by $\gamma'$.

\begin{defn}[Crystalline site]\label{def: crystalline site}
    For a scheme $X/S$, we define the crystalline site $Crys(X/S)$ of $X$ relative to $S$ to be a category whose objects are $T_U:=(U\into T,\gamma)$ where $U$ is open subscheme of $X$, and $U\into T$ is a locally nilpotent immersion. There exists a locally nilpotent DP structure $\gamma=(\gamma_i)$ on the ideal $\cI$ of $U$ in $T$ compatible with DP structure on $S$. 
    
    The morphisms from $(U\into T,\gamma)$ to $(U'\into T',\gamma')$ are morphisms $(f, \bar{f}))$ making the following diagram commutative
    \begin{equation*}
\begin{tikzcd}
U \arrow[d, "f"] \arrow[r, hook] & T \arrow[d, "\bar{f}"] \\
U' \arrow[r, hook]               & T'                    
\end{tikzcd}
    \end{equation*}
    and $\bar{f}$ is a DP morphism.

    A covering of an object $(U\into T,\gamma)$ is a collection of morphisms $\{(U_i\into T_i,\gamma_i)\to (U\into T,\gamma)\}$ such that $T_i$ is an open subscheme of $T$, $U_i$ is an open subscheme of $U$ and $\bigcup U_i=U$.

    A crystalline sheaf on $X$ is a sheaf on the crystalline site $Crys(X)$ with respect to covering mentioned above (see \cref{def: sheaf on site}). 
\end{defn}
\begin{example}
    The structural sheaf $\cO_{X/S}$ on $Crys(X/S)$ is defined by 
    \[
    (\cO_{X/S})_{(U,T,\gamma)}=\cO_T
    \]
    for every object $(U\into T, \gamma)$.
\end{example}
\begin{defn}
    A crystal of $\cO_{X/S}$-modules is a sheaf $\cF$ of $\cO_{X/S}$-modules such that for any morphism $f:(U\into T,\gamma)\to (U'\into T',\gamma')$ in $Crys(X/S)$, the induced map
    \[
    f^*\cF_{(U'\into T',\gamma')}\to \cF_{(U\into T,\gamma)}
    \]
    is an isomorphism.
\end{defn}
By the method of exponential, one can associate to certain p-divisible groups on $S_0$ a crystal in finite locally free $\cO_{S/S_0}$-modules. More precisely, for those p-divisible groups that are locally liftable to infinitesimal neighborhoods. To such p-divisible group $G$, Messing defined:
\begin{enumerate}
    \item a crystal in fppf groups: $\E(G)$,
    \item a crystal in formal Lie groups: $\hat{\E(G)}$,
    \item a crystal in finite locally free modules: $\D(G)$.
\end{enumerate}

\subsection{Construction of $\Ex(G)$ (\cite{messing_crystals_1972})}
    Recall that an extension of $G$ by a vector group $G$ is called a vector extension. It is said to be universal if for any group $M$ the natural map 
    \[
    \Hom_S(V,M)\to \Ext^1_S(G,M)
    \] obtained by pushout is a bijective.

    If $S$ is a scheme in characteristic $p$ (i.e. $p^n\cO_S=0$ for some integer $N$), then for a p-divisible group $G$ over $S$ the universal vector extension $\Ex(G)$ of $G$ by $V(G)$ exists (\cite[Chapter IV 1.10]{messing_crystals_1972}). Moreover, if $u:G\to H$ is a homomorphism of p-divisible groups over $S$, there is a unique homomorphism $\Ex(u):\Ex(G)\to \Ex(H)$ such that we can obtain a morphism of extensions:
\begin{equation}
\begin{tikzcd}
0 \arrow[r] & V(G) \arrow[d, "V(u)"] \arrow[r] & \Ex(G) \arrow[d, "\Ex(u)"] \arrow[r] & G \arrow[d, "u"] \arrow[r] & 0 \\
0 \arrow[r] & V(H) \arrow[r]                   & \Ex(H) \arrow[r]                    & H \arrow[r]                & 0
\end{tikzcd}
\end{equation}
where $V(u)$ is induced on invariant differentials by the Cartier dual of $u$.

Now assume that $S=\Spec(A)$ such that $A$ is in characteristic $p$. Let $S_0=\Spec(A/I)$, where $I$ is an ideal of $A$ with nilpotent DP structure. Let $G$ and $H$ be two p-divisible groups over $S$ and assume that $u_0:G_0\to H_0$ is a homomorphism between $G_0:=G\times_S S_0$ and $H_0:=H\times_S S_0$. We have the diagram 
\begin{equation*}
\begin{tikzcd}
0 \arrow[r] & V(G_0) \arrow[d, "V(u_0)"] \arrow[r] & \Ex(G_0) \arrow[d, "\Ex(u_0)"] \arrow[r] & G \arrow[d, "u_0"] \arrow[r] & 0 \\
0 \arrow[r] & V(H_0) \arrow[r]                   & \Ex(H_0) \arrow[r]                    & H_0 \arrow[r]                & 0.
\end{tikzcd}
\end{equation*}
Then
\begin{enumerate}
    \item There is a unique morphism $v:\Ex(G)\to \Ex(H)$ lifting $\Ex(u_0)$.
    \item Assume that $w:V(G)\to V(H)$ is a lifting of $V(u_0)$ such that $d=i\circ w-v\mid_{V(G)}:V(G)\to V(H)$ induces zero on $S_0$, where $i: V(H)\to V(G)$ is inclusion. We denote $E_S(u_0):=v$.
    \item If $u_0:G_0\to H_0$ and $u'_0:H_0\to K_0$ are homomorphisms of p-divisible groups over $S_0$, then 
    \[
    E_S(u'_0\circ u_0)=E_S(u'_0)\circ E_S(u_0).
    \]
    \item $E_S(1_{G_0})=1_{\Ex(G)}$.
    \item If $u_0:G_0\cong H_0$ is an isomorphism, then $E_S(u_0)$ is an isomorphism.
\end{enumerate}

We claim that the assignment $G_0\mapsto \Ex(G_U)$, for a lifting $G_U$ of $G_0\mid_{U_0}$ to $U$, is a crystal. It suffices to give the value of the crystal on sufficiently small objects $(U_0\into U)$ of the crystalline site of $S_0$. Take an affine scheme $U_0$, from the above observation $\Ex(G_U)$ is independent of the choice of lifting up to canonical isomorphism. If $V_0\into V$ is another object and there is a morphism 
\begin{equation*}
\begin{tikzcd}
U_0 \arrow[r, hook] \arrow[d, "f"] & U \arrow[d, "\bar{f}"] \\
V_0 \arrow[r, hook]                & V                     
\end{tikzcd}
\end{equation*}
then a lifting $G\mid_U$ of $G_0\mid_{U_0}$ to $U$ and a lifting $G_V$ of $G_0\mid_{V_0}$ to $V$ induce a canonical isomorphism $\bar{f^*}(\Ex(G_U))\cong \Ex(G_V)$.

The formal completion $\hat{\Ex(G)} :=\colim \text{Inf}^k \Ex(G)$ of $\Ex(G)$ along its zero section is a formal Lie group (\cite[Chapter IV 1.19]{messing_crystals_1972}).
\begin{defn}\label{def: three crystals}{\cite[Chapter IV, 2.5.4]{messing_crystals_1972}.}
    We define three crystals
    \begin{gather*}
        \E(G_0)_{(U_0\into U)}:=\Ex(G_U),\\
        \hat{\E(G_0)}_{U_0\into U}:=\hat{\Ex(G_U)},\\
        \D(G_0)_{(U_0\into U)}:=\Lie(\hat{\E(G_0)}_{U_0\into U}).
    \end{gather*}
    We call $\D$ our covariant Dieudonn\'e crystal. The contravariant Dieudonn\'e crystal is \[\D^*(G_0)_{(U_0\into U)}:=\Lie(\hat{\Ex(G\ve_0)}_{U\into U}).\]
\end{defn}
\begin{thm}[\cite{messing_crystals_1972}]
Let $\Spec(A_0)\into\Spec(A)$ be a locally nilpotent DP extension. Suppose that $G$ is a p-divisible group over $A$ and $G_0=G\times_{\Spec A}\Spec(A_0)$. Then there is a canonical exact sequence
\begin{equation}\label{1.6.2}
0\to \Lie(G\ve)\ve\to \D(G_0)_{\Spec(A_0)\into\Spec(A)}\to\Lie(G)\to 0.
\end{equation}
Here $G\ve$ is the Cartier dual of $G$, and $\Lie(G\ve)\ve$ is the dual of the tangent space of $G\ve$.

Suppose that $A$ is an abelian scheme over $A$ such that there exists an isomorphism
\[
A[p^{\infty}]\times_{\Spec(A)}\Spec(A_0)\cong G_0.
\]
Then there exists a natural isomorphism
\[
\D(G_0)_{\Spec(A_0)\into\Spec(A)}\cong \operatorname{H^{dR}_1}(A),
\]
such that it identifies the exact sequence \ref{1.6.2} with the Hodge filtration
\[
0\to \Lie(A\ve)\ve\to \operatorname{H^{dR}_1}(A)\to\Lie(A)\to 0.
\]
\end{thm}
\begin{remark}\label{crystal on crystalline site k}
Let $k$ be a perfect field of characteristic $p$ and $G$ a p-divisible group over $k$. For $p\geq 3$ we define $\operatorname{W}_n:=\W(k)/p^n\W(k)$, with a surjective homomorphism $\operatorname{W}_n\to \W(k)/p\W(k)=k$ giving a nilpotent immersion $\Spec k\into\Spec W_n$ with nilpotent DP structure on the ideal $p\W(k)/p^n\W(k)$. The Grothendieck-Messing crystal recovers the classical Dieudonn\'e theory by 
\begin{equation}\label{1.6.3}
\D(G)=\lim\D(G\ve)_{(\Spec k\into \Spec\W_n)},
\end{equation}
where $\D(G)$ is Dieudonn\'e module associated to $G$ via the Dieudonn\'e functor as described in \cref{Theorem anti-equivalence dieudonne module and p-divisbible group}. If $p=2$, we take $\W_n:=\W(k)/4^n\W(k)$.
\end{remark}

\section{Foundations of p-adic Hodge theory}
P-adic Hodge theory is a powerful and intricate branch of arithmetic geometry that seeks to understand the relationship between p-adic representations and the various cohomological theories associated with algebraic varieties over p-adic fields. Originating from the pioneering work of John Tate, Jean-Marc Fontaine, Pierre Colmez, and others, p-adic Hodge theory provides profound insights into the structure of p-adic Galois representations by connecting them to more classical objects in Hodge theory and algebraic geometry. One of the fundamental question that p-adic Hodge theory aims to find an answer to is which p-adic representations come from geometry. At its core, p-adic Hodge theory explores the ways in which p-adic analytic methods can be employed to study the cohomology of algebraic varieties. A central theme in p-adic Hodge theory is the comparison theorems, which establish deep connections between different cohomology theories.

In this section, we will highlight the basic concepts that will be used throughout the dissertation. In this section we set $K$ to be a a discrete-valued and complete extension of $\Q_p$ with ring of integers $\cO_K$ and prefect residue field $k$. We denote the Witt vectors of $k$ by $\W(k)$, and its fraction field by $K_0$. The completion of the algebraic closure of $K$ is denoted by $\C_p:=\hat{\bar{K}}$. The absolute Galois group of $K$ is denoted by $\Gamma_K$. We denote by $\Rep(\Gamma_K)$ the category of finite dimensional p-adic $\Gamma_K$-representations over $\Q_p$.
\begin{defn}
Let $M$ be a $\Z_p[\Gamma_K]$-module. We define $n$-th Tate twist of $M$ to be the $\Z_p[\Gamma_K]$-module
$$M(n):=\begin{cases}
        M\otimes_{\Z_p[\Gamma_k]}\Tp(\mu_{p^{\infty}})^{\otimes n} &  n\geq 0\\
        \Hom_{\Z_p[\Gamma_k]}(\Tp(\mu_{p^{\infty}}^{\otimes -n}),M) & n<0.
    \end{cases}
$$
The homomorphism $\chi\colon\Gamma_K\to \Aut(\Z_p(1))\cong\Z^{\times}_p$ which represents the action of $\Gamma_K$ on $\Z_p(1)$ is called the p-adic cyclotomic character of $K$.
\end{defn}
\begin{prop}
Let $M$ be a $\Z_p[\Gamma_K]$-module.
\begin{enumerate}
    \item We have canonical $\Gamma_K$-equivariant isomorphisms
    $$M(m)\otimes_{\Z_p}\Z_p(n)\cong M(m+n)$$ and $$M(n)\ve\cong M\ve(-n)$$
    for each $m,n\in\Z$.
    \item If $\rho$ is the action of $\Gamma_K$ on $M$, then $\rho\otimes\chi^n$ represents the action of $\Gamma_K$ on $M(n)$ for each $n$.
\end{enumerate}
\end{prop}

\begin{defn}
    Let $B$ be a $\Q_p$-algebra with an action of $\Gamma_K$, and $F$ a p-adic field. We say that $B$ is $F$-regular if
    \begin{enumerate}
    \item[(I)] $B^{\Gamma_K}=C^{\Gamma_K}$, where $C$ is the fraction field of $B$, endowed with a natural action of $\Gamma_K$ which extends the action on $B$.
    \item[(II)] $b\in B$ is unit if the set $F.b=\{c.b\mid c\in F\}$ is stable under the action of $\Gamma_K$.
    \end{enumerate}
\end{defn}
\begin{defn}
Assume that $B$ is a regular $\Q_p$-algebra. Let $E:=B^{\Gamma_K}$, and let $\Vect(E)$ denote the category of finite-dimensional vector spaces over $E$, and $\operatorname{Rep_F}(\Gamma_K)$ the category of $F$-linear $\Gamma_K$-representations of finite type.
\begin{enumerate}
    \item We define the functor $D_B\colon \operatorname{Rep_F}(\Gamma_K)\to \Vect(E)$ by 
    $$D_B(V):=(V\otimes_{F}B)^{\Gamma_K}\,\,\, \text{for every }V\in \operatorname{Rep_F}(\Gamma_K)$$
    \item We say $V\in \operatorname{Rep_F}(\Gamma_K)$ is $B$-admissible if $\dim_E D_{B}(V)=\dim_{F}V$.
    \item We denote the category of $B$-admissible $p$-adic representations by $\operatorname{Rep_{F}^B}(\Gamma_K)$.
\end{enumerate}
\end{defn}
When $B$ is regular, $\dim_E D_B(V)\leq\dim_F(V)$ with equality if and only if the induced map $D_B(V)\otimes_E B\to V\otimes_{F}B$ is an isomorphism.
\subsection*{Prefectoid fields and tilt}
\textbf{References:} \cite{brinon2009cmi}, \cite{scholze2012perfectoid}
\begin{defn}
Let $C$ be a complete nonarchimedean field with residue field of characteristic $p$. We say that $C$ is a perfectoid field if it satisfies the following conditions:
\begin{enumerate}
    \item[(I)] The valuation on $C$ is nondiscrete.
    \item[(II)] The $p$-th power map on $\cO_C/p\cO_C$ is surjective.
\end{enumerate}
\end{defn}
\begin{prop}
Let $C$ be a complete nonarchimedean field of residue characteristic $p$. Assume that the $p$-th power map is surjective on $C$, then, $C$ is a perfectoid field. In particular, a nonarchimedean field of characteristic $p$ is perfectoid if and only if it is complete and perfect.
\end{prop}
\begin{defn}
We define the tilt of $C$ by 
$C^{\flat}:=\lim_{x\mapsto x^p}C$ equipped with the natural multiplication.
For every $c=(c_n)\in C^{\flat}$, we write $c^{\sharp}:=c_0$.
\end{defn}
For every element $\varpi\in C^{\times}$ with $0<\nu(\varpi)\leq\nu(p)$, we have a natural multiplicative bijection 
    $$\lim_{x\mapsto x^p}\cO_{C}\cong \lim_{x\mapsto x^p}\cO_C/\varpi\cO_C.$$
The ring structure on $\cO_C/\varpi\cO_C$ induces a ring structure on $\lim_{x\mapsto x^p}\cO_{C}$. We define the tilt ring to be $\cO_{C^{\flat}}:=\lim_{x\mapsto x^p}\cO_{C}$ that does not depend on the choice of $\varpi$. This becomes naturally a complete valuation ring with fraction field $C^{\flat}$ of characteristic $p$ and valuation $\nu^{\flat}$ given by $\nu^{\flat}(c):=\nu(c^{\sharp})$. Moreover, $C^{\flat}$ is a prefectoid field of characteristic $p$.

\begin{example}
Let $\hat{\Q_p(p^{1/p^{\infty}})}$ be the completion of $\bigcup_{n\geq 1}\Q_p(p^{1/p^n})$. Then, $\hat{\Q_p(p^{1/p^{\infty}})}$ is a perfectoid field and the argument below shows that its tilt is isomorphic to $\hat{\F_p(u^{1/p^{\infty}})}$ which is the $u$-adic completion of the perfection of the polynomial ring $\F_p(u)$.
$$\lim_{x\mapsto x^p}\hat{\Z_p[1/p^{\infty}]/p}\cong\lim_{x\mapsto x^p}\Z_p[1/p^{\infty}]/p\cong\lim_{x\mapsto x^p}\F_p[u^{1/p^{\infty}}]/u\cong \hat{\F_p[u^{1/p^{\infty}}]}.$$

Similar argument shows that the completion of $\Q_p(\mu_{p^{\infty}})=\bigcup_{n\geq 1}\Q_p(\zeta_{p^n})$ is also a prefectoid field whose tilt is isomorphic to $\hat{\F_p(u^{1/p^{\infty}})}$.
\end{example}
According to the above example, the tilting functor is not fully faithful on the category of prefectoid fields over $\Q_p$. However, Peter Scholze showed that for every perfectoid field $C$, the tilting functor induces an equivalence between the category of prefectoid fields over $C$ and the category of prefectoid fields over $C^{\flat}$ (\cite{scholze2012perfectoid}).

\subsection*{Period rings}\label{sec: period rings}
\textbf{References:} \cite{brinon2009cmi}, \cite{caruso2019p-adicperiod}, \cite{szamuely2016p-adicHodgeAccordingtoBeilinson}, \cite{fontaine2008theory}\\

Let $K$ be a local field with prefect residue field $k$ of characteristic $p$. Let $\C_p$ be the completion of the algebraic closure of $K$. We also fix the valuation $\nu$ on $\C_p$ with $\nu(p)=1$ inducing the valuation $v^{\flat}$ on $\C^{\flat}_p$ given by $\nu^{\flat}(c)=\nu(c^{\sharp})$ for any $c\in \C^{\flat}_{p}$.
\begin{defn}
We define the infinitesimal period ring, denoted by $\Ainf$, to be the ring of Witt vectors over $\cO_{\C^{\flat}_{p}}$, $\Ainf:=\W(\cO_{\C^{\flat}_{p}})$. For any $c\in \cO_{\C^{\flat}_{p}}$, we write $[c]$ for its Teichm\"uller lift in $\Ainf$.
\end{defn}

There exists a surjective ring homomorphism $\theta\colon\Ainf\to\cO_{\C_p}$ with 
$$\theta(\sum^{\infty}_{n=0}[c_n]p^n)=\sum^{\infty}_{n=0}c^{\sharp}_np^n$$
for all $c_n\in \cO_{\C^{\flat}_{p}}$. Let $\theta[1/p]\colon \Ainf[1/p]\to\C_p$ be the induced map on $\Ainf[1/p]$. Choose an element $p^{\flat}\in \cO_{\C^{\flat}_{p}}$ such that $(p^{\flat})^{\sharp}=p$, and set $\xi:=[p^{\flat}]-p\in\Ainf$.

We define the de Rham local ring by 
$$\BdRp:=\lim_i\Ainf[1/p]/\ker(\theta[1/p])^i,$$
and denote by $\theta^+_{\text{dR}}$ by the natural projection $\BdRp\onto\Ainf[1/p]/\ker(\theta[1/p])$.
We also define the de Rham period ring $\BdR$ as the fraction field of $\BdRp$.

We can observe that both $\ker(\theta)$ and $\ker(\theta[1/p])$ are generated by $\xi$ as ideals in the ring $\Ainf$ and $\Ainf[1/p]$, respectively. Moreover, the natural map $\Ainf[1/p]\to \BdRp$ is injective. Therefore, we can canonically identify $\Ainf[1/p]$ as a subring of $\BdRp$. 

    The ring $\BdRp$ is a complete discrete valuation ring with maximal ideal $\ker(\theta^+_{\text{dR}})$ and residue field $\C_p$. In addition, $\xi$ is a uniformizer of $\BdRp$. We have a filtered ring structure $\{\xi^n\BdRp\}_{n\in\Z}$. This filtration does not depend on the choice of the uniformizer. In fact, $\xi^n\BdRp=\ker(\theta^+_{\text{dR}})^n$ for each $n\in\Z$. 

    Let $K_0$ be the fraction field of $\W(k)$. The field $K$ is a finite totally ramified extension of $K_0$ and there is a natural commutative diagram
    \begin{equation*}
        \begin{tikzcd}
K_0 \arrow[r, hook] \arrow[d] & \bar{K} \arrow[r, hook] \arrow[d, hook] & \C_p \\
{\Ainf[1/p]} \arrow[r, hook]  & \BdRp \arrow[ru, two heads]             &     
\end{tikzcd}
    \end{equation*}

    There exists a continuous map $\log:\Z_p(1)\to\BdRp$ given by 
    $$\log(\epsilon)=\sum^{\infty}_{n=1}(-1)^{n+1}\frac{([\epsilon]-1)^n}{n}$$
    for every $\epsilon\in \Z_p(1)=\Tp\mu_{p^{\infty}}=\lim\mu_{p^n}(\bar{K})=\{\epsilon\in \cO_{\C^{\flat}_p}:\epsilon^{\sharp}=1\}$. The power series $\sum^{\infty}_{n=1}(-1)^{n+1}\frac{([\epsilon]-1)^n}{n}$ converges with respect to discrete valuation topology on $\BdRp$. Fix a $\Z_p$-basis $\epsilon\in\Z_p(1)$ and let $t:=\log(\epsilon)$. The element $t\in\BdRp$ is a uniformizer.

\begin{defn}
We define $\BHT:=\bigoplus_{i\in\Z}\C_p(i)$. The ring $\BHT$ is regular. We say that the p-adic Galois representation $V$ is Hodge-Tate if $V$ is $\BHT$-admissible. 

Assume that $V$ is a Hodge-Tate representation. This implies that the natural map $\DBHT(V)\otimes_K \C_p\to V\otimes_{\Q_p}\C_p$ is an isomorphism. As $V$ is finite dimensional, there are finitely many $i$ such that $(V\otimes_{\Q_p}\C_p(i))^{\Gamma_{K}}\neq 0$. The Hodge-Tate weights of $V$ are the $i$ for which $(V\otimes_{\Q_p}\C_p(i))^{\Gamma_{K}}\neq 0$ and the multiplicity of the weight $i$ is $\dim_K((V\otimes_{\Q_p}\C_p(i))^{\Gamma_{K}})$.
\end{defn}
\begin{remark}
The natural action of $\Gamma_K$ on $\BdR$ satisfies the following properties:
\begin{enumerate}
    \item The $log$ and $\theta^+_{\text{dR}}$ are $\Gamma_K$-equivariant.
    \item For every $\gamma\in\Gamma_K$ we have $\gamma(t)=\chi(\gamma)t$.
    \item Each $t^n\BdRp$ is stable under the action of $\Gamma_K$.
    \item We have natural $\Gamma_K$-equivariant isomorphisms $$\BdRp/\ker(\theta^+_{\text{dR}})\cong\BdRp/t\BdRp\cong\Ainf[1/p]/\ker(\theta[1/p])\cong\C_p$$ and $$ \ker(\theta^+_{\text{dR}})^n/\ker(\theta^+_{\text{dR}})^{n+1}\cong t^n\BdRp/t^{n+1}\BdRp\cong\C_p(n) \text{ for all }n\in\Z.$$ 
    \item There exists a natural $\Gamma_K$-equivariant isomorphism of graded $K$-algebras $$\gr(\BdR)=\bigoplus_{n\in\Z}t^n\BdRp/t^{n+1}\BdRp\cong\bigoplus_{n\in\Z}\C_p(n)=\BHT.$$
    \item $\operatorname{B^{\Gamma_K}_{dR}}=(\operatorname{B^{+}_{dR}})^{\Gamma_K}\cong K$, and $\BdR$ is regular.
\end{enumerate}
 
\end{remark}

\begin{defn}\label{def B2}
    We define the rings $$\Bt:=\frac{\Ainf[1/p]}{(\ker\theta[1/p])^2}\cong \frac{\BdRp}{t^2\BdRp} \text{ and }\At:=\frac{\Ainf}{(\ker\theta)^2}.$$
The ring $\At$ embeds into $\Bt$, as $\Ainf$ has no $p$-power torsion, and we have $\Bt=\At[1/p]$.
\end{defn}

\begin{defn}
We define the integral crystalline period ring by
$$\Acris:=\big\{\sum^{\infty}_{n=0}a_n\frac{\xi^n}{n!}\in\BdRp:a_n\in\Ainf \text{ with } \operatorname{lim}_{n\to\infty}a_n=0 \big\}.$$
We write $\Bcrisp:=\Acris[1/p]$.
\end{defn}
\begin{prop}(\cite[\S 7.1]{fontaine2008theory})
Let $\Aocris$ be the divided power envelope of $\Ainf$ with respect to $\ker\theta$, that is, by adding all elements $\gamma_n(a):=\frac{a^n}{n!}$ for all $a\in \ker\theta$ and $n\in\N$. Then the p-adic completion of $\Aocris$, $\lim \Aocris/p^n\Aocris$, is canonically isomorphic to $\Acris$.
\end{prop}
We have $t\in\Acris$ and $t^{p-1}\in p\Acris$. Moreover, we have an identification $\Bcrisp[1/t]=\Acris[1/t]$. We define the crystalline period ring by 
    $$\Bcris:=\Bcrisp[1/t]=\Acris[1/t].$$

   By construction we have  $$\Ainf[1/p]\subseteq\Acris[1/p]=\Bcrisp\subseteq\Bcris\subseteq\BdR.$$
We have the following universal property for $\Acris$:
The map $\Acris\to\cO_{\C_p}$ is a universal p-adically complete divided power thickening of $\cO_{\C_p}$. This means that for any separated, complete, p-adic valuation ring $A$, and for any continuous surjective ring map $\alpha\colon A\to\cO_{\C_p}$ whose kernel has P.D structure compatible with the one on $pA$, there exists a unique homomorphism $\lambda_{\alpha}\colon\Acris\to A$ such that the diagram 
\begin{equation*}
\begin{tikzcd}
\Acris \arrow[rd, "\theta", two heads] \arrow[rr, "\lambda_{\alpha}", dotted] &            & A \arrow[ld, "\alpha", two heads] \\
                                                                              & \cO_{\C_p} &                                  
\end{tikzcd}
\end{equation*}
commutes.

\begin{remark}
$\Acris$ is not regular. However, $\Bcris$ is naturally a filtered subalgebra of $\BdR$ over $K_0$ with filtration $\Fil^n(\Bcris):=\Bcris\cap t^n\BdRp$ which is table under the action of $\Gamma_K$. Moreover, we have canonical isomorphisms $\gr(\Bcris\otimes_{K_0}K)\cong\gr(\BdR)\cong\BHT$ and $(\Bcris)^{\Gamma_K}\cong K_0$. Hence, $\Bcris$ is regular.

Let $V\in\Rep(\Gamma_K)$. Then $\DdR(V)$ is naturally a filtered vector space over $K$ with the filtration 
\begin{equation}\label{filtration on DdR(V)}
\Fil^n(\DdR(V)):=(V\otimes_{\Q_p}\Fil^n(\BdR))^{\Gamma_K}.    
\end{equation}
If $V\in\Rep(\Gamma_K)$, then $\Dcris(V)$ is naturally a filtered isocrystal over $K$ with the Frobenius automorphism $1\otimes\varphi$ and the filtration on $\Dcris(V)_K=\Dcris(V)\otimes_{K_0}K$ given by $$\Fil^n(\Dcris(V)_K):=(V\otimes_{\Q_p}\Fil^n(\Bcris\otimes_{K_0}K))^{\Gamma_K}.$$
\end{remark}

\begin{prop}[\cite{fontaine1994corps},\cite{injectivfroben-Conrad}]
The Frobenius automorphism of $\Ainf$ naturally extends to a $\Gamma_K$-equivariant endomorphism $\varphi_{cris}$ on $\Bcris$ with $\varphi_{cris}(t)=pt$. The Frobenius $\varphi_{cris}$ of $\Bcris $ is injective.
\end{prop}

\begin{prop}\label{log_cris}
For any $x\in 1+\fm_{\cO_{\C_p^{\flat}}}$, the sum 
\[
\logcris([x])=\sum_{n\geq 1}(-1)^{n+1}\frac{([x]-1)^n}{n}
\]
converges in $\Bcrisp$, and
we have $\Gamma_K$-equivariant homomorphism 
\[
x\mapsto \logcris([x])
\]
such that $\varphi_{cris}(\logcris([x]))=\logcris([x^p])=p\logcris([x])$.
\end{prop}
\begin{proof}
    \cite[Lemma 9.2.2]{brinon2009cmi}
\end{proof}

The exact sequence provided in the following theorem is called the fundamental exact sequence of p-adic Hodge theory.
\begin{thm}{\cite{fontaine2008theory},\cite[Theorem 9.1.8]{brinon2009cmi}.}\label{Fundamental exact sequence of p-adic Hodge theory}
There exists a natural exact sequence
$$0\to \Q_p\to(\Bcris)^{\varphi=1}\to\BdR/\BdRp\to 0.$$
\end{thm}

\begin{defn}
    We say that $V\in\Rep(\Gamma_K)$ is de Rham (crystalline, or Hodge-Tate resp.) if it is $\BdR$-admissible ($\Bcris$-admissible, or $\BHT$-admissible resp.). We write $\DdR$ ($\Dcris$, or $\DHT$ resp.) for the functor $D_{\BdR}$ ($D_{\Bcris}$, or $D_{\BHT}$ resp.). We denote by $\operatorname{Rep^{dR}_{\Q_p}}(\Gamma_K)$ and $\operatorname{Rep^{cris}_{\Q_p}}(\Gamma_K)$, the category of de Rham and crystalline representations respectively.
\end{defn}

\begin{thm}\label{Theorem: filtered isocrystal va crystalline representaion}{\cite{colmez2000construction},\cite{fargues_courbes_2018}, or \cite{berger_construction_2008}.}
\begin{enumerate}
\item Let $V$ be a de Rham representation of $\Gamma_K$. Then $V$ is Hodge-Tate with a natural $K$-isomorphism of graded vector spaces
$$\gr(\DdR(V))\cong\DHT(V).$$
\item The functor $\DdR$ with values in $\Fil_K$ is faithful and exact on $\operatorname{Rep^{dR}_{\Q_p}}(\Gamma_K)$.
\item Let $V$ be a crystalline representation of $\Gamma_K$. Then $V$ is de Rham with a natural isomorphism of filtered vector spaces
$$\Dcris(V)_K=\Dcris(V)\otimes_{K_0}K\cong\DdR(V).$$
\item The functor $\Dcris$ with values in $\MF$ is faithful and exact on $\operatorname{Rep^{cris}_{\Q_p}}(\Gamma_K)$. 
\item The functor $\Dcris$ is an equivalence between the category $\operatorname{Rep^{cris}_{\Q_p}}(\Gamma_K)$ and $\MFw$. The inverse functor is given by 
\[
N\mapsto (N\otimes_{K_0}\Bcris)^{\varphi=1}\cap \Fil^0(N_K\otimes_K\BdR).
\]
\end{enumerate}
\end{thm}
\begin{remark}\label{Hodge-Tate weights of tate module}
Let $G$ be a p-divisible group over $\cO_K$. The Hodge-Tate decomposition (\cref{Theorem: Hodge-Tate decomposition}) implies that $\Tp(G)\cong \coLie(G\ve)_{\C_p}\oplus \Lie(G)_{\C_p}(1)$. Hence, $\Tp(G)$ is Hodge-Tate with Hodge-Tate weights of $0$ and $1$ with multiplicity $\dim(G\ve)$ and $\dim(G)$ respectively.
\end{remark}

\begin{remark}\label{tate module is crystalline}
In this remark, we summarize the results of \cite[\S 12.2]{brinon2009cmi} and \cite[\S 6]{faltings_integral_1999}. Let $G$ be a p-divisible group over $\cO_K$. Recall that $\Tp(G)=\Hom(\Q_p/\Z_p,G_{\cO_{\bar{K}}})$. Let $f:\Q_p/\Z_p\to G_{\cO_{\bar{K}}}$, and we denote by $f$ its base change to $\ocp$. Recall the contravariant Dieudonn\'e functor from \cref{def: three crystals}. Evaluating at the DP thickening $\Spec\ocp\to\Spec\Acris$ provides a map
\[
\D^*(G)_{(\ocp\to\Acris)}\to \D^*(\Q_p/\Z_p)_{\Acris}\cong\Acris.
\]
Therefore, we get a pairing
\[
\Tp(G)\times\D(\bar{G})\to\Acris,
\]
where here $\bar{G}:=G\times_{\Spec\cO_K}\Spec k$, and $\D(\bar{G})$ is the classical Dieudonn\'e module as described in \ref{1.6.3}. If we tensor with $\Bcrisp$, we get an isomorphism
\[
\Tp(G)\otimes_{\Z_p}\Bcrisp\cong \D(\bar{G}\ve)\otimes_{\W(k)}\Bcrisp.
\]
This isomorphism is $\Gamma_K$-equivariant and respects the filtration. The filtration on the right side is induced by the filtered isocrystal $\D(\bar{G}\ve)\otimes_{\W(k)}K$ and natural filtration on $\Bcrisp$. The filtration on the left side is induced by the filtration on $\Bcris$. It also respects the action of the Frobenius. The action of the Frobenius on the right side is induced by $F\otimes\phi_{cris}$ and on the left side is induced by $1\otimes\phi_{cris}$. This identification shows that $\Vp(G)$ is a crystalline representation i.e. \[\Dcris(\Vp(G))=\D(\bar{G})\otimes_{\W(k)}K_0.\]
\end{remark}



\clearpage
\chapter{1-motives and their realisation functors}
1-motives were originally defined in \cite{deligne1974theorie} over an algebraically closed field.
Deligne's construction of 1-motives was motivated by the desire to study Hodge and étale realizations in a broader context than that provided by pure motives. Pure motives, associated with smooth projective varieties, have well-established Hodge structures, but they do not capture the complexities of non-smooth or non-projective varieties. By including algebraic tori and extending to mixed motives, 1-motives offer a more inclusive theory that addresses these cases.
Deligne also constructs several realization functors carrying a filtration and provides comparison isomorphism relating them which is compatible with the filtrations. For more detailed information on 1-motives, refer to \cite{barbieri2016derived}. 
\section{Basic definitions}
\begin{defn}\label{definition of 1-motive}
An integral Deligne 1-motive $M$ over $S$ is a two-term complex of $S$-commutative group schemes $M=[L\xrightarrow{u}G]$, where:
\begin{enumerate}
    \item $L$ is a lattice over $S$, i.e., it is an $S$-group scheme which, locally for the \'{e}tale topology on $S$, is isomorphic to a  constant finitely generated free $\Z$-module;
    \item $G$ is a semi-abelian $S$-scheme, i.e., it is an extension of an abelian $S$-scheme $A$ by an $S$-torus $T$; and
    \item $u$ is a morphism of $S$-group schemes.
\end{enumerate}

An effective torsion 1-motive over $S$ is $M=[L\xrightarrow{u}G]$, where $G$ is a semi-abelian $S$-scheme but $L$ is finitely generated $\Z$-module and \'{e}tale locally constant. From this point forward, we will refer to integral Deligne 1-motives simply as 1-motives.
\end{defn}
Morphisms of Deligne 1-motives (effective torsion 1-motives) are morphisms of complexes $L\to G$. We denote the category of 1-motives over $S$ by $\Mi(S)$. There is a canonical exact sequence
\begin{equation}\label{canonical exact sequence for any 1-motive}
    0\to[0\to G]\to M\to [L\to 0]\to 0
\end{equation} for any 1-motive $M=[L\to G]$ over $S$.

\begin{example}
Assume that $S=\Spec K$ for some field $K$. If the $M=[L\xrightarrow{u} G]$ is a Deligne 1-motive (resp. torsion 1-motive), then $L$ is a finite free $\Z$-module (resp. discrete $K$-group scheme) equipped with a continuous action of $\Gamma_K$, and $u:L\to G(\bar{K})$ is a $\Gamma_K$-equivariant group homomorphism.
\end{example}
Let $S=\Spec(R)$, where $R$ is a henselian local ring with residue field $k$. If $M=[L\xrightarrow{u} G]$ is a (Deligne) 1-motive over $R$, then $L$ is an isotrivial lattice (see \cref{Appendix isotrivial lattice torus}(1)), and $u$ induces a $\Gamma_k$-equivariant group homomorphism $L\to G(R^s)$, where $R^s$ is the universal covering of $R$ at $x:\Spec(\bar{k})\to\Spec(k)$. If $K$ is the field of fractions of $R$, then the map $u$ gives a group homomorphism $L\to G(R')$, where $R'$ is the integral closure of $R$ in some finite unramified extension $K'$ of $K$. For more details, we refer to \cite[\S 1.6]{matev2014good}.
\begin{remark}
Let $K$ be any field, and $S=\Spec(K)$. The category $\tMi$ of 1-motives with torsion is the localization of effective torsion 1-motives with respect to the multiplicative class of quasi-isomorphisms. This is an abelian category if $char(K)=0$ (see \cite[C.2,C.3,C.5]{barbieri2016derived}). By \cite[Prop. C.7.1]{barbieri2016derived}, the canonical embedding $\Mi\to\tMi$ has a left adjoint given by $$M=[L\xrightarrow{u} G]\mapsto M_{fr}:=[L/L_{\text{tor}}\to G/u(L_{\text{tor}})].$$ The category of 1-motives up to isogeny $\Mi^{\Q}:=\Mi\otimes\Q$ is abelian and $\Mi^{\Q}\cong\tMi^{\Q}:=\tMi\otimes\Q$.
\end{remark}
From now on, by the category of 1-motives, we refer to the isogeny category of 1-motives, and we will use the same notation $\Mi$ to denote this isogeny category.
\begin{defn}
    There is a standard filtration associated to a 1-motive $M=[L\to G]$ called weight filtration. It is defined as follows
    \[
    \W_i(M)=\begin{cases}
        0 & i<-2\\
        T & i=-2\\
        G & i=-1\\
        M & i\geq 0.
    \end{cases}
    \]
    The graded pieces are defined as follows:
       \[
    \gr_i(M)=\begin{cases}
        0 & i\leq -3 \text{ or } i\geq 1,\\
        T & i=-2,\\
        A & i=-1,\\
        L & i= 0.
    \end{cases}
    \]
    By $T, G, A, $ and $L$, we mean the complexes $[0\to T], [0\to G], [0\to A],$ and $[L\to 0]$, respectively.
\end{defn}
In general, for any abelian category $\cA$, one can define a filtered object as a pair $(X, F)$, where $X\in\cA$, and $F =(F^n(X))_{n\in\Z}$ is a sequence of objects in $\cA$ such that for any $n\leq m$, it satisfies that $F^n(X)\subseteq F^m(X)$. To any such filtered object, one can associate a graded object $Gr^F (X)$. Although the category of 1-motives over $S$ is not abelian, it can be viewed as a subcategory of the category of complexes of representable abelian sheaves on $S\fppf$. Thus, a 1-motive $M=[L\to G]$ over $S$ can be identified with a complex, where $L$ is in degree $-1$ and $G$ is in degree $0$. Under this identification, the pair $(M, W)$ forms a filtered object. For more details on filtrations, we refer to \cite{Deligne1971}. 

\begin{defn}
    Let $\cT\colon\Mi(S)\to \Mod(R)$ be an additive exact covariant functor from the category of 1-motives over $S$ to the category of modules over $R$, where $R$ is a commutative ring with unity. We define the standard weight filtration on $\cT$ to be
    \[
    \W_i\cT(M)=\begin{cases}
        0 & i<-2\\
        \cT(T) & i=-2\\
        \cT(G) & i=-1\\
        \cT(M) & i\geq 0.
    \end{cases}
    \]
    for any $M=[L\to G]\in\Mi$. Here, $0\to T\to G\to A\to 0$ is an extension of the abelian $S$-scheme $A$ by an $S$-torus $T$. By $\cT(T)$ ($\cT(G)$ and $\cT(L)$ resp.), we mean $\cT([0\to T])$ ($\cT([0\to G])$ and $\cT([L\to 0])$ resp.).
\end{defn}

\begin{prop}
The category $\Mi$ has all finite limits and colimits. In particular, for any  given two effective torsion 1-motives $M=[L\xrightarrow{u}G]$, $M'=[L'\xrightarrow{u}G']$, a morphism $\varphi=(f,g):M\to M'$ admits the kernel and cokernel as a morphism of complexes.The kernel of $\varphi$ is given by $\ker\varphi=[\ker^0 f\xrightarrow{u}\ker^0 g]$ and the cokernel of $\varphi$ is given by $\coker\varphi=[\coker f\xrightarrow{\bar{u}'}\coker g]$, where $\ker^0 g$ is the reduced connected component of the kernel of $g$ in the category of commutative group schemes, $\ker^0f$ is the pullback of $\ker^0g$ along $u:\ker f\xrightarrow{u}G$, and $\bar{u}'$ is the map induced by $u'$.
\end{prop}
\begin{proof}
See \cite[Prop. C.1.3]{barbieri2016derived}.
\end{proof}

\section{Cartier duality}
Given a 1-motive $M$ over $S$ we can construct another 1-motive $M\ve$ called the Cartier
dual of $M$. This defines in fact a contravariant functor on the category of
1-motives
\[(.)\ve:\Mi(S)\to\Mi(S)\]
with the property that there exist a canonical isomorphisms $(M\ve)\ve\cong M$. This construction generalizes the duality of abelian schemes, given by the functor $\Ext^1_S(.,\G_m)$, and the Cartier duality of affine commutative group schemes taking tori to lattices
and vice versa, which is given by $\Hom_S(.,\G_m)$ (see \cref{sec: Cartier dual in appendix}).

\subsection*{Torsors.}
    We recall general facts about torsors which can be found in \cite[\S III]{Giraud_1971}. We fix a site $\cS$ and work in the topos $Sh(\cS)$ of sheaves on $S$. Let $H\in Sh(\cS)$ with $H$ a sheaf of abelian groups over $S$. By an $H$-torsor over $S$ we will mean a sheaf $P$ over $S$ endowed with an $H$-action $m: H\times_{S} P\to P$ such that:
    \begin{enumerate}
        \item[(1)] the morphism 
        \[
       H\times_S P\to P\times_S P,\, (h,e)\mapsto (m(h,e),e) 
        \]
        is an isomorphism, and 
        \item[(2)] the structural morphism $P\to S$ is an epimorphism.
    \end{enumerate}
    A morphism of torsors is a morphism of corresponding sheaves which is compatible with the actions. The trivial $H$-torsor is just $H$ with action given by multiplication.

    Condition (1) is equivalent to the following:
    For any $T\in Sh(\cS)$, the action of $H(T)=\Hom_S(T,H)$ on $P(T)=\Hom(T,P)$ is simply transitive.

    Condition (2) is equivalent to the following:
    There exists an epimorphic cover $\{S_i\to S\}$ such that $P\times_S S_i$ is the trivial $H\times_S S_i$-torsor over $S_i$.

\begin{defn}[\cite{Grothendieck_1972}]
Let $H, A,B$ be sheaves of abelian groups on a
site. A biextension of $(A,B)$ by $H$ is an $H_{A\times B}$-torsor $P$ over $A\times B$ which is endowed with a structure of extension of $B_A$ by $H_A$ and a structure of extension of $A_B$ by $H_B$, such that both structures are compatible.  
\end{defn}
\begin{example}
A biextension of abelian groups is a biextension on the punctual topos,
i.e., the category of sheaves on the site with one object $*$ and one morphism $id_{*}:*\to *$. Thus for abelian groups $H, A,B$, a biextension of $(A,B)$ by $H$ is given by a set $P$ endowed with a simply transitive $H$-action and a surjective function $P\to A\times B$ such that the fibers $P_a$, for every $a\in A$, have the structure of the extension of
$B$ by $H$, and the fibers $P_b$, for every $b\in B$ have the structure of the extension of $A$ by $H$.
\end{example}
\begin{example}
Given three commutative group schemes $G,H,N$ over a base field $K$, a biextension of $(G,H)$ by
$N$ is a morphism $B\to G\times_K H$ plus two relative homomorphisms. The first one, for $E\to H$, makes
$E\to H$ an extension of $G_H:=G\times_K H$ by $N_H:= N\times_K H$, while the second group law, for $E\to G$, makes $E\to G$ an extension of $H_G$ by $N_G$.
\end{example}

Consider the fppf site $S\fppf$. Let $A$ be an abelian scheme over $S$. The dual abelian scheme of $A$ is characterized by a pair $(A\ve,P_A)$, where $A\ve$ is isomorphic to $\underline{\Pic}^0_{A/S}$ as fppf sheaves (\cite{Faltings_Chai_1990}) and $P_A$ is a biextension of $(A,A\ve)$ by $\G_m$, called a Poincar\'e biextension. When composing with the isomorphism $\underline{\Pic}^0_{A/S}\cong \iext^1_S(A,\G_{m,S})$ we get the Weil-Barsotti formula $A\ve\cong\iext^1(A,\G_{m,S})$, which maps a section $a\ve\in A\ve$ to the fiber $P_{A_S,a\ve}$ of $P_{A_S}$ over $a\ve$. Since $\G_m$ is affine over $S$, $P_A$ is representable and it is locally trivial with respect to the \'etale topology on $S$, as $\G_m$ is smooth, i.e., $P_{A_S}$ is a $\G_m$-torsor over $A_S\times_SA\ve_S$ on \'etale site $S\etale$.
If $S$ is $\Spec K$, where $K$ is a complete discrete valuation field with ring of algebraic integers $\cO_K$, then $P_{A_{K}}$ extends canonically to a biextension $P_{A}$ of $(A_0,A_0\ve)$ by $\G_{m,\cO_K}$, where $A_0$ is the N\'eron model of $A_K$ over $\cO_K$, and $A_0\ve$ is the N\'eron model of $A\ve_K$ (see \cite[Expos\'e VIII]{Grothendieck_1972}).

In \cite{deligne1974theorie}, Deligne generalized the notion of biextension to complexes of sheaves.
\begin{defn}{\cite[\S 10.2.1]{deligne1974theorie}}
Let $C_1=[A_1\to B_1]$ and $C_2=[A_2 \to B_2]$ be two complexes of sheaves of abelian groups concentrated in degrees $0$ and $-1$. A biextension of $(C_1, C_2)$ by a sheaf of abelian groups $H$ consists of:
\begin{enumerate}
    \item[(1)] a biextension $P$ of $(B_1,B_2)$ by $H$,
    \item[(2)] a trivialization of the biextension of $(B_1,A_2)$ by $H$ obtained as the pullback of $P$ over $B_1\times A_2$, and
    \item[(3)] a trivialization of the biextension of $(A_1,B_2)$ by $H$ obtained as the pullback of $P$ over $A_1\times B_2$.
\end{enumerate}
Moreover, the trivialization conditions in (2) and (3) must coincide on $A_1\times A_2$.

We now review the Poincar\'e biextension associated with a 1-motive $M$, as described in \cite[\S 1.2]{andreatta2005crystalline}. We want to construct the Poincar\'e biextension $P$ of $(M,M\ve)$ by $\G_m$. Every $x\ve\in T\ve$ corresponds to the extension of $M_A$ by $\G_m$ obtained by the pushout of the exact sequence \[0\to T\to M\to M_A\to 0\] along $-x:T\to\G_m$
\begin{equation}
\begin{tikzcd}
0 \arrow[r] & T \arrow[r] \arrow[d, "-x"]                   & M \arrow[r] \arrow[d, "\phi", dashed]      & M_A \arrow[r] \arrow[d]                  & 0 \\
0 \arrow[r] & \G_m \arrow[r]                                & P'_{u\ve(x\ve)} \arrow[r]                  & M_A \arrow[r]                            & 0 \end{tikzcd}
.\end{equation}
The map $\phi$ induces a trivialization of the pullback $P_{u\ve(x\ve)}$ of $P'_{u\ve(x\ve)}$ along $M$
\begin{equation}
\begin{tikzcd}
0 \arrow[r] & \G_m \arrow[r] \arrow[d] & P_{u\ve(x\ve)} \arrow[r] \arrow[d] & M \arrow[r] \arrow[ld, "\phi"] \arrow[d] & 0 \\
0 \arrow[r] & \G_m \arrow[r]                                & P'_{u\ve(x\ve)} \arrow[r]                  & M_A \arrow[r]                            & 0
\end{tikzcd}
.\end{equation}
We can see that $P$ is indeed a canonical trivialization over $G\times L\ve$, as the pullback of $P$ to $G\times{x\ve}$ is $P_{u\ve(x\ve)}$ by construction. The biextension $P$ is called the Poincar\'e biextension of $(M,M\ve)$ by $\G_m$.

\end{defn}
\subsection*{Construction of the dual of $M$}
Let $F^{\bullet}$ and $G^{\bullet}$ be objects in $D^b(S\fppf)$. We have an internal $\ihom$ denoted by $\ihom_{S\fppf}(F^{\bullet},G^{\bullet})$, and also we have \[\iext^i_{S\fppf}(F^{\bullet},G^{\bullet})=H^i(\ihom_{S}(F^{\bullet},G^{\bullet})).\] This $\iext^i_{S\fppf}$ sheaf is indeed sheafification of the presheaf
\begin{gather*}
    (S\fppf)^{opp}\to Set\\ (T\to S)\mapsto \Ext^i_T(F^{\bullet}\mid_T,G^{\bullet}\mid_T)
\end{gather*}
with respect to the fppf site on $S$, where 
\[\Ext^i_S(F^{\bullet},G^{\bullet}):=\Hom_{D^b(S\fppf)}(F^{\bullet},G^{\bullet}[i])=\Hom_{D^b(S\fppf)}(F^{\bullet}[-i],G^{\bullet}).\]

We now proceed to the construction of the Cartier dual of $M=[L\xrightarrow{u} G]$. Let $S$ be a locally notherian scheme. As before, $G$ is a an extension of abelian scheme $A$ by a torus $T$. Denote $M_A:=[L\xrightarrow{v}A]$, where $v$ is a map that makes the following diagram commute 
\begin{equation}\label{5.2.1}
\begin{tikzcd}
            &             & L \arrow[d, "u"] \arrow[rd, "v"] &             &   \\
0 \arrow[r] & T \arrow[r] & G \arrow[r]                      & A \arrow[r] & 0
.\end{tikzcd}
\end{equation}
The dual of $M$ is a 1-motive $M=[T\ve\xrightarrow{u\ve} G\ve]$, where $G\ve$ is an extension of $A\ve$ by $L\ve$ and defined as follows:
\begin{enumerate}
    \item The dual of the torus $T$ is a lattice over $S$, i.e., $T\ve$ is the group scheme which represents the sheaf $\ihom_{S}(T,\G_m)$ (see \cref{sec: Cartier dual in appendix}).
    \item The dual of the lattice $L$ is a torus over $S$, i.e., $L\ve$ is the group scheme which represents $\ihom_{S}(L,\G_m)$.
    \item The dual of abelian scheme $A$ is $A\ve$, which is an abelian scheme over $S$ representing the sheaf $\iext^1_S(A,\G_m)$, by the Weil-Barsotti formula $A\ve\cong \iext^1_S(A,\G_m)$ (see \cite[Chapter III]{oort2006commutative}).
    \item We define $G\ve$ to be the group scheme over $S$ which represents the sheaf $\iext^1_S(M_A,\G_m)$. More explicitly, applying the functor $\iext^{\bullet}_S(.,\G_m)$ to the diagram 
    \begin{equation}\label{eq: 5.2.2}
        \begin{tikzcd}
            &             & 0                       &             &   \\
0 \arrow[r] & A \arrow[r] & M_A \arrow[u] \arrow[r] & L \arrow[r] & 0\\
            &             & M\arrow[u]                       &             &   \\
                        &             & T\arrow[u]                       &             &   \\
                                                &             & 0\arrow[u]                       &             &   
\end{tikzcd}
    \end{equation}
    gives the diagram 
    \begin{equation}\label{eq: 5.2.3}
\begin{tikzcd}[column sep=tiny]
                 &                               &                & \vdots\arrow[d]                                   &                              &\\
                 &                               &                & \ihom_S(M,\G_m) \arrow[d]                         &                              &        \\
                 &                               &                & T\ve \arrow[d, "u\ve"]                   &                              &        \\
\cdots \arrow[r] & {\ihom_{S}(A,\G_m)} \arrow[r] & L\ve \arrow[r] & {\iext_S^1(M_A,\G_m)} \arrow[r] \arrow[d] & {\iext_S^1(A,\G_m)} \arrow[r] &\iext^1_S(L,\G_m)\arrow[r]  & \cdots \\
                 &                               &                & \iext^1_S(M,\G_m)\arrow[d]                                   &                              &
                 \\
                 &                               &                & \vdots                                   &                              &
\end{tikzcd}
\end{equation}
The sheaf $\iext^1_S(M_A,\G_m)$ is representable and we denote by $G\ve$ the group scheme over $S$ that represents it. As $A$ is an abelian scheme (it is proper), $\ihom_{S}(A,\G_m)=0$, and by \cref{1.2.5}, $\iext^1_S(L,\G_m)=0$.
\end{enumerate}
Therefore, the dual of $M$ is $M\ve=[T\ve\xrightarrow{u\ve} G\ve]$ and we have the commutative diagram
\begin{equation}\label{eq: 5.2.4}
\begin{tikzcd}
            &                & T\ve \arrow[d, "u\ve"] \arrow[rd, "v\ve"] &                &   \\
0 \arrow[r] & L\ve \arrow[r] & G\ve \arrow[r]                            & A\ve \arrow[r] & 0
\end{tikzcd}
\end{equation}
A morphism $\varphi:M_1\to M_2$ induces a morphism of complexes in $\ref{eq: 5.2.2}$ and $\ref{eq: 5.2.3}$. As a result, it gives a morphism $\varphi\ve:M\ve_2\to M\ve_1$.

We can repeat the above construction to get a 1-motive $(M\ve)\ve$, 
\begin{equation}
\begin{tikzcd}
            &                & (L\ve)\ve \arrow[d, "(u\ve)\ve"] \arrow[rd, "(v\ve)\ve"] &                &   \\
0 \arrow[r] & (T\ve)\ve \arrow[r] & (G\ve)\ve \arrow[r]                            & (A\ve)\ve \arrow[r] & 0
\end{tikzcd}
\end{equation}
There are natural isomorphisms $(L\ve)\ve\cong L$, $(T\ve)\ve\cong T$, and $(A\ve)\ve\cong A$. This implies that $(G\ve)\ve\cong G$, and in particular $(M\ve)\ve\cong M$ as 1-motives.

The Cartier duality on $L$, $T$, and $A$, can uniquely extend to a duality on $\Mi(S)$.
If the functor $(.)^{D}:\Mi(S)\to\Mi(S)$ induces the Cartier duality $(.)\ve$ on $L$, $T$, and $A$, then $(.)^D$ induces natural isomorphisms between the weight filtrations $M^D$ and $M\ve$, for any $M\in\Mi(S)$. Hence, $(.)^D=(.)\ve$. 

\begin{cor}
   The Cartier dual $(.)\ve:\Mi(S)\to\Mi(S)$ is an exact contravariant functor with the property that $(.)\ve\,\ve\cong id_{\Mi}$. 
\end{cor}

\begin{thm}\label{1.2.5}
Let $X$ be either a finite flat group scheme, a torus or a lattice over $S$, and let $A$ be an abelian scheme over $S$. Then
\begin{enumerate}
    \item $\iext_S^1(X,\G_m)=0$.
    \item $\iext_S^2(A,\G_m)=0$.
    \item $\ihom_S(X,A)\cong \iext^1_S(A\ve,X\ve)$.
\end{enumerate}
\end{thm}

\begin{proof}
\begin{enumerate}
    \item If $X$ is a finite flat group, the proof can be found in \cite[Th. III.16.1]{oort2006commutative}. If $X$ is a constant group scheme or a torus, then it follows from \cite[exp. XIII]{Grothendieck_Demazure_2011} and \cite[exp. VIII]{grothendieck1967elements}.

    \item It is shown in \cite{Breen_1969} that $\iext^i_S(A,\G_m)$ are torsion for all $i>1$. Using the statement (1), we can see that the multiplication by $[m]$ on $\iext^2_S(A,\G_m)$ is injective so $\iext^2(A,\G_m)=0$.

    \item Let $f:X\to A$ be a morphism. It yields an exact sequence of complexes
\[0\to [0\to A]\to [X\to A]\to [X\to 0]\to 0\]
By (1), we know that $\iext^1(X,\G_m)=\ihom_{S}(A,\G_m)=0$. Applying $\iext^i_S(.\,,\G_m)$, we obtain the exact sequence 
\[
0\to X\ve\to \iext^1(M_A,\G_m)\to A\ve\to 0
\]
Therefore, we have obtained a map $\ihom_S(X,A)\to \iext^1_S(A\ve,X\ve),\, f\mapsto\iext^1(M_A,\G_m)$. For the inverse map, consider the exact sequence \[0\to X\ve\to G\to A\ve\to 0\]
Taking $\iext^i_S(.\,,(\G_m))$, we get 
\[0\to \ihom_S(G,\G_m)\to X\to A\to \iext^1_S(G,\G_m)\to 0\] This gives us a map $X\to A$.
\end{enumerate}
\end{proof}
\section{Points on 1-motive}
Let $M$ be a 1-motive over $S$. The definition of $S$-points of $M$ is inspired by \cite[\S 4.3]{deligne1974theorie} and also \cite[\S 7.1]{Andreatta_Bertapelle_2011}.
\begin{defn}
The group of $S$ points of $M$ is 
\[
M(S):=\Ext^1_S(M\ve_S,\G_m)\cong\Ext^1_S(\underline{\Z},M_S)\cong H^0\fppf(S,M)\cong\Hom_S(\Z,M_S)
\]
where the $\Hom$ and $\Ext^1$ are considered in the derived category of abelian fppf sheaves on fppf site $S\fppf$.
\end{defn}
For an abelian scheme $A$ over $S$, the previous identifications reduce to the Weil-Barsotti formula 
\[
M(S)=\Ext^1(A_S\ve,\G_{m,S})=\Hom_S(\Z,A_S)=A_S(S)
\]
\begin{remark}\label{S-point of 1-motive when split}
Let $M\ve=[T\ve\to G\ve]$ be the dual of the 1-motive $M=[L\to G]$ over $S$ and $M\ve_A=[T\ve\xrightarrow{v\ve}A\ve]$ (see \ref{1.2.5}). We have
\begin{equation}
    M\ve_A(S)=\Ext^1(G,\G_m)=G(S)
\end{equation}
The exact sequence 
\[
0\to L\ve\to M\ve\to M\ve_A\to 0
\]
of complexes of fppf sheaves induces a long exact sequence
\begin{equation}
0\to \Hom_S(M\ve,\G_m)\to L(S)\to G(S)\to M(S)\to H^1(S\fppf,L)\to \cdots   
\end{equation}
The group $H^1(S\fppf,L)$ is equal to $H^1(S\etale,L)$ as $L$ is \'etale over $S$. By \cite[Prop. VIII.5.1]{grothendieck1967elements}, if $S'\to S$ is a finite flat \'etale Galois extension with Galois group $G$ such that $L\times_SS'$ is constant, then we have $H^1(S'\etale,L)=0$.
Thus, the group $H^1(S\etale,L)=H^1(G,L)$. The group $H^1(G,L)$ is a finite Galois cohomology, so it is torsion.

Hence, over a locally noetherian base scheme $S$, we have 
\begin{equation}
    M(S)\otimes_{\Z}\Q=(G(S)/\im(L))\otimes_{\Z} \Q.
\end{equation}
Particularly, when $S=\Spec(K)$ and $L$ is split over $K$, then $H^1(S\fppf,L)=0$ and we have an exact sequence 
\[
L(K)\to G(K)\to M(K)\to 0,
\]
and $M(K)=G(K)/\im(u_K)$.
\end{remark}
\begin{remark}
    We can give a more explicit description of an element in $M(S)\otimes_{\Z}\Q$. Assume that $M=[L\xrightarrow{u}G]$ is a 1-motive over $S$. Let $x\in G(S)$ and  $M_x=[L\oplus\Z\to G]$ be a 1-motive induced by the map $(\ell,1)\mapsto u(\ell)+x$. The canonical exact sequence
\[
0\to M\to M_x\to \Z\to 0
\]
in $\Ext^1_S(\Z,M)$ depends only on the class of $x$ in $G(S)/\im(u)$. Taking the Cartier dual $(.)\ve$ gives the element 
\[
0\to \G_m\to M_x\ve\to M\ve\to 0
\]
in $M(S)$. Conversely, if $y\in M(S)$, the identification $(M(S)\otimes_{\Z}\Q=G(S)/\im(u))\otimes_{\Z}\Q$ implies that there is a power of $y$ which comes from a point $x\in G(S)$. Moreover, $y$ corresponds to the pull-back of $x$ along $M\ve\to G\ve$, i.e.,
\begin{equation*}
    \begin{tikzcd}
0 \arrow[r] & \G_m \arrow[r] \arrow[d] & E_y \arrow[r] \arrow[d] & M\ve \arrow[r] \arrow[d] & 0 \\
0 \arrow[r] & \G_m \arrow[r]           & E_x \arrow[r]           & G\ve \arrow[r]           & 0
\end{tikzcd}
\end{equation*}
Taking $(.)\ve$ gives
\begin{equation*}
\begin{tikzcd}
0 \arrow[r] & M \arrow[r] \arrow[d] & E_y\ve \arrow[r] \arrow[d] & \Z \arrow[r] \arrow[d] & 0 \\
0 \arrow[r] & G\ve \arrow[r]           & E_x\ve \arrow[r]           & \Z \arrow[r]           & 0
\end{tikzcd}
\end{equation*}
As $E_x\ve$ corresponds to a map $\Z\to G$, it gives us a 1-motive $[\Z\xrightarrow{u}\ G]$. Therefore, $E_y$ is a 1-motive of the form $[L\oplus\Z\xrightarrow{v} G]$. As the above diagram is commutative, we have $v(\ell,n)=u(\ell)+nx$.
\end{remark}
We can summarize the above remark in the following
\begin{cor}\label{cor: explicit description of M(S) points}
Let $M=[L\xrightarrow{u} G]$ be a 1-motive over $S$. We have 
\begin{equation}
    M(S)\otimes_{\Z}\Q=\Big\{0\to M\to M_x\to \Z\to 0 \st \, x\in G(S),\, M_x=[L\oplus\Z\xrightarrow{(\ell,n)\mapsto u(\ell)+nx}G] \Big\}\otimes_{\Z}\Q
\end{equation}
\end{cor}

\section{The de Rham realization for 1-motive}\label{sec: de Rham realisation}
In this section, we start by defining the universal vector extension for 1-motives. For the construction of $M^{\natural}$, the universal vector extension of a Deligne 1-motive $M$ over a field of characteristic zero, we refer to \cite[10.1.7]{deligne1974theorie}. For the more general case over a base scheme $S$, see \cite[Prop. 2.2.1]{barbieri2009sharp}, and \cite[\S 2.3, 2.4]{andreatta2005crystalline}.

Recall that a vector group scheme over $S$ is an $S$-group scheme that is locally isomorphic for the fpqc topology\footnote{fpqc stands for fidèlement plate et quasi-compacte.} to a finite product of $\G_a$'s. If $V$ is a vector group over $S$ then the sheaf $\ihom_S(.,V)$ is a locally free $\cO_S$-module of finite rank. Conversely, every locally free $\cO_S$-module $\cM$ of finite rank induces a vector group $V$ whose sections over an $S$-scheme $T$ are $\cM(T)=\Gamma(T,\cO_T\otimes_{\cO_S}\cM)$.

A commutative group scheme $G$ over $S$ is semi-abelian if and only if $\ihom_S(G,V)=0$ for all vector group scheme $V$ over $S$.

\begin{defn}
Let $G$ be a semi-abelian group scheme.
A vector extension of $G$ over $S$ is an extension of $G$ by a vector group scheme $V$ over $S$.

The vector extension \[0\to V(G)\to \Ex(G)\to G\to 0\] is called the universal vector extension of $G$ if for any vector extension 
\[0\to V\to G'\to G\to 0,\] there exists a unique homomorphism of $S$-vector group schemes $\varphi:V(G)\to V$ such that $G'$ is the push-out of $\Ex(G)$ by $\varphi$. 
\end{defn}
\begin{remark}
Indeed, the vector extension 
   \[0\to V(G)\to \Ex(G)\to G\to 0\]
   is universal if and only if the map \[\Hom_{\cO_S}(V(G),V)\to\Ext^1_S(G,V)\] induced by push-out is an isomorphism for all vector groups $V$ over $S$.
    
    If the following conditions are satisfied:
    \begin{enumerate}
        \item[I.] $\ihom_{S}(G,\G_{a,S})=0$,
        \item[II.] $\iext^1_S(G,\G_{a,S})$ is a locally free $\cO_S$-module of finite rank,
    \end{enumerate}
    then, $V(G):=\ihom_{\cO_S}(\iext^1_S(G,\G_{a,S}),\cO_S)$ is a vector group scheme over $S$ and the universal vector extension of $G$ exists. Moreover, we have
    \[\iext^1_S(G,V)=\iext^1_S(G,\G_{a,S})\otimes_{\cO_S}V=\ihom_{\cO_S}(V(G),\cO_S)\otimes_{\cO_S}V=\ihom_{\cO_S}(V(G),V).
    \]
    Thus, $\Ex(G)$, by definition, is the extension corresponding to the identity morphism on $V(G)$.
\end{remark}
\begin{prop}
Let $G$ be a semi-abelian group scheme over $S$. The universal vector extension 
\begin{equation}\label{Universal vector extension G}
    0\to V(G)\to \Ex(G)\to G\to 0
\end{equation}
exists, where $V(G)=\ihom_{\cO_S}(\iext^1_S(G,\G_{a,S}),\cO_S)$ and there is an isomorphism $$\Ex(G)\cong \Ex(A)\times_A G.$$
\end{prop}
\begin{proof}
See \cite[\S 1.7]{mazur2006universal}.
\end{proof}
\begin{defn}
A vector extension of a 1-motive $M=[L\to G]$ over $S$ is an extension of $M$ by a vector group scheme $V$ over $S$. That is an exact sequence \[0\to [0\to V]\to [L\to G']\to [L\to G]\to 0\] as complexes of $S$-group schemes. We usually denote $V$ by $V:=[0\to V]$.
\end{defn}
\begin{defn}
The vector extension \[0\to V(M)\to M^{\natural}\to M\to 0\] is called the universal vector extension of $M$ if for any vector extension $$0\to V\to M'\to M\to 0,$$ there exists a unique homomorphism of $S$-vector group schemes $\varphi:V(M)\to V$ such that $M'$ is the push-out of $M^{\natural}$ by $\varphi$. 
\end{defn}

The vector extension 
   \[0\to V(M)\to M^{\natural}\to M\to 0\]
   is universal if and only if the map
   \begin{equation}\label{2.4.2}
   \Hom_{\cO_S}(V(M),V)\to\Ext^1_S(M,V)    
   \end{equation}
   induced by push-out is an isomorphism for all vector groups $V$ over $S$.

    If the following conditions are satisfied:
    \begin{enumerate}
        \item[I.] $\ihom_{S}(M,\G_{a,S})=0$,
        \item[II.] $\iext^1_S(M,\G_{a,S})$ is a locally free $\cO_S$-module of finite rank,
    \end{enumerate}
    then, $V(M):=\ihom_{\cO_S}(\iext^1_S(M,\G_{a,S}),\cO_S)$ is a vector group scheme over $S$ and the universal vector extension of $M$ exists.

\begin{prop}\label{prop 2.4.2}
Let $M=[L\xrightarrow{u}G]$ be a 1-motive over $S$. The universal vector extension 
\begin{equation}\label{Universal extension M}
    0\to V(M)\to M^{\natural}\to M\to 0
\end{equation}
exists, where $V(M)=\ihom_{\cO_S}(\iext^1_S(M,\G_{a,S}),\cO_S)$ and $M^{\natural}$ is given by $M^{\natural}=[L\xrightarrow{u^{\natural}}G^{\natural}]$. Moreover, we have an extension of $S$-group schemes 
\begin{equation}\label{universal' extension Gnat}
0\to V(M)\to G^{\natural}\to G\to 0 
\end{equation}
such that $G^{\natural}$ is the push-out of the universal vector extension \[0\to V(G)\to \Ex(G)\to G\to 0\] of semi-abelian scheme $G$ along the inclusion \[V(G)=\ihom_{\cO_S}(\iext^1_S(G,\G_{a,S}),\cO_S)\hookrightarrow\ihom_{\cO_S}(\iext^1_S(M,\G_{a,S}),\cO_S)=V(M)\] and consequently, there is a non-canonical isomorphism $G^{\natural}\cong \Ex(G)\times_S (L\otimes_{\Z}\G_{a,S})$.
\end{prop}
%
\begin{proof}
We prove the conditions (I), and (II). Observe that $\ihom_{S}(G,\G_{a,S})=0$ implies that $\ihom_{S}(M,\G_{a,S})=0$. To prove (II), consider the exact sequence 
\[
0\to G\to M\to L\to 0
\]
and apply $\iext^i(.\,,\G_{a,S})$ to get an exact sequence
\[
0\to \ihom_{S}(L,\G_{a,S})\to \iext^1_{S}(M,\G_{a,S})\to\iext^1_{S}(G,\G_{a,S})\to 0.
\]
Notice that $\ihom_{S}(G,\G_{a,S})=\iext^1_S(L,\G_a)=0$. Then $\iext^1_{S}(M,\G_{a,S})$ will be a locally free sheaf of $\cO_S$-modules of finite rank if both $\ihom_{S}(L,\G_{a,S})$ and $\iext^1_{S}(G,\G_{a,S})$ have finite rank. For $\ihom_{S}(L,\G_{a,S})$, we know that it is of finite rank. For $\iext^1(G,\G_{a,S})$, notice that the exact sequence 
\[
0\to T\to G\to A\to 0
\]
yields an isomorphism $\iext_S^1(A,\G_{a,S})\cong\iext_S^1(G,\G_{a,S})$, since we know that $\ihom_S(T,\G_{a,S})\cong\iext_S^1(T,\G_{a,S})=0$. Hence, $\iext_S^1(G,\G_{a,S})$ is of finite rank as well.

For the proof of the second part of the statement, see \cite[Prop. 2.2.1]{barbieri2009sharp}.
\end{proof}
\begin{example}
Consider the 1-motive $L=[L\to 0]$ over $S$. We have 
\[
V(L)=\ihom_{\cO_S}(\iext^1_S(L,\G_{a,S}),\cO_S)=\ihom_{\cO_S}(\ihom_S(L,\G_{a,S},\cO_S),\cO_S)=L\otimes\G_{a,S}.
\]
Let $\psi:L\to\ihom_{\cO_S}(\ihom_S(L,\G_{a,S},\cO_S),\cO_S)=L\otimes\G_{a,S}=L\otimes\G_{a,S}$ given by
\[
x\mapsto ( f\mapsto f(x))
\]
for any $f\in\ihom_{S}(L,\G_{a,S})$. Then the isomorphism \ref{2.4.2} sends the identity to $$\psi\in \Hom(L,L\otimes\G_{a,S})=\Ext^1(L,L\otimes\G_{a.S}).$$ It follows that $L\nat=[L\xrightarrow{f}L\otimes\G_{a,S}]$.
\end{example}

\begin{remark}\label{V(M)=coLie(Gv)}
By \cite[Propposition 2.3]{bertapelle_delignes_2009}, we know that 
\[
\iext^1_S(M,\G_{a,S})=\iext^1_S(M_A,\G_{a,S})=\underline{\Lie}( G\ve),
\]
and therefore $$V(M)=\ihom_{\cO_S}(\iext^1_S(M,\G_{a,S}),\cO_S)=\ihom_{\cO_S}(\Lie G\ve,\cO_S)=\underline{\coLie}(G\ve).$$ Hence, we have that
\begin{equation*}
\begin{tikzcd}
0 \arrow[r] & V(G) \arrow[r] \arrow[d,Rightarrow, no head] & V(M) \arrow[r] \arrow[d,Rightarrow, no head] & V(L) \arrow[r] \arrow[d, Rightarrow, no head] & 0 \\
0 \arrow[r] & \coLie(A\ve) \arrow[r]                        & \coLie(G\ve) \arrow[r]                        & \coLie(T\ve) \arrow[r]                        & 0
.\end{tikzcd}
\end{equation*}

\end{remark}

\begin{defn}
The de Rham realisation of the 1-motive $M=[L\to G]$ over $S$ is defined as \[\TdR(M):=\underline{\Lie}_{G^{\natural}}(S)=\Lie(G\nat).\]

In addition to the standard weight filtration, the de Rham realisation carries the Hodge filtration 
\[
\Fil^i\TdR(M)=\begin{cases}
    V(M)=\ker(\Lie(G\nat)\to\Lie(G)), & i=0,\\
    \TdR(M), & i=-1,\\
    0 & i\neq 0,-1.
\end{cases}
\]
\end{defn}
We have a canonical exact sequence 
\[
0\to V(M)\to\Lie(G\nat)\to\Lie(G)\to 0.
\]

Assume that $M$ is a 1-motive defined over a field $K$. Let $\pi:M^{\natural}\to M$ in \ref{Universal extension M} be the map $\pi=(id_L,\pi_G)$, where $\pi_G:G^{\natural}\to G$ and $\ker \pi_G=V(M)$. Therefore, the kernel of the induced map $d\pi_G: \Lie(G^{\natural})=\TdR(M)\to \Lie(G)$ is again $V(M)\subseteq \TdR(M)$. Thus, $\TdR(M)$ together with the $K$-subspace $V(M)$ can be regarded as a filtered $K$-vector space.

\begin{prop}{\cite[Lemma 2.3.2]{barbieri2009sharp}.}
    The functor $M\mapsto M\nat$ is exact.
\end{prop}

\begin{defn}
If $M=[L\xrightarrow{u}G]$ is a 1-motive defined over a subfield $K$ of $\C$. The singular realisation of $M$, denoted by $\Tsing(M)$, is the fibre product of $L$ and $\Lie(G^{an})$ over $G^{an}$ under the structure map $u:L\to G$ and the exponential map $\exp: \Lie(G^{an})\to G^{an}$.
\end{defn}
\begin{prop}\label{Betti-de Rham comparison iso}
Let $M$ be a Deligne 1-motive (or 1-motive with torsion) defined over the subfield $K$ of $\C$.
\begin{enumerate}
    \item We have a natural isomorphism $$\TdR(M_{\C})\xrightarrow{\cong}\TdR(M_K)\otimes_K\C$$ and 
    $$T_{\text{sing}}(M_{\C})\xrightarrow{\cong}T_{\text{sing}}(M_K)\otimes_K\C$$
    where $M_{\C}$ is the base change of $M_K$ to $\C$.
    \item The unique homomorphism $$\varpi: T_{\text{sing}}(M_{\C})\to \TdR(M)\otimes_K\C$$ that yields $d\pi_G\circ\varpi=\tilde{u}_{\C}$ and $exp\circ\varpi =u^{\natural}_{\C}\circ\tilde{exp}$ induces an isomorphism 
    \begin{equation}\label{Betti-de Rham comparison}
    \varpi_{\C}: T_{\text{sing}}(M_{\C})\otimes_{\Z}\otimes\C\to \TdR(M)\otimes_K\C,
    \end{equation}
    where $\tilde{u}_{\C}$ is the pull-back of $u_{\C}:L_{\C}\to G_{\C}$ along $exp$ and $\tilde{exp}: T_{\text{sing}}(M_{\C})\to L_{\C}$ is the structural map given by the definition of $T_{sing}(M_{\C})$. The isomorphism $\varpi$ is called the period isomorphism.
    \item The period pairing map $per:\TdR^{\vee}(M)\times T_{\text{sing}}(M)\to \C$ given by $per(\omega,\sigma)=\omega_{\C}(\varpi(\sigma))$ is indeed equal to $\int_{\gamma}\omega $ where $\gamma$ is a path from $0$ to a point $exp(\sigma)\in G(\bar{\Q})$.
\end{enumerate}
\end{prop}
\begin{proof}
See \cite[Chap. 9]{huber2022transcendence}, or \cite[\S 2.2]{andreatta2020motivic}. 
\end{proof}
\begin{defn}
    The isomorphism \ref{Betti-de Rham comparison} is called Betti-de Rham comparison isomorphism.
\end{defn}

\section{$\ell$-adic realisations for 1-motive}\label{sec: tate-module of motives}
Let $M=[L\xrightarrow{u}G]$ be a 1-motive over $S$ in which we shall place $L$ in degree $-1$ and $G$ in degree $0$. Let $n$ be a positive integer. Consider the multiplication by $n$ on $M$, $n\colon M\to M$, consisting of multiplication-by-$n$ maps on both $L$ and $G$. Its associated commutative diagram 
\begin{equation*}
\xymatrix{L \ar[r]^{u} \ar[d]^{n}& G\ar[d]^{n}\\
L\ar[r]^{u}& G
 }\end{equation*}
induces a morphism of $S$-group schemes $L\to L\times_G G, x\mapsto (nx,-u(x))$ and we define 
$$M[n]:=\coker(L\to L\times_G G).$$
In other words, $M[n]$ is $H^{-1}(M/n)$ where $M/n$ is the cone of multiplication by $n$ on $M$. The exact sequence (\ref{canonical exact sequence for any 1-motive}) yields a short exact sequence of cohomology sheaves
\begin{equation}\label{exact sequence M[n]}
0\to G[n]\to M[n]\to L[n]\to 0.
\end{equation}
Indeed, we have
\begin{equation}\label{formula M[n]}
M[n]=\frac{\{(x,g)\in L\times G\mid u(x)=-ng\}}{\{(nx,-u(x)\mid x\in L\}}
\end{equation}
as an fppf quotient. In particular, when $S=\Spec K$, $M[n]$ turns into a $(\Z/n\Z)[\Gamma_K]$-module.
\begin{remark}
    Let $n$ be a positive integer and let $G$ be a semi-abelian scheme over $S$. We have the commutative exact diagram 
\begin{equation}
\begin{tikzcd}
0 \arrow[r] & T \arrow[d, "{[n]}"] \arrow[r] & G \arrow[d, "{[n]}"] \arrow[r] & A \arrow[d, "{[n]}"] \arrow[r] & 0 \\
0 \arrow[r] & T \arrow[r]                    & G \arrow[r]                    & A \arrow[r]                    & 0.
\end{tikzcd}
\end{equation}
    Hence, applying the Snake lemma yields the exact sequences
    \[
    0\to T[n]\to G[n]\to A[n]\to 0
    \]
    and 
    \begin{equation}\label{3.2.6}
    0\to G[n]\to G\xrightarrow{[n]}G\to 0.
    \end{equation}
    We know that multiplication by $[n]$ is finite and faithfully flat on both $A$ and $T$, and both $T[n]$ and $A[n]$ are finite flat group schemes over $S$, and they are \'etale if $n$ is coprime to the characteristics of all residue fields of $S$. As an extension of finite flat group schemes is again a finite flat group scheme, we conclude that $G[n]$ is a finite flat group scheme over $S$. As the sequence \ref{3.2.6} is exact and $G[n]$ is finite flat, therefore $G\xrightarrow{[n]}G$ is finite and faithfully flat, and it is \'etale if $n$ is coprime to the characteristics of all residue fields of $S$.
\end{remark}
\begin{remark}
If $M$ is a 1-motive over $S$, the exact sequence \ref{exact sequence M[n]} implies that $M[n]$ is a finite flat group scheme over $S$, which is \'etale  if $S$ can be defined over $\Z[\frac{1}{n}]$. In particular, for $S=\Spec K$, $M[n]$ is a finite flat \'etale group scheme if $n$ is coprime to the characteristic of $K$. 
\end{remark}
\begin{defn}
The p-divisible group (or Barsotti-Tate group) of $M$ is $$M[p^{\infty}]:=\colim_{n\to\infty} M[p^n]$$
where the direct limit is taken over maps $M[p^m]\to M[p^n]$, for $m\geq n$, induced by $(x,g)\mapsto (p^{m-n}x,g)$.
\end{defn}
\begin{defn}
Let $p$ be a fixed prime number and $M$ a 1-motive over $K$. The $p$-adic Tate module (or $p$-adic realization) of 1-motive $M$ is $$\Tp(M):=\lim_{m}M[p^m]$$ where the inverse limit is taken over maps $M[p^m](\bar{K})\to M[p^n](\bar{K})$, for $m\geq n$, induced by $(x,g)\mapsto (x,p^{m-n}g)$. We also denote $\Vp(M):=\Tp(M)\otimes_{\Z_p}\Q_p$.
\end{defn}
The exact sequence (\ref{exact sequence M[n]}) yields the exact sequence
\begin{equation}\label{exact sequence M[p^infty]}
    0\to G[p^{\infty}]\to M[p^{\infty}]\to L[p^{\infty}]\to 0
\end{equation}
where $L[p^{\infty}]=L\otimes \Q_p/\Z_p$. For $M=[0\to A]$ an abelian scheme we recover the Barsotti-Tate group of $A$.

Note that we also have the canonical exact sequence 
\begin{equation}\label{exact sequence for semi-abelian G[p^infty]}
0\to T[p^{\infty}]\to G[p^{\infty}]\to A[p^{\infty}]\to 0,
\end{equation}
and if we take Tate-modules, we obtain a canonical exact sequence 
\begin{equation}\label{3.2.7}
    0\to \Tp(G)\to \Tp(M)\to\Tp(L)\to 0
\end{equation}
where, $\Tp(L)\cong L\otimes\Z_p$.
\begin{prop}
There exists a finite flat extension $S'$ of $S$ such that the exact sequence (\ref{exact sequence M[p^infty]}) is split over $S'$.
\end{prop}
\begin{proof}
    Choose a finite flat extension $S'\to S$ and a compatible set of homomorphism 
    $$\{u_n\colon\frac{1}{p^n}(L\times_S S')\to G\times_S S'\}_{n\in\N}$$
    such that $u_0=u\times_S S'$. For any $n$, we define $L[p^{n}]\times_S S'\to M[p^n] $ by $\bar{x}\mapsto (x,-u_n(p^{-n}x))$. Hence, we a a splitting of (\ref{exact sequence M[p^infty]}) over $S'$.
\end{proof}
\begin{example}
The $p$-adic Tate module of $[0\to A]$ is just a Tate-module of the abelian scheme $A$, i.e., $\Tp(A)=\lim_{m}A[p^m]$ which has rank $2\dim(A)$, if $p$ is coprime to the characteristic of all residue fields of $S$. 
\end{example}

\begin{prop}
Let $M=[L\to G]$ be a 1-motive over arbitrary field $K$, where $G$ is an extension of abelian variety $A$ by a torus $T$.
\begin{enumerate}
    \item The Tate module $\Tp(M)$ is a free $\Z_{p}$-module. If $p$ is coprime to the characteristic of $K$, then
    \[
    \rank\Tp(M)=\rank(L)+\dim(T)+2\dim(A).
    \]
    \item The action of the absolute Galois group $\Gamma_K$ on $\Tp(M)$ is continuous.
    \item The association $M\mapsto \Tp(M)$ is a covariant exact functor from $\Mi(K)$ to the category of finitely generated free $\Z_{p}$-modules with continuous Galois action.
\end{enumerate}
\end{prop}
\begin{proof}
    It is clear if one considers the exact sequences (\ref{exact sequence M[p^infty]}), (\ref{3.2.7}), and (\ref{exact sequence for semi-abelian G[p^infty]}).
\end{proof}

\begin{remark}\label{Faltings' theorem for Tate module of 1-motives}
The Tate-Faltings theorem on homomorphisms of abelian varieties can be generalized to 1-motives. For 1-motives $M$ and $M'$ over finitely generated field $k$ one has isomorphism 
$$\Hom(M,M')\otimes\Z_{\ell}\xrightarrow{\cong}\Hom_{G_k}(T_{\ell}(M),T_{\ell}(M')) $$
when $\ell$ is coprime to characteristic of $k$. See \cite{jannsen1995mixed}.
\end{remark}

\section{Reduction types and deformation theory}\label{section reduction types and deformation theory}
\begin{defn}
Let $G$ be a group scheme over a field $K$ and let $G$ be a p-divisible group (abelian variety, torus, lattice, or semi-abelian variety resp.).
\begin{enumerate}
    \item If $K$ is a local field of mixed characteristic $(0,p)$ with residue field $k$, we say that $G$ has a good reduction if $G$ can be extended to $\cO_K$ i.e. there exists a group scheme $G_0$ over $\cO_K$ such that $G_0$ is a p-divisible group (resp. abelian variety, torus, lattice, or semi-abelian variety) and $G_0\times_{\Spec(\cO_K)}\Spec(K)$ is isomorphic to $G$. In this case, we call $\bar{G}:=G_0\times_{\cO_K}\Spec(k)$ the reduction of $G$.
    \item If $K$ is a global field, we say that $G$ a has good reduction at a prime $\fp\subset\cO_K$ if $G$ can be extended to $(\cO_K)_{\fp}$ i.e. there exists a group scheme $G_0$ over $(\cO_K)_{\fp}$ such that $G_0$ is a p-divisible group (abelian variety, torus, lattice, or semi-abelian variety resp.) and $G_0\times_{(\cO_K)_{\fp}}\Spec(K)$ is isomorphic to $G$. In this case, we call $\bar{G}:=G_0\times_{(\cO_K)_{\fp}}\Spec(k)$ the reduction of $G$ modulo $\fp$.
\end{enumerate}
\end{defn}

    Assume the above notations. Let $K$ be a global field. $G$ has a good reduction at a prime $\fp$ in $\cO_K$ if and only if $G\times_{K}\Spec(K_{\fp})$ has a good reduction.

\begin{thm}[N\'eron-Ogg-Shafarevich criterion, \cite{bosch1990werner}, Theorem 7.4.5]
Let $A$ be an abelian variety over a field $K$ with perfect residue field of characteristic $p$. Let $\ell\neq p$ be a prime. Then $A$ has good reduction at $p$ if and only if the $\ell$-adic Tate module $\operatorname{T_{\ell}}(A)$ is unramified as a Galois representation of $\Gamma_K$.
\end{thm}

\begin{thm}[\cite{de1998homomorphisms}, \S 2.5]
Let $R$ be a henselian discrete
valuation ring with fraction field $K$. Let $p$ be any prime, and let $A$ be an abelian variety over $K$. Then, $A$ has good reduction if and only if the $p$-divisible group associated to $A$, $A[p^{\infty}]$, has good reduction.
\end{thm}

Let $S=\Spec(R)$ be a base scheme on which $p$ is nilpotent, $I$ a nilpotent ideal of $R$ and $S_0=\Spec(R/I)$.  Let $\operatorname{Def_{p-div}}(S)$ be the category of triples $(\bar{A},G,\epsilon)$ consisting of an abelian scheme $\bar{A}$ over $S_0$, a $p$-divisible group $G$ over $S$ and an isomorphism $\epsilon\colon \bar{A}[p^{\infty}]\cong G\times_S S_0$. Regarding the deformation of abelian schemes, we have the following theorem:
\begin{thm}[Serre-Tate deformation theorem, \cite{katz1981serre}]
The functor $A\mapsto (\bar{A},A[p^{\infty}],\epsilon)$ from the category of abelian schemes over $R$ to the category $\operatorname{Def_{p-div}}(S)$ is an equivalence.
\end{thm}
\begin{remark}
As a special case of Serre-Tate deformation theorem, if $R=\cO_K$ and  $\bar{A}$ is an
abelian variety over the residue field $k$, then a deformation of $\bar{A}[p^{\infty}]$ to a p-divisible group $G$ over $\cO_K$ corresponds to a deformation of $\bar{A}$ to an abelian scheme $A$ over $\cO_K$.
\end{remark}

We can generalize the theorem of Serre-Tate and Grothendieck on the deformation of abelian schemes to 1-motives. This is done in \cite{bertapelle_deformations_2019}. Let $S$ be a base scheme on which $p$ is locally nilpotent. Let $S_0\to S$ be a nilpotent thickening of schemes i.e. $S_0\to S$ is a closed immersion whose the ideal of definition $I$ is locally nilpotent. If $M$ is an object over $S$, its base change to $S_0$ is denoted by $\bar{M}$. Let $\operatorname{Def^{\Mi}_{p-div}(S)}$ be the category of triples $(\bar{M},G,\epsilon)$ consisting of a 1-motive $\bar{M}$ over $S_0$, a p-divisible group $G$ over $S$ and an isomorphism $\epsilon\colon \bar{M}[p^{\infty}]\cong G\times_S S_0$. 
\begin{thm}[\cite{bertapelle_deformations_2019}]
The functor $M\mapsto (\bar{M},M[p^{\infty}],\epsilon)$ from the category $\Mi(S)$ to the category $\operatorname{Def^{\Mi}_{p-div}(S)}$ is an equivalence.
\end{thm}

\begin{defn}[1-motive with good reduction]
Let $M$ be a 1-motive over field $K$.
\begin{enumerate}
    \item If $K$ is a local field of mixed characteristic $(0,p)$ with ring of algebraic integers $(\cO_K,\fp,k)$, we say that $M$ has a good reduction if it can be extended to $\cO_K$, i.e., there exists a 1-motive $M_0$ in $\Mi(\cO_K)$ whose generic fibre $M_0\times_{\Spec(\cO_K)} \Spec(K)$ is isomorphic to $M$. In that case we call the 1-motive $\bar{M}:=M_0\times_{\Spec(R)}\Spec(k)$ the reduction of $M$ modulo $\fp$.
    
    \item If $K$ is a global field, we say that $M$ has a good reduction at prime $\fp\subset\cO_K$ if $M$ can be extended to $(\cO_K)_{\fp}$ i.e. there exists a 1-motive $M_0$ in $\Mi((\cO_K)_{\fp})$ whose generic fibre $M_0\times_{\Spec(\cO_K)_{\fp}} \Spec(K)$ is isomorphic to $M$. In that case we call the 1-motive $\bar{M}:=M_0\times_{\Spec(\cO_K)_{\fp}}\Spec(k)$ the reduction of $M$ modulo $\fp$.
\end{enumerate}
\end{defn}
Let $\fp$ be a fixed prime in a field $K$. The category of 1-motives over $K$ with good reductions at $\fp$ is denoted by $\Mi^{gr}(K)$. A 1-motive $M$ over a number field $K$ has a good reduction at $\fp$ if and only if its lift $M\otimes_K K_\fp$ has good reduction, where $K_\fp$ denotes the completion of $K$ at $\fp$.

\begin{prop}
Let $M$ be a 1-motive over number field $K$. $M$ has good reduction at all but finitely many primes $\fp$.
\end{prop}
\begin{proof}
    See \cite[Corollary 4.2.7]{matev2014good}.
\end{proof}

\begin{thm}\label{good reduction and unramified Tate module}
Let $M$ be a 1-motive over $K$, where $K$ is either a number field or a p-adic local field. Then $M$ has a good reduction at $\fp$ if and only if $\Tl(M)$ is unramified at a prime $\ell$ which is different from $\fp$.
\end{thm}
\begin{proof}
    See \cite[Theorem 4.1.1 ,Theorem 4.2.8]{matev2014good}.
\end{proof}
\begin{prop}
Let $R$ be a henselian discrete valuation ring with fraction field $K$ and the maximal ideal $p$. Let $M$ be a 1-motive over $K$. Then $M$ has a good reduction if and only if the Barsotti-Tate group associated to $M$, $M[p^{\infty}]$, has a good reduction.
\end{prop}
\begin{proof}
When $M=[0\to A]$ an abelian variety or $M=[0\to \G_m]$ a torus, the result follows from \cite[Expos\'e IX]{Grothendieck_1972} and \cite[\S 2.5]{de1998homomorphisms}. Consequently, for lattices $[L\to 0]$, which are dual of tori, the statement holds. Therefore, the statement is valid for any 1-motive $M$. 
\end{proof}
\begin{prop}
The isogeny category of $\Mi^{gr}(K)$ is an abelian category.
\end{prop}
\begin{proof}
Without loss of generality we can assume that $K$ is a complete local field. We know that the isogeny category $\Mi(K)\otimes\Q$ is an abelian category. To prove that the isogeny category of $\Mi^{gr}(K)$ is an abelian category, it suffices to show that for each exact sequence 
    \[
    0\to M_1\to M\to M_2\to 0
    \]
    in $\Mi(K)\otimes\Q$, $M$ has a good reduction if and only if both $M_1$ and $M_2$ have good reductions. Taking the Tate module, we obtain the exact sequence
    \[
    0\to \Tl(M_1)\to \Tl(M)\to \Tl(M_2)\to 0
    \]
    of $\Gamma_K$-equivariant $\Z_p$-modules, where $\ell$ is prime different from $p$. The Galois representation $\Tl(M)$ is unramified if and only if both $\Tl(M_1)$ and $\Tl(M_2)$ are unramified. By \cref{good reduction and unramified Tate module}, the statement follows. 
\end{proof}
\begin{cor}\label{cor: good reduction exact sequence of 1-motives}
    Let 
    \[
    0\to M_1\to M\to M_2\to 0
    \]
    be an exact sequence of 1-motives over $K$. Then $M$ has a good reduction at $p$ if and only if both $M_1$ and $M_2$ have good reductions at $p$.
\end{cor}
\begin{proof}
    The proof closely follows the same reasoning as the previous argument.
\end{proof}

\begin{prop}\label{Hodge-Tate weights of tate module of 1-motives} 
Let $K$ be a complete local field with perfect residue field and $M$ a 1-motive with good reduction over $K$. The Tate module $\Vp(M)=\Tp(M)\otimes_{\Z_p}\Q_p$ is a Hodge-Tate representation of weights $0$ and $1$ with multiplicity $\rank(L)+\dim(A)$ and $\dim(T)+\dim(A)$ respectively. 
\end{prop}
\begin{proof}
Without loss of generality, we can assume that $M$ is a 1-motive over $\cO_K$. Consider the canonical exact sequence 
\[
    0\to G[p^{\infty}]\to M[p^{\infty}]\to L[p^{\infty}]\to 0.
    \]
As $L[p^{\infty}]$ is \'etale, the connected component of $ M[p^{\infty}]$ is the same as the connected component of $G[p^{\infty}]$. Thus, $\Lie( M[p^{\infty}])=\Lie( G[p^{\infty}])$. Taking Cartier dual from the above sequence, we get the exact sequence 
\[
0\to (L[p^{\infty}])\ve\to(M[p^{\infty}])\ve\to (G[p^{\infty}])\ve\to 0 
\]
of p-divisible groups over $\cO_K$. The dimension of p-divisible group $(L[p^{\infty}])\ve$ is equal to $\rank(L)$, since $L[p^{\infty}]=L\otimes\Q_p/\Z_p$, and $(L[p^{\infty}])\ve\cong (\mup)^{\rank(L)}$. We need to find the dimension of $(G[p^{\infty}])\ve$ as well. The exact sequence \ref{exact sequence for semi-abelian G[p^infty]} gives us an exact sequence 
\[
0\to (A[p^{\infty}])\ve\to (G[p^{\infty}])\ve\to (T[p^{\infty}])\ve \to 0
\]
of p-divisible groups over $\cO_K$. The dimension of p-divisible group $(G[p^{\infty}])\ve$ is equal to $\dim(A)$, since $(A[p^{\infty}])\ve\cong A\ve[p^{\infty}]$ and $(T[p^{\infty}])\ve$ is \'etale and of dimension $0$. Thus, $\dim(\Lie(M[p^{\infty}]\ve))=\rank(L)+\dim(A)$.

\cref{Hodge-Tate weights of tate module} implies that $\Vp(M)$ is Hodge-Tate of weights $0$ and $1$ with multiplicity $\dim(\Lie(M[p^{\infty}]\ve))=\rank(L)+\dim(A)$ and $\dim (\Lie (G))=\dim(T)+\dim(A)$ respectively.
\end{proof}

\section{Crystalline realization for 1-motive}\label{section crystalline realization for 1-motive}
Recall our notations in \cref{sec: Dieudonne}. Let $k$ be a prefect field of characteristic $p$, $\W(k)$ the Witt vectors of $k$ and $M[p^{\infty}]$ the p-divisible group associated to $M$ defined over $k$. If we take the contravariant Dieudonn\'e functor, it gives us a module $\D(M[p^{\infty}])$ over the Dieudonn\'e ring 
\[\sD_k:=\W(k)[\sF,\sV]/(\sF\sV-p,\sV\sF-p,Fc-\sigma(c)\sF,c\sV-\sV\sigma(c) ,\,\, \forall\sigma\in\W(k)).\]
As pointed out in \cref{sec: crystalline nature of p-divisible groups}, this can be further extended to define a crystal on the nilpotent crystalline site on $\Spec k$. With the notations in \cref{crystal on crystalline site k}, we can define the Barsotti-Tate crystal of the 1-motive $M$ as follows.
 
\begin{defn}[\cite{andreatta2005crystalline}]
Let $k$ be a perfect field of characteristic $p$ and let $\D$ denote the contravariant Dieudonn\'e crystal. The crystalline realization of 1-motive $M$ is the following $\W(k)$-module
\[\Tcrys(M):=\lim_{n}\D(M[p^{\infty}]^{\vee}])(\Spec k\to \Spec \W_n(k)).\]
We also define 
\[\Tcrys\ve(M):=\lim_{n}\D(M[p^{\infty}]])(\Spec k\to \Spec \W_n(k)).\]
We call $\Tcrys\ve(M)$ the Barsotti-Tate crystal of the 1-motive $M$.
\end{defn}
The functor associating to a 1-motive its p-divisible group is exact and covariant. The Dieudonn\'e functor and Cartier dual are exact and contravariant. Therefore, the functor $M\mapsto\Tcrys(M)$ is exact and covariant.

For 1-motive $M=[L\to G]$, we define a weight filtration on $\Tcrys(M)$ as follows
\begin{equation}
    \W^i(\Tcrys(M))\begin{cases}
        \Tcrys(M), & i\geq 0,\\
        \Tcrys(G), & i=-1,\\
        \Tcrys(T), & i=-2,\\
        0, & i\leq -3.
    \end{cases}
\end{equation}
Where $G$ is an extension of an abelian scheme $A$ by a torus $T$.

We have a canonical isomorphism between crystalline realization of $M$ and the de Rham realization of its lift to $\W(k)$, as described in \cite[Theorem A']{andreatta2005crystalline}. Hence, the following theorem provides a crystalline-de Rham comparison isomorphism for 1-motives with good reduction.

\begin{thm}[\cite{andreatta2005crystalline}]\label{crystalline-de Rham comparison isomorphism}
    Let $M$ be a 1-motive with good reduction over local p-adic field $K$. We have a canonical isomorphism
    \[
    \Tcrys(\bar{M})\otimes_{\W(k)}K\cong\TdR(M_K).
    \]
\end{thm}
We can define a filtration on the isocrystal $N=\Tcrys(\bar{M})\otimes_{\W(k)}K$ by transferring the Hodge filtration on $\TdR(M)$ via the crystalline-de Rham comparison isomorphism from \cref{crystalline-de Rham comparison isomorphism}. This turns $N$ into a filtered isocrystal\footnote{Recall \cref{definition of filtered isocrystal} for filtered isocrystal.}. Moreover, we can view $D=\Tcrys(\bar{M})$ as a filtered Dieudonn\'e module\footnote{Recall \cref{filtered Dieudonn\'e module} for filtered Dieudonn\'e module.}, where the filtration $(D^i)$ is the one induced by the Hodge filtration on the de Rham realization of the (formal) lifting $\bar{M}$ to $\W(k)$. Thus, we can state the following:
\begin{cor}\label{cor: Hodge filtration and tangent space on Dieudonne module}
    Let $M=[L\to G]$ be a 1-motive over p-adic local filed $K$ with good reduction. Let $D=\Tcrys(\bar{M})$. Then we have an exact sequence 
    \[
    0\to D^0\to D\to D/D^0\to 0
    \]
    with natural identification $D/D^0\otimes_{\W(k)}K\cong\Lie(G)$.
\end{cor}


\clearpage
\chapter{P-adic integration theory for 1-motives}
P-adic integration theory plays a crucial role in modern number theory, arithmetic geometry, and algebraic geometry. Developed and refined by several mathematicians, including Fontaine, Messing, Colman, and Colmez, this theory extends classical integration methods into the realm of p-adic analysis, offering a powerful framework for understanding arithmetic properties of varieties over p-adic fields.

Jean-Marc Fontaine's contributions laid the groundwork for understanding p-adic Hodge theory, a field closely linked to p-adic integration. His work established a deep connection between the arithmetic of p-adic representations and geometric structures, particularly through the use of period rings. This connection has become fundamental in modern p-adic integration.

William Messing furthered the development of p-adic integration theory by studying the deformation theory of p-divisible groups and crystalline cohomology. His work provided essential tools for understanding the deformation of varieties over p-adic fields, a key aspect of p-adic integration.

Peter Coleman and Pierre Colmez made significant advancements by refining and expanding the theory. Colman's contributions include the development of the theory of p-adic differential equations, which has profound implications for p-adic integration. Colmez, on the other hand, is known for his work on p-adic periods and his exploration of the p-adic analogs of classical integrals, particularly in the context of p-adic L-functions and p-adic Hodge theory.

In the first part of this chapter, we seek to extend the p-adic integration theory introduced in \cite{colmez1992periodes} to 1-motives with good reduction. We will show that this p-adic integration pairing is perfect and respects Hodge filtration (\cref{theorem p-adic integration pairing for M}). Furthermore, we will explore our p-adic integration maps in greater detail through \cref{kernel Fontaine map phi} and \cref{kernel p-adic integration map varpi}. The p-adic integration theory in \cite{colmez1992periodes} focuses on constructing p-adic periods for an abelian variety with good reduction defined over a local field. The methods employed highlight the analogies between p-adic theory and classical theory over the complex field, making these connections explicit. The paper concludes with a detailed comparison of this calculation with the theory of Hodge-Tate periods, initially established by Tate and further developed by Fontaine, Coleman, and others. The periods arising from this pairing cannot generally be expressed in terms of function-theoretically when the abelian variety varies in a family (i.e. this integration is nit compatible with morphisms between abelian varieties). However, there is another context in which abelian periods appear: when the integrand is viewed as a solution to certain linear differential equations (Picard-Fuchs or Gauss-Manin). For a more detailed comparison of these methods, see \cite{andre_period_2003}.

In the second part of this chapter, we will study the p-adic logarithm through Barsotti-Tate groups and identify its local inverse. The main goal is to gain a clearer interpretation of the image of $\barQ$-rational points under the $\log$ map (e.g., \cref{theorem classification for algebraic poin in image of logarithm}), using techniques from p-adic Hodge theory and Galois cohomology, inspired by \cite{cartier_l-functions_2007}.

For this chapter, we assume that $K$ is a local discrete valued field with prefect residue field $k$ of characteristic $p$, and with ring of integers $\cO_K$. Recall our notations in \cref{sec: period rings} and the definition of period rings $\BdRp$ and $\BdR$. We define $\At=\Ainf/(\ker\theta)^2$ and $\Bt:=\BdRp/t^2\BdRp$; see \cref{def B2}. 

\section{Geometric interpretation of period rings}
Consider the natural inclusion $\BdRp\into\BdR$. Recall that there exists a canonical section of $\theta:\BdRp\to\C_p$ above $\bar{K}\subset\C_p$, that is $\bar{K}\into\BdRp$ whose image is dense in $\BdRp$. This embedding induces a map $\bar{K}\to\Bt$ which is injective.

From \cref{sec: period rings}, recall the surjective ring homomorphism $\theta\colon \Ainf\to\cO_{\C_p}$. We let $I:=\ker\theta$, $J:=\ker(\theta[1/p])$, $\bar{I}=I/I^2$, and $\bar{J}=J/J^2$. Notice that $\At/\bar{I}\cong\ocp$ and $\bar{J}\cong\C_p(1)$. By construction we have the following commutative diagram with exact rows of $\Gamma_K$-equivariant maps
\begin{equation}\label{diagram Ibar A2 Jbar B2}
\begin{tikzcd}
0 \arrow[r] & \bar{I} \arrow[r] \arrow[d, hook] & \At \arrow[r] \arrow[d, hook] & \cO_{\C_p} \arrow[r] \arrow[d, hook] & 0 \\
0 \arrow[r] & \bar{J} \arrow[r]                 & \Bt \arrow[r]                 & \C_p \arrow[r]                       & 0
\end{tikzcd}
\end{equation}
    Let $\Omega=\Omega_{\Spec(\cO_{\bar{K}})/\Spec(\cO_K)}$ be the sheaf of modules of differentials with the universal derivation $d\colon\cO_{\bar{K}}\to \Omega$ and let $\cA:=\ker d$. The multiplication by $p^n$ induces a commutative diagram of $\cO_K$-modules
    \begin{equation}
\begin{tikzcd}
0 \arrow[r] & \cA \arrow[r] \arrow[d, "{[p^n]}", tail] & \cO_{\bar{K}} \arrow[r] \arrow[d, "{[p^n]}", tail] & \Omega \arrow[r] \arrow[d, "{[p^n]}", two heads] & 0 \\
0 \arrow[r] & \cA \arrow[r]                            & \cO_{\bar{K}} \arrow[r]                            & \Omega \arrow[r]                                 & 0.
\end{tikzcd}
    \end{equation}
Notice that multiplication by $p^n$ is surjective on $\Omega$, and it is injective on $\cO_{\bar{K}}$. The snake lemma yields an exact sequence
$$0\to \Omega[p^n]\to\cA/p^n\cA\to\cO_{\bar{K}}/p^n\cO_{\bar{K}}\to 0.$$
The direct system $\{\Omega[p^n]\}_{n\geq 1}$ with natural transition maps $\Omega[p^{n+1}]\onto\Omega[p^n]$ satisfies the Mitag-Leffler condition (\cite[\href{https://stacks.math.columbia.edu/tag/0594}{Section 0594}]{stacks-project}). Hence, taking projective limits yields an exact sequence
$$0\to\Tp\Omega\to \hat{\cA}\to\cO_{\C_p}\to 0,$$
where $\hat{\cA}=\lim\cA/p^n\cA$. The ring $\BdRp$ is the universal pro-infinitesimal thickening of $\cO_{\C_p}$. This universal property for $\At$ provides a unique map $\At\to\hat{\cA}$ which is an isomorphism and induces a commutative diagram
\begin{equation}
\begin{tikzcd}
0 \arrow[r] & \bar{I} \arrow[r] \arrow[d, "\cong"] & \At \arrow[r] \arrow[d, "\cong"] & \cO_{\C_p} \arrow[r] \arrow[d, Rightarrow, no head] & 0 \\
0 \arrow[r] & \Tp\Omega \arrow[r]                  & \hat{\cA} \arrow[r]              & \cO_{\C_p} \arrow[r]                                & 0
\end{tikzcd}
\end{equation}
where the maps are $\Gamma_K$-equivariant. Inverting $p$ we obtain isomorphisms $\bar{J}\cong(\Tp\Omega)[1/p]\cong\C_p(1)$, and $\Bt\cong\hat{\cA}[1/p]$.

For the rest of this chapter, assume that $M=[L\xrightarrow{u}G]$ is a 1-motive over $\cO_K$. The Tate module of $M$ is $\Tp(M)=\lim M[p^n](\bar{K})=\lim M[p^n](\cO_{C_p})$. The first goal of this chapter is to construct a bilinear pairing map 
\[\Tp(M)\times \TdR\ve(M)\to\At\]
whose generic fibre induces a bilinear, perfect pairing
\[\int\colon \Tp(M_K)\times \TdR\ve(M_K)\to\Bt\]
which is $\Gamma_K$-equivariant in the first argument and respects the Hodge filtration. This pairing is called p-adic integration pairing for the 1-motive $M$. By $\displaystyle\int_{x}\omega$, we mean $\displaystyle\int(x,\omega)$ for any $x\in\Tp(M)$ and $\omega\in\TdR\ve(M)$.  By perfect, we mean that $\displaystyle\int$ is non-degenerate:
\begin{align}
    \bullet \text{ For each } x\in\Tp(M),\text{ if }\displaystyle\int_x\omega=0 \text{ for all } \omega\in\TdR\ve(M_K), \text{ then } x=0.\nonumber \\
    \bullet \text{ For each } \omega\in\TdR\ve(M_K), \text{ if } \displaystyle\int_x\omega=0 \text{ for all } x\in\Tp(M), \text{ then } \omega=0.\label{perfect def meaning}
\end{align}
This is equivalent to saying that the induced map $\Tp(M)\otimes_{\Z_p}\Bt\to \TdR(M_K)\otimes_K\Bt$ is an isomorphism.

\begin{defn}
    Let $M=[L\to G]$ be a 1-motive. We define the Hodge filtration on $\TdR\ve(M)$ as follows
    \begin{equation}
        \Fil^i\TdR\ve(M)=\begin{cases}
            \TdR\ve(M), & i\leq 0,\\
            \coLie(G), & i=1,\\
            0, & i\geq 2.
        \end{cases}
    \end{equation}
\end{defn}
\begin{defn}
Recall that $\Fil^i\BdR=t^i\BdRp/t^{i+1}\BdRp\cong\C_p(i)$, for $i\in\Z$. We define the following filtration for the ring $\Bt$:
\[\Fil^i\Bt=\begin{cases}
    \Bt, & i\leq 0\\
    \C_p(1), & i=1\\
    0, & i\leq 2.
\end{cases}\]
\end{defn}
Saying that the pairing $\displaystyle\int$ respects filtration means that $\displaystyle\int_x\omega\in\Fil^i\Bt$ if $\omega\in\Fil^i\TdR\ve(M)$.
\section{Construction of $\displaystyle\int$ for $M=[0\to G]$}
Let $G$ be a semi-abelian scheme over $\cO_K$ which is an extension of abelian scheme $A$ by a torus $T$.

    Infinitesimal lifting criterion for formal smoothness of group schemes $G$ implies that the natural map $G(\At)\to G(\cO_{\C_p})=G(\At/\bar{I})$ is surjective and its kernel is $\Lie(G)\otimes_{\cO_K}\bar{I}$. We can write the following exact sequence
    \begin{equation}\label{exact seq for Fontaine'spairing}
        \begin{tikzcd}
    0 \arrow[r] & \Lie(G)\otimes_{\cO_K}\bar{I} \arrow[r] & G(\At) \arrow[r] & G(\cO_{\C_p}) \arrow[r] & 0.
        \end{tikzcd}
    \end{equation}
Multiplication by $p^n$ gives a commutative diagram
    \begin{equation}
        \begin{tikzcd}
0 \arrow[r] & \Lie(G)\otimes_{\cO_K}\bar{I} \arrow[r] \arrow[d, "{[p^n]}"] & G(\At) \arrow[r] \arrow[d, "{[p^n]}"] & G(\cO_{\C_p}) \arrow[r] \arrow[d, "{[p^n]}"] & 0 \\
0 \arrow[r] & \Lie(G)\otimes_{\cO_K}\bar{I} \arrow[r]                      & G(\At) \arrow[r]                      & G(\cO_{\C_p}) \arrow[r]                       & 0
\end{tikzcd}
    \end{equation}
    and the snake lemma gives a $\Gamma_K$-equivariant map 
    \[\phi_n\colon G[p^n](\cO_{\C_p})\to \Lie(G)\otimes_{\cO_K}\bar{I}/p^n\bar{I}.\]
    Recall that we have a natural isomorphism $\Tp\Omega\cong\bar{I}$, hence $\bar{I}/p^n\bar{I}\cong\Omega[p^n]$. This yields a map
    \[\phi_n\colon G[p^n](\cO_{\C_p})\to \Lie(G)\otimes_{\cO_K}\Omega[p^n]\]
    Taking inverse limit and inverting $p$, we obtain a map 
    \[\displaystyle\varphi_G\colon\Tp(G)\to \Lie(G)\otimes_{K}\C_p(1). \]
\begin{defn}
We call the map $\varphi_G$, defined above, Fontaine's map for semi-abelian scheme $G$. The pairing $\Tp(G)\times \Lie\ve(G)\to \C_p(1)$ induced by Fontaine's map $\varphi_G$ is called Fontaine's pairing.
\end{defn}

\begin{thm}\label{Fontaine map tensored by Cp surjective}
The Fontaine's map induces a surjective map \[\phi_G\otimes\C_p:\Tp(M)\otimes\C_p\to\Lie(G)\otimes_{\cO_K}\C_p(1).\]
\end{thm}
\begin{proof}
Let $x\colon\Spec_{\cO_{\bar{K}}}\to G$ be an integer point. This induces the map $$x^*\colon \Omega_{G/\Spec\cO_K}\to\Omega_{\Spec\cO_{\bar{K}}/\Spec\cO_K}=\Omega,$$ and the map $\Omega_{G}(G)\to\Hom_{\Z[\Gamma_K]}(G(\cO_{\bar{K}}),\Omega)$ given by $\omega\mapsto (\eta_{\omega}\colon x\mapsto x^*(\omega))$ for any $\omega\in \Omega_G(G)$ and $x\in G(\cO_{\bar{K}})$. Composing it with natural injective map $$\Hom_{\Z[\Gamma_K]}(G(\bar{K}),\Omega)\into\Hom_{\Z[\Gamma_K]}((\Tp G)[1/p],(\Tp\Omega)[1/p])$$ and the restriction map on $\Tp(G)$, we get a $K$-linear map
\[\rho\colon \Omega_G(G)\to\Hom_{\Z[\Gamma_K]}(\Tp(G),\C_p(1)). \]
The Tate-Raynaud theorem, \cite[Theorem 2]{fontaine1982formes}, states that $\rho$ is injective and functorial on $G$. Notice that the dual of $\rho$ is $\rho\ve:\Tp(G)\otimes_{\cO_K}\C_p\to \Lie(G)\otimes\C_p(1)$. It suffices to show that the map $\phi_G\otimes\C_p:\Tp(M)\otimes\C_p\to\Lie(G)\otimes_{\cO_K}\C_p(1)$ coincides with $\rho\ve$. We outline the proof from \cite{colmez1992periodes}, or \cite[3.7.1]{p-adicintegration-iovita} for the case of an abelian scheme, which extends in a similar manner to the semi-abelian scheme $G$.

Note that we have an exact sequence
\[0\to \Tp\Omega\to\Vp\Omega\xrightarrow{s}\Omega\to 0\]
of $\Gamma_K$-modules, where $s$ is defined on elementary
tensors as:
\[s((x_n)_{n\geq 0}\otimes (1/p^m)) = x_m.\]
One can check that $s$ is well-defined and surjective $\Z_p[\Gamma_K]$-homomorphism, such that $\ker s=\Tp\Omega$. We can define a map $\alpha:\cO_{\bar{K}}\to\bar{J}/\bar{I}$, as $x\mapsto x_1-x_2$, where $x_1$ is the image of $x$ in $\Bt$, and $x_2$ is a lift of $x$ in $\At$. It can be shown that this map is well-defined. Moreover, via the isomorphism
\[
\bar{J}/\bar{I}\cong\Vp(\Omega)/\Tp(\Omega)\cong\Omega
\]
we can identify $\alpha$ with the differential $d:\cO_{\bar{K}}\to\Omega$.

Similarly, we extend this construction to $G(\cO_{\bar{K}})$ in the following way. Let $x\in G(\cO_{\bar{K}})$, $x_1$ the image of $x$ in $G(\Bt)$, and $x_2$ be a lift of $x$ in $G(\At)$. Then $x_1\equiv x_2\pmod{\bar{I}}$, and $$\beta:G(\cO_{\bar{K}})\to \ker(G(\Bt)\to G(\bar{I})),\, \beta(x)=x_1-x_2$$ is independent of the choices of $x_1$ and $x_2$. But, 
\[\ker(G(\Bt)\to G(\bar{I}))=\Lie(G)\otimes_{\cO_K}\bar{J}/\bar{I}\cong\Lie(G)\otimes_{\cO_K}\Omega. \] Therefore the map $\eta_{\omega}:\omega\mapsto x^*(\omega)$ coincides with the map $$\delta_x:\cO_{X,x}\to\Omega,\, \omega\mapsto\omega(x_1)-\omega(x_2)\pmod{\bar{I}}.$$ Recall that the we had 
\[
\phi_n:G[p^n](\cO_{\C_p})\to\Lie(G)\otimes(\bar{J}/p^n\bar{J})\cong\Lie(G)\otimes(\Omega[p^n])
\]
where $\phi_n(x)=p^nx_2\pmod{p^n\bar{J}}$, as $p^nx_1=p^nx=0$. Moreover, we have $\beta(x)=\phi_n(x)$. The result follows.
\end{proof}

Next, we introduce the p-adic integration pairing \[\int^{\varpi}:\Tp(G)\times\coLie(\Ex(G))\to \Bt,\]

where \[0\to V\to \Ex(G)\to G\to 0\] is the universal vector extension of $G$, and $V=\Ext^1(G,\G_a)\ve$. We construct the following diagram with exact rows and columns
\begin{equation}\label{big diagram for integration G}
\begin{tikzcd}
            & 0 \arrow[d]                                       & 0 \arrow[d]                                            & 0 \arrow[d]                                       &   \\
0 \arrow[r] & \Lie(V)\otimes_{\cO_K}\bar{I} \arrow[d] \arrow[r] & \Lie(\Ex(G))\otimes_{\cO_K}\bar{I} \arrow[d] \arrow[r] & \Lie(G)\otimes_{\cO_K}\bar{I} \arrow[r] \arrow[d] & 0 \\
0 \arrow[r] & V(\At) \arrow[d] \arrow[r]                        & \Ex(G)(\At) \arrow[d] \arrow[r]                        & G(\At) \arrow[r] \arrow[d]                        & 0 \\
0 \arrow[r] & V(\cO_{\C_p}) \arrow[d] \arrow[r]                 & \Ex(G)(\cO_{\C_p}) \arrow[d] \arrow[r]                 & G(\cO_{\C_p}) \arrow[r] \arrow[d]                 & 0 \\
            & 0                                                 & 0                                                      & 0                                                 &  
\end{tikzcd}
\end{equation}
From the above diagram we can observe that the map $\Ex(G)(\At)\to G(\cO_{\C_p})$ is surjective and its kernel is \[\cK:=\frac{(V\otimes_{\cO_K}\At)\oplus(\Lie(\Ex(G))\otimes_{\cO_K}\bar{I})}{\Lie(V)\otimes_{\cO_K}\bar{I}}\]
where we view $\Lie(V)\otimes_{\cO_K}\bar{I}$ as a submodule of $$(V\otimes_{\cO_K}\At)\oplus(\Lie(\Ex(G))\otimes_{\cO_K}\bar{I})$$ via diagonal embedding. Notice that we have an identification $\Lie(V)\cong V$. Hence, we have an exact sequence
\begin{equation}\label{4.2.7}
\begin{tikzcd}
0 \arrow[r] & \cK \arrow[r] & \Ex(G)(\At) \arrow[r] & G(\cO_{\C_p}) \arrow[r] & 0.
\end{tikzcd}
\end{equation}
Multiplication by $p^n$ gives a commutative diagram 
\begin{equation}
\begin{tikzcd}
0 \arrow[r] & \cK \arrow[d, "{[p^n]}"] \arrow[r] & \Ex(G)(\At) \arrow[d, "{[p^n]}"] \arrow[r] & G(\cO_{\C_p}) \arrow[d, "{[p^n]}"] \arrow[r] & 0 \\
0 \arrow[r] & \cK \arrow[r]                      & \Ex(G)(\At) \arrow[r]                      & G(\cO_{\C_p}) \arrow[r]                      & 0
\end{tikzcd},
\end{equation}
and the snake lemma yields a map $G[p^n](\cO_{\C_p})\to \cK/p^n\cK$. By composing with the natural map \[\cK/p^n\cK\to \Lie(\Ex(G))\otimes_{\cO_K}\At/p^n\At\] we get a map
\[\varpi_{n,G}\colon G[p^n](\cO_{\C_p})\to \Lie(\Ex(G))\otimes_{\cO_K}\At/p^n\At\]
Taking the inverse limits, we obtain the map
\[\varpi_G\colon \Tp(G)\to\TdR(G)\otimes_{\cO_{K}}\At\]
which is called p-adic integration map. We call the pairing 
\[\int^{\varpi}\colon \Tp(G_K)\times\TdR\ve(G_K)\to \Bt\]
induced by $\varpi_G$ the p-adic integration pairing, where $G_K:=G\times_{\Spec\cO_K}\Spec K$ is the generic fibre of $G$. We arrive at the following theorem, which in the case of abelian scheme $G=A$, was proven by P. Colmez (\cite{colmez1992periodes}).
\begin{thm}\label{theorem integration pairing for G}
The p-adic integration pairing $\displaystyle\int^{\varpi}$ is bilinear, perfect, and $\Gamma_K$-equivariant in the first argument. Moreover, It respects the Hodge filtration in the following sense: for all $\omega\in\Fil^1\TdR\ve(G_K)$ and all $\nu\in\Tp(G)$,
\[\int^{\varpi}_\nu\omega\in \Fil^1\Bt.\]
\end{thm}
\begin{proof}
    The fact that the pairing is bilinear and $\Gamma_K$-equivariant in the first argument is by construction. For the case $G=\G_m$, we have $G\nat=\G_m$, and $V(G)=0$ since $\Ext^1(\G_m,\G_a)=0$. Therefore, $\cK=\Lie(\G_m)\otimes_{\cO_K}\bar{I}$ and the exact sequence \ref{4.2.7} becomes 
    \[
    0\to \Lie(\G_m)\otimes_{\cO_K}\bar{I}\to \G_m(\At)\to \G_m(\ocp)\to 0
    \]
    The map $\cK/p^n\cK\to\Lie(\G_m)\otimes_{\cO_K}\At/p^n\At$ coincides with the inclusion map $$\Lie(\G_m)\otimes_{\cO_K}\bar{I}/p^n\bar{I}\to \Lie(\G_m)\otimes_{\cO_K}\At/p^n\At.$$ The map $\G_m[p^n](\ocp)\to\Lie(\G_m)\otimes_{\cO_K}\bar{I}/p^n\bar{I}$ is induced by $x\mapsto[p^n]\tilde{x}$, where $\tilde{x}$ is a lift of $x$ to $\G_m(\At)$. Therefore, after taking the limit and generic fibre, we get the natural map
    \begin{equation}\label{4.2.9}
    \Z_p(1)\to \Lie(\G_m)\otimes_{K}\C_p(1)\to \Lie(\G_m)\otimes_{K}\Bt.
    \end{equation}
    Thus, tensoring with $\Bt$ gives a $\Gamma_K$-equivariant isomorphism 
    \[
    \Z_p(1)\otimes_{\Z_p}\Bt\to\Lie(\G_m)\otimes_{\cO_K}\Bt.
    \]
    The corresponding pairing $\Z_p(1)\times\coLie(\G_m)\to\Bt$ respects the filtration, since, by \ref{4.2.9}, the image of $\Z_p(1)$ lies in $\Lie(\G_m)\otimes_{\cO_K}\C_p(1)$.

    For the case of abelian scheme $G=A$, the statement follows from \cite[Theorem 5.2]{colmez1992periodes}. We now deal with the case of semi-abelian scheme $G$ over $\cO_K$ with the exact sequence 
    \[
    0\to T\to G\to A\to 0.
    \]
    We have the commutative diagram 
\begin{equation*}
\begin{tikzcd}
0 \arrow[r] & \Tp(T) \arrow[r] \arrow[d]     & \Tp(G) \arrow[r] \arrow[d]  & \Tp(A) \arrow[r] \arrow[d]  & 0 \\
0 \arrow[r] & \TdR(T)\otimes\Bt \arrow[r] & \TdR(G)\otimes\Bt \arrow[r] & \TdR(A)\otimes\Bt \arrow[r] & 0
\end{tikzcd}
\end{equation*}
with exact rows. Now, the integration pairing which is induced by the middle vertical arrow is perfect because both integration pairings induced by left and right vertical arrows are perfect. The fact that the pairing respects the filtration will be shown in a more general case in \cref{theorem p-adic integration pairing for M}.
\end{proof}

\section{Construction of $\displaystyle\int$ for $M=[L\to G]$}
The universal vector extension of $M=[L\xrightarrow{u}M]$ is given by $M\nat=[L\xrightarrow{u\nat} G\nat]$ which induces the exact sequence (see \cref{prop 2.4.2})
\begin{equation}\label{exact sequence for colman pam for 1-motive}
0\to V(M)\to G\nat\to G\to 0.
\end{equation}
In the same way we obtained the diagram $\ref{big diagram for integration G}$, we can come up with the following commutative diagram with exact rows and columns
\begin{equation}\label{big diagram for integration M}
\begin{tikzcd}
            & 0 \arrow[d]                                       & 0 \arrow[d]                                            & 0 \arrow[d]                                       &   \\
0 \arrow[r] & \Lie(V(M))\otimes_{\cO_K}\bar{I} \arrow[d] \arrow[r] & \Lie(G\nat)\otimes_{\cO_K}\bar{I} \arrow[d] \arrow[r] & \Lie(G)\otimes_{\cO_K}\bar{I} \arrow[r] \arrow[d] & 0 \\
0 \arrow[r] & V(M)(\At) \arrow[d] \arrow[r]                        & G\nat(\At) \arrow[d] \arrow[r]                        & G(\At) \arrow[r] \arrow[d]                        & 0 \\
0 \arrow[r] & V(M)(\cO_{\C_p}) \arrow[d] \arrow[r]                 & G\nat(\cO_{\C_p}) \arrow[d] \arrow[r]                 & G(\cO_{\C_p}) \arrow[r] \arrow[d]                 & 0 \\
            & 0                                                 & 0                                                      & 0                                                 &  
\end{tikzcd}
\end{equation}
According to this diagram, the map $G\nat(\At)\to G(\cO_{\C_p})$ is clearly surjective and its kernel is 
\[\cK:=\frac{(V(M)\otimes_{\cO_K}\At)\oplus(\Lie(G\nat)\otimes_{\cO_K}\bar{I})}{\Lie(V(M))\otimes_{\cO_K}\bar{I}}.\]
Hence, we have an exact sequence
\begin{equation}\label{7.3.2}
\begin{tikzcd}
0 \arrow[r] & \cK \arrow[r] & G\nat(\At) \arrow[r] & G(\cO_{\C_p}) \arrow[r] & 0.
\end{tikzcd}
\end{equation}
Define the maps
\begin{align*}
q\colon L\times_{\cO_K} G\to G, \, (x,g)\mapsto u(x)+p^ng \\
q\nat\colon L\times_{\cO_k} G\nat\to G\nat,  \, (x,g)\mapsto u\nat(x)+p^ng,
\end{align*}
We denote the kernel of $q$ by $\ker q:=\tilde{M[p^n]}$. Then, by \cref{formula M[n]} the group scheme $M[p^n]$ is the quotient of $\tilde{M[p^n]}$ by the image of $L$.
These maps together with the exact sequence (\ref{7.3.2}) induces a commutative diagram with exact rows
\begin{equation}\label{4.3.4}
\begin{tikzcd}
0 \arrow[r] & \cK \arrow[r] \arrow[d, "{[p^n]}"] & L\times_{\cO_K}G\nat(\At) \arrow[r] \arrow[d, "q\nat"] & L\times_{\cO_K}G(\cO_{\C_p}) \arrow[r] \arrow[d, "q"] & 0 \\
0 \arrow[r] & \cK \arrow[r]                         & G\nat(\At) \arrow[r]                                   & G(\cO_{\C_p}) \arrow[r]                               & 0.
\end{tikzcd}
\end{equation}
Applying the snake lemma gives a map
\[\tilde{M[p^n]}\to \cK/p^n\cK.\]
By composing with the natural map 
\[\cK/p^n\cK\to \Lie(G\nat)\otimes_{\cO_K}\At/p^n\At\] we get map
$\tilde{M[p^n]}\to \Lie(G\nat)\otimes_{\cO_K}\At/p^n\At$. Notice that this map factors through the quotient of $\tilde{M[p^n]}$ by the image of $L$ since the image of $L$ in $\tilde{M[p^n]}$ vanishes under the map $q\nat$. Hence, we can get a map
\[\varpi_{n,M}\colon M[p^n]\to \TdR(M)\otimes_{\cO_K}\At/p^n\At.\]
Taking the inverse limits, we get the map
\begin{equation}\label{Integration map for M}
\varpi_{M}\colon \Tp(M)\to \TdR(M)\otimes_{\cO_K}\At
\end{equation}
which we call it the p-adic integration map for motive $M$. The pairing 
\begin{equation}\label{integration pairing for M}
\int^{\varpi}\colon\Tp(M_K)\times\TdR\ve(M_K)\to\Bt
\end{equation}
induced by $\varpi_M$ will be called the p-adic integration pairing, where $M_K:=M\times_{\Spec\cO_K}\Spec K$ is the generic fibre of $M$.

If we repeat the same argument as above but this time for the  exact sequence \ref{exact seq for Fontaine'spairing}, we get a diagram 
\begin{equation}\label{equation 4.3.6}
\begin{tikzcd}
0 \arrow[r] & \Lie(G)\otimes_{\cO_K}\bar{I} \arrow[r] \arrow[d, "{[p^n]}"] & L\times_{\cO_K}G(\At) \arrow[r] \arrow[d, "q_2"] & L\times_{\cO_K}G(\cO_{\C_p}) \arrow[r] \arrow[d, "q_1"] & 0 \\
0 \arrow[r] & \Lie(G)\otimes_{\cO_K}\bar{I} \arrow[r]                         & G(\At) \arrow[r]                                   & G(\cO_{\C_p}) \arrow[r]                               & 0
\end{tikzcd}
\end{equation}
where $q_1:L\times_{\cO_K}G(\ocp)\to G(\ocp)$ and $q_2:L\times_{\cO_K}G(\At)\to G(\At)$ are induced by $(x,g)\mapsto u(x)+p^ng$.
This yields a map $\phi_{n,M}:M[p^n](\ocp)\to \Lie(G)\otimes_{\cO_K}\bar{I}/p^n\bar{I}$. By taking the limit, we have the analogue of the Fontaine's map \[\varphi_M:\Tp(M)\to \Lie(G)\otimes_{\cO_K}\C_p(1)\] for 1-motive $M$. We call the pairing 
\begin{equation}\label{Fontaine's pairing for M}
    \int^{\varphi}:\Tp(M)\times\coLie(G_K)\to\C_p(1)
\end{equation}
induced by the map $\varphi_M$ Fontaine's pairing for the motive $M$.

If we apply the same reasoning as above for the exact sequence \ref{exact sequence for colman pam for 1-motive}, we get a diagram
\begin{equation}\label{4.3.9}
\begin{tikzcd}
0 \arrow[r] & V(M) \arrow[r] \arrow[d, "{[p^n]}"] & L\times G\nat(\ocp) \arrow[r] \arrow[d, "q_2"] & L\times G(\cO_{\C_p}) \arrow[r] \arrow[d, "q_1"] & 0 \\
0 \arrow[r] & V(M) \arrow[r]  & L\times G\nat(\ocp) \arrow[r]                                   & L\times G(\cO_{\C_p}) \arrow[r]                               & 0
\end{tikzcd}
\end{equation}
Using the same argument, we obtain the map $\psi_M:\Tp(M)\to V(M)\otimes_{\cO_K}\ocp$, which we call it Coleman's map for the motive $M$.

Next, we present three theorems. The \cref{kernel Fontaine map phi} and \cref{kernel p-adic integration map varpi} provide information about the kernels of Fontaine's map $\phi_M$ and the p-adic integration map $\varpi_M$. Moreover, in the \cref{theorem p-adic integration pairing for M}, we prove that the p-adic integration pairing $$\displaystyle\int^{\varpi}:\Tp(M)\times\TdR\ve(M_K)\to\Bt$$ is perfect and respects the Hodge filtration.  

\begin{remark}\label{remark before injectivity of fontaine map}
    Let $M$ be a 1-motive over finite extension $K$ of $\Q_p$. Let $F:=K^{un}$ be the maximal unramified extension of $K$ and $N:=\hat{F}$ its completion. We want to show that the kernel of Fontaine map $\phi_M:\Tp(M)\to\Lie(G)\otimes_{\cO_K}$ is $\Tp(M)^{\Gamma_F}$. By Krasner's lemma, we know that $F\cap \bar{N}=N$. Therefore, we get a map $f:\Gamma_N\to \Gamma_F$ induced by restriction. The map $f$ is an isomorphism, since $f$ equals to the composition $\Gamma_N\xrightarrow{\cong} \Aut_{cont}(\C_p/N)\xrightarrow{\cong}\Aut_{cont}(\C_p/F)$. Moreover, $\Tp(M_N)=\Tp(M)$ as $\Z_p[\Gamma_N]$-modules, where the action of $\Gamma_N$ is via $f$. We have a commutative diagram 
    \begin{equation*}
\begin{tikzcd}
\Tp(M_N) \arrow[r] \arrow[d, "\cong", Rightarrow, no head] & \Lie(G_N)\otimes_{N}\C_P(1) \arrow[d, "\cong", Rightarrow,  no head] \\
\Tp(M) \arrow[r]                                           & \Lie(G)\otimes_K\C_p(1)                                            
\end{tikzcd}
    \end{equation*}
and $x$ belongs to the kernel of the top map if and only if it belongs to the kernel of the bottom one. 
\end{remark}
Next, we aim to study the kernel of Fontaine's map $\phi_M$. In the case of an abelian variety $A$ with good reduction, Yeuk Hay Joshua Lam and Alexander Petrov independently proved that that the $\ker(\phi_A)=\Tp(M)^{\Gamma_F}$, where $F=K^{un}$ (see \cite[Theorem A.4]{iovita_p_2022}). In the following theorem, we prove a similar result for the Fontaine map $\phi_M$ that we developed for 1-motive $M$ with good reduction, representing a more general case.

\begin{thm}\label{kernel Fontaine map phi}
Let $F:=K^{un}$ be the maximal unramified extension of $K$. The kernel of the  Fontaine map $\phi_M:\Tp(M)\to \Lie(G)\otimes_{\cO_K}\C_p(1)$ is  $\ker(\phi_M)=\Tp(M)^{\Gamma_F}$ and when tensored by $\C_p$ the resulting map $\Tp(M)\otimes_{\Z_p}\C_p\to \Lie(G)\otimes_{\cO_K}\C_p(1)$ is surjective.
\end{thm}
\begin{proof}
We first prove the surjectivity. Consider the restriction of  $q_1$ and $q_2$ on $0\times G(\ocp)$ and $0\times G(\At)$ respectively in diagram \ref{equation 4.3.6}, then the snake lemma gives the Fontaine map $\phi_G: \Tp(G)\to \Lie(G)\otimes_{\cO_K}\C_p(1)$ of $G$ which factors through the Fontaine map $\phi_M:\Tp(M)\to\Lie(G)\otimes_{\cO_K}\C_p(1)$ of $M$.  We have a commutative diagram 
    \begin{equation}
\begin{tikzcd}
0 \arrow[r] & \Tp G \arrow[r] \arrow[d]               & \Tp M \arrow[r] \arrow[d]               & L\otimes\Z_p \arrow[r] \arrow[d] & 0 \\
0 \arrow[r] & \Lie(G)\otimes_{\cO_K}\C_p(1) \arrow[r] & \Lie(G)\otimes_{\cO_K}\C_p(1) \arrow[r] & L\otimes\C_p(1) \arrow[r]           & 0
\end{tikzcd}
    \end{equation}
of $\Z_p[\Gamma_K]$-modules. Tensoring with $\C_p$, the right vertical map induces a surjective map $L\otimes\C_p\to L\otimes\C_p(1)$ and the left vertical map induces a surjective map $\Tp(G)\otimes_{\Z_p}\C_p\to\Lie(G)\otimes_{\cO_K}\C_p(1)$ by \cref{Fontaine map tensored by Cp surjective}. Thus the induced map $\phi_M\otimes \C_p:\Tp(M)\otimes_{\Z_p}\C_p\to\Lie(G)\otimes_{\cO_K}\C_p(1)$ is surjective.

We now determine $\ker(\phi_M)$. By \cref{remark before injectivity of fontaine map}, we can assume that $K$ is the completion of the maximal unramified extension $\Q^{un}_p$ of $\Q_p$. Therefore, we need to show that $\ker(\phi_M)=\Tp(M)^{\Gamma_K}$.
By \cref{Hodge-Tate weights of tate module of 1-motives}, the Hodge-Tate weights of $\Vp(M)$ are $0$ and $1$ with multiplicity $\rank(L)+\dim(A)$ and $\dim(T)+\dim(A)$ respectively. Therefore, $$\Tp(M)\otimes_{\Z_p}\C_p\cong \C_p^{\rank(L)+\dim(A)}\oplus \C_p(1)^{\dim(T)+\dim(A)},$$ and the kernel of $\phi_M\otimes \C_p:\Tp(M)\otimes_{\Z_p}\C_p\to\Lie(G)\otimes_{\cO_K}\C_p$ is isomorphic to $\C_p^{\rank(L)+\dim(A)}$.

Let $S:=\ker(\phi_M\otimes \Q_p: \Vp(M)\to \Lie(G)\otimes_{K}\C_p(1))$. Notice that the composition $S\otimes_{\Q_p}\C_p\into \Vp(M)\otimes_{\Q_p}\C_p\xrightarrow{\phi_M\otimes\C_p}\Lie(G)\otimes_K\C_p(1)$ is zero. Thus, \[S\otimes_{\Q_p}\C_p\subseteq \ker(\phi_M\otimes\C_p)=\C_p^{\rank(L)+\dim(A)}.\]
       It follows that $S$ is a Hodge-Tate Galois representation with all Hodge-Tate weights $0$. By \cite[Corollary 1]{sen1980continuous}, this implies that the representation $\rho:\Gamma_K\to \End_{\Q_p}(S)$ is finite. There exists a finite totally ramified extension $L/K$ such that $\rho$ factors trough $\Gal(L/K)\to \End_{\Q_p}(S)$. As $S$ is a crystalline $\Gamma_K$-representation, by \cite[Proposition 6.10]{Fontaine_Laffaille_1982} we have $(S\otimes_{\Q_p}\Bcris)^{\Gamma_L}=S\otimes_{\Q_p}(\Bcris)^{\Gamma_L}=S\otimes_{\Q_p}L$. Hence, we get
       \[
       \Dcris(S)=(S\otimes_{\Q_p}\Bcris)^{\Gamma_K}=((S\otimes_{\Q_p}\Bcris)^{\Gamma_L})^{\Gal(L/K)}=(S\otimes_{\Q_p}L)^{\Gal(L/K)}=S^{\Gal(L/K)}\otimes_{\Q_p}K.
       \]
       Thus, $\dim_{\Q_p}(S^{\Gal(L/K)})=\dim_{\Q_p}(\Dcris(S))=\dim_{\Q_p}(S)$. This means that $S=S^{\Gal(L/K)}$. On the other hand, the action of $\Gamma_K$ on $S$ factored through $\Gal(L/K)$, thus $S=S^{\Gamma_K}$. As $\rho_M\otimes\Q_p: \Vp(M)\to\Lie(G)\otimes_K\C_p(1)$ is $\Gamma_K$-equivariant and $H^0(K,\Lie(G)\otimes_{K}\C_p(1))=0$ (by \cite[Theorem 2]{tate1967pdivisble}), we can conclude that $\Vp(M)^{\Gamma_K}\subseteq S$. But $S\subset\Vp(M)$, hence $S=S^{\Gamma_K}\subset\Vp(M)^{\Gamma_K}$, and the result follows. 
\end{proof}
For the p-adic integration map $\varpi_M:\Tp(M)\to\TdR(M)\otimes_{\cO_K}\Bt$, we cannot make a statement similar to \cref{kernel Fontaine map phi}. However, we have the following:
\begin{thm}\label{kernel p-adic integration map varpi}
Let $M=[L\to G]$ be a 1-motive over $\cO_K$, and $\sG^0$ denote the connected component of $G[p^{\infty}]$. If $\Tp(M)^{\Gamma_K}=0$ and $\ker(\varpi_G)\cap\Tp(\sG^0)=0$, then $\ker(\varpi_M)=0$. In other words, if $\Tp(M)^{\Gamma_K}=0$, the restriction $\left.\varpi\right|_{\Tp(\sG^0)}$ is injective if and only if $\varpi_M$ is injective.
\end{thm}
\begin{proof}
    The canonical exact sequence 
    \[
    0\to \Tp(G)\to \Tp(M)\to \Tp(L)\to 0
    \]
    gives the long exact sequence in Galois cohomology:
    \[
    0\to\Tp(G)^{\Gamma_K}\to\Tp(M)^{\Gamma_K}\to\Tp(L)^{\Gamma_K}\to H^1(K,\Tp(G)).
    \]
    We have $\Tp(G)^{\Gamma_K}=\Tp(M)^{\Gamma_K}=0$, and $\Tp(L)^{\Gamma_K}=\Tp(L)$, since $L[p^{\infty}]$ is \'etale. Assume that $\varpi_M(x)=0$, for some $x\in\Tp(M)$. Then $\varpi_{M}((\sigma-1)x)=0$, for all $\sigma\in\Gamma_K$. But, $(\sigma-1)x$ lies in the Tate module $\Tp(\sG^0)$, where $\sG^0$ is the connected component of $G[p^{\infty}]$ which is the same as the connected component of $M[p^{\infty}]$. Hence, $(\sigma-1)x\in\ker(\varpi_G)\cap\Tp(\sG^0)$ for all $\sigma\in\Gamma_K$. Therefore, $(\sigma-1)x=0$ for all $\sigma\in\Gamma_K$. Let $f:\Tp(M)\to\Tp(L)$ and $\delta:\Tp(L)\to H^1(K,\Tp(G))$ be the connecting map. As $\sigma(x)=x$ for all $\sigma\in\Gamma_K$ and $f$ is $\Gamma_K$-equivariant, we can conclude that $\delta(f(x))=0$ in $H^1(K,\Tp(G))$. Since $\delta$ is injective, it follows that $f(x)=0$, which implies that $x\in\Tp(G)\cap\ker(\varpi_M)=\ker(\varpi_G)$. Consider the exact sequence
    \begin{equation}\label{3.3.10}
    0\to \Tp(\sG^0)\to \Tp(G)\to \Tp(G)/\Tp(\sG^0)\to 0.
    \end{equation}
    The $\Z_p$-module $\Tp(G)/\Tp(\sG^0)$ is the Tate-module of the \'etale part of $G[p^{\infty}]$. If we repeat the same approach as above for the exact sequence \ref{3.3.10}, we obtain that $x\in\Tp(\sG^0)\cap\ker(\varpi_G)=0$. This shows that $\varpi_M$ is injective.
\end{proof}

In the following theorem, we show that the integration pairing $\displaystyle\int^{\varpi}:\Tp(M)\times\TdR\ve(M_K)\to\Bt$ is perfect and respects Hodge filtration. By perfect we mean the condition \ref{perfect def meaning}.

\begin{thm}\label{theorem p-adic integration pairing for M}
The p-adic integration pairing $\displaystyle\int^{\varpi}$ for a 1-motive $M$ is bilinear, perfect, and $\Gamma_K$-equivariant in the first argument. Moreover, it respects the Hodge filtration in the following sense: for all $\omega\in\Fil^1\TdR\ve(M_K)$ and $x\in\Tp(M)$,
\[\int^{\varpi}_{x}\omega\in\Fil^1\Bt.\]
In particular, $\displaystyle\int^{\varpi}_x\omega=\int^{\phi}_x\omega$ if $\omega\in\Fil^1\TdR\ve(M_K)=\coLie(G)_K$.
\end{thm}
\begin{proof}
The fact that the pairing is bilinear and $\Gamma_K$-equivariant in the first argument is by construction. We first show that the integration pairing is perfect. For the case $M=[L\to 0]$, we have
\[G\nat=L\otimes\G_a=V(L), \text{ and }\cK=V(L)\otimes_{\cO_K}\At/\bar{I}=V(L)\otimes\ocp.\]
Thus, the diagram \ref{equation 4.3.6} vanishes, i.e., every therm becomes zero, and the diagrams \ref{4.3.4} and \ref{4.3.9} are the same. Therefore, the p-adic integration map $\varpi_L$ coincides with the Coleman's map 
\[
\psi_{L}:\Z_p\otimes L\to V(L)\otimes_{\cO_K}\ocp,
\]
which is induced by $(x_n)\mapsto ([p^n]x_n)$, for $x_n\in L[p^n]$. Clearly, this pairing respects the filtration, and the induced map 
\[
L\otimes\Bt\to V(L)\otimes\Bt
\]
is an isomorphism.

The p-adic integration pairing for the case $M=[0\to G]$ is perfect by \cref{theorem integration pairing for G}. Finally for the case $M=[L\to G]$, by the construction of p-adic integration applied to the canonical exact sequence \ref{canonical exact sequence for any 1-motive}, we can obtain the following commutative diagram with exact rows
\begin{equation}
\begin{tikzcd}[column sep=small]
0 \arrow[r] & {\Tp G} \arrow[r] \arrow[d, "\varpi_{G}"]                    & {\Tp M} \arrow[r] \arrow[d, "\varpi_{M}"]                   & L\otimes\Z_p \arrow[r] \arrow[d]  & 0 \\
0 \arrow[r] & \Lie(\Ex(G))\otimes_{\cO_K}\Bt \arrow[r] & \Lie(G\nat)\otimes_{\cO_K}\Bt \arrow[r] & V(L)\otimes\Bt \arrow[r] & 0
.\end{tikzcd}
\end{equation}
 The pairing induced by the middle vertical arrow is perfect, because both pairings induced by left and right vertical arrows are perfect.
 
 We now want to show that the p-adic integration pairing for $M=[L\to G]$ respects the filtration. As $V(M)$ is the kernel of $\pi:G\nat\to G$, the map $\Lie(G\nat)\otimes_{\cO_K}\bar{I}\to\Lie(G)\otimes_{\cO_K}\bar{I}$ factors through its quotient by $V(M)\otimes_{\cO_K}\bar{I}$, and it gives a map $g:\cK\to\Lie(G)\otimes_{\cO_K}\bar{I}$. We can obtain the commutative diagram
\begin{equation}
\begin{tikzcd}
0 \arrow[r] & \cK \arrow[d, "g"] \arrow[r]            & L\times G\nat(\At) \arrow[d, "f"] \arrow[r] & L\times G(\ocp) \arrow[d, Rightarrow, no head] \arrow[r] & 0 \\
0 \arrow[r] & \Lie(G)\otimes_{\cO_K}\bar{I} \arrow[r] & L\times G(\At) \arrow[r]                    & L\times G(\ocp) \arrow[r]                                & 0
\end{tikzcd}
\end{equation}
with the property that $q=q_1$, $\pi\circ q\nat=q_2\circ f$, and $g\circ[p^n]=[p^n]\circ g$, where $q, q\nat$ are the maps shown in diagram \ref{4.3.4}, and $q_1, q_2$ are the maps shown in diagram \ref{equation 4.3.6}. This shows that applying the snake lemma on diagrams \ref{4.3.4} and \ref{equation 4.3.6} together with the above property yield the commutative diagram 
\begin{equation*}
\begin{tikzcd}
{M[p^n]} \arrow[r]                                & \cK/p^n\cK \arrow[d]             \\
{M[p^n]} \arrow[u, Rightarrow, no head] \arrow[r] & \Lie(G)\otimes\bar{J}/p^n\bar{J}.
\end{tikzcd}
\end{equation*}
By pushing-out along $\cK/p^n\cK\to\Lie(G\nat)\otimes\Bt/p^n\Bt$, we get
\begin{equation*}
\begin{tikzcd}
{M[p^n]} \arrow[r]                                & \cK/p^n\cK \arrow[d] \arrow[r]             & \Lie(G\nat)\otimes\Bt/p^n\Bt \arrow[d] \\
{M[p^n]} \arrow[u, Rightarrow, no head] \arrow[r] & \Lie(G)\otimes\bar{J}/p^n\bar{J} \arrow[r] & \Lie(G)\otimes\Bt/p^n\Bt.              
\end{tikzcd}
\end{equation*}
Hence, the p-adic integration map $\varpi_M$ factors through the Fontaine's map $\phi_M$, i.e., 
\begin{equation}\label{commutative diagram respect filtration}
\begin{tikzcd}
\Tp(M) \arrow[r, "\varpi_M"] \arrow[d, "\varphi_M"] & \Lie(G\nat)\otimes_{\cO_K}\Bt \arrow[d, two heads] \\
\Lie(G)\otimes_{\cO_K}\C_p(1) \arrow[r, hook]       & \Lie(G)\otimes_{\cO_K}\Bt                        
\end{tikzcd}
\end{equation}
which completes the proof.
\end{proof}
\subsection*{Crystalline integration}
    \cref{tate module is crystalline} gives the map $\Tp(M)\to\Tcrys(\bar{M})\otimes_{\W(k)}\Acris$ which is called the crystalline integration map. This map induces the filtered isomorphism $\Tp(M)\otimes_{\Z_p}\Bcris\cong\Tcrys(\bar{M})\otimes_{\W(k)}\Bcris$. The induced pairing
    \begin{equation}\label{crystalline integration map}
    \displaystyle\int^{cris}:\Tp(M)\times\Tcrys\ve(\bar{M})_K\to\Bcrisp
    \end{equation}
    is called crystalline integration pairing. The crystalline pairing factors through the p-adic integration pairing via the crystalline-de Rham identification $\TdR\ve(M)_K\cong\Tcrys\ve(\bar{M})_K$ (\cite{colmez_periodes_1993}).

\begin{prop}\label{prop 4.3.2}
    Let $\tilde{\Tp}(M):=\Tp(\sG^0)$, where $\sG^0$ is the connected component of the p-divisible group associated with $M$, and let $D:=\tilde{\Tcrys}\ve(\bar{M})$ the Dieudonn\'e submodule of $\Tcrys\ve(\bar{M})$ associated with $\sG^0$. We have 
    \[
    \int^{cris}_xF^n\omega=\phi_{cris}^n(\int^{cris}_x\omega)
    \]
    for any $x\in\tilde{\Tp}(M)$, $\omega\in D$, and $n\geq 0$, where $\phi_{cris}$ is Frobenius on $\Bcrisp$.
\end{prop}
\begin{proof}
The proof follows directly from \cite[Proposition 3.1]{colmez1992periodes}.
\end{proof}
\begin{lemma}\label{lemma 4.3.1}
    Let $V:=\Vp(M)$. We have 
    \[
    (V\otimes\Bt)^{\Gamma_K}\cong(V\otimes\BdRp)^{\Gamma_K}\text{ and }(V\ve\otimes\Bt)^{\Gamma_K}\cong(V\ve\otimes\BdRp)^{\Gamma_K}\cong\DdR(V\ve).
    \]
\end{lemma}
\begin{proof}
Let us first show that the quotient map $\BdRp\onto \Bt$ induces an isomorphism $(V\otimes_{\Q_p}\BdRp)^{\Gamma_K}\cong (V\otimes_{\Q_p}\Bt)^{\Gamma_K}$.
Since $t\BdRp/(t^2\BdRp)\cong\C_p(1)$, we have an exact sequence
$$0\to\C_p(1)\to\BdRp/t^2\BdRp\to \BdRp/t\BdRp\to 0$$
of $K$-Banach spaces. We tensor this exact sequence with $V$ to obtain an exact sequence of $\Gamma_K$-modules
\[
0\to V\otimes_{\Q_p}\C_p(1)\to V\otimes_{\Q_p}\BdRp/t^2\BdRp\to V\otimes_{\Q_p}\BdRp/t\BdRp\to 0.
\]
The Hodge-Tate weights of $V$ are $0$ and $1$ (\cref{Hodge-Tate weights of tate module of 1-motives}), hence $V\otimes_{\Q_p}\C_p(1)\cong\C_p(1)^r\oplus\C_p(2)^m$ for some positive integers $r$ and $m$. As the Galois cohomology of $\C_p(1)^r\oplus\C_p(2)^m$ vanishes in all degrees (\cite[Theorem 2]{tate1967pdivisble}), the corresponding long exact sequence of Galois cohomology yields an isomorphism
$$(V\otimes_{\Q_p}\Bt)^{\Gamma_K}\cong (V\otimes_{\Q_p}(\BdRp/t\BdRp))^{\Gamma_K}.$$
One continues by induction to show that for all $n\geq 1$, the natural map $\BdRp/(t^n\BdRp)\onto\Bt$ induces an isomorphism 
$$(V\otimes_{\Q_p}(\BdRp/t^n\BdRp))^{\Gamma_K}\cong(V\otimes_{\Q_p}\Bt)^{\Gamma_K}.$$
Taking the inverse limit gives
\begin{equation}\label{4.3.14}
(V\otimes_{\Q_p}\BdRp)^{\Gamma_K}\cong(V\otimes_{\Q_p}\lim(\BdRp/t^n\BdRp))^{\Gamma_K}\cong\lim(V\otimes_{\Q_p}(\BdRp/t^n\BdRp))^{\Gamma_K} \cong(V\otimes_{\Q_p}\Bt)^{\Gamma_K}.  
\end{equation}

Similarly, for every $n\geq 2$, the exact sequence
\[
0\to V\ve\otimes\C_p(n)\to V\ve\otimes\BdRp/t^{n+1}\BdRp\to V\ve\otimes\BdRp/t^n\BdRp\to 0
\]
yields an isomorphism $(V\ve\otimes\BdRp/t^{n+1}\BdRp)^{\Gamma_K}\cong (V\ve\otimes\BdRp/t^n\BdRp)^{\Gamma_K}$, since Galois cohomology of $V\ve\otimes\C_p(n)\cong\C_p(n)^r\oplus\C_p(n-1)^m$ vanishes in all degrees. The latter isomorphism follows from the fact that the Hodge-Tate weights of $V\ve$ are $0$ and $-1$. This implies that 
\[
(V\ve\otimes\Bt)^{\Gamma_K}\cong(V\ve\otimes\BdRp)^{\Gamma_K}.
\]

We now want to show that the induced map $i\colon (V\otimes_{\Q_p}\BdRp)^{\Gamma_K}\to\DdR(V)$ is an isomorphism. Since $t^{-1}\BdRp/\BdRp\cong\C_p(-1)$, we have an exact sequence
\[0\to \BdRp\to t^{-1}\BdRp\to \C_p(-1)\to 0\]
of $K$-Banach spaces. We then tensor this sequence with $V\ve$ to obtain the following exact sequence of $\Z[\Gamma_K]$-modules
\[0\to V\ve\otimes_{\Q_p}\BdRp\to V\ve\otimes_{\Q_p}t^{-1}\BdRp\to V\ve\otimes_{\Q_p}\C_p(-1)\to 0.\]
The Hodge-Tate weights of $\Vp(M)$ are $0$ and $1$, therefore the Hodge-Tate weights of $V\ve$ are $0$ and $-1$, and so \[V\ve\otimes_{\Q_p}\C_p(-1)\cong\C_p(-1)^r\oplus\C_p(-2)^m\]
for some positive integers $r$ and $m$. As the Galois cohomology of $\C_p(-1)^r\oplus\C_p(-2)^m$ vanishes in all degrees, the exact sequence gives an isomorphism 
\begin{equation*}
(V\ve\otimes_{\Q_p}\BdRp)^{\Gamma_K}\cong (V\ve\otimes_{\Q_p}t^{-1}\BdRp)^{\Gamma_K}.    
\end{equation*}
Similarly, we can show that for all $n\geq 2$, the natural map $t^{-n+1}\BdRp\into t^{-n}\BdRp$ induces an isomorphism
$$(V\ve\otimes_{\Q_p}t^{-n+1}\BdRp)^{\Gamma_K}\cong(V\ve\otimes_{\Q_p}t^{-n}\BdRp)^{\Gamma_K}.$$
Thus, we have $(V\ve\otimes_{\Q_p}\BdRp)^{\Gamma_K}\cong(V\ve\otimes_{\Q_p}t^{-n}\BdRp)^{\Gamma_K}$ for all $n\geq 1$. As $\BdR=\colim t^{-n}\BdRp$, we conclude that
\begin{align}\label{4.3.13}
\DdR(V\ve)\cong (V\ve\otimes_{\Q_p}\colim t^{-n}\BdRp)^{\Gamma_K}=\colim(V\ve\otimes_{\Q_p}t^{-n}\BdRp)^{\Gamma_K}\cong(V\ve\otimes_{\Q_p}\BdRp)^{\Gamma_K}.   
\end{align}

\end{proof}
\begin{cor}\label{tate module of M is de Rham}
Let $M$ be a 1-motive over $K$ with good reduction. There is a canonical isomorphism of filtered $K$-vector spaces 
\[\DdR(\Tp M)\cong\TdR(M)\]
and the p-adic Galois representation $\Vp(M)$ is de Rham.
\end{cor}
\begin{proof}
    By the above lemma together with \cref{theorem p-adic integration pairing for M}, we can write
    \[
    \TdR(M)\cong(\Vp(M)\otimes_{\Q_p}\Bt)^{\Gamma_K}\cong(\Vp(M)\otimes_{\Q_p}\BdRp)^{\Gamma_K}\into\DdR(\Vp(M)).
    \]
  Since $\dim_K\TdR(M_K)=\dim_{\Q_p}(\Vp(M))$ and $\dim(\DdR(\Vp(M)))\leq\dim(\Vp(M))$, the above embedding must be an isomorphism. This completes the proof.   
\end{proof}
The following corollary will be a key ingredient in the next chapter.
\begin{cor}\label{cor 4.3.2}
With the notations of \ref{crystalline integration map}, for all $x\in\Vp(M)$ and $\omega\in\TdR\ve(M)_K$, we have $\displaystyle\int^{cris}_x\omega=0$ if and only if $\displaystyle\int^{\varpi}_x\omega=0$.
\end{cor}
\begin{proof}
    We have seen in chapter 1 that $\Bcrisp\subset\BdRp$ and the filtration of $\Bcrisp$ is induced by the filtration of $\BdRp$. Let $\omega\in\TdR\ve(M)$, and $\omega'\in\Tcrys\ve(\bar{M})_K$ be its corresponding element via the canonical identification $\TdR\ve(M)_K\cong\Tcrys\ve(\bar{M})_K$. Consider the maps
    \begin{align*}
        f:\Vp(M)\to\Bt,\,\, x\mapsto\int^{\varpi}_x\omega,\\
        g:\Vp(M)\to \Bcris\into\BdRp,\,\, x\mapsto\int^{cris}_x\omega'.
    \end{align*}
    We have
    \begin{align*}
        f\in \Hom_{\Z_p[\Gamma_K]}(\Vp(M),\Bt)\cong(\Vp\ve(M)\otimes_{\Q_p}\Bt)^{\Gamma_K},\\
        g\in \Hom_{\Z_p[\Gamma_K]}(\Vp(M),\BdRp)\cong(\Vp\ve(M)\otimes_{\Q_p}\BdRp)^{\Gamma_K}.
    \end{align*}
    By \cref{lemma 4.3.1}, we know that 
    \[
    (\Vp\ve(M)\otimes_{\Q_p}\Bt)^{\Gamma_K}\cong(\Vp\ve(M)\otimes_{\Q_p}\BdRp)^{\Gamma_K}.
    \]
    Therefore, the map $f$ identifies $g$ through the canonical isomorphism $\TdR\ve(M)_K\cong\Tcrys\ve(\bar{M})_K$. Consequently, $f(x)=0$ if and only if $g(x)=0$. 
\end{proof}


\section{P-adic logarithm through Barsotti-Tate groups}\label{section p-adic logarithm}
The p-adic logarithm plays a crucial role in the construction of our p-adic periods. In this section, our main goal is to study the p-adic logarithm map of a semi-abelian variety with good reduction through its associated p-divisible (Barsotti-Tate) group. We will further investigate the image of this logarithm map using techniques from p-adic Hodge theory and Galois cohomology.
    Let $G$ be a commutative group defined over a complete p-adic subfield $K$ of $\C_p$. Recall \cite[Chapter III, 7.6]{bourbaki1975elements} and \cite{zarhin1996p} the properties of the p-adic logarithm map
    \[\log_{G(K)}\colon G(K)_f\to \Lie(G(K)) \]
    where $G(K)_f$ is the smallest open subgroup of $G(K)$ such that the quotient group $G(K)/G(K)_f$ is torsion free. The map $\log_{G(K)}$ is a $K$-analytic homomorphism whose tangent map 
    \[
   d\log_{G(K)}\colon\Lie(G)\to\Lie(\Lie(G))=\Lie(G)
    \] is the identity map. Moreover, the p-adic logarithm map $\log_{G(K)}$ can be extended to a map \[
    \log^{(c)}_{G(K)}\colon G(K)\to \Lie(G)
    \] if one fixes a branch $c$ of $K$.

    The logarithm map $\log_{G(K)}$ is compatible with the base change and is functorial on $G$. Moreover, the map $\log_{G(K)}$ and the subgroup $G(K)_f$ are uniquely determined by the above properties. Specifically, the subgroup $G(K)_f$ consists of all elements $x\in G(K)$ such that the identity element of $G(K)$ is an accumulation point of the set $\{x^n\st n>0\},$ \ie there exists an increasing sequence $(n_i)$ of positive integers such that $x^{n_i}$ tends to $0$ in $G(K)$.   
\begin{example}\label{6.1.2}
Consider the multiplicative group $\G_m$ over $K$. We have 
\[
\G_m(K)_f=\{x\in K^{\times}\st \nu(x)=0\}=\cO_K^{\times}
\]
and the logarithm map $\log_{\G_m(K)}$ coincides with the usual p-adic logarithm on the open subgroup of principle units $\{x\in K^{\times}\st \nu(1-x)>0\}=1+\fm_K$.

The group $1+\fm_K$ is indeed the $\cO_K$-valued formal points on the p-divisible group $\mu_{p^{\infty}}$ associated to $\G_m$. That is
\[
\mu_{p^{\infty}}(\cO_K)=\lim_i\mu_{p^{\infty}}(\cO_K/\fm^i_K)=1+\fm_K
\]
The elements on the left side are all $x\in\cO_K^{\times}$ such that $\nu(x^{p^i}-1)$ can get arbitrary large. As the residue field of $K$ has characteristic $p$, we also have the opposite inclusion $\mu_{p^{\infty}}(\cO_K)\subseteq 1+\fm_K$. Therefore, the logarithm map of the p-divisible group $\mu_{p^{\infty}}(\cO_K)$ (see \cref{logarith of p-divisible multiplicative}) factors through $\log_{\G_m(K)}\colon \G_m(K)_f\to K$.
\end{example}

    We can generalize the above example to any semi-abelian variety $G$ over $\cO_K$, i.e., the logarithm map $\log_{G(K)}$ recovers the logarithm of the p-divisible group associated to $G$. As the multiplication map $[p^n]$ is an isogeny on $G$, we can associate a p-divisible group $\sG:=G[p^{\infty}]$ with $G$. The $\cO_{K}$-valued formal points $\sG(\cO_{K})$ is isomorphic to $\Hom_{\cO_K-cont}(\sA,\cO_K)$, where $\sG=\Spf(\sA)$ and $\sA$ equipped with $\fm$-adic topology. In particular, when $\sG$ is connected, $\sG(\cO_K)\cong \fm_K^d$ as a set, where $d$ is dimension of $\sG$ and the formal group law $\mu\colon \sA\to \sA$ induces the structure of a p-adic analytic group over $K$. Thus, we can view $\sG(\cO_{K})$ as an analytic subgroup of $G(K)$.
\begin{prop}
 The group $\sG(\cO_K)$ is a subgroup of $G(K)_f$. Moreover, the logarithm map $\log_{G}\colon G(K)_f\to \Lie(G)$ factors through the logarithm of $\sG(\cO_K)$ in \cref{def logarithm p-divisible group}.
\end{prop}
\begin{proof}
\cref{6.1.2} shows the proposition for $G=\G_m$. For an arbitrary semi-abelian scheme $G$, the Hodge-Tate decomposition gives us an injection $\sG(\cO_{\C_p})\to \Hom(\Tp(G\ve),\mup(\cO_{\C_p}))$ which induces an isomorphism if we take Galois invariant elements. We have
\begin{gather*}
\Tp(\sG)=\lim G[p^n](\bar{K})=\lim\Hom(G[p^n]\ve(\bar{K}),\mu_{p^n}(\bar{K}))\\=\Hom(\lim G[p^n]\ve(\bar{K}),\lim\mu_{p^n}(\bar{K}))=\Hom(\Tp(\sG\ve),\Z_p(1))=\Tp\ve(\sG\ve)\otimes\Z_p(1)\\=\Tp\ve(\sG\ve)(1).    
\end{gather*}
 Therefore, $\Tp(\sG\ve)=\Tp\ve(\sG)(1)=\Tp(\sG)(-1)\ve$ and we obtain
\begin{equation}
\Hom(\Tp(\sG\ve),\mup(\cO_{\C_p}))=\Tp\ve(\sG\ve)\otimes\mup(\cO_{\C_p})=\Tp(\sG)(-1)\otimes\mup(\cO_{\C_p}).
\end{equation}

The topology of $\Tp(G)(-1)\otimes\mup(\cO_{\C_p})$ is induced from the product topology of the adic topologies on $\Tp(G)$ and $\mup(\cO_{\C_p})$. As elements in $\mup(\cO_{\C_p})$ are topologically nilpotent, so is any element in $\Tp(G)(-1)\otimes\mup(\cO_{\C_p})$. The result follows from the fact that $\sG(\cO_K)\subset G(K)_f$ and that $\log_{\sG}:\sG(\cO_K)\to\Lie(G)_K$ is a homomorphism whose tangent map is identity. Furthermore, $\displaystyle\left.\log_{G(K)}\right|_{\sG(\cO_K)}:\sG(\cO_K)\to \Lie(G)$ is determined uniquely by these properties. 
\end{proof}

\begin{remark}
    There are some obstructions in studying the group $G(K)_f$. For instance, although $(.)_f$ is functorial on $G$, it is not exact. Moreover, the group $G(K)_f$ is not easily computable, whereas p-divisible groups are better understood. Furthermore, if $G'$ is a vector extension of a semi-abelian group $G$, then logarithms $\log_{\sG}$ and $\log_{\sG'}$ associated with their p-divisible groups share the same image.
\end{remark}

Let $G_{\K}$ be a semi-abelian variety over a number field $\K\subset K$ with a good reduction at $p$. Then $G$ can be extended to a semi-abelian scheme over $\cO_K$, which we also denote by $G$. Let $\sG:=G[p^{\infty}]$ be the associated p-divisible group over $\cO_K$.
\begin{defn}\label{def: algebraic points on G}
With above notations, we define 
    \begin{gather}  
    G(\barQ)_f:=G_{\K}(\barQ)\cap G(\C_p)_f;\\
    \sG(\barQ):=G_{\K}(\barQ)\cap\sG(\ocp).\label{5}
    \end{gather}
\end{defn}
The intersections above are inside $G(\C_p)$, as $\sG(\ocp)\subset G(\C_p)_f$ and $G_{\K}(\barQ)\subset G(\C_p)$. 

\begin{lemma}
    Let $G$ be a commutative group scheme over a field $K$. Assume that $G$ is a vector extension of $H$ i.e. there is an exact sequence \[0\to V\to G\to H\to 0\] of commutative group schemes over $K$. Let $f:\Lie(H)\to \Lie(G)$ be a splitting of \[0\to \Lie(V)\to\Lie(G)\to \Lie(H)\to 0,\]
    then there exists a canonical homomorphism $\phi: H\to G$ which is a splitting of \[0\to V(K)\to G(K)\to H(K)\to 0\]
    and we have $\Lie\,\phi=f$.
\end{lemma}
\begin{proof}
As $V$ is a vector group, $\iext^1(V,(.))=0$ and we have an exact sequence \[0\to V(K)\to G(K)\to H(K)\to 0.\]
Fix a branch of the logarithm map $\log^{(c)}$. We can construct the following commutative diagram
\begin{equation}
\begin{tikzcd}
0 \arrow[r] & V(K) \arrow[r] \arrow[d, equal] & G(K) \arrow[r] \arrow[d, "\log^{(c)}_{G(K)}"] \arrow[l, "\log^{(c)}_{G(K)}\circ \bar{f}"', bend right=49] & H(K) \arrow[r] \arrow[d, "\log^{(c)}_{H(K)}"] \arrow[l, "\phi"', dashed, bend right=49] & 0 \\
0 \arrow[r] & \Lie V \arrow[r]                              & \Lie G \arrow[r] \arrow[l, "\bar{f}", bend left=49]                                                       & \Lie H \arrow[r] \arrow[l, "f", bend left=49]                                           & 0
\end{tikzcd}
\end{equation}
where $\bar{f}$ is the retraction induced by splitting $f$ and $\phi:H(K)\to G(K)$ is the section induced by $\log^{(c)}_{G(K)}\circ \bar{f}$. It is easy to check that $\Lie\,\phi=f$. The map $\phi$ is a splitting since $\log^{(c)}_{G(K)}\circ \bar{f}:G(K)\to V(K)$ is a splitting. 
\end{proof}

\section{$\exp$: the local inverse of $\log$}
The goal of this section is to identify the local inverse of $\log_{\sG}$ when $\sG$ is a p-divisible group over $\cO_K$.

Let $M=[L\xrightarrow{u}G]$ be a 1-motive over $K$. There is a finite extension $F$ of $K$ such that $L_F$ is split. By \cref{S-point of 1-motive when split}, we have $M(F)=G(F)/\im(u_F)$.  In particular, $M(\bar{K})=G(\bar{K})/\im(u)$.

Recall the definition of $M[p^n](K)$ (see \ref{formula M[n]}).
\begin{defn}
Let $M$ be a 1-motive over $K$. We define \[M(\bar{K})[p^n]:=\ker([p^n]:M(\bar{K})\to M(\bar{K}))\] and \[M(\bar{K})[p^{\infty}]:=\colim M(\bar{K})[p^n]\]
\end{defn}

We have the following
\begin{prop}
 If $u$ is injective, then $M[p^n](\bar{K})=M(\bar{K})[p^n]$.
\end{prop}
\begin{proof}
    We know that \[
    M[p^n]=\frac{\{(x,g)\in L\times G\mid u(x)=-p^ng\}}{\{(p^nx,-u(x)\mid x\in L\}}
    \]
    The map $(x,g)\mapsto g$ defines a map $\phi:M[p^n](\bar{K})\to M(\bar{K})[p^n]$ because if $(x,g)\in M[p^n]$, then $p^ng\in\im(L)$ and it is well-defined.  Conversely, if $g\in M(\bar{K})[p^n]$, then $p^ng=u(x_g)$ for some $x_g\in L$. Since $u$ is injective, then $g\mapsto x_g$ defines a well-defined map $\psi:M(\bar{K})[p^n]\to M[p^n](\bar{K})$. We have $\phi\circ\psi=id$ and $\psi\circ\phi=id$. The result follows.
\end{proof}
\begin{cor}
Let $M$ be a 1-motive over $\cO_K$. Then $M(\bar{K})[p^{\infty}]=M[p^{\infty}](\bar{K})$ and $\Tp(M)=\lim M(\bar{K})[p^{n}]$.   
\end{cor}

Let $\sG$ be the p-divisible group associated to the 1-motive $M=[L\to G]$ and $D$ the Dieudonn\'e module associated to $M=[L\to G]$ i.e. $D:=\Tcrys(\bar{M})$. The Dieudonn\'e module $D$ is equipped with a filtration which is induced from the Hodge filtration on $\TdR(M)$. Then \cref{cor: Hodge filtration and tangent space on Dieudonne module} implies that $(D/D^0)\otimes\Q$ identifies the tangent space $\Lie(G)_K=\Lie(\sG)_K$. Our goal is to define a local inverse of $\log_{\sG}$. We define 
\begin{equation}\label{eq 3.5.1}
\exp_D:(D/D^0)\otimes\Q\to (D/(1-F)D^0)\otimes\Q,  \text{ induced by } x\mapsto  x-F(x) \text{ for all }x\in D/D^0
\end{equation}
Clearly, $\exp_D$ is surjective. On the other hand, by \cref{cor: explicit description of M(S) points}, any point $x\in M(\cO_K)$ corresponds to an exact sequence 
\[
0\to M\to M_x\to \Z\to 0
\]
where $M_x=[\Z\oplus L\xrightarrow{f} G]$, and $f$ is given by $(1,\ell)\mapsto x+u(\ell)$. This gives the exact sequence
\[
0\to \Tcrys(\bar{M})\to \Tcrys(\bar{M_x})\to 1_{FD}\to 0
\]
of filtered Dieudonn\'e modules, where $1_{FD}$ is the unit filtered Dieudonn\'e module (recall \cref{unit filtered Dieudonn\'e module}). Thus, we get a map
\[
M(\cO_K)\to \Ext^1(D,1_{FD})
\]
where $\Ext^1(D,1_{FD})$ is in the category of filtered Dieudonn\'e modules over $\W(k)$. In this section, we want to show that the diagram
\begin{equation}\label{diagram 3.5.1}
\begin{tikzcd}
                                         &  & \sG(\cO_K) \arrow[d, hook] \arrow[lldd, "\log_{\sG}"'] \\
                                         &  & G(\cO_K) \arrow[d]                                     \\
\Lie(G)_K \arrow[d, Rightarrow, no head] &  & M(\cO_K) \arrow[d]                                     \\
(D/D^0)\otimes\Q \arrow[rr, "\exp"]      &  & {\Ext^1(D,1_{FD})}                                    
\end{tikzcd}
\end{equation}
commutes. By this, we actually find a local inverse of $\log_{\sG}$.

Consider the Kummer sequence \cite[\S 2.4]{tate1967pdivisble} 
\[
0\to \Tp(M)\to \cB_{\sG}\to \sG(\cO_{\bar{K}})\to 0,
\]
where \[\cB_{\sG}:=\lim(\sG(\cO_{\bar{K}})\xleftarrow{[p]}\sG(\cO_{\bar{K}})\xleftarrow{[p]}\dots).\]
The Hodge-Tate decomposition (\cref{Theorem: Hodge-Tate decomposition}) yields the $\Gamma_K$-equivariant commutative diagram
\begin{equation}
\begin{tikzcd}[column sep=small]
0 \arrow[r] & \Tp(M) \arrow[r] \arrow[d, Rightarrow, no head] & \cB_{\sG} \arrow[r] \arrow[d]                                            & \sG(\cO_{\bar{K}}) \arrow[r] \arrow[d]                                & 0 \\
0 \arrow[r] & \Tp(M) \arrow[r] \arrow[d, Rightarrow, no head] & {\lim(\Hom(\Tp(\sG\ve),\mup(\cO_{\bar{K}})))} \arrow[r] \arrow[d, "\cong"] & {\Hom(\Tp(\sG\ve),\mup(\cO_{\bar{K}}))} \arrow[r] \arrow[d, "\cong"] & 0 \\
0 \arrow[r] & \Tp(M) \arrow[r]                                & \lim(\mup(\cO_{\bar{K}})\otimes \Tp(M)(-1)) \arrow[r]                    & \mup(\cO_{\bar{K}})\otimes\Tp(M)(-1) \arrow[r]                       & 0
.\end{tikzcd}
\end{equation}
Taking Galois cohomology leads to 
\begin{equation}\label{6.2.3}
\begin{tikzcd}
\sG(\cO_K) \arrow[r, "\delta"] \arrow[d, "\log"] & {H^1(K,\Tp(M))} \arrow[d] \\
\Lie(\sG)(K) \arrow[r, "\exp", dashed]           & {H^1(K,\Vp(M))}          
,\end{tikzcd}
\end{equation}
where $\delta$ is the connecting map. We define the dotted arrow to make the above diagram commute.

We use the terminologies defined in \cite{cartier_l-functions_2007}. The fundamental exact diagram in \cite{cartier_l-functions_2007} is 
\begin{equation*}
\begin{tikzcd}
0 \arrow[r] & \Q_p \arrow[r] \arrow[d,  Rightarrow, no head] & \Bcris^{\phi=1}\oplus\BdRp \arrow[r, "\beta"] \arrow[d, hook] & \BdR \arrow[d, hook]\arrow[r] & 0 \\
0 \arrow[r] & \Q_p \arrow[r]                                & \Bcris\oplus\BdRp \arrow[r, "\gamma"]                      & \Bcris\oplus\BdR  \arrow[r]   & 0
\end{tikzcd}
\end{equation*}
where,
\[
\beta(x,y):=x-y,\,\, \gamma(x,y)=(x-\phi(x),x-y)
\]
and $\phi$ is the Frobenius on $\Bcris$. Recall the natural filtration on $\DdR(V)$ as described in \ref{filtration on DdR(V)}. We denote $\Fil^0(\DdR(V))=\DdR(V)^0$. Tensoring the above diagram with $V:=\Vp(M)$ and passing to cohomology gives
\begin{equation*}
\begin{tikzcd}
0 \arrow[r] & {H^0(K,V)} \arrow[r] & \Dcris(V)^{\phi=1}\oplus\DdR(V)^0 \arrow[r] & \DdR(V)\\ \arrow[r] & {H^1(K,V)}\arrow[r] & {H^1(K,V\otimes(\Bcris^{\phi=1}\oplus\BdRp))} \arrow[r] & {H^1(K,V\otimes\BdR)}
\end{tikzcd}
\end{equation*}
and 
\begin{equation*}
\begin{tikzcd}
0 \arrow[r] & {H^0(K,V)} \arrow[r] & \Dcris(V)\oplus\DdR(V) \arrow[r] & \Dcris(V)\oplus\DdR(V)\\ \arrow[r] & {H^1(K,V)} \arrow[r] & {H^1(K,V\otimes(\Bcris\oplus\BdRp))} \arrow[r] & {H^1(K,V\otimes(\Bcris\oplus\BdR))}
,\end{tikzcd}
\end{equation*}
where $\DdR(V)^0:=\Fil^0\DdR(V)=(V\otimes\BdRp)^{\Gamma_K}$. By \cite[Lemma 3.8.1]{cartier_l-functions_2007}, we know that $H^1(K,V\otimes\BdRp)\to H^1(K,V\otimes\BdR)$ is injective, so we get a diagram 
\begin{equation*}
\begin{tikzcd}
\Dcris(V)^{\phi=1}\oplus\DdR(V)^0 \arrow[r] \arrow[d, hook] & \DdR(V) \arrow[r] \arrow[d, hook] & {H^1_e(K,V)} \arrow[r] \arrow[d, hook] & 0 \\
\Dcris(V)\oplus\DdR(V)^0 \arrow[r]                          & \Dcris(V)\oplus\DdR(V) \arrow[r]  & {H^1_f(K,V)} \arrow[r]                 & 0
\end{tikzcd}
\end{equation*}
where 
\begin{align}\label{def: H^1_e, H^1_f}
    H^1_e(K,V):=\ker(H^1(K,V)\to H^1(K,V\otimes\Bcris^{\phi=1}))\\
    H^1_f(K,V):=\ker(H^1(K,V)\to H^1(K,V\otimes\Bcris)).
\end{align}
$\DdR(V)^0$ is in the kernel of $\DdR(V)\to H^1_e(K,V)$, so we can get a surjective map
\begin{equation}\label{exp for tate module Tp}
\exp:\DdR(V)/\DdR(V)^0\to H^1_e(K,V)    
\end{equation}
with kernel $\ker(\exp)=\Dcris(V)^{\phi=1}/H^0(K,V)$.

Since $V$ is de Rham, $\DdR(V)/\DdR(V)^0$ is indeed the tangent space i.e. $\DdR(V)/\DdR(V)^0=\Lie(\sG)=\Lie(G)_{\cO_K}$ (\cite[\S 6]{fontaine_sur_1982}). We want to show that the map $\exp$ makes the diagram \ref{6.2.3} commutative. We follow an argument similar to one in \cite{cartier_l-functions_2007}.

In the case of multiplicative group, we have the commutative diagram 
\begin{equation*}
\begin{tikzcd}
0 \arrow[r]          & \Z_p(1) \arrow[d]\arrow[r]                      & \cB_{\mup}\arrow[r] \arrow[d, "\log_{cris}"]          & \mup(\cO_{\bar{K}}) \arrow[d, "\log"]\arrow[r] & 0  \\
0 \arrow[r]          & \Q_p(1)\arrow[r] \arrow[d, Rightarrow, no head] & \Bcris^{\phi=p}\cap\BdRp\arrow[r] \arrow[d, "\psi"] & \C_p \arrow[d]\arrow[r]                        & 0  \\
0 \arrow[r] & \Q_p(1) \arrow[r]                      & \Bcris^{\phi=1}\otimes\Z_p(1) \arrow[r]    & (\BdR/\BdRp)(1) \arrow[r]             & {0}
\end{tikzcd}
\end{equation*}
where $\psi$ is given by $\psi(x)=xt^{-1}\otimes t$, $\log_{cris}$ is the map in \cref{log_cris}, and the bottom arrow can actually be obtained by tensoring the fundamental exact sequence of \cref{Fundamental exact sequence of p-adic Hodge theory} with $\Z_p(1)$. Put $T:=\Tp(M)$.

Tensoring with $T(-1)=\Tp(M)(-1)$ and passing to cohomology gives the commutative diagram
\begin{equation*}
\begin{tikzcd}[column sep=tiny]
{H^0(K,\mup(\cO_{\bar{K}})\otimes T(-1))} \arrow[d] \arrow[r, "\delta"] & {H^1(K,T)} \arrow[r, hook]          & {H^1(K,V)} \arrow[r] & {H^1(K,V\otimes\Bcris^{\phi=1})} \\
{H^0(K,\C_p\otimes T(-1))} \arrow[r]                                    & \DdR(V)/\DdR(V)^0 \arrow[ru, "\exp"] &                      &                                
\end{tikzcd}
\end{equation*}
Notice that the top row is not exact, but gives a complex. The image of $\delta$ is indeed $H^1_e(K,T)$ which is the set of all classes in $H^1(K,T)$ whose image under $H^1(K,T)\to H^1(K,V)$ lie in $H^1_e(K,V)$. Moreover, $H^0(K,\mup(\cO_{\bar{K}})\otimes T(-1))=\sG(\cO_K)$ and $H^0(K,\C_p\otimes T(-1))=\Lie(\sG)$. Thus the left vertical map coincides with the logarithm map $\log_{\sG}: \sG(\cO_K)\to \Lie(\sG)$, and we have the commutative diagram
\begin{equation}\label{6.2.8}
\begin{tikzcd}
                                          & \sG(\cO_K) \arrow[d, "\delta"] \arrow[ld, "\log_{\sG}"'] \\
\DdR(V)/\DdR(V)^0 \arrow[r, two heads, "\exp_{\sG}"] & {H^1_e(K,V)}                                  
.\end{tikzcd} 
\end{equation}
This means that $\log_{\sG}$ is a local inverse of $\exp_{\sG}$. 
This diagram is compatible with finite extensions $K'$ of $K$. This means that we have the commutative diagram
\begin{equation*}
\begin{tikzcd}
                                                                        & \sG(\cO_K) \arrow[d] \arrow[ld, "\log_{\sG}"'] \arrow[r, hook] & \sG(\cO_{K'}) \arrow[d] \arrow[rd, "\log"] &                                                \\
\DdR(V)/\DdR(V)^0 \arrow[r, "\exp_{\sG}"] \arrow[rrr, hook, bend right] & {H^1_e(K,V)} \arrow[r, hook]                                   & {H^1_e(K',V)}                              & \DdR(V)/\DdR(V)^0\otimes K' \arrow[l, "\exp"']
\end{tikzcd}
\end{equation*}
for any finite extension $K'$ of $K$.

\begin{prop}\label{proposition 6.2.2}
    If $\sG$ is connected, then $\delta$ is injective and $\exp_{\sG}:\DdR(V)/\DdR(V)^0\to H^1_e(K,V)$ is a bijection.
\end{prop}
\begin{proof}
The kernel of $\exp_{\sG}$ is $\ker(exp_{\sG})=\Dcris(V)^{\phi=1}/H^0(K,V)$. Let $P(V,u)=\det_{K_0}(1-\sigma^{[K_0:\Q_p]}u:\Dcris(V))$ be the characteristic polynomial of the $K_0$-linear action of $\sigma^{[K_0:\Q_p]}$ induced by the Frobenius $\sigma$ of $K_0$ (see \cite[\S 4]{cartier_l-functions_2007}). Since $\sG$ is connected, we have $P(V,1)\neq 0$. Therefore, $\exp_\sG$ is bijective. The injectivity of $\delta$ follows from \cite[\S 5]{cartier_l-functions_2007}.
%
\end{proof}

Now, we want to show that the map $\exp_{\sG}$ coincides with the exponential map introduced in \ref{eq 3.5.1}.
\begin{defn}
Let $D$ be a filtered Dieudonn\'e module (see \cref{filtered Dieudonn\'e module}). We define
\begin{equation*}
   h^i(D)= \begin{cases}
        \ker(1-F:D^0\to D),& i=0\\
        \coker(1-F:D^0\to D),& i=1\\
        0,& i\geq 2.
    \end{cases}
\end{equation*}
\end{defn}

The category of filtered Dieudonn\'e modules is an abelian category. Therefore, we have the identification
\[
h^i(D)\cong \Ext^i(1_{FD},D) \text{ for all } i\in\Z
\]
where $1_{FD}$ is the unit filtered Dieudonn\'e module (see \cref{unit filtered Dieudonn\'e module}). We define
\[
T(D):=\ker(1-\phi:\Fil^0(D\otimes_{\W(k)}\Acris)\to D\otimes_{\W(k)}\Acris).
\]
By \cite{Fontaine_Laffaille_1982}, it is known that the functor $D\mapsto T(D)$ is an exact and fully faithful functor from the category of filtered Dieudonn\'e modules satisfying the condition
\[
(*)\hspace{2cm} \exists\, i,j\in\Z,\,\, D^i=D, D^j=0, \text{ and } j-i<p
\]
to the category of finite $\Z_p$-modules endowed with a continuous action of $\Gamma_K$. Furthermore, $T(D)$ gives a crystalline representation and \[\Dcris(T(D)\otimes\Q)\cong D\otimes_{\W(k)} K.\]

In \cite[Lemma 4.5]{cartier_l-functions_2007}, it is shown that when condition $(*)$ holds, the canonical map
\begin{equation}\label{isomorphism for h1 with H1(K,T)}
h^i(D)=\Ext^i(1_{FD},D)\to\Ext^i_{\Z_p[\Gamma_K]}(\Z_p,T(D)) 
\end{equation}
is an isomorphism if $i=0$, and it is injective when $i=1$. Moreover, the image of $h^1(D)\to \Ext^1_{\Z_p[\Gamma_K]}(\Z_p,T(D))=H^1(K,T(D))$ is indeed $H^1_e(K,T(D))$ if $D$ has no p-torsion. Recall that $H^1_e(K,T(D))$ is the set of all classes in $H^1(K,T(D))$ whose image lie in $H^1_e(K,T(D)\otimes\Q)$.

\begin{remark}
If $D$ is the filtered Dieudonn\'e module associated to the 1-motive $M$, then $D$ satisfies condition $(*)$. According to \cref{Theorem: filtered isocrystal va crystalline representaion}(4), we have
\[
\Vp(M)\cong\Fil^0(D\otimes_{\W(k)}\Bcris)^{\phi=1}\cong \Fil^0(D\otimes_{\W(k)}\Acris)^{\phi=1}\otimes\Q=T(D)\otimes_{\Z_p}\Q.
\]
Thus,
\begin{equation}\label{diagram 3.5.10}
 h^0(D)\cong H^0(K,\Tp(M)), \text{ and } h^1(D)\cong H^1_e(K,\Tp(M))   
\end{equation}
and the following canonical diagram commutes
\begin{equation}
\begin{tikzcd}
(D/D^0)\otimes\Q \arrow[r, "1-F"] \arrow[d, "\cong"] & h^1(D)\otimes\Q \arrow[d, "\cong"] \\
\DdR(\Vp(M))/\DdR(\Vp(M))^0 \arrow[r, "\exp"]        & {H^1_e(K,\Vp(M))}        
\end{tikzcd}
\end{equation}
where the top arrow map is $1-F:(D/D^0)\otimes\Q\to (D/(1-F)D^0)\otimes\Q=h^1(D)\otimes\Q$. Both $(D/D^0)\otimes\Q$ and $\DdR(\Vp(M))/\DdR(\Vp(M))^0$ identify the tangent space. Hence, this diagram together with the commutative diagram \ref{6.2.8} gives the commutative diagram
\begin{equation}\label{diagram 3.5.11}
\begin{tikzcd}
                                       & \sG(\cO_K) \arrow[d] \arrow[ld, "\log_{\sG}"'] \\
(D/D^0)\otimes\Q \arrow[r, "\exp_{D}"] & h^1(D)\otimes\Q                               
\end{tikzcd}
\end{equation}
\end{remark}
\begin{cor}
    If $\sG$ is connected, then \[\exp:(D/D^0)\otimes\Q\to h^1(D)\otimes\Q\]
    is an isomorphism.
\end{cor}
\begin{proof}
    By construction, we know that $\exp$ is surjective and its kernel is $(D^{F=1}/h^0)\otimes\Q$. \cref{proposition 6.2.2} implies that $\exp_\sG$ is bijective. The result now follows from the diagram \ref{diagram 3.5.10}.
\end{proof}
\section{Crystalline algebraic points}\label{section crystalline algebraic points}
We aim to provide a classification of the vectors within the image of the logarithm map from the perspective of p-adic representations, as obtained in \cref{theorem classification for algebraic poin in image of logarithm}. Additionally, we define a $\Q$-structure denoted $\hp(M,\bar{\K})$, which will play an important role in the construction of our p-adic periods in the following chapter.

Let $M$ be a 1-motive over a number field $\K$ with good reduction at $p$ and $K$ a complete local field containing $\K$ (fix an embedding $\K\into K$). Let $M_K$ and $M_{\cO_K}$ denote the base change of $M$ to $K$ and extension of $M$ to $\cO_K$, respectively. We may drop the index $K$, and $\cO_K$ when it is known from the context.

\begin{defn}\label{def: tilde tate module}
Let $\sG$ be the p-divisible group associated to $M_{\cO_K}$ over $\cO_K$ and $\sG^0$ its connected component. We define
    \[
    \tilde{\Tp}(M):=\Tp(\sG^0)\subset\Tp(M) \text{ and } \tilde{\Vp}(M)=\tilde{\Tp}(M)\otimes_{\Z_p}\Q_p.
    \]
\end{defn}
Consider the above notations. The Galois group $\Gamma_{\K}$ acts continuously on $\Tp(M)$ and the action of $\Gamma_{K}$ is via $\Gamma_K\into\Gamma_{\K}$. This gets a restriction map $H^1(\K,\Tp(M))\to H^1(K,\Tp(M))$ which does not depend on the choice of embedding $\bar{\K}\into\bar{K}$. In fact, if $\lambda_1, \lambda_2$ are two such embedding, then $\lambda_1=\lambda_2\circ\sigma$ for some $\sigma\in\Gamma_{\bar{\K}}$, and since the $\Gamma_{\bar{\K}}$ acts trivially on $H^1(\K,\Tp(M))$, we conclude that both embeddings $\lambda_1$ and $\lambda_2$ induce the same map on the cohomology groups.
\begin{prop}\label{3.5.2}
Let $\K\subseteq\barQ$ and $K$ a complete local field with a fixed embedding $\K\into K$. If $H^0(K,\Tp(M))=0$, then the map $H^1(\K,\Tp(M))\to H^1(K,\Tp(M))$ is injective. 
\end{prop}
\begin{proof}
From \cite[B.2]{silverman_arithmetic_2009}, we can derive the exact sequence     \[
    0\to H^1(\Gamma_{\K}/\Gamma_K, H^0(\K,\Tp(M)))\to H^1(\K,\Tp(M))\to H^1(K,\Tp(M)).
    \]
We only need to show that $H^0(\K,\Tp(M))=0$ which follows from
\[
H^0(\K,\Tp(M))\into H^0(K,\Tp(M))=0.
\]
\end{proof}

Notice that the exponential map $\DdR(\tilde{\Vp}(M))/\DdR(\tilde{\Vp}(M))^0\xrightarrow{\exp_{\sG^0}} H^1_e(K,\tilde{\Vp}(M))$ and the connecting map $\sG^0(\cO_K)\to H^1_e(K,\tilde{\Vp}(M))$ for $\sG^0$ are both injective, since $\sG^0$ is connected.
\begin{defn}\label{def tilde exp}
    Define the map $\tilde{exp_{\sG}}$ as the composition
    \[
    \DdR(\tilde{\Vp}(M))/\DdR(\tilde{\Vp}(M))^0\xrightarrow{\exp_{\sG^0}} H^1_e(K,\tilde{\Vp}(M))\to H^1_e(K,\tilde{\Vp}(M))/\sG^0(\cO_K),
    \]
    and set $\hp(M,K):=\ker(\tilde{\exp_{\sG}})\otimes_{\Z}\Q$.
    We have the following commutative diagram by construction:
    \begin{equation*}
\begin{tikzcd}
{\hp(M,K)} \arrow[r, hook] \arrow[d, hook] & \sG^0(\cO_K)\otimes\Q \arrow[d, hook] \\
\Lie(G)_K \arrow[r, "\exp"]                & {H^1_e(K,\tilde{\Vp}(M))}            
.\end{tikzcd}
    \end{equation*}
\end{defn}

For the $\Q$-structure $\hp(M,K)$, we have the following
\begin{prop}\label{fundamental equality for hp in chapter 5}
    We have 
    \[
\hp(M,K)=\ker(\tilde{\exp_{\sG}})\otimes_{\Z}\Q=\im(\log_{\sG(\cO_K)})\otimes_{\Z}\Q.
    \]
\end{prop}
\begin{proof}
    First one can establish the equality $\ker(\tilde{\exp_{\sG}})=\im(\log_{\sG^0})$ from the diagram \ref{6.2.8} for $\sG^0$. Thus, it suffices to show that $\im(\log_{\sG})\otimes\Q=\im(\log_{\sG^0})\otimes\Q$. Assume that $x$ belong to $\sG(\cO_K)$. Then, there exists some positive integer $n$ such that $p^nx\in\sG^0(\cO_K)$ and $\log_{\sG}(x)=\frac{\log_{\sG^0}(p^nx)}{p^n}$, as indicated in \cref{def logarithm p-divisible group}. Therefore, $\log_{\sG}(x)$ belongs to $\im(\log_{\sG^0})\otimes\Q$. This completes the proof.
\end{proof}
Notice that $\hp(M,K)$ is compatible with extensions $K'$ of $K$. We have the following definition.
\begin{defn}\label{definition 3.6.3}
 We define $\hp(M,K')$ to  be the extension of $\hp(M,K)$ to $\cO_{K'}$ i.e. $\hp(M,K')=\im(\log_{\sG(\cO_{K'})})\otimes\Q$.
 When $M$ is a 1-motive over a number field $\K$ contained in $K$, we define
 \[
 \hp(M,\K):=\im(\left.\log\right|_{\sG(\K)}:\sG(\cO_K)\to\Lie(G)_K)\otimes\Q,
 \]
 where $\sG(\K)=\sG(\cO_K)\cap G(\K)$, following the notation in \cref{def: algebraic points on G}. Finally, we define
 \[
 \hp(M)=\hp(M,\bar{\K}):=\colim \hp(M,\K'),
 \]
 where the direct limit is taken over all finite extensions $\K'$ of $\K$.
\end{defn}
Let $\tilde{V}:=\tilde{\Vp}(M)=\Vp(\sG^0)$.
Let $x\in H^1(K,\tilde{V})=\Ext^1(\Q_p,\tilde{V})$. The point $x$ corresponds to an exact sequence 
\begin{equation}\label{3.5.13}
0\to \tilde{V}\to E_x\to \Q_p\to 0
\end{equation}
of $\Gamma_K$ p-adic representations. Via the canonical isomorphism $H^1_e(K,\tilde{V})\cong\Ext^1(1_{FD},D)$ induced by \ref{isomorphism for h1 with H1(K,T)}, where $D$ is the covariant Dieudonn\'e module associated with $\sG^0$, the extension \ref{3.5.13} corresponds to an extension of the unit Dieudonn\'e module $1_{FD}$ by $D$ in the category of filtered Dieudonn\'e modules if and only if $x$ lies in $H^1_e(K,\tilde{V})$. In other words, for any $x\in H^1(K,\tilde{V})$ the corresponding extension $E_x$ is a crystalline p-adic representation if and only if $x\in H^1_e(K,\tilde{V})$.

By \cref{kernel Fontaine map phi}, we know that the map $\phi_M\left.\right|_{\tilde{V}}:\tilde{V}\to \Lie(G)\otimes\C_p(1)$ is injective. We obtain the diagram 
\begin{equation}\label{3.5.14}
\begin{tikzcd}
0 \arrow[r] & \tilde{V} \arrow[d, "\cong"] \arrow[r] & E_x \arrow[d, hook] \arrow[r]   & \Q_p \arrow[d, hook] \arrow[r]                            & 0 \\
0 \arrow[r] & \phi_M(\tilde{V}) \arrow[r]            & \Lie(G)\otimes\C_p(1) \arrow[r] &\displaystyle \frac{\Lie(G)\otimes\C_p(1)}{\phi_M(\tilde{V})} \arrow[r] & 0
\end{tikzcd}
\end{equation}
Passing to Galois cohomology, we get an exact sequence
\begin{align*}
H^0\left(K,\Lie(G)\otimes\C_p(1)\right)\to H^0\left(K,\frac{\Lie(G)\otimes\C_p(1)}{\phi_M(\tilde{V})}\right)\to H^1(K,\phi_M(\tilde{V}))\\=H^1(K,\tilde{V})\to H^1(K,\Lie(G)\otimes\C_p(1)).    
\end{align*}

We have $H^i(K,\Lie(G)\otimes\C_p(1))=0$ for all integers $i$, therefore
\begin{equation}\label{3.5.15}
H^0\left(K,\frac{\Lie(G)\otimes\C_p(1)}{\phi_M(\tilde{V})}\right)\cong H^1\left(K,\tilde{V}\right).    
\end{equation}
We have a diagram similar to \ref{3.5.14} and a canonical isomorphism similar to \ref{3.5.15} for any finite extension $K'$ of $K$. By \cref{3.5.2}, we know that the map $H^1(\K,\tilde{V})\to H^1(K,\tilde{V})$ is injective.

We now present the following definition, inspired by \cite[Definition 3.7]{iovita_p_2022}.
\begin{defn}\label{definition algebraic crystalline points}
    Let $\displaystyle x\in \frac{\Lie(G)\otimes\C_p(1)}{\phi_M(\tilde{V})}$.
    \begin{enumerate}
        \item We say that $x$ is an algebraic point over $K$, if the orbit $\Gamma_K.x\subset\frac{\Lie(G)\otimes\C_p(1)}{\phi_M(\tilde{V})}$ is finite. In other words, if there exists a finite extension $K'$ of $K$ such that $$\displaystyle x\in\left(\frac{\Lie(G)\otimes\C_p(1)}{\phi_M(\tilde{V})}\right)^{\Gamma_{K'}}.$$
        \item We say that $x$ is an algebraic point over $\K$, if there exists a finite extension $\K'$ of $\K$ such that $x$ belongs to the image of 
        \[
        H^1(\K',\tilde{V})\to H^1(\K'_p,\tilde{V})\cong \left(\frac{\Lie(G)\otimes\C_p(1)}{\phi_M(\tilde{V})}\right)^{\Gamma_{\K'_p}},
        \]
        where $\K'_p$ is the p-adic completion of $\K'$.
        \item We say that $x$ is a crystalline point, if $\alpha^{-1}(x\Q_p)$ is a crystalline representation, where \[\alpha:\Lie(G)\otimes\C_p(1)\to\frac{\Lie(G)\otimes\C_p(1)}{\phi_M(\tilde{V})}\] is the quotient map. 
    \end{enumerate}
\end{defn}

\begin{thm}\label{theorem classification for algebraic poin in image of logarithm}
Let $\tilde{V}:=\tilde{\Vp}(M)$ and $\operatorname{h_e}(\tilde{V}):=\colim H^1_e(K,\tilde{V})$, where the colimit runs over all finite extensions of $K$.
\begin{enumerate}

\item There is a one-to-one correspondence between $\operatorname{h_e}(\tilde{V})$ and crystalline points in $(\Lie(G)\otimes_{\cO_K}\C_p(1))/\phi(\tilde{V})$ that are algebraic over $K$.

\item Every point in $\hp(M)$ corresponds to a crystalline point in $(\Lie(G)\otimes_{\cO_K}\C_p(1))/\phi(\tilde{V})$ which is algebraic over $\K$.

\end{enumerate}
\end{thm}

\begin{proof}
    \begin{enumerate}
        \item Let $x\in \operatorname{h_e}(\tilde{V})=\colim H^1_e(K,\tilde{V})$. There exists a finite extension $K'$ of $K$ such that $x\in H^1_e(K',\tilde{V})$. Via the isomorphism \ref{3.5.15}, $x$ maps to an element $y\in\Lie(G)\otimes\C_p(1)/\phi_M(\tilde{V})$, which is algebraic over $K$, as it lies in $$(\Lie(G)\otimes\C_p(1)/\phi_M(\tilde{V}))^{\Gamma_{K'}}.$$
        We consider the diagram \ref{3.5.14} associated with $x$. Since $x$ is in $H^1_e(K,\tilde{V})$, we know that $E_x$ is crystalline. Moreover, $\alpha^{-1}(y\Q_p)=E_x$ by the definition. Therefore, $y$ is a crystalline point which is algebraic over $K$.

        Conversely, assume that $y\in (\Lie(G)\otimes_{\cO_K}\C_p(1))/\phi(\tilde{\Vp}(M))$ is a crystalline point that is algebraic over $K$. Thus, $y$ lies in $(\Lie(G)\otimes_{\cO_K}\C_p(1))/\phi(\tilde{\Vp}(M))^{\Gamma_{K'}}$ for some finite extension $K'$ of $K$. Again, via the isomorphism $\ref{3.5.15}$, $y$ maps to an element $x$ in $H^1(K',\tilde{V})$. Since $\alpha^{-1}(y\Q_p)=E_x$ is crystalline, this implies that $x$ lies in $H^1_e(K',\tilde{V})$.

        \item Every point $x\in\hp(M,\bar{\K})$ corresponds to a point in $y\in\sG(\K')\subset H^1(\K',\tilde{V})\into H^1(\K'_p,\tilde{V})$ (see \cref{def tilde exp}), where $\K'$ is a finite extension of $\K$. However, due to the diagram \ref{6.2.8}, $y$ belongs to $H^1_e(\K'_p,\tilde{V})$. Thus the result follows from part (1).
    \end{enumerate}
\end{proof}

The $\Q$-structure $\hp(M):=\hp(M,\bar{\K})$ is the main ingredient in the construction of our p-adic periods in the next chapter.

\clearpage
\chapter{P-adic periods of 1-motives}
The classical period numbers arise from the Grothendieck's comparison isomorphism between the de Rham cohomology of a smooth variety $X$ over a subfield $K\into\C$ with the singular cohomology of the analytification $X^{an}/\C$. The entries of this isomorphism with respect to all choice of $\barQ$-bases are called periods of $X$ which are indeed induced by integrals of global algebraic differential forms $\omega$ on $X$ over a differentiable relative chains $\sigma$.

The period conjecture asserts that there are no relations among the classical periods other than obvious relations  (for more details, see \cite{huber2017periods} and \cite{huber2018galois}). There are two primary approaches to studying these periods: investigating their algebraic relations or exploring the linear relations among these periods. For studying the algebraic relations, a neutral Tannakian category is required to construct the period algebras. In this context, the classical period conjecture predicts that the transcendental degree of the period algebra associated with a motive $M$ is equal to the dimension of the motivic Galois group $G(M)$. 

To study only the linear relations of periods, we do not need a rigid tensor category in this setting. In \cite{huber2022transcendence}, it was demonstrated that the period conjecture holds for the isogeny category of Deligne 1-motives over $\barQ$, using the celebrated Wustholz analytic subgroup theorem \cite{wustholz1989algebraische}. In fact, their results show that the Kontsevich-Zagier period conjecture holds for all smooth curves over $\barQ$.

The primary object of this chapter is to present and prove a p-adic version of the period conjecture as well as a p-adic version of the subgroup theorem for 1-motives with good reductions. We develop a formalism that enables the study of periods arising from the integration pairings introduced in Chapter 3. This formalism will encompass classical periods as well. It yields linear spaces of periods, and, as will be demonstrated at the end of this chapter, it provides insights into their vanishing behaviour and describes all the linear relations among them. Drawing inspiration from \cite{hormann2021note}, we introduce period conjectures within this framework, considering different depths of formal relations among periods.

Finally, we will demonstrate that the p-adic integration pairing (\cref{theorem p-adic integration pairing for M}) that we introduced in Chapter 3 induces three distinct pairings. By applying a p-adic version of the subgroup theorem for 1-motives, which we will establish in \cref{p-adic subgroup theorem for Fontaine pairing}, we can prove the period conjectures at depths 1 and 2 for these pairings, as outlined in \cref{level 2 period cinjecture for Hpp} and \cref{depth 2 and 1 period cinjecture for Hp and hp}.

\section{Formalism of periods}

\begin{defn}[Realization category]
Let $K$ and $L$ be two fields containing $\Q$, and $\barQ$ respectively. Let $B$ be an algebra over both $K$ and $L$. The realisation category with coefficients in $B$, denoted $\Mod^{B}_{K,L}$, is a category with the following information:
\begin{enumerate}
    \item[(i)] Objects are triples $(H_K,H_L,\varpi)$ where $H_K$ and $H_L$ are finite-dimensional vector spaces over $L$ and $K$ respectively, and $\varpi\colon H_K\otimes_K B\to H_L\otimes_L B$ is a $B$-linear isomorphism.
    \item[(ii)] Morphisms are pairs $\varphi:=(\varphi_K,\varphi_L)$ where $\varphi_L\colon H_L\to H'_L$ is a $L$-linear map, and $\varphi_K\colon H_K\to H'_K$ is a $K$-linear map such that the diagram
    \begin{equation*}
    \begin{tikzcd}
    H_{K}\otimes_K B \arrow[r, "\varpi"] \arrow[d, "\varphi_{K}\otimes 1_B"] & H_{L}\otimes_{L}B \arrow[d, "\varphi_{L}\otimes 1_{B}"] \\
    H'_{K}\otimes_K B \arrow[r, "\varpi' "]                    & H'_{L}\otimes_{L}B             
    \end{tikzcd}
    \end{equation*}
commutes.
\end{enumerate}
\end{defn}
\begin{remark}
The realisation category $\Mod^B_{K,L}$ is an abelian rigid tensor category. We have tensor structure, dual, and identity object as follows:
\begin{itemize}
    \item For $H=(H_{K},H_{L},\varpi)$ and $H'=(H'_{K}, H'_{L}, \varpi')$ we define $$H\otimes H':=(H_{K}\otimes_{K} H'_{K}, H_{L}\otimes_{L} H'_{L}, \varpi\otimes\varpi').$$
    \item We have duals $H^{\vee}=(H^{\vee}_{K},H^{\vee}_{L},\varpi^{\vee})$, where $H^{\vee}_{L}=\Hom(H_{L},L)$ is the dual $L$-vector space, $H^{\vee}_{K}=\Hom(H_{K},K)$ is the dual $K$-vector space, and $\varpi^{\vee}\colon H^{\vee}_{L}\to H^{\vee}_{K}\otimes_{B}L$ is given by $\varpi^{\vee}(f)=f\circ\varpi^{-1}$ under the identification $(H_{K}\otimes_{B}L)^{\vee}\cong H^{\vee}_{K}\otimes_{B}L$. Clearly $(H^{\vee})^{\vee}=H$.
    \item The identity object in $\Mod^B_{K,L}$ is $\mathbbm{1}=(K,L,1)$. 
\end{itemize}
The category $\Mod^B_{K,L}$ admits an internal $\ihom$ structure. We have $H^{\vee}=\ihom(H,\mathbbm{1})$.
\end{remark}

\begin{defn}\label{period pairing for T}
Let $\cC$ be an additive category, and \[T:\cC\to \Mod^B_{K,L},\, X\mapsto(T_K(X),T_L(X),\varpi_X)\] an additive functor. Let $T\ve_L:\cC\to\Vect(L)$ denote the functor $X\mapsto T_L(X)\ve$, and let $U$ be an algebraic extension of $\Q$ that makes $B$ a $U$-algebra. The additive functor $\cH=(F,G):\cC\to\Vect(\Q)\times\Vect(U)$ is called a period pairing for $T$, if
\begin{enumerate}
\item $F$ is an additive covariant exact functor and $G$ is an additive contravariant exact functor.
    \item There exist natural embeddings $F\into T_K\otimes_K B$, and $G\into T\ve_L\otimes_L B$, i.e., for every object $X$ of $\cC$, $F(X)$ and $G(X)$ are $\Q$-linear and $U$-linear subspaces of $T_K(X)\otimes_K B$ and $T_L\ve(X)\otimes_L B$, respectively. Moreover, the diagrams
    \begin{equation*}
    \begin{tikzcd}
    F(X) \arrow[r, hook] \arrow[d, "F(f)"] & T_K(X)\otimes_K B \arrow[d, "T_{K}(f)\otimes B"] & &  G(X) \arrow[r, hook] \arrow[d, "G(f)"] & T_L\ve(X)\otimes_L B \arrow[d, "T\ve_{L}(f)\otimes B"] \\
    F(X') \arrow[r, hook]                    & T_{K}(X')\otimes_K B & &   G(X') \arrow[r, hook]                    & T_{L}\ve(X')\otimes_L B          
\end{tikzcd}
    \end{equation*}
    commute for any morphism $X\xrightarrow{f}X'$ in $\cC$.
\end{enumerate}
\end{defn}

\begin{defn}\label{Space of H-periods}
Let $\cH=(F,G)$ be a period pairing for $T:\cC\to\Mod_{K,L}^B$.
\begin{enumerate}
    \item We define the $\cH$-periods of $X$ in $\cC$ as \[\cP_\cH(X):=\im(F(X)\times G(X)\to B), \, (\nu,\gamma)\mapsto \gamma(\varpi_X(\nu)), \]
    where we view $\nu$ and $\gamma$ in $T_K(X)$ and $T_{L}\ve(X)$ respectively. The above pairing is denoted by $\langle\nu,\gamma\rangle_\cH:=\gamma(\varpi_X(\nu))$.
    \item The space of $\cH$-periods $\cP_\cH\langle X\rangle$ is the $U$-vector space generated by $\cP_\cH(X)$.
    \item If $\cD$ is a full additive subcategory of $\cC$, we define the $\cH$-periods of $\cD$ to be
    \[\cP_\cH(\cD):=\bigcup_{X\in\cD}\cP_\cH(X). \]
\end{enumerate}
\end{defn}
\begin{remark}
\begin{enumerate}
    \item Assume that $F(X)$, and $G(X)$ are finite-dimensional. The $\cH$-period space $\cP_\cH\langle X\rangle$ is indeed the $U$-vector space generated by the entries of all matrix $M_{S,S'}(X)$ where $S$ is a $\Q$-basis for $F(X)$, $S'$ is a $U$-basis for $G(X)\ve$ and the matrix $M_{S,S'}(X)$ is a vertical rectangular matrix whose entries are the coordinates of $S$ with respect to $S'$ through the $B$-isomorphism $\varpi_X$.
    \item The set $\cP_\cH(\cC)$ only depends on the objects in the image of $F$ and $G$.
    \item  We have two types of obvious relations between the $\cH$-periods of $X$ in $\cC$:
        \begin{itemize}
            \item Bilinearity: \[\langle a_1\nu_1+a_2\nu_2,b_1\gamma_1+b_2\gamma_2\rangle=a_1b_1\langle \nu_1,\gamma_1\rangle+a_1b_2\langle \nu_1,\gamma_2\rangle+a_2b_1\langle \nu_2,\gamma_1\rangle+a_2b_2\langle \nu_2,\gamma_2\rangle\] where $a_1,a_2\in\Q$, $b_1,b_2\in U$, $\nu_1,\nu_2\in F(X)$, and $\gamma_1,\gamma_2\in G(X)$.
            \item Functoriality: \[\langle f_*\nu,\gamma\rangle=\langle \nu,f^*\gamma\rangle\] where $f\colon X\to Y$ is a morphism in $\cC$, $\nu\in F(X)$, and $\gamma\in G(Y)$.

        \end{itemize}

This motivates the definition of the space of formal $\cH$-periods.
    \end{enumerate}
\end{remark}

\begin{defn}\label{space of formal H-periods}
\begin{enumerate}
\item Let $\cH=(F,G)$ be a period pairing for $T:\cC\to\Mod_{K,L}^B$. The space of formal $\cH$-periods $\tilde{\cP}_\cH(\cC)$ is 
\[\displaystyle\tilde{\cP}_\cH(\cC):=(\bigoplus_{X\in\cC}F(X)\otimes_{\Q}G(X))/\text{functoriality} \]

\item The $\cH$-period map is $\text{eval}_p\colon\tilde{\cP}_\cH(\cC)\to\cP_\cH(\cC)$ which is induced by $(\nu,\gamma)\mapsto \gamma(\varpi(\nu))$ where $(\nu,\gamma)\in F(X)\times G(X)$.

\item We say that all relations among $\cH$-periods of $\cC$ induced by bilinearlity and functoriality if  the evaluation map $\tilde{\cP}_\cH(\cC)\to\cP_\cH(\cC)$ is injective.
\end{enumerate}
Note that $\cP_\cH(\cC)$ is a subspace of $B$, and the evaluation map $\tilde{\cP}_\cH(\cC)\to\cP_\cH(\cC)$ is clearly surjective. 
\end{defn}
The following definition is inspired by \cite{hormann2021note}.
\begin{defn}\label{space of formal periods at depth i}
    Let $\cH=(F,G)$ be a period pairing for $T:\cC\to \Mod^B_{K,L}$, and $X\in\cC$. The space of formal $\cH$-periods of depth $i$ of $X$, denoted $\tilde{\cP}_\cH^i(X)$, is the quotient of the vector space $F(X)\otimes_{\Q}G(X)$ modulo the subspace generated by
    \[\sum^m_{j=1}\nu_j\otimes\gamma_j\]
    for every exact sequence 
    \[
    0\to X'\to X^m\to X''\to 0
    \]
    with $\nu_j\in F(X')$, and $\gamma_j\in G(X'')$ such that $m\leq i$.

    We also define 
    \[
    \tilde{\cP}^{\infty}_\cH(X):=\colim \tilde{\cP}^i_\cH(X), \text{ and } \tilde{\cP}^{i}_\cH(\cC):=\colim_{X\in\cC}\tilde{\cP}^{i}_\cH(X).
    \]
\end{defn}

We denote the class of $\sum^m_{j=1}\nu_j\otimes\gamma_j$ in $\tilde{\cP}^i_{\cH}(X)$ by   \[\displaystyle(\sum^m_{j=1}\nu_j\otimes\gamma_j)_{\tilde{\cP}^i_{\cH}(X)}.\]

\begin{prop}\label{dimension prop}
Let $(\cC, T, \cH)$ be as above.
\begin{enumerate} 
\item $\cP_\cH(\cC)$ is a $U$-vector space. In particular, $\cP_\cH\langle X\rangle=\cP_\cH(\langle X\rangle)$, where $\langle X\rangle$ is the full additive subcategory of $\cC$ generated by $X$.

\item If $\cC$ is abelian, then $\tilde{\cP}_\cH\langle X\rangle$ is generated by the elements of $F(X)\otimes_{\Q}G(X)$ as a $U$-vector space, where here $\langle X\rangle$ is the full abelian subcategory of $\cC$ generated by $X$. We denote $\tilde{\cP}_\cH(X):=\tilde{\cP}_\cH(\langle X\rangle)$. In particular, $\dim_{U}\tilde{\cP}_\cH(X)\leq \max(\dim_{\Q}F(X),\dim_{U}G(X))^2$.

\item If $\cC$ is abelian, then $\tilde{\cP}_\cH^i\langle X\rangle$ is generated by the elements of $F(X)\otimes_{\Q}G(X)$ as a $U$-vector space.

\item If $\cC$ is abelian, then $\cP_\cH\langle X\rangle=\cP_\cH(\langle X\rangle)$.

\end{enumerate}
\end{prop}
\begin{proof}
(1) It suffices to show that $\cP_\cH(\cC)$ is closed under the addition. If $\alpha_1$ is a $\cH$-period of $X_1$ and $\alpha_2$ is a $\cH$-period of $X_2$ in $\cC$, then $\alpha_1+\alpha_2$ is a $\cH$-period of $X_1\oplus X_2$.\\
\

(2) We need to verify that all pure tensors in $\tilde{\cP}_\cH(\langle X\rangle)$ can be expressed in terms of elements in $F(X)\otimes_{\Q}G(X)$. Every object in $\langle X\rangle$ is a subquotient of some $X^n$. Note that $F(X^n)\cong F(X)^n,\, G(X^n)\cong G(X)^n$, and we can write an element $\nu\otimes\gamma\in F(X^n)\otimes_{\Q}G(X^n)$ as 
\[
\nu\otimes\gamma=\sum^n_{k=1}(i_k)_*\nu_k\otimes\gamma=\sum^n_{k=1}\nu_k\otimes(i_k)^*\gamma\in F(X)\otimes_{\Q}G(X),\] where $\nu_1,\dots,\nu_n$ are components of $\nu$. Now, if $f\colon X^n\to Y$ is a surjective morphism in $\cC$, we want to show that all pure tensors in $F(Y)\otimes_{\Q}G(Y)$ can be identified with pure tensors in $F(X^n)\otimes_{\Q}G(X^n)$ (and as a result, in $F(X)\otimes_{\Q}G(X))$). This is because every element in $F(Y)$ can be expressed as $f_*\nu$ for some $\nu$ in $F(X^n)$. Hence,
\[
f_*\nu\otimes\gamma=\nu\otimes f^*\gamma\in F(X^n)\otimes_{\Q}G(X^n).
\]
If $f\colon Y\to X^n$ is an injective morphism in $\cC$, then all pure tensors in $F(Y)\otimes_{\Q}G(Y)$ can be identified with certain pure tensors in $F(X^n)\otimes_{\Q}G(X^n)$. Since every element in $G(Y)$ has the form $f^*\gamma$ for some $\gamma\in G(X^n)$, it follows that
\[
\nu\otimes f^*\gamma=f_*\otimes\gamma\in F(X^n)\otimes_{\Q}G(X^n).\]

(3) The proof is similar to (2).\\

(4) If $0\to X_1\to X\to X_2\to 0$ is an exact sequence in $\cC$, a similar argument as above, allows us to conclude that $\cP_\cH\langle X_1\rangle +\cP_\cH\langle X_2\rangle\subseteq \cP_\cH\langle X\rangle$. This implies that for any subquotient $Y$ of $X$, $\cP_\cH(Y)\subseteq\cP_\cH\langle X\rangle$. From (1), we have $\cP_\cH(Y^n)\subseteq\cP_\cH\langle Y\rangle$. Since all objects in $\langle X\rangle$ are subquotients of some $X^n$ for some $n$, it follows that $\cP_\cH(\langle X\rangle)=\cP_\cH\langle X\rangle$.
\end{proof}

\begin{remark}
For a given short exact sequence $$0\to X_1\xrightarrow{i} X^n\xrightarrow{p}X_2\to 0$$ in $\cC$, let $\sum^n_{i=1}\nu_i\otimes\gamma_i\in F(X)\otimes_{\Q}G(X)$ with \[(\nu_1,\dots,\nu_n)\in i_*(F(X_1)), (\gamma_1,\dots,\gamma_n)\in p^*(G(X_2)),\]
that is, $\nu_i=i_*\nu'_i$, $\gamma_i=p^*\gamma'_i$, for some $\nu'_i\in F(X_1)$ and, $\gamma'_i\in G(X_2)$. Then, we have
$$\sum^n_{i=1}\nu_i\otimes\gamma_i=\sum^n_{i=1}i_*\nu'_i\otimes\gamma_i=\sum^n_{i=1}\nu'_i\otimes i^*p^*\gamma'_i=0.$$
\end{remark}

\begin{defn}\label{definition period conjecture at depth i}
    Let $(\cC, T, \cH)$ be as above. We say that the $\cH$-period conjecture at depth $i$ holds for $\cC$, if the evaluation map
    \[
    \tilde{\cP}^i_{\cH}(\cC)\to\cP_\cH(\cC)
    \] is injective.
\end{defn}

\begin{lemma}\label{colim injective period conjecture}
Let $\cH$ be a period pairing for $T:\cC\to\Mod^B_{K,L}$. Let $\cC$ be an abelian category. For any object $X$ in $\cC$, we denote by $\langle X\rangle$ the full abelian subcategory of $\cC$ generated $X$.
\begin{enumerate}
    \item The evaluation map $\tilde{\cP}_\cH(\cC)\to\cP_\cH(\cC)$ is bijective if and only if the evaluation map $\tilde{\cP}_\cH(X)\to\cP_\cH\langle X\rangle$ is bijective for every object $X$ of $\cC$.
    \item The $\cH$-period conjecture holds at depth $i$ for $\cC$, if and only if $\cH$-period conjecture holds at depth $i$ for $\langle X\rangle$ for every object $X$ of $\cC$.
\end{enumerate}
\end{lemma}
\begin{proof}
We prove (1); the proof of (2) is similar.
The natural map $\tilde{\cP}_\cH\langle X\rangle\to\tilde{\cP}_\cH(\cC)$ is injective. If $\tilde{\cP}_\cH(\cC)\to \cP_\cH(\cC)$ is injective then so is the composition $\tilde{\cP}_\cH\langle X\rangle\to\tilde{\cP}_\cH(\cC)\to\cP_\cH(\cC)$ for every object $X$ of $\cC$.

Conversely, we have $\cC=\bigcup_{X\in\cC}\langle X\rangle$ because a morphism $f\colon X\to Y$ in $\cC$ is a morphism in the full abelian subcategory $\langle X\oplus Y\rangle$. As a result, we have 
$$\tilde{\cP}_{\cH}(\cC)=\colim_{X\in\cC}\tilde{\cP}_\cH\langle X\rangle.$$
On the other hand, the evaluation map $\tilde{\cP}_\cH\langle X\rangle\to\cP_\cH\langle X\rangle$ is injective for every $\langle X\rangle$, and thus it is injective on the colimit.
\end{proof}

\begin{cor}\label{cor: P1=P for period conjecture}
For any object $X$ in $\cC$,
    we have 
    \[\cP_\cH\langle X\rangle\cong\tilde{\cP}_{\cH}^1(X)=\tilde{\cP}_{\cH}^2(X)=\dots=\tilde{\cP}_{\cH}^n(X)=\dots=\tilde{\cP}_{\cH}^{\infty}(X).\]
    if and only if $\tilde{\cP}_{\cH}(X)\to\cP_{\cH}\langle X\rangle$ is injective. In particular, all relations among the $\cH$-periods of $X$ are induced by bilinearity and functoriality if and only if the period conjecture holds at depth $1$ for $X$.
\end{cor}
\begin{proof}
By \cite[Theorem 1.3]{hormann2021note}, we have that $\tilde{\cP}^{\infty}_\cH(X)=\tilde{\cP}_\cH(X)$. We have the successive quotients
\[
\tilde{\cP}^{1}_\cH(X)\to\tilde{\cP}^{2}_\cH(X)\to\dots \to \tilde{\cP}^{\infty}_\cH(X)=\tilde{\cP}_\cH(X).
\]
If the map $\tilde{\cP}^{1}_\cH(X)\to \cP_\cH\langle X\rangle$ is injective, then clearly  $\tilde{\cP}_\cH(X)=\tilde{\cP}^{\infty}_\cH(X)\to \cP_\cH\langle X\rangle$ must also be injective, due to the following diagram:
\begin{equation*}
\begin{tikzcd}
\tilde{\cP}^1_\cH(X) \arrow[d] \arrow[r, two heads]    & \tilde{\cP}^{\infty}_\cH(X) \arrow[d] \\
\cP_\cH\langle X\rangle \arrow[r, Rightarrow, no head] & \cP_\cH\langle X\rangle              
\end{tikzcd}
\end{equation*}
It also follows that $\tilde{\cP}^i_\cH(X)\cong\cP_\cH\langle X\rangle$ for all $i\geq 1$. Thus, 
\[\cP_\cH\langle X\rangle\cong\tilde{\cP}_{\cH}^1(X)=\tilde{\cP}_{\cH}^2(X)=\dots=\tilde{\cP}_{\cH}^n(X)=\dots=\tilde{\cP}_{\cH}^{\infty}(X)=\tilde{\cP}_{\cH}(X).\]

Conversely, if the map $\tilde{\cP}_\cH(X)=\tilde{\cP}^{\infty}_\cH(X)\to \cP_\cH\langle X\rangle$ is injective, then by \cite[Lemma 3.2 and Corollary 4.2]{hormann2021note}, we can show that the map $\tilde{\cP}^{1}_\cH(X)\to \cP_\cH\langle X\rangle$ is also injective.

\end{proof}

\begin{defn}
Let $(\cC, T, \cH)$ be as above, with $\cH=(F,G)$. For each object $X$ of $\cC$, we define $$\cH_X:=\Nat(F\mid_{\langle X\rangle},G\mid_{\langle X\rangle})),$$ where $\Nat$ is the space of natural transformations. 
\end{defn}
\label{dual Hx=fromal p-adic periods}
We have the following identifications for the dual space $\cH^{\vee}_X$
\begin{multline*}
\cH^{\vee}_X=\Nat(F\mid_{\langle X\rangle},G\mid_{\langle X\rangle})\ve=\\ \Big\{(f_Y)\in\prod_{Y\in\langle X\rangle}\Hom_{\Q}(F(Y),G\ve(Y))\mid\forall g\colon Y\to Y', f_{Y'}\circ T(g)=T(g)\circ f_Y\Big\}^{\vee}\\=\left(\bigoplus_{Y\in\langle X\rangle}F(Y)\otimes_{\Q}G(Y)\right)/\text{functoriality},
\end{multline*}
where we have used the identification $\Hom_{\Q}(F(Y),G\ve(Y))^{\vee}\cong F(Y)\otimes_{\Q}G(Y)$ and where the naturality can be interpreted as functoriality relations generated by elements of the form $\nu\otimes g^*\gamma-g_*\nu\otimes\gamma$ for all $g:Y\to Y'$ in $\langle X\rangle$ and $\nu\in F(Y), \gamma\in G(Y')$. The right hand side is indeed the space of formal $\cH$-periods $\tilde{\cP}_\cH(X)$ and so, $\cH_X^{\vee}=\tilde{\cP}_\cH (X)$.
\begin{cor}
Assume that $F(Y)$ and $G(Y)$ are finite-dimensional for every $Y\in\langle X\rangle$. Then the $\cH$-period conjecture holds at depth 1 if and only if $\dim_U(\cH\ve_X)=\dim_U \cP_{\cH}\langle X\rangle$.
\end{cor}
\begin{proof}
 Assume that $F(Y)$ and $G(Y)$ are finite-dimensional for every $Y\in\langle X\rangle$. By \cref{dimension prop}, we know that $\dim_U(\cH_X^{\vee})=\dim_U(\tilde{\cP}_\cH (X))$ is finite. Therefore, the evaluation map $\tilde{\cP}_\cH (X)\to\cP_\cH\langle X\rangle$ is injective if and only if $\dim_U(\cH_X^{\vee})=\dim_U\cP_\cH\langle X\rangle$. By \cref{cor: P1=P for period conjecture}, this is equivalent to the $\cH$-period conjecture at depth $1$.
\end{proof}

\begin{example}
    Let $\cC=\Mi(\barQ)$ be the isogeny category of 1-motives over $\barQ$. Let $T:\Mi(\barQ)\to\Mod^{\C}_{\Q,\barQ}$ be the Betti-de Rham realization functor define by 
    \[
    M\mapsto (\Tsing(M)\otimes_{\Z}\Q,\TdR(M),\omega)
    \]
    where $\omega:\Tsing(M)\otimes_{\Z}\C\cong\TdR(M)\otimes_{\barQ}\C$ is the Betti-de Rham comparison isomorphism. We define the period pairing $\cH=(F,G):\Mi(\barQ)\to\Vect(\Q,\barQ)$ as follows:
    \[
    F(M):=\Tsing(M)\otimes_{\Z}\Q, \text{ and } G(M):=\TdR\ve(M)_{\barQ}.
    \]
    This defines the classical periods for 1-motives. In \cite[Theorem 9.7]{huber2022transcendence}, Hubber and Wüstholz showed that the $\tilde{\cP}^1_\cH(M)\to \cP_\cH\langle M\rangle$ is injective. It follows that $\tilde{\cP}_\cH(M)\cong\cP_\cH\langle M\rangle$, i.e, the Kontsevich-Zagier period conjecture for 1-motives over $\barQ$ holds (\cite[Theorem 9.10]{huber2022transcendence}).
\end{example}

Recall that $\gMi(K)$ denotes the isogeny category of 1-motives with good reduction over $K$ and consider the functor $T:\gMi(K)\to\Mod^{\Bt}_{\Q_p,\barQ}$ which is given by 
\[
M\mapsto (\Vp(M),\TdR(M),\varpi\otimes\Bt).
\]
Here, $\varpi\otimes\Bt:\Tp(M)\otimes_{\Z_p}\Bt\to\TdR(M)\otimes_K\Bt$ is the isomorphism induced by the integration map $\varpi:\Tp(M)\to\TdR(M)\otimes_K\Bt$ (see \cref{theorem p-adic integration pairing for M}). In the next section, we will define three distinct period pairings for $T$, namely $\displaystyle \int^{\hp}, \int^{\Hp}$, and $\displaystyle\int^{\Hpp}$, and demonstrate that the $\displaystyle\int^{\hp}$-period conjecture holds at depth 1, while both $\displaystyle\int^{\Hp}$-period conjecture and $\displaystyle\int^{\Hpp}$-period conjecture hold at depth 2.

\section{$\Q$-structures $\Hp$ and $\Hpp$}
In this section, we will define two $\Q$-structures $\Hp$ and $\Hpp$, within $\Tp(M)\otimes\C_p$ and $\Tp(M)\otimes\Bt$ by pulling-back the space $\hp(M)$ along the Fontaine's map $\varphi_M$ (see \ref{Fontaine's pairing for M}) and the p-adic integration map $\varpi_M$ (\ref{Integration map for M}). We then obtain the corresponding pairings relative to these structures. The p-adic numbers $\alpha\in\Bt$ that arise from these pairings are indeed Fontaine-Messing p-adic periods. Recall that $\alpha\in\Bt$ is called a Fontaine-Messing p-adic period of $M$ if it lies in the image of the pairing
\[
(\Tp(M)\otimes_{\Z_p}\Bt)\times(\TdR\ve(M)\otimes_K\Bt)\to\Bt
\]
induced by the integration map $\varpi_M$ (\ref{Integration map for M}).

We always assume that $M=[L\xrightarrow{u}G]$ is a 1-motive over a number field $\K$ which has a good reduction at $p$. We fix an embedding $\K\into K$ to a p-adic complete discrete-valued local field $K$ with residue field $k$. The base change of $M$ to $K$ is the 1-motive $M_K$ which has a good reduction and it can be extended to the 1-motive $M_{\cO_K}$ over $\cO_K$. For simplicity, we often drop the indices $K$ and $\cO_K$.

Recall \cref{def tilde exp}, \cref{definition 3.6.3}, and \cref{fundamental equality for hp in chapter 5}.
\begin{defn}\label{definition Hp}
  We define $\Hp(M)$ to be the fibre product of $\hp(M,\bar{\K})$  and $\Tp(M)\otimes_{\Z_p}\C_p$ over $\Lie(G)\otimes_{\cO_K}\C_p(1)$ along the map \[\phi_M\otimes 1_{\C_p}:\Tp(M)\otimes_{\Z_p}\C_p\to \Lie(G)\otimes_{\cO_K}\C_p(1)\]
  where $\phi_M$ is the Fontaine's map \ref{Fontaine's pairing for M}. We have 
\begin{equation}
\begin{tikzcd}
\Hp(M) \arrow[d] \arrow[r]                                      & {\hp(M,\bar{\K})} \arrow[d, "\iota"] \\
\Tp(M)\otimes_{\Z_p}\C_p \arrow[r, "\phi_M\otimes 1_{\C_p}"] & \Lie(G)\otimes_{\cO_K}\C_p(1)      
\end{tikzcd}
\end{equation}
  where $\iota$ is the $\Q$-linear map \[\hp(M,\bar{\K})\into\Lie(G)\otimes_{\cO_K}\C_p\to\Lie(G)\otimes_{\cO_K}\C_p(1),\,\,x\mapsto x\otimes 1.\]
\end{defn}
Notice that $\Hp(M)$ is just a $\Q$-vector space inside $\Tp(M)\otimes_{\Z_p}\C_p$, and by construction, we have the following commutative diagram with exact rows 
\begin{equation}\label{7.2.1}
\begin{tikzcd}
0 \arrow[r] & S \arrow[r] \arrow[d, Rightarrow, no head] & \Hp(M) \arrow[d, hook] \arrow[r]      & {\hp(M,\bar{\K})} \arrow[r] \arrow[d, hook] & 0 \\
0 \arrow[r] & S \arrow[r]                                & \Tp(M)\otimes_{\Z_p}\C_p \arrow[r] & \Lie(G)\otimes_{\cO_K}\C_p(1) \arrow[r]     & 0
\end{tikzcd}
\end{equation}
By \cref{Hodge-Tate weights of tate module of 1-motives}, the Hodge-Tate weights of $\Vp(M)$ are $0$ and $1$ with multiplicity $n=\rank(L)+\dim(A)$ and $m=\dim(T)+\dim(A)$ respectively. Thus, $S\cong\C_p^{n}$. The bottom sequence is split, and the splitting is unique because
\[
\Ext^1(\Lie(G)\otimes_{\cO_K}\C_p(1),S)=\Ext^1(\C_p(1)^{m},\C_p^{n})=H^1(K,\C_p(-1)^{mn})=0
\]
and, 
\[
\Hom(\Lie(G)\otimes_{\cO_K}\C_p(1),S)=\Hom(\C_p(1)^{m},\C_p^{n})=H^0(K,\C_p(-1))^{mn}=0.
\]
Notice that the above $\Ext^i(.,.)$ is in the category of $\Gamma_K$-representations. Therefore, by the Hodge-Tate decomposition, $S=\coLie(\sG\ve)\otimes_{\cO_K}\C_p$. By \cref{V(M)=coLie(Gv)}, we know that $\coLie(\sG\ve)=\coLie(G\ve)=V(M)$. Considering the exact sequence \ref{canonical exact sequence for any 1-motive}, we can obtain the following diagram
\begin{equation}\label{7.1.4}
\begin{tikzcd}
            & 0                                              & 0                          &                                                            &   \\
0 \arrow[r] & V(L)\otimes\C_p \arrow[r, Rightarrow, no head] \arrow[u] & \Hp(L) \arrow[r] \arrow[u] & 0                                                          &   \\
0 \arrow[r] & V(M)\otimes\C_p \arrow[u] \arrow[r]                      & \Hp(M) \arrow[u] \arrow[r] & {\hp(M,\bar{\K})} \arrow[u] \arrow[r]                      & 0 \\
0 \arrow[r] & V(G)\otimes\C_p \arrow[r] \arrow[u]                      & \Hp(G) \arrow[r] \arrow[u] & {\hp(G,\bar{\K})} \arrow[u, Rightarrow, no head] \arrow[r] & 0 \\
            & 0 \arrow[u]                                    & 0 \arrow[u]                & 0 \arrow[u]                                                &  
\end{tikzcd}
\end{equation}
Notice that the rows are split.
\begin{prop}
    The functor $\Hp:\gMi(\K)\to\Vect(\Q)$ is faithful and exact.
\end{prop}
\begin{proof}
    For exactness, we only need to show that $\hp(.)$ exact, as the Lie algebras, the Tate modules, and the pull-backs functors are exact. Assume that 
    \begin{equation}\label{exact seq in proof exactness hp}
    0\to M_1\to M\to M_2\to 0
    \end{equation}
    is an exact sequence of 1-motives with good reductions at $p$, where $M_1=[L_1\to G_1]$ and $M_2=[L_2\to G_2]$. To demonstrate the exactness of the sequence 
    \[
    0\to \hp(M_1)\to \hp(M)\to \hp(M_2)\to 0,
    \]
    we use our terminologies and results from \cref{section crystalline algebraic points}. Note that we only need to show that $\hp(M)\to \hp(M_2)$ is surjective. The sequence \ref{exact seq in proof exactness hp} induces a natural Galois-equivariant exact sequence
    \[
    0\to \Lie(G_1)\otimes\C_p(1)/\phi_{M_1}(\tilde{V_1})\to \Lie(G)\otimes\C_p(1)/\phi_{M}(\tilde{V})\to \Lie(G_2)\otimes\C_p(1)/\phi_{M_2}(\tilde{V_2})\to 0,
    \]
    where $\tilde{V_i}=\Vp(\sG_i^0)$ and $\sG^i$ is the connected component of the p-divisible group associated to $M_i$. Let $x\in \Lie(G_2)\otimes\C_p(1)/\phi_{M_2}(\tilde{V_2})$ and $y\in\Lie(G_2)\otimes\C_p(1)/\phi_{M}(\tilde{V})$ belongs to the fibre of $x$. If $x$ is crystalline, then so is $y$. Furthermore, The orbit of $y$ maps bijectively to the orbit of $x$ via $\Lie(G)\otimes\C_p(1)/\phi_{M}(\tilde{V})\to \Lie(G_2)\otimes\C_p(1)/\phi_{M_2}(\tilde{V_2})$. This means that if the orbit of $x$ is finite, then so is the orbit of $y$. But, by \cref{theorem classification for algebraic poin in image of logarithm}, any point in $\hp(M_2)$ corresponds to an algebraic crystalline point $x$ in $\Lie(G_2)\otimes\C_p(1)/\phi_{M_2}(\tilde{V_2})$. Consequently, the point $y\in\Lie(G)\otimes\C_p(1)/\phi_{M}(\tilde{V})$ in the fibre of $x$ is also an algebraic crystalline point. This concludes the result.
    
    In order to verify faithfulness it suffices to show that $\Hp(M)=0$ implies that $M=0$ in $\gMi(\K)$. Consider some $M$ in $\gMi(\K)$ with $\Hp(M)=0$. The diagram \ref{7.1.4} implies that $\Hp(L)=\Hp(G)=0$, so $\hp(G,\bar{\K})=0$ and $V(L)=0$. $\hp(G,\bar{\K})=0$ implies that $G=0$, and $V(L)=L\otimes\G_a=0$ implies that $L=0$.
\end{proof}
\begin{defn}\label{definition Hpp}
We define $\Hpp(M)$ as the pullback of $\Hp(M)\into\Tp(M)\otimes_{\Z_p}\C_p$ via the map $\Tp(M)\otimes_{\Z_p}\Bt\to\Tp(M)\otimes_{\Z_p}\C_p$ induced by the quotient map $\Bt\to\C_p$.    
\end{defn}
We have the commutative diagram
\begin{equation}\label{7.2.6}
\begin{tikzcd}
{0} \arrow[r] & \Tp(M)\otimes_{\Z_p}\C_p(1) \arrow[r] \arrow[d, Rightarrow, no head] & \Hpp(M) \arrow[r] \arrow[d, hook] & \Hp(M) \arrow[r] \arrow[d, hook]   & {0} \\
0 \arrow[r]  & \Tp(M)\otimes_{\Z_p}\C_p(1) \arrow[r]                                & \Tp(M)\otimes_{\Z_p}\Bt \arrow[r] & \Tp(M)\otimes_{\Z_p}\C_p \arrow[r] & 0 
\end{tikzcd}
\end{equation}
with exact rows, where the bottom row is obtained by tensoring $\Tp(M)$ with the canonical exact sequence 
\[
0\to \C_p(1)\to\Bt\to\C_p\to 0.
\]
\begin{cor}
The functor $M\mapsto\Hpp(M)$ is faithful and exact.
\end{cor}
\begin{proof}
This is straightforward since both $\Tp$ and $\Hp$ are faithful and exact.   
\end{proof}

\begin{remark}
Let $\langle ,\,  \rangle$ denotes the pairing $\Lie(G\nat)_{\Bt}\times\coLie(G\nat)_{\Bt}\to\Bt$. The splitting of the first row in \ref{7.2.1} implies that $\Hp(M)$ embeds naturally into $\Lie(G\nat)_{\C_p}$. If $x\in\Hp(M)$ and $\omega\in\coLie(G)$, then diagram \ref{7.2.1} implies that 
\[
\langle x,\omega\rangle=\int^{\phi}_x\omega=\int^{\varpi}_x\omega.
\]
The latter equality follows from the fact that the p-adic integration pairing respects filtration (\cref{theorem p-adic integration pairing for M}). Moreover, we have the diagram 
\begin{equation*}
\begin{tikzcd}
0 \arrow[r] & \Tp(M)\otimes_{\Z_p}\C_p(1) \arrow[r] \arrow[d, Rightarrow, no head] & \Hpp(M) \arrow[r] \arrow[d, hook]                                & \Hp(M) \arrow[r] \arrow[d, hook]                                              & 0 \\
0 \arrow[r] & \Tp(M)\otimes_{\Z_p}\C_p(1) \arrow[r]                                & \Tp(M)\otimes_{\Z_p}\Bt \arrow[r] \arrow[dd, "\varpi_M\otimes\Bt"] & \Tp(M)\otimes_{\Z_p}\C_p \arrow[r] \arrow[d, "\varphi\otimes\C_p", two heads] & 0 \\
            &                                                                      &                                                                  & \Lie(G)\otimes_{K}\C_p(1) \arrow[d, hook]                                     &   \\
0 \arrow[r] & V(M)\otimes_K\Bt \arrow[r]                                           & \Lie(G\nat)\otimes_K\Bt \arrow[r]                                & \Lie(G)\otimes_{K}\Bt \arrow[r]                                               & 0
,\end{tikzcd}
\end{equation*}
where the top squares are commutative, and the bottom square is also commutative due to the diagram \ref{commutative diagram respect filtration}. Assume that $x\in\Hpp(M)$. We can write $x$ as $x=u+v$, where the image of $u$ belongs to $\Hp(M)$ and $v\in\Tp(M)\otimes_{\Z_p}\C_p(1)$, and hence $v$ is in the kernel of $\Tp(M)\otimes_{\Z_p}\Bt\to\Tp(M)\otimes_{\Z_p}\C_p$. Notice that this representation is not necessarily unique. As the above diagram commutes, $v$ belongs to $V(M)\otimes_K\Bt$, and for $\omega\in\coLie(G)$, we have 
\[
\displaystyle\int^{\varpi}_x\omega=\int^{\varpi}_u\omega + \int^{\varpi}_v\omega=\int^{\varpi}_u\omega+0=\int^{\varphi}_u\omega.
\]

We summarize theses results in the following corollary.
\end{remark}
\begin{cor}\label{comparing parings Hp and Hpp}
Let $\omega\in\coLie(G)_{\bar{\K}}$, and let $\pi$ denote the surjective map $\Hpp(M)\to\Hp(M)$. We have
\begin{equation}
    \int^{\varpi}_{x}\omega=\int^{\varphi}_{\pi(x)}\omega.
\end{equation}    
\end{cor}
The $\Q$-structure $\Hpp(M)$ is very large. It contains $\Tp(M)\otimes_{\Z_p}\C_p(1)$ which we do not know anything about their image under the integration pairing $\displaystyle\int^{\varpi}$ in general. This is why we focus on some specific elements in $\Hpp(M)$. Let $\Bct:=\Bcrisp/t^2\Bcrisp\subseteq\Bt$. As $\phi_{cris}$ is injective on $\Bcrisp$, and acts on $t$ by multiplication by $p$, it follows that $\phi_{cris}$ induces an injective endomorphism on $\Bct$. Recall the \cref{def: tilde tate module} for $\tilde{\Tp}(M)$.
\begin{defn}
 We define $\tilde{\Hpp}(M)$ to be the pull-back of $\Hpp(M)\into\Tp(M)\otimes_{\Z_p}\Bt$ via $\tilde{\Tp}(M)\otimes_{\Z_p}\phi_{cris}(\Bct)\into\Tp(M)\otimes_{\Z_p}\Bt$; i.e.
 \begin{equation}
\begin{tikzcd}
\tilde{\Hpp}(M) \arrow[d, hook] \arrow[r, hook]               & \Hpp(M) \arrow[d, hook] \\
\tilde{\Tp}(M)\otimes_{\Z_p}\bigcap_{m=1}^{\infty}\phi^m_{cris}(\Bct) \arrow[r, hook] & \Tp(M)\otimes_{\Z_p}\Bt
.\end{tikzcd}
 \end{equation}
\end{defn}
We need one more step to introduce our period pairings relative to the $\Q$-structures $\Hp$, and $\Hpp$. The description of $\hp(M)$ in \cref{fundamental equality for hp in chapter 5} indicates that these elements are obtained from $\sG^0$ the connected component of $\sG$, whose associated isocrystal has non-zero slopes. Consequently, it makes sense to pair them with forms that also have non-zero slopes in the isocrystal associated to $M$. Actually, this pairing seems to be necessary, as demonstrated in the proof of our main theorem (\cref{level 2 period cinjecture for Hpp}); otherwise, we may encounter some vanishing periods, making it difficult to trace the underlying relations.

\begin{defn}\label{def: non-zero slope isocrystal}
    Let $N$ be an isocrystal over $K_0$. By \cref{isoclinic decomposition lemma}, there is a unique decomposition 
\[
N=\bigoplus N(\alpha_i)
\]
into direct sum of nonzero sub-filtered isocrystals with slopes $\alpha_i$ for $i=1,\dots, n$. We define $\tilde{N}$ to be sub-object of $N$ containing all nonzero slopes, i.e., 
\[
\tilde{N}:=\bigoplus_{\alpha_i\neq 0}N(\alpha_i)\subset N.
\]
\end{defn}

\begin{remark}\label{remark 6.2.2}
Let $\K'$ be a finite extension of $\K$ and $k'$ its residue field at the prime above $p$. Then, by the crystalline-de Rham comparison isomorphism (\cref{crystalline-de Rham comparison isomorphism}), we can obtain a canonical identification
\[
\TdR\ve(M)\otimes_{\K}K'\cong\Tcrys\ve(\bar{M})\otimes_{\W(k')}K',
\]
where $K'$ is the p-adic completion of $\K'$. Let $K'_0$ be the field of fraction of $\W(k')$, and let $N=\Tcrys\ve(\bar{M})\otimes_{\W(k')}K'_0$ and $\tilde{N}_{K'_0}$ the subobject of $N$ containing all nonzero slopes (see \cref{def: non-zero slope isocrystal}). We define $\tilde{N}(\K')$ to be the pull-back of $\TdR\ve(M)_{\K'}\into \Tcrys\ve(\bar{M})\otimes_{\W(k')}K'$ and $\tilde{N}_{K'_0}\into \Tcrys\ve(\bar{M})\otimes_{\W(k')}K'$ i.e.
\begin{equation}\label{eq 6.2.8}
\begin{tikzcd}
\tilde{N}(\K') \arrow[d, hook] \arrow[r, hook] & \tilde{N}_{K'_0} \arrow[d, hook]  \\
\TdR\ve(M)_{\K'} \arrow[r, hook]                  & \Tcrys\ve(\bar{M})\otimes_{\W(k')}K'
\end{tikzcd}
\end{equation}
where the bottom arrow is the composition 
\[
\TdR\ve(M)_{\K'}\into\TdR\ve(M)\otimes_{\K}K'\cong\Tcrys\ve(\bar{M})\otimes_{\W(k')}K'.
\]
The action of the Frobenius is compatible with unramified extensions of local fields and the above diagram is compatible with extensions of $\K$. We define 
\begin{equation}
    \tilde{N}(M)=\colim \tilde{N}(\K'),
\end{equation}
where direct limit is taken over all finite extensions of $\K$. By construction, this is the same as taking direct limit of $\colim\tilde{N}(\K')$ over all finite extensions of $\K$ where $p$ is unramified, or equivalently, over all finite extensions $\K'$ such that $\K\subseteq\K'\subseteq K^{ur}\cap\bar{\K}:=\K^{u}$. The field $\K^u$ is indeed the maximal unramified extension of $\K$ at $p$ which is the extension obtained by taking the compositum of all finite extensions of $\K$ in which the prime 
$p$ does not ramify. This extension is infinite and is a Galois extension of $\K$, where its decomposition group at $p$ is isomorphic to the profinite completion of the absolute Galois group of the residue field at $p$ (see \cite[Chapter II, \S 5]{neukirch_algebraic_1999}, or \cite[Chapter III]{serre_local_1979}).

For any $\omega\in\Lie(G\nat)_{\bar{\K}}$, there exists a finite extension $\K'$ of $\K$ such that $\omega$ belongs to $\Lie(G\nat)_{\K'}$. We have \[
\tilde{N}(M)\into\TdR\ve(M_{\K})\otimes \bar{K}\into\TdR\ve(M_K)\otimes_K\Bt=\coLie(G\nat)\otimes_K\Bt.
\]
In fact, $\tilde{N}(M)$ is a $\K^u$-linear subspace of $\coLie(G\nat)_{\bar{\K}}$.
\end{remark}
\begin{defn}\label{three pairings hp Hp Hpp}
    The restriction of the paring $(\Tp(M)\otimes_{\Z_p}\C_p)\times(\coLie(G))\to\C_p(1)$ induced by Fontaine's map $\phi_M$, on $\Hp(M)$, is denoted by 
    \[
    \displaystyle \int^{\Hp}:\Hp(M)\times\coLie(G)_{\bar{\K}}\to\C_p(1)
    \]
    and the restriction of the paring $(\Tp(M)\otimes\Bt)\times(\TdR\ve(M)\otimes\Bt)\to \Bt$ induced by the integration map $\varpi_M$, on $\tilde{\Hpp}(M)\times\tilde{N}(M)$, is denoted by 
    \[
    \displaystyle \int^{\Hpp}:\tilde{\Hpp}(M)\times \tilde{N}(M)\to\Bt.
    \]
    The restriction of the pairing $\displaystyle\int^{\Hp}$ on $\hp(M,\bar{\K})$ is denoted by $\displaystyle\int^{\hp}$.
    
    We say that a p-adic number $\alpha\in\Bt$ is an $\Hpp$-period, $\Hp$-period, or $\hp$-period of a 1-motive $M\in\gMi(\K)$, if it is in the image of the pairing $\displaystyle\int^{\Hpp}$, $\displaystyle\int^{\Hp}$, or $\displaystyle\int^{\hp}$, respectively.
\end{defn}
\begin{remark}
Let $T:\gMi\to\Mod_{\Q,\barQ}^{B_2}$ be defined by \[M\mapsto (\Vp(M),\TdR(M)_{\barQ},\varpi\otimes\Bt).\]
The pairing $\displaystyle\int^{\Hpp}$ serves as a period pairing for $T$ in the sense of \cref{period pairing for T}. In the sense of \cref{period pairing for T}, one can put $\displaystyle\int^{\Hpp}=(F,G)$, where we define
\[F:\gMi\to\Vect(\Q),\, M\mapsto\tilde{\Hpp}(M), \text{ and } G:\gMi\to \Vect(\K^u), M\mapsto\tilde{N}(M).\] For convenience, $\displaystyle\int^{\Hpp}$-period will simply be referred to as an $\Hpp$-period. The space of these $\Hpp$-periods is denoted by $\cP_{\Hpp}\langle M\rangle$ (\cref{Space of H-periods}), which is a $\K^u$-vector space. The space of formal $\Hpp$-periods and formal $\Hpp$-periods of depth $i$ are denoted by $\tilde{\cP}_{\Hpp}(M)$ and $\tilde{\cP}^i_{\Hpp}(M)$, respectively (\cref{space of formal H-periods} and \cref{space of formal periods at depth i}).

We can also view $\displaystyle\int^{\Hp}$ as a period pairing for $T$. One can put $\displaystyle\int^{\Hp}=(F',G')$, where 
\[F':\gMi\to\Vect(\Q),\, M\mapsto \Hp(M), \text{ and } G':\gMi\to \Vect(\barQ), M\mapsto\coLie(G)_{\barQ}.\]
The space of $\Hp$-periods is denoted by $\cP_{\Hp}\langle M\rangle$, which is a $\barQ$-vector space. The space of formal $\Hp$-periods and formal $\Hp$-periods of depth $i$  are denoted by $\tilde{\cP}_{\Hp}(M)$ and $\tilde{\cP}^i_{\Hp}(M)$, respectively. Analogously, The space of $\hp$-periods, formal $\hp$-periods, and formal $\hp$-periods of depth $i$ are denoted by $\cP_{\hp}\langle M\rangle$, $\tilde{\cP}_{\hp}(M)$, and $\tilde{\cP}^i_{\hp}(M)$, respectively.
\end{remark}

\begin{prop}
    The assignment $\tilde{N}:\gMi(\bar{\K})\to\Vect(\bar{\K}),\, M\mapsto\tilde{N}(M)$ is a contravariant exact functor.
\end{prop}
\begin{proof}
    Taking the functor $\tilde{(.)}$ from an exact sequence of isocrystals over $K_0$ (\cref{def: non-zero slope isocrystal}) results in an exact sequence of isocrystals with non-zero slopes. Additionally, we know that $\TdR$  and base change are exact. Therefore, the result follows since pullback is also exact. 
\end{proof}

\section{P-adic subgroup theorem for 1-motives}
In this section, we first establish a p-adic version of the subgroup theorem for 1-motives (\cref{p-adic subgroup theorem for Fontaine pairing}). Then using this theorem, we prove the $\Hpp$-period conjecture at depth 2, the $\Hp$-period conjecture at depth 2, and the $\hp$-period conjecture at depth 1 for all 1-motives with good reductions.

\begin{defn}\label{def: left-right kernel duality pairing}
Let $\fg$ be a vector space over a field $F$. Let $\langle,\,\rangle$ be the duality pairing $\fg\times\fg\ve\to K$. We define the left kernel and right kernel as follows: for any $\fa\subseteq\fg$ and $\fb\subseteq \fg\ve$
\begin{gather*}
    \Ann(\fa):=\{f\in\fg\ve\st \langle \fa,f\rangle=0\}\\
    \Ann(\fb):=\{u\in\fg\st \langle u,\fb\rangle=0\}.
\end{gather*}
Put $\Ann(u):=\Ann(\{u\})$, for any $u\in\fg$.
\end{defn}
By the definition and maximality of $\Ann(u)$, we have $\Ann\Ann(\Ann(u))=\Ann(u)$.\\
\

 Assume that $x\in\Hp(M)$. We can view $x$ in $\Lie(G\nat)_{\C_p}$, due to diagram \ref{7.1.4}. Thus, by $\Ann(x)$, we mean $\Ann(x)$ within $\coLie(G\nat)_{\C_p}$.

We need to prove the following theorem to which we refer as the p-adic subgroup theorem for 1-motives with good reduction:
\begin{thm}[P-adic subgroup theorem for 1-motives]\label{p-adic subgroup theorem for Fontaine pairing}
Let $M$ be a 1-motive over number field $\K$ with good reduction at $p$ and let $x\in\Hp(M)$. There exists an exact sequence 
\[
0\to M_1\to M^n\to M_2\to 0
\]
of 1-motives over a finite extension of $\K$ with good reductions at $p$, with $n\in\{1,2\}$, such that $x\in\Hp(M_1)$, and $\Ann(x)\subseteq\TdR\ve(M_2)$.
\end{thm}
To prove this theorem, our main tool is the p-adic analytic subgroup theorem which is stated in \cite{bertrand1985lemmes} and proved in \cite{fuchs2015p}, and which is the p-adic version of the celebrated Wustholz's (classical) analytic subgroup theorem \cite{wustholz1989algebraische}. 

\begin{thm}[P-adic analytic subgroup theorem]\label{original p-adic analytic subgroup theorem}
Let $G$ be a commutative algebraic group defined over $\barQ$ and let $V\subseteq\Lie(G)$ be a non-trivial $\barQ$-linear subspace. For any $\gamma\in\G(\barQ)_f$ with $0\neq\log_{G(\C_p)}(\gamma)\in V\otimes_{\barQ}\C_p$, there exists an algebraic group $H\subseteq G$ defined over $\barQ$ such that $\Lie(H)\subseteq V$ and $\gamma\in H(\barQ)$.
\end{thm}

\begin{lemma}\label{lemma annihilator for exact sequence}
    Assume that \[0\to \fh\to\fg\xrightarrow{\pi}\fg/\fh\to 0\] is an exact sequence of Lie algebras. We have $\Ann(\fh)=\pi^*(\fg/\fh)\ve$.
\end{lemma}
\begin{proof}
    If $f\in \Ann(\fh)$, this means that the restriction $f\mid_{\fh}=0$. Thus, $f$ belongs to the kernel of $\fg\ve\to\fh\ve$ which implies that $f\in \pi^*(\fg/\fh)\ve$. Conversely $\pi(\fh)=0$ and so, \[\langle \fh,\pi^*f\rangle=\langle \pi(\fh),f\rangle=0 \] for any $f\in(\fg/\fh)\ve$. Thus, $\Ann(\fh)=\pi^*(\fg/\fh)\ve$. 
\end{proof}
We need the following reformulation of p-adic analytic subgroup theorem. 
\begin{prop}\label{general reformulation of p-adic analytic subgroup theroem for algebraic groups}
Let $G$ be a commutative connected algebraic group over a number field $\K$. Assume that $u_1,\dots,u_n\in \log_{G(\C_p)}(G(\barQ)_f)$ (or in $\log_{\sG}(\sG(\barQ))$ when $G$ is a semi-abelian variety). There exists an exact sequence 
\[0\to H_1\to G\to H_2\to 0\] of connected commutative algebraic groups over a finite extension of $\K$ such that $u_1,\dots,u_n\in \Lie(H_1)_{\C_p}$ and $\Ann(u_1,\dots,u_n)=\coLie(H_2)_{\C_p}$. The sequence is uniquely determined by these properties.
\end{prop}
\begin{proof}
Assume that $n=1$. Let $u=\log(\gamma)$ for some $\gamma\in G(\barQ)_f$. If $u=0$, then the theorem holds with $H_1=0$. Otherwise, put $V=\Ann(\Ann(u))\subset\Lie(G)_{\C_p}$. We apply \cref{original p-adic analytic subgroup theorem} to get a connected algebraic subgroup $H_1$ of $G_{\K'}$\footnote{We drop the index $L$ if it is clear from the context.} defined over a finite extension $\K'/\K$ with $u\in \Lie(H_1)_{\C_p}\subseteq V$. Taking the annihilator, we obtain $\Ann(\Ann(\Ann(u))) \subseteq \Ann(\Lie(H_1))\subseteq\Ann(u)$. Thus, $\Ann(u)=\Ann\Lie(H_1)$. Now, for $H_2=G/H_1$, we get an exact sequence of Lie algebras
\[0\to \Lie(H_1)\to\Lie(G)\to\Lie(H_2)\to 0\] which corresponds to an exact sequence \begin{equation}\label{eq: exact sequence in p-adic analytic subgroup theorem} 0\to H_1\to G\to H_2\to 0\end{equation} of algebraic groups over a finite extension of $\K$. If $G$ is a semi-abelian variety, then so are $H_1$ and $H_2=G/H_1$. By \cref{lemma annihilator for exact sequence}, we obtain $\Ann(u)=\Ann(\Lie(H))=\pi^*(\Lie(G)/\Lie(H_1))\ve$. As $\pi^*$ is injective, we can view $(\Lie(G)/\Lie(H_1))\ve$ as a subspace of $\Lie(G)$.

For $n>1$, we apply the above argument for $u_1,\dots, u_n$ and obtain subgroups $H_1,\dots, H_n$ defined over a finite extension of $\K$ such that $u_i\in\Lie(H_i)_{\C_p}$ and $\Ann(u_i)=\coLie(G/H_i)$. Let $H=H_1+\dots+H_n$. This is an algebraic subgroup of $G$ with 
\[
u_1,\dots,u_n\in\Lie(H)_{\C_p}=\Lie(H_1)_{\C_p}+\dots+\Lie(H_n)_{\C_p}
\]
and 
\begin{align*}
\coLie(G/H)_{\C_p}=\coLie(G/(H_1+\dots+H_n))_{\C_p}\cong\bigcap\coLie(G/H_i)_{\C_p}\\=\bigcap\Ann(u_i)=\Ann(u_1,\dots,u_n),
\end{align*}
where the above intersection occurs inside $\coLie(G)$. Now, the exact sequence 
\[
0\to H\to G\to G/H\to 0
\]
is the desired exact sequence.

Suppose that \[0\to H'\to G\to H''\to 0\] is another exact sequence with the same properties. We have 
\[\Ann(u_1,\dots,u_n)=\coLie(G/H)_{\C_p}=\coLie(H'')_{\C_p}=(\Lie(G)/\Lie(H'))_{\C_p}\ve.\] This implies that $H=H'$, as $H_1$ and $H'$ are connected.
\end{proof}

\textit{\large Proof of \cref{p-adic subgroup theorem for Fontaine pairing}}.
Without loss of generality, we can assume that $L\to G$ is injective. Indeed, we have the decomposition of $$M=[L'\to 0]\oplus[L/L'\pmod{\text{torsion}}\to G]$$ in the isogeny category of $\gMi(\K)$, where $L/L'\to G$ is injective. For the motive $[L'\to 0]$, we have $\TdR(L')_{\C_p}=V(L')_{\C_p}=\Hp(L')$. We apply the p-adic analytic subgroup theorem to the element $x\in V(L')$, yielding the desired exact sequence of constant finitely generated free groups modulo torsion (1-motives), thereby completing the proof. Thus, we only need to consider the case where $L\to G$ is injective.

Let $x\in\Hp(M)$. We can assume that $x=u+v$, where $u\in\hp(M,\bar{\K})$ and $v\in V(M)$. By the construction of $\hp(M,\bar{K})$ and \cref{fundamental equality for hp in chapter 5}, there exists some $\gamma\in\sG(\bar{\K})$ such that $\log_{\sG}(\gamma)=u$, up to scalar multiplication by a rational number. As $G\nat$ is a vector extension of $G$, we can view $u$ in the image of $\log_{G\nat}$. Recall that the logarithm map on the vector group $V(M)$ is the identity map. We apply \cref{general reformulation of p-adic analytic subgroup theroem for algebraic groups} to $u+v$ to get the exact sequence 
\[
0\to H_1\to G\nat\to H_2\to 0
\]
over a finite extension of $\K$ with $u+v\in H_1$ and $\Ann(u+v)=\coLie(H_2)_{\C_p}$. By the structure theory of commutative algebraic groups (\cref{Structure theory of algebraic groups}), there are canonical exact sequences
\[
0\to V_1\to H_1\to G_1\to 0\]
\[0\to V_2\to H_2\to G_2\to 0\]
where $G_i$ is semi-abelian, and $V_i$ is a vector group for $i=1,2$. Note that $v\in (V_1)_{\C_p}$. We have the following diagram with exact rows and columns
\begin{equation}
    \begin{tikzcd}
            & 0 \arrow[d]             & 0 \arrow[d]               & 0 \arrow[d]             &   \\
0 \arrow[r] & V_1 \arrow[r] \arrow[d] & V(M) \arrow[r] \arrow[d]     & V_2 \arrow[r] \arrow[d] & 0 \\
0 \arrow[r] & H_1 \arrow[r] \arrow[d] & G\nat \arrow[r] \arrow[d] & H_2 \arrow[r] \arrow[d] & 0 \\
0 \arrow[r] & G_1 \arrow[r] \arrow[d] & G \arrow[r] \arrow[d]     & G_2 \arrow[r] \arrow[d] & 0 \\
            & 0                       & 0                         & 0                       &  
\end{tikzcd}
\end{equation}
We define $L_1:=L\cap H_1$ and $L_2:=L/L_1\pmod{\text{torsion}}$. By construction, $L_1\to L\to G$ factors through $G_1$ and $L\to G\to G_2$ factors through $L_2\to G_2$. There is an exact sequence
\begin{equation}\label{7.3.3}
    0\to M_1\to M\to M_2\to 0,
\end{equation} 
where
$M_1=[L_1\to G_1]$ and $M_2=[L_2\to G_2]$. Notice that both $M_1$ and $M_2$ have good reductions at $p$ since $M$ has a good reduction at $p$ (see \cref{cor: good reduction exact sequence of 1-motives}). Therefore, the exact sequence \ref{7.3.3} can be lifted to $\cO_K$, where $K$ is a finite extension of $\Q_p$ containing $\K$. By the property of universal vector extensions, the exact sequence \[0\to V_i\to [L_i\to H_i]\to [L_i\to G_i]\to 0 \]
is the push-out of $G\nat_i$ for $i=1,2$. The compositions $G\nat_1\to H_1\to G\nat$ and $G\nat\to G\nat_2\to H_2$ are injective and surjective respectively. So, the maps $G\nat_1\to H_1$ and $G\nat_2\to H_2$ are injective and surjective respectively. It follows that $V(M_1)\to V_1$ is also injective. Since $H_1$ is a vector extension of $G_1$, we deduce that $u\in\hp(G_1,\bar{\K})$. 

Now, if $v\in V(M_1)_{\C_p}$, then $x=u+v\in\Hp(M_1)$. Because of the short exact sequence \ref{7.3.3}, we have $\TdR\ve(M_2)_{\C_p}=\coLie(G\nat_2)_{\C_p}\subseteq\Ann(x)$. However, $G\nat_2\twoheadrightarrow H_2$ is surjective, hence 
\begin{equation}\label{ann(u)=TdR(M2)}
\Ann(u+v)=\coLie(H_2)_{\C_p}\subseteq\coLie(G\nat_2)_{\C_p}\subseteq\TdR\ve(M_2)_{\C_p},      
\end{equation}
as desired.

We now assume that $v\notin V(M_1)$. From \cref{V(M)=coLie(Gv)}, we know that $V(M_1)=\coLie(G\ve_1)$, and $\coLie(G\ve_1)\into\coLie(G\ve)=V(M)$. Choose a semi-abelian scheme $N$ such that its Cartier dual $N\ve$ is a subgroup of $G$, with a surjection $N\twoheadrightarrow G_1\ve$, and such that $v\in\coLie(N)\subset\coLie(G\ve)$ (e.g. $G\ve$ is a candidate that satisfies this condition\footnote{The Zorn's lemma can be applied to identify the smallest $N$ that satisfies these conditions; however, this is not essential for the argument in our proof.}). Taking the Cartier duality, we have $G_1\into N\ve\into G$. Thus, $u\in\hp(N\ve)$ as well. Therefore, we obtain the diagram 
\begin{equation}\label{7.3.4}
\begin{tikzcd}
0 \arrow[r] & H_1 \arrow[r] \arrow[d]       & G\nat \arrow[r] \arrow[d]                  & H_2 \arrow[r] \arrow[d, two heads] & 0 \\
0 \arrow[r] & G_1 \arrow[d, hook] \arrow[r] & G \arrow[d, Rightarrow, no head] \arrow[r] & G_2 \arrow[d, two heads] \arrow[r] & 0 \\
0 \arrow[r] & N\ve \arrow[r]                & G \arrow[r]                                & G/N\ve \arrow[r]                   & 0
.\end{tikzcd}
\end{equation}
Furthermore, we define 1-motives $M'_1=[L_1\to N\ve]$ and $M'_2=[L_2\to G/N\ve]$ to get the diagram 

\begin{equation}\label{diagram 435}
\begin{tikzcd}
0 \arrow[r] & M_1 \arrow[r] \arrow[d, hook] & M \arrow[r] \arrow[d, Rightarrow, no head] & M_2 \arrow[r] \arrow[d, two heads] & 0 \\
0 \arrow[r] & M'_1 \arrow[r]                & M \arrow[r]                                & M_2' \arrow[r]                     & 0
.\end{tikzcd}
\end{equation}
Based on the condition we impose on $N$, we have $v\in\coLie(N)=V(M'_1)_{\C_p}$ and $u\in\hp(M'_1)$. Consequently, $x=u+v\in\Hp(M'_1)$, and diagram \ref{diagram 435} implies that $$\TdR\ve(M'_2)_{\C_p}\subseteq\Ann(x)\subseteq\TdR\ve(M_2)_{\C_p},$$ where the latter inclusion is from \ref{ann(u)=TdR(M2)}. Moreover, if we consider the exact sequence
\[
0\to M'\to M^2\to M''\to 0,
\]
where $M'=M_1\oplus M'_1$ and $M''=M_2\oplus M'_2$, then $x\in\Hp(M')$ and $\Ann(x)\subset\TdR\ve(M'')$. This completes the proof.

\qed
Along the way of the proof we get the following:
\begin{cor}\label{cor 5.6.1}
Let $x\in\Hp(M)$. We can write $x$ uniquely as $x=u+v$, where $u\in\hp(M,\bar{\K})$ and $v\in V(M)\otimes\C_p$. Moreover, there exists a commutative diagram 
\begin{equation}
\begin{tikzcd}
0 \arrow[r] & M_1 \arrow[r] \arrow[d, hook] & M \arrow[r] \arrow[d, Rightarrow, no head] & M_2 \arrow[r] \arrow[d, two heads] & 0 \\
0 \arrow[r] & M'_1 \arrow[r]                & M \arrow[r]                                & M_2' \arrow[r]                     & 0
\end{tikzcd}
\end{equation}
with the following properties:
\begin{enumerate}
    \item $u\in\Hp(M_1)$.
    \item $x\in\Hp(M'_1)$.
    \item $\Ann(x)=\TdR\ve(M_2)_{\C_p}\supseteq\TdR\ve(M'_2)$.
    \item If $\hp(M)=\Hp(M)$ (or $x\in\hp(M)$), then the exact sequence provided in \cref{p-adic subgroup theorem for Fontaine pairing} occurs with $n=1$.
\end{enumerate}
\end{cor}

We are now ready to prove the p-adic period conjecture relative to the pairing $\displaystyle\int^{\Hpp}$ at depth 2.

\begin{lemma}\label{lemma 5.6.2}
    Let $N$ be an admissible filtered isocrystal over $K_0$ with non-zero slopes and equipped with the filtration
    \begin{equation*}
        \Fil^i(N)=\begin{cases}
            N,\,\, i\leq 0\\
            X,\,\, i=1\\
            0,\,\, i\geq 2.
        \end{cases}
    \end{equation*}
Then $\sum_{n\geq 0}F^n(X)=N$.
\end{lemma}
\begin{proof}
    We first claim that $Y:=\sum_{n\geq 0}F^n(X)$ is a weekly admissible filtered sub-isocrystal. Recall the definition of Newton numbers (\cref{Newton number}) and the definition of weakly admissible filtered isocrystals (\cref{definition weakly addmissible isocrystal}). The submodule $Y$ is clearly a filtered sub-isocrystal of $N$, therefore $t_N(Y)\geq t_H(Y)$, by admissibility of $N$. Assume that $\dim_{K_0}(X)=r$ and $\dim_{K_0}(N)=s$. According to the filtration of $N$, we can observe that the Hodge polygon of $D$ consists of slope a horizontal segment of length $s-r$, and a segment of slope $1$ with multiplicity $r$. As the filtration of $Y$ is given by
    \begin{equation*}
          \Fil^i(Y)=\begin{cases}
            Y,\,\, i\leq 0\\
            X,\,\, i=1\\
            0,\,\, i\geq 2.
        \end{cases}
    \end{equation*}
    The Hodge polygon of $Y$ has a horizontal segment of length $\dim_{K_0}(Y)-r$, and a segment of slope 1 with multiplicity $r$. As all slopes of $N$ and $Y$ are non-zero, we have that $t_N(Y)\leq t_N(N)=t_H(N)=r$. Therefore, $Y$ is indeed a weakly admissible sub-isocrystal of $N$. Then quotient $N/Y$ is also weakly admissible in the category of weakly admissible isocrystals over $K_0$. However, $N/Y$ has again positive slopes. Assume that $N/Y\neq 0$. Its filtration is given by
    \begin{equation*}
        \Fil^i(N/Y)=\begin{cases}
            N/Y,\,\, i\leq 0\\
            0,\,\, i\geq 1.
        \end{cases}
    \end{equation*}
This means that the Hodge polygon of $N/Y$ has only one horizontal segment, and since all of its slopes are positive, it cannot be weakly admissible. Therefore, $N/Y=0$, and so $N=Y$.
    
\end{proof}

\begin{thm}\label{level 2 period cinjecture for Hpp}
The $\Hpp$-period conjecture holds at depth $2$.
\end{thm}
\begin{proof}
We need to prove that the evaluation map \[\tilde{\cP}^2_{\Hpp}(\gMi(\K))\to\cP_{\Hpp}(\gMi(\K))\]
is injective. Let $\displaystyle\alpha_i=\int^{\varpi}_{x_i}\omega_i$ be an $\Hpp$-period of the 1-motive $M_i$, where $x_i\in\tilde{\Hpp}(M_i)$ and $\omega_i\in\tilde{N}(M_i)$ for $i=1,\dots,n$. Assume that 
\[
c_1\alpha_1+\dots+c_n\alpha_n=0
\]
for some $c_i\in\K^u$. As it is shown in \cref{dimension prop}(1), a linear combination of $\Hpp$-periods is again a $\Hpp$-period. More precisely, we can write 
\begin{equation}\label{4.3.7}
\displaystyle c_1\alpha_1+\dots+c_n\alpha_n=\sum_{i,j}c_j\int_{x_i}\omega_j=\sum_ic_i\int^{\Hpp}_{x_i}\omega_i=\int^{\Hpp}_{x}\omega
,\end{equation}
where $x=x_1+\dots+x_n\in\tilde{\Hpp}(\bigoplus M_i)$ and $\omega=c_1\omega_1+\dots+c_n\omega_n\in\tilde{N}(\bigoplus M_i)$. The 1-motive $M=\bigoplus M_i$ is defined over a finite extension $\K'$ of $\K$ and has still a good reduction at $p$ (\cref{section reduction types and deformation theory}). The preceding equality \ref{4.3.7} follows from the bilinearity relation. Now, it suffices to show that for the period $\displaystyle\int^{\Hpp}_x\omega=0$ of $M$, the element $(x\otimes\omega)_{\tilde{\cP}_{\Hpp}^2(M)}=0$. Recall the notations in \cref{remark 6.2.2}. Since $\omega\in \tilde{N}(M)$, there exists an isocrystal $N_{K'_0}$ defined over a finite unramified extension $K'_0$ such that $\omega\in \tilde{N}_{K'_0}\cap\TdR\ve(M)_{\K'}$ and $\tilde{N}_{K'_0}\subseteq\Tcrys\ve(\bar{M})\otimes_{\W(k')}K'$, as shown in the diagram \ref{eq 6.2.8}. But $\Tcrys\ve(\bar{M})\otimes_{\W(k')}K'\cong \TdR\ve(M)_{K'}$ is a filtered isomorphism as discussed in \cref{section crystalline realization for 1-motive}, which induces the following filtration on $\tilde{N}_{K'_0}$ that is identical to the filtration induced as a subobject of $N_{K'_0}$.
    \begin{equation*}
        \Fil^i(\tilde{N}_{K'_0})=\begin{cases}
            \TdR\ve(M)_{K'}\cap \tilde{N}_{K'_0},& i\leq 0\\
            \coLie(G)_{K'}\cap\tilde{N}_{K'_0},& i=1\\
            0& i\geq 2.
        \end{cases}
    \end{equation*}
By \cref{lemma 5.6.2}, we have 
\[
\sum_{n\geq 0}F^n(\coLie(G)\cap\tilde{N}_{K'_0})=\TdR\ve(M)_{K'}\cap \tilde{N}_{K'_0}=\tilde{N}_{K'_0}
\]
Therefore, there exist $\gamma\in\coLie(G)\cap\tilde{N}_{K'_0}$ and $n\geq 0$ such that $\omega=F^n(\gamma)$. As $x\in\tilde{\Hpp}(M)$, we write it as $x=b_1\otimes x_1+\dots+b_m\otimes x_m$, where $x_i\in\tilde{\Tp}(M)$ and $b_i\in\bigcap_{m=1}\phi^m_{cris}(\Bct)$. There exists elements $b_i'\in\Bct$ such that $\phi^n_{cris}(b_i')=b_i$. Since $\phi_{cris}$ is injective, one can show that $b'_i\in\bigcap_{m=1}\phi_{cris}^m(\Bct)$. \cref{cor 4.3.2} implies that $\displaystyle 0=\int^{\varpi}_x\omega=\int^{cris}_x\omega$. Meanwhile, we have 
\[
\displaystyle \int^{cris}_{x}\omega=\int^{cris}_{b_1\otimes x_1+\dots+b_m\otimes x_m}F^n(\gamma)=\sum_i\int^{cris}_{x_i}b_i\otimes F^n(\gamma)=\sum_i\int^{cris}_{x_i}(\phi^n_{cris}\otimes F^n)(b'_i\otimes\gamma).
\]
On the other hand, the action of Frobenius $F$ on $\tilde{N}_{K'_0}$ is $\sigma$-semilinear and $\phi_{cris}$ is a lift of Frobenius $\sigma$ over $K_0$, then the endomorphism \[F\otimes\phi_{cris}:N\otimes_{K_0}\Bct\to N\otimes_{K_0}\phi_{cris}(\Bct)\] is bijective and $\sigma$-semilinear. Therefore, by \cref{prop 4.3.2} we obtain $$ \int^{cris}_{x_i}(\phi^n_{cris}\otimes F^n)(b'_i\otimes\gamma)=\phi^n_{cris}(b'_i)\phi^n_{cris}\left(\int^{cris}_{x_i}\gamma\right)=\phi^n_{cris}\left(b'_i\int^{cris}_{x_i}\gamma\right).$$
Therefore, $$\displaystyle 0=\int^{cris}_x\omega=\sum_i\phi^n_{cris}\left(b'_i\int^{cris}_{x_i}\gamma\right)=\phi^n_{cris}\left(\sum_i\int^{cris}_{b'_i\otimes x_i}\gamma\right)=\phi^n_{cris}\left(\int^{cris}_{\sum b'_ix_i}\gamma\right).$$
Since $\phi_{cris}$ is injective on $\Bcris$, we obtain \[0=\int^{cris}_{\sum b'_ix_i}\gamma=\int^{cris}_{\sum \phi^n(b'_i)x_i}\gamma=\int^{cris}_x\gamma=\int^{\varpi}_x\gamma,\]
where the second equality follows from the Frobenius-equivariance of the crystalline integration on the first argument (\cref{tate module is crystalline}), and the last equality follows from \cref{cor 4.3.2}. 
As $\gamma\in\coLie(G)$, we have $\displaystyle\int^{\varpi}_x\gamma=\int^{\varphi}_x\gamma=0$ by \cref{theorem p-adic integration pairing for M}. Moreover, \cref{comparing parings Hp and Hpp} implies that 
\[
\displaystyle 0=\int^{\varphi}_x\gamma=\int^{\varphi}_{\pi(x)}\gamma,
\]
where $\pi(x)\in\Hp(M)$. Now, we can apply the p-adic subgroup theorem for motives (\cref{p-adic subgroup theorem for Fontaine pairing}) to get an exact sequence 
\begin{equation}\label{exact seq in proof of period conjecture}
0\to M_1\xrightarrow{\iota} M^n\xrightarrow{\pi} M_2\to 0
\end{equation}
of 1-motives over a finite extension of $\K'$ such that $M_1$ and $M_2$ have good reductions at $p$, $\pi(x)\in\Hp(M_1)$, $n\in\{1,2\}$, and $\Ann(\pi(x))\subset\TdR\ve(M_2)_{\C_p}$. Now, by definition, we only need to show that $\omega\in\tilde{N}(M_2)$ and $x\in\tilde{\Hpp}(M_1)$.

Since $\Ann(\pi(x))\subset\TdR\ve(M_2)_{\C_p}$, we have $\gamma\in\TdR\ve(M_2)_{\C_p}$. As $\gamma$ is in $\coLie(G)_{K'}\cap\tilde{N}_{K'_0}$, it has non-zero slopes and the same holds for every power $F^n(\gamma)$. Thus, $\omega=F^n(\gamma)\in\tilde{N}(M_2)$. 

We can obtain a commutative exact diagram 
\begin{equation*}
\begin{tikzcd}
            & 0                                                 & 0                             & 0                            &   \\
0 \arrow[r] & \Tp(M_2)\otimes_{\Z_p}\C_p(1) \arrow[u] \arrow[r] & \Hpp(M_2) \arrow[u] \arrow[r] & \Hp(M_2) \arrow[u] \arrow[r] & 0 \\
0 \arrow[r] & \Tp(M^n)\otimes_{\Z_p}\C_p(1) \arrow[u] \arrow[r] & \Hpp(M^n) \arrow[u] \arrow[r] & \Hp(M^n) \arrow[u] \arrow[r] & 0 \\
0 \arrow[r] & \Tp(M_1)\otimes_{\Z_p}\C_p(1) \arrow[u] \arrow[r] & \Hpp(M_1) \arrow[u] \arrow[r] & \Hp(M_1) \arrow[u] \arrow[r] & 0 \\
            & 0 \arrow[u]                                       & 0 \arrow[u]                   & 0 \arrow[u]                  &  
\end{tikzcd}
\end{equation*}
such that $x\in\Hpp(M^n)$ and $\pi(x)$ is in both $\Hp(M_1)$ and $\Hp(M^n)$. The rows are exact by the diagram \ref{7.2.6}, and the columns are exact because $\Hp$ and $\Hpp$ are exact functors. By diagram chasing, it follows that $x$ belongs to the kernel of the map $\Hpp(M^n)\to\Hpp(M_2)$, which implies that $x\in\Hpp(M_1)$. However, $x$ was taken in $\tilde{\Hpp}(M)\subset\Hpp(M)$, and so $x\in\tilde{\Hpp}(M_1)$. Thus, we have shown that $\left(x\otimes\omega\right)_{\tilde{\cP}^2_{\Hpp}}=0$ is zero in $\tilde{\cP}^2_{\Hpp}(M)$ and, as a result, the evaluation map
\[\tilde{\cP}^2_{\Hpp}(\gMi(\K))\to\cP_{\Hpp}(\gMi(\K))\]
is bijective, by \cref{colim injective period conjecture}(2).
\end{proof}

\begin{thm}\label{depth 2 and 1 period cinjecture for Hp and hp}
The $\Hp$-period conjecture holds at depth 2 and the $\hp$-period conjecture holds at depth 1.
\end{thm}
\begin{proof}
 Similar to the proof of \cref{level 2 period cinjecture for Hpp}, we can reduce a linear combination of $\Hp$-periods to a single relation $\displaystyle\int^{\Hp}_x\omega=0$. We now assume that $\displaystyle\int^{\Hp}_x\omega=0$, where $x\in\Hp(M)$, and $\omega\in\coLie(G)_{\bar{\K}}$. By using a similar approach to the one in the above proof and applying the p-adic subgroup theorem for motives (\cref{p-adic subgroup theorem for Fontaine pairing}) to $x\in\Hp(M)$, we obtain an exact sequence 
    \begin{equation}\label{6.3.2 exact seq}
      0\to M_1\to M^n\to M_2\to 0
    \end{equation}
of 1-motives over a finite extension of $\K$ with good reductions at $p$ such that $x\in\Hp(M_1)$ and $\Ann(x)\subseteq\TdR\ve(M_2)_{\C_p}=\coLie(G\nat_2)_{\C_p}$. This means that $\omega\in\coLie(G\nat_2)$. The exact sequence \ref{6.3.2 exact seq} yields the commutative diagram
    \begin{equation*}
\begin{tikzcd}
0 \arrow[r] & \coLie(G_2) \arrow[r] \arrow[d, hook] & \coLie(G)^n \arrow[r] \arrow[d, hook] & \coLie(G_1) \arrow[r] \arrow[d, hook] & 0 \\
0 \arrow[r] & \coLie(G\nat_2) \arrow[r]                 & \coLie(G\nat)^n \arrow[r]                 & \coLie(G\nat_1) \arrow[r]                 & 0,
\end{tikzcd}
    \end{equation*}
  and, as $\omega\in\coLie(G\nat_2)\cap\coLie(G)^n_{\bar{\K}}$, we can conclude that $\omega\in\coLie(G_2)_{\bar{\K}}$. Hence, in $\tilde{\cP}^2_{\Hpp}(M)$ we have
  \[
  \left(x\otimes\omega\right)_{\tilde{\cP}^2_{\Hpp}(M)}=0
  .\] 
This completes the proof of the $\Hp$-period conjecture at depth $2$.

 As for the proof of the $\hp$-period conjecture, we proceed exactly in the same way as above. In this case, since $x\in\hp(M)$, and \cref{cor 5.6.1} implies that in the exact sequence \ref{6.3.2 exact seq}, we must have $n=1$. Therefore, the evaluation map 
  \[
  \tilde{\cP}^{1}_{\hp}(\gMi)\to \cP_{\hp}(\gMi)
  \]
  is bijective.
\end{proof}
\begin{cor}
All relations between the $\hp$-periods of $\gMi$ are induced by bilinearilty and functoriality. More precisely, the evaluation map $\tilde{\cP}_{\hp}(\gMi)\to\cP_{\hp}(\gMi)$ is bijective.
\end{cor}
\begin{proof}
This follows directly from \cref{depth 2 and 1 period cinjecture for Hp and hp}, together with \cref{cor: P1=P for period conjecture}.
\end{proof}

\begin{remark}\label{remark 6.3.2}
Let $\displaystyle\alpha=\int^{\Hpp}_x\omega=0$ be a vanishing $\Hpp$-period of $M$, where $x\in\tilde{\Hpp}(M)$ and $\omega\in\tilde{N}(M)$. Let $\pi:\tilde{\Hpp}(M)\to\Hp(M)$. The above observations, together with \cref{cor 5.6.1}, show that if $\pi(x)$ belongs to $\hp(M)$, then $(x\otimes\omega)_{\tilde{\cP}^1_{\Hpp}(M)}$ is zero in $\tilde{\cP}^1_{\Hpp}(M)$, i.e., there exists an exact sequence 
\[
0\to M_1\to M\to M_2\to 0
\]
such that $x\in\tilde{\Hpp}(M_1)$ and $\omega\in\tilde{N}(M_2)$. In particular, if $\Hp(M)=\hp(M)$, then the evaluation map
\[
\tilde{\cP}^1_{\Hpp}(M)\to \cP_{\Hpp}\langle M\rangle.
\] is injective, and as a result, \[\tilde{\cP}^1_{\Hpp}( M)=\tilde{\cP}^2_{\Hpp}(M)=\cP_{\Hpp}\langle M\rangle.\]
The same applies for $\Hp$-periods as well.
\end{remark}
\begin{remark}\label{remark Hpp period does not hold at depth 1}
    In general, $\tilde{\cP}^2_{\Hpp}(\gMi)\neq \tilde{\cP}^1_{\Hpp}(\gMi)$ (and also $\tilde{\cP}^2_{\Hp}(\gMi)\neq \tilde{\cP}^1_{\Hp}(\gMi)$), see \cref{example M[LtoGm]} in the following section. Therefore, by \cref{cor: P1=P for period conjecture}, the map $\tilde{\cP}_{\Hpp}(\gMi)\to \cP_{\Hpp}(\gMi)$ is not injective. In other words, there are relations among the $\Hpp$-periods (or $\Hp$-periods) beyond those induced by bilinearity and functoriality. In fact, as the Theorems \ref{level 2 period cinjecture for Hpp} and \ref{depth 2 and 1 period cinjecture for Hp and hp} demonstrate, all relations among the $\Hpp$-periods ($\Hp$-periods, respectively) are exactly those that are induced by the depth 2 formal space $\tilde{\cP}^2_{\Hpp}$ ($\tilde{\cP}^2_{\Hp}$, respectively).
\end{remark}

\section{Examples}\label{section: examples}
In the following examples, $M$ is a 1-motive over a number field $\K$ with good reduction at $p$.
\begin{example}[Periods of 0-motives]
Let $M=[L\to 0]$ be a 1-motive with good reduction at $p$. Let $r=\rank(L)$. We have:
\begin{itemize}
    \item The Hodge filtration of $L$ is
\[
0\to V(L) \to V(L)\to 0 \to 0
.\]
\item $\Tcrys\ve(\bar{M})=\D(\bar{M}[p^{\infty}])=\D((\underline{\Q_p/\Z_p})^r)=(1_{FD})^r$, where $1_{FD}$ is the unit Dieudonn\'e module. Then, all the slopes are zero. 
\item $\tilde{N}(M)=0$.
\item $\hp(M)=0$.
\item $\Hp(M)=V(L)_{\C_p}$.
\item $\cP_{\Hpp}(M)=\cP_{\Hp}(M)=\cP_{\hp}(M)=0$.
\end{itemize}
\end{example}

\begin{example}[Torus]
Let $M=[0\to\G_m]$ be a 1-motive with good reduction. We have:
\begin{itemize}
    \item The Hodge filtration of $M$ is given by
    \[
    0\to 0\to \Lie(\G_m)\to\Lie(\G_m)\to 0,
    \]
    since $V(M)=\Ext^1(\G_m,\G_a)\ve=0$.
    \item $\hp(M)=\log(\mup(\barQ))\otimes\Q=\log_{\mup}((1+\fm_{\ocp})\cap\barQ)\otimes_{\Z}\Q$.
    \item As $\mup$ is connected, $\Tp(M)=\tilde{\Tp}(M)$.
    \item $\Tcrys\ve(\bar{M})=\D(\bar{M}[p^{\infty}])=\D(\mup)=\Delta_1$. Then, all slopes are non-zero.
    \item $\Hp(M)=\hp(M)$, as $V(M)=0$. By \cref{remark 6.3.2}, this implies that $\tilde{\cP}^1_{\Hpp}\langle M\rangle=\tilde{\cP}^2_{\Hpp}\langle M\rangle\cong\cP_{\Hpp}\langle M\rangle$, i.e., all relations among $\Hpp$-periods of a torus are induced by bilinearity and functoriality. The same holds for the space of formal $\Hp$-periods.
    \item As $\hp(M)=\Hp(M)$, we have also $\cP_{\hp}(M)=\cP_{\Hp}(M)$. Since all forms in $\coLie(G)$ have non-zero slopes, we have 
    \[
    \int^{\Hpp}_{x}\omega=\int^{\Hp}_{\pi(x)}\omega
    \]
    for any $\omega\in \tilde{N}(M)\subseteq\coLie(\G_m)$ and $x\in\tilde{\Hpp}(M)$. The above equality follows from \cref{comparing parings Hp and Hpp}. This implies that $\cP_{\Hpp}(M)\subseteq\cP_{\Hp}(M)$.
\end{itemize}
\end{example}
\begin{example}[Kummer motive]\label{example M[LtoGm]}
The 1-motive $M=[\Z\to\G_m], 1\mapsto g$ is called a Kummer motive. If $g$ is a root of unity, the canonical exact sequence 
\begin{equation}\label{exact seq in examples}
 0\to \G_m\to M\to \Z\to 0
\end{equation}
splits. We have
\begin{itemize}
    \item $G\nat$ is the extension of $\G_m$ by $\Hom(\Z,\G_a)=\G_a$, and by the structure theory of algebraic groups, we have $G\nat=\G_a\times\G_m$.
    \item By applying the functor $\Tcrys\ve(.)$ to the natural sequence of motives (\ref{exact seq in examples}), we get a splitting sequence \[ 0\to 1_{FD}\to\Tcrys\ve(\bar{M})\to\Delta_1\to 0.\] Therefore, we have slopes 1 and 0 with multiplicity 1, and $\tilde{N}(M)$ is contained in the filtered isocrystal associated with $\Delta_1$.
    \item If the sequence \ref{exact seq in examples} splits, taking formal p-divisible groups yields a split exact sequence. Therefore, we have \[\hp(M)=\hp(\G_m)\oplus\hp(\Z)=\hp(\G_m),\]
since $\hp(\Z)=0$. In fact, this equality holds even if $g$ is not a root of unity. This is because $\Z[p^{\infty}]$ is \'etale, and consequently, $\hp(M)=\hp(\G_m)$.
    \item $V(M)=V(\Z)=\G_a$, and $\Hp(M)=\hp(M)\oplus V(M)_{\C_p}=\hp(\G_m)\oplus V(\Z)_{\C_p}$.
    \item $\cP_{\hp}\langle M\rangle=\cP_{\hp}\langle \G_m\rangle$, $\cP_{\Hp}\langle\G_m\rangle\subset\cP_{\Hp}\langle M\rangle$.
\end{itemize}
Assume that $\displaystyle\int^{\Hpp}_{x}\omega$ is an $\Hpp$-period of $M$, where $\omega\in\tilde{N}(M)\cap\coLie(\G_m)_{\bar{\K}}$, and $x\in\tilde{\Hpp}(M)$ such that the image of $x$ in $\Hp(M)$ lies in $V(\Z)_{\C_p}$, i.e., $\pi(x)\in V(M)_{\C_p}$, where $\pi:\tilde{\Hpp}(M)\to \Hp(M)$. Suppose further that $x,\omega,\pi(x)\neq 0$. By \cref{comparing parings Hp and Hpp}, we have 
\[
\displaystyle\int^{\Hpp}_x\omega=\int^{\Hpp}_{\pi(x)}\omega=\int^{\Hp}_{\pi(x)}\omega=0.
\]
The latter equality holds because $\Ann(\pi(x))$ contains $\coLie(G)$, as $\pi(x)\in V(M)$. Now, by \cref{level 2 period cinjecture for Hpp}, $x\otimes\omega$ is zero in $\tilde{\cP}^2_{\Hpp}( M)$, i.e., there exists an exact sequence
\[
0\to M_1\to M^n\to M_2\to 0
\]
with $n\leq 2$, such that $x\in\tilde{\Hpp}(M_1)$ and $\omega\in\tilde{N}(M_2)$. We want to show that $n=2$. Assume that $n=1$. When $g$ is not a root of unity, there are only three possibilities for $M_1$:
\[
0,\,\, [0\to\G_m],\text{ and}\,\, M.
\]
The case $M_1=0$ is impossible, because $x\in\tilde{\Hpp}(M_1)$ and $x\neq 0$. Now, if $M_1=\G_m$, then $x\in\tilde{\Hpp}(M_1)=\tilde{\Hpp}(\G_m)$, and $\pi(x)\in\Hp(\G_m)=\hp(\G_m)$, but $\pi(x)\in V(\Z)_{\C_p}\subset\Hp(M)$, which is a contradiction as both $\hp(\G_m)$ and $V(\Z)_{\C_p}$ are direct summand of $\Hp(M)$. Finally, we can also exclude the case $M_1=M$, as $\omega\neq 0$. We conclude that $n=2$.  In other words, for the Kummer motive $M=[\Z\to\G_m]$, we have 
\[
\tilde{\cP}^1_{\Hpp}( M)\neq \tilde{\cP}^2_{\Hpp}( M)\cong\cP_{\Hpp}\langle M\rangle.
\]
The latter identification is due to the fact that the $\Hpp$-period conjecture holds at depth 2 for $\langle M\rangle$ (\cref{level 2 period cinjecture for Hpp}). 

Note that we took $\omega\in\coLie(G)$. The above argument still applies to $\displaystyle\int^{\Hp}_{\pi(x)}\omega=0$. We have the same conclusion for the $\Hp$-periods of $\langle M\rangle$, namely \[\tilde{\cP}^1_{\Hp}( M)\neq\tilde{\cP}^2_{\Hp}( M)\cong\cP_{\Hp}\langle M\rangle.\]
By \cref{cor: P1=P for period conjecture}, we can conclude that there are relations among $\Hpp$-periods and $\Hp$-periods of the Kummer motive $M$ beyond those induced by bilinearlity and functoriality.
\end{example}

\begin{prop}\label{prop 4.4.1}
Assume that $\alpha_1,\dots,\alpha_n$ are nonzero $\Hpp$-periods ($\Hp$-periods, resp.) of $M_1,\dots,M_n$. Let $\cC_i=\langle M_i\rangle$ be the abelian subcategories generated by $M_i$. If $\Hom(\cC_i,\cC_j)=\Hom(\cC_j,\cC_i)=0$, for $i\neq j$, then $\alpha_1,\dots,\alpha_n$ are $\K^u$-linearly ($\barQ$-linearly, resp.) independent.
\end{prop}
\begin{proof}
We prove the case for $n=2$. The result for arbitrary $n$ then follows by induction. Recall that $\langle M_1\rangle$ and $\langle M_2\rangle$ are full additive subcategories, generated by $M_1$ and $M_2$, respectively, and closed under subquotients.
    By the assumption and the definition of formal space of periods at depth $i$, we have \[\tilde{\cP}^i_{\Hpp}( M_1\oplus M_2)=\tilde{\cP}^i_{\Hpp}( M_1)\oplus\tilde{\cP}^i_{\Hpp}( M_2).\]
    Assume that there exist some $\lambda_1,\lambda_2\in\K^u$ such that 
    \[
    \lambda_1\alpha_1+\lambda_2\alpha_2=0.
    \]
    The period $\lambda_1\alpha_1+\lambda_2\alpha_2$ is a $\Hpp$-period of $M_1\oplus M_2$. The validity of $\Hpp$-period conjecture at depth 2 (\cref{level 2 period cinjecture for Hpp}) implies that $(\lambda_1\alpha_1+\lambda_2\alpha_2)_{\tilde{\cP}^2_{\Hpp}( M_1\oplus M_2)}$ is zero in $\tilde{\cP}^i_{\Hpp}( M_1\oplus M_2)$. But \[\tilde{\cP}^2_{\Hpp}(M_1\oplus M_2)=\tilde{\cP}^2_{\Hpp}( M_1)\oplus\tilde{\cP}^2_{\Hpp}( M_2)\cong\cP_{\Hpp}\langle M_1\rangle\oplus\cP_{\Hpp}\langle M_2\rangle,\] thus $\lambda_1\alpha_1=0$ and $\lambda_2\alpha_2=0$. As $\alpha_1,\alpha_2\neq 0$, this means that $\lambda_1=\lambda_2=0$.

    For the $\Hp$-periods the proof is similar.
\end{proof}

\begin{cor}
    Let $A$ be an abelian variety with good reduction over $\K$. Assume that $\alpha_1$ is a nonzero $\Hpp$-period (or $\Hp$-period, resp.) of $A$, and $\alpha_2$ is a nonzero $\Hpp$-period (or $\Hp$-period, resp.) of $\G_m$. Then $\alpha_1, \alpha_2$ are $\K^u$-linearly independent ($\barQ$-linearly independent, resp.). 
\end{cor}
\begin{proof}
    The proof follows from the fact that there are no non-trivial morphisms between $\G_m$, and $A$, combined with the application of the above proposition.
\end{proof}


\subsection*{1-motives associated with smooth varieties} Varieties provide another rich source of examples for our p-adic periods. Following \cite{BS01}, we begin by associating a 1-motive to any equidimensional variety $X$ over a field $K$ of characteristic $0$. Let $S\subset X$ be the singular locus, $f:\tilde{X}\to X$ a resolution of singularities of $X$ and $\tilde{S}$ denote the reduced inverse image of $S$. Consider a smooth compactification $\bar{X}$ of $X$ with boundary $Y=\bar{X}-\tilde{X}$ and denote by $\bar{S}$ the Zariski closure of $\tilde{S}$ in $X$. We can choose the resolution $\tilde{X}$ and the compactification $\bar{X}$ of $X$ such that $\bar{X}$ is a projective variety with a reduced normal crossing divisor $\bar{S}+Y$. The fpqc sheaf 
\[
T\mapsto\Pic(\bar{X}\times_{\Spec K} T,Y\times_{\Spec K}T)
\]
is representable by a $K$-group scheme which is locally of finite type over $K$ whose group of $K$-rational points is $\Pic(\bar{X},Y)$ (\cite[Lemma 2.1]{BS01}). Its identity component is denoted by $\Pic^0(\bar{X},Y)$. Let $Y_i$ be the smooth irreducible components of $Y$. The identity component of the kernel $A(\bar{X},Y):=\ker^0(\Pic^0(\bar{X})\to\oplus_i\Pic^0(Y_i))$ is an abelian variety and we have an exact sequence 
\[
0\to T(\bar{X},Y)\to\Pic^0(\bar{X},Y)\to A(\bar{X},Y)\to 0, 
\]
where $T(\bar{X},Y)$ is a torus (see \cite[Proposition 2.2]{BS01}). Thus, $\Pic^0(\bar{X},Y)$ should represent the semi-abelian part of the 1-motive associated with $X$ ($\W_{-1}(M)=\Pic^0(\bar{X},Y)$). We now identify its lattice part. Denote by $\Div^0_{\bar{S}}(\bar{X},Y)$ the subgroup of divisors $D$ on $\bar{X}$ such that $\operatorname{supp}(D)\cap Y=\emptyset$, $\operatorname{supp}(D)\subset\bar{S}$, and $[D]\in\Pic^0(\bar{X},Y)$\footnote{The divisor $D$ has a section trivializing it on $\bar{X}-D$, therefore $(\cO_{\bar{X}}(D),1)$ identifies a class $[D]\in\Pic^0(\bar{X},Y)$.}. Consider the push-forward of Weil divisors $f_*:\Div_{\tilde{S}}(\tilde{X})\to\Div_S(X)$ and let $\Div_{\tilde{S}/S}(\tilde{X},Y)$ be its kernel and $\Div^0_{\bar{S}/S}(\bar{X},Y)$ denote the intersection of $\Div_{\tilde{S}/S}(\tilde{X},Y)$ and $\Div^0_{\bar{S}}(\bar{X},Y)$, i.e., the group of divisors on $\bar{X}$ which are linear combinations of compact components in $\tilde{S}$
which have trivial push-forward under $f:\tilde{X}\to X$ and which are algebraically equivalent to zero
relative to $Y$. We have the following definition:
\begin{defn}{\cite[Definition 2.3]{BS01}}
    The homological Picard 1-motive of $X$ (or the 1-motive associated with $X$) is defined as 
    \[
    \Pic^{-}(X)=[\Div^0_{\bar{S}/S}(\bar{X},Y)\xrightarrow{u}\Pic^0(\bar{X},Y)],
    \]
    where $u(D)=[D]$. The cohomological Albanese 1-motive $\Alb^+(X)$ of $X$ is defined as the Cartier dual of $\Pic^{-}(X)$.
\end{defn}

\begin{example}
    Let $C$ be a curve. For the homological Picard 1-motive $\Pic^{-}(C)$ of $C$, $\W_{-1}(\Pic^{-}(C))$ is isomorphic to the generalized Jacobian $J(C)$ of $C$ (see \cite[\S 1.8]{miyanishi_algebraic_2020}). Now, assume that $C$ is a smooth curve and $D\subset C$ a subvariety of dimension $0$. To the pair $(C,D)$, we can associate the 1-motive 
    \[
    [\Div^0(D)\to J(C) ],
    \]
    where $\Div^0(D)$ is the subgroup of the degree-zero divisors supported on $D$. Let $H$ be a cohomology theory. Applying the long exact sequence for relative cohomology to the inclusion $C\to J(C)$ yields
    \begin{equation*}
\begin{tikzcd}
H^0(J(C)) \arrow[r] \arrow[d, "f"] & H^0(D) \arrow[r] \arrow[d, Rightarrow, no head] & {H^1(J(C),D)} \arrow[r] \arrow[d, "h"] & H^1(J(C)) \arrow[r] \arrow[d, "g"] & 0 \\
H^0(C) \arrow[r]              & H^0(D) \arrow[r]                                & {H^1(C,D)} \arrow[r]              & H^1(C) \arrow[r]              & 0
.\end{tikzcd}
    \end{equation*}
    By the five lemma, if both $f$ and $g$ are isomorphisms, then so is $h$. According to \cite[Chapter V]{serrealgebraicgroupsandclaasfields}, this condition holds for the de Rham, singular, and the \'etale cohomology. For the crystalline cohomology over field of characteristic $p\geq 3$, this follows from \cite[Theorem B']{andreatta2005crystalline}. Thus, one can transfer all arithmetic information obtained for 1-motives to the pair $(C,D)$. In particular, we can identify our $\Q$-structures within the homology classes of $(C,D)$ and define our notion of p-adic periods for $(C,D)$ when the 1-motive associated with $(C,D)$ has a good reduction at $p$.
\end{example}

\begin{defn}
    Let $X$ be a variety defined over a number field $\K$. Assume that $\Pic^{-}(X)$ is a 1-motive with a good reduction at $p$. An $\Hpp$-period ($\Hp$-period or $\hp$-period, resp.) of $X$ is a p-adic number which is an $\Hpp$-period ($\Hp$-period or $\hp$-period, resp.) of $\Pic^{-}(X)$. The space of $\Hpp$-periods ($\Hp$-periods or $\hp$-periods, resp.) of $X$ is denoted by $\cP_{\Hpp}(X)$ ($\cP_{\Hp}(X)$ or $\cP_{\hp}(X)$, resp.) and its space of formal periods at depth $i$ is denoted by $\tilde{\cP}^i_{\Hpp}(X)$ ($\tilde{\cP}^i_{\Hp}(X)$ or $\tilde{\cP}^i_{\hp}(X)$, resp.).
\end{defn}
When the 1-motive associated with $X$ has a good reduction at $p$, these p-adic periods of $X$ arise from the p-adic integration pairing of the 1-motive $\Pic^{-}(X)$. \cref{level 2 period cinjecture for Hpp} and \cref{depth 2 and 1 period cinjecture for Hp and hp} provide the following motivic classification of their vanishing behaviour:
\begin{thm}
Assume the homological Picard 1-motive $M=\Pic^{-}(X)$ of the variety $X$ defined over a number field $\K$ has a good reduction at $p$. Let $\alpha$ be an $\Hpp$-period ($\Hp$-period or $\hp$-period, resp.) of $X$. If $\displaystyle\alpha=\int^{\Hpp}_x\omega=0$ ($\displaystyle\alpha=\int^{\Hp}_x\omega=0$ or $\displaystyle\alpha=\int^{\hp}_x\omega=0$, resp.), then there exists an exact sequence
     \[
     0\to M_1\to M^n\to M_2\to 0
     \]
     of 1-motives over a finite extension of $\K$ with good reductions at $p$, where $n\in\{1,2\}$ ($n\in\{1,2\}$ or $n=1$, resp.), $x\in\tilde{\Hpp}(M_1)$ ($x\in\Hp(M_1)$ or $x\in\hp(M_1)$, resp.), and $\omega\in\tilde{N}(M_2)$ ($\omega\in\Fil^1(\TdR\ve(M_2))_{\bar{\K}}$, resp.).
\end{thm}

\clearpage
\appendix
\begin{appendix}
\chapter{Schemes}
\section{On morphisms of schemes}

Recall that a ring homomorphism $R\to A$ is of finite presentation if $A$ is isomorphic to $R[x_1,\dots,x_n]/(f_1,\dots,f_m)$ as an $R$-algebra. We say $R\to A$ is of finite type if $A$ is isomorphic to a quotient of $R[x_1,\dots,x_n]$ as an $R$-algebra.\\
The commutative ring $R$ is regular if all the localization rings $R_{\fp}$ are regular local rings for every prime ideal $\fp$ of $R$. For detailed information on smooth ring maps, see \cite[\href{https://stacks.math.columbia.edu/tag/00T1}{Section 00T1}]{stacks-project}.\\
For any morphism $f\colon X\to S$ of schemes, the residue fields at point $x\in X$ and $s=f(x)\in S$ are denoted by $k(x)$ and $k(f(x))$ respectively.\\
An $R$-module $M$ is faithfully flat if any complex of $R$-modules $N_1\to N_2\to N_3$ is exact if and only if the sequence $M\otimes_R N_1\to M\otimes_R N_2\to M\otimes_R N_3$ is exact.
The ring homomorphism $R\to A$ is called faithfully flat if $A$ is faithfully flat as an $R$-module.\\
Let $(X,\cO_X)$ be a locally ringed space. Let $\cF$ be a sheaf of $\cO_X$-modules. We say that $\cF$ is finite locally free as $\cO_X$-modules if there exists an open covering $\{U_i\to X\}_{i\in I}$ of $X$ such that each restriction $\cF\mid_{U_i}$ is a finite free $\cO_{X/U_i}$-module. We say that a sheaf $\cF$ of $\cO_X$-modules is quasi-coherent (on Zariski topology of $X$) if $\cF$ is locally an $\cO_X$-module which has a global presentation. See \cite[\href{https://stacks.math.columbia.edu/tag/01BD}{Section 01BD}]{stacks-project} for the details.
\begin{defn}
Let $f: X\to S$ be a morhism of schemes.
\begin{enumerate}
    \item We say that $f$ is \textbf{locally of finite presentation} if for each $x\in X$ there exists an affine neighbourhood $U=\Spec A\subseteq X$ of $x$ and affine $V=\Spec R\subseteq S$ with $f(U)\subseteq V$ such that the induced ring homomorphism $f^*:R\to A$ is of finite presentation. 
    \item We say that $f$ is of \textbf{finite presentation} if it is locally of finite presentation, quasi-compact and quasi-separated.
        
    \item We say that $f$ is \textbf{locally of finite type} if for each $x\in X$ there exists an affine open neighbourhood $U=\Spec A\subseteq X$ of $x$ and an open affine $V=\Spec R\subseteq S$ with $f(U)\subseteq V$ such that the induced ring homomorphism $R\to A$ is of finite type.
    
    \item We say that $f$ is \textbf{of finite type} if it is locally of finite type and quasi-compact.

    \item We say that $f$ is \textbf{integral} if for each $x\in X$ there exists an affine open neighborhood $U=\Spec A\subseteq X$ of $x$ and an affine open $V=\Spec R\subseteq S$ with $f(U)\subseteq V$ such that the induced ring homomorphism $R\to A$ is integral.
    
    \item We say that $f$ is \textbf{unramified} if $f$ is locally of finite type and for each $x\in X$ we have $\fm_{f(x)}\cO_{x,X}=\fm_x$ and the extension $k(x)/k(f(x))$ is finite separable.

    \item We say that $f$ is \textbf{flat} if for each $x\in X$ the induced ring map on stalks $\cO_{S,f(x)}\to \cO_{X,x}$ is flat.

    \item We say that $f$ is \textbf{faithfully flat} if for each $x\in X$ the induced ring map on stalks $\cO_{S,f(x)}\to \cO_{X,x}$ is faithfully flat.
    
    \item We say that $f$ is regular if for any $x\in X$, $X_{f(x)}\to\Spec k(f(x))$ is regular.

    \item We say that $f$ is \textbf{smooth} if $f$ is locally of finite presentation and flat and for any $x\in X$, $X_{f(x)}\to \Spec(k(f(x))$ is smooth.
    \item We say that $f$ is \textbf{\'etale} if $f$ is smooth and unramified or equivalently, for every $s$ the fibre $X_s$ is a disjoint union $\coprod^{m}_{i=1}\Spec(l_i) $ where each $l_i/k(s)$ is a finite separable extension.

    \item We say that $f$ is \textbf{proper} if $f$ is separated, finite type, and universally closed.
    
    \item We say that $f$ is \textbf{affine} (\textbf{quasi-projective or projective} resp.) if the inverse image of every affine open of $S$ is affine (quasi-projective or projective resp.).

    \item We say that $f$ is \textbf{finite locally free} if $f$ is affine and $f_*\cO_X$ is finite locally free as $\cO_S$-modules.

    \item We say that $f$ is \textbf{finite} if $f$ is affine and $f_*\cO_X$ is locally finitely generated as $\cO_S$-modules i.e. $\cO_S(U)\to\cO_X(f^{-1}(U))$ is finite for any affine open $U\subseteq S$.
    
    \item We say that $f$ is of \textbf{relative dimension $d$} if all nonempty fibres $X_s$ have the same dimension $d$.
\end{enumerate}
See \cite{stacks-project}.
\end{defn}
\begin{prop}
Let $f:X\to S$ be a morphism of schemes.
\begin{enumerate}
    \item The morphism $f\colon X\to S$ is faithfully flat if and only if it is flat and surjective.
    \item (Infinitesimal lifting criterion) The morphism $f\colon X\to S$ is smooth if and only if $f$ is locally of finite presentation and for every affine scheme $\Spec A$ over $S$ and for every nilpotent ideal $I\subset A$, the natural map $X(A)\to X(A/I)$ is surjective.
    \item If one replaces “surjective” in (2) by “injective”
    or “bijective”, one gets equivalent definitions for $X\to S$ unramified
    or \'etale morphisms, respectively.
    \item One can get an equivalent condition in both (2) and (3) if one allows only ideals $I$ for which $I^2=0$.
    \item $f$ is integral if and only if $f$ is affine and universally closed.
    \item If $f$ is integral and locally of finite type, then $f$ is finite.
    \item $f$ is finite if and only if $f$ is affine and proper.
    \item (Valuative criterion for properness) $f$ is proper if and only if for every valuation ring $A$ (defined over $S$) with field of fraction $K$, the natural map $X(A)\to X(K)$ is surjective.
\end{enumerate}
\end{prop}

\begin{proof}
(1): \cite[\href{https://stacks.math.columbia.edu/tag/00HQ}{Lemma 00HQ}]{stacks-project}.\\
(2): \cite[\href{https://stacks.math.columbia.edu/tag/02H6}{Lemma 02H6}]{stacks-project}, or \cite[\S 17.5.2]{grothendieck1967elements}.\\
(3): \cite[\href{https://stacks.math.columbia.edu/tag/02HE}{Lemma 02HE}]{stacks-project} , or \cite[\S 17.1.1, 17.3.1]{grothendieck1967elements}.\\
(5):\cite[\href{https://stacks.math.columbia.edu/tag/01WM}{Lemma 01WM}]{stacks-project}\\
(6): \cite[\href{https://stacks.math.columbia.edu/tag/01WJ}{Lemma 01WJ}]{stacks-project}\\
(7): \cite[\href{https://stacks.math.columbia.edu/tag/01WN}{Lemma 01WN}]{stacks-project}\\
(8): \cite[\href{https://stacks.math.columbia.edu/tag/0BX5}{Lemma 0BX5}]{stacks-project}
\end{proof}

\begin{defn}{\cite[\href{https://stacks.math.columbia.edu/tag/04EW}{Section 04EW}]{stacks-project}}
We say a scheme $X'$ is a thickening of a scheme $X$ if $X$ is a closed subscheme of $X'$ and the underlying topological spaces are equal.
\end{defn}
\begin{defn}
Let $k$ be a field and $X$ a scheme over $k$. We say that $X$ is geometrically reduced (irreducible, or connected, or integral, or normal resp.) over $k$ if for any field extension $k'/k$, $X_{k'}$ is reduced (irreducible, or connected, or integral, or normal resp.).
\end{defn}
Recall that a $k$-algebra $A$ is called normal if the localisation $A_{\fp}$ is integrally closed in its field of fractions for each prime ideal $\fp$ in $A$. Recall that a $k$-algebra $A$ is called connected if $\Spec A$ is a connected scheme. Moreover, $\Spec A$ is connected if and only if there is no nontrivial idempotents in $A$ (\cite{326463}).

\begin{prop}
Let $k$ be a field and $X$ a scheme over $k$.
\begin{enumerate}
\item If $k$ is a prefect field, then the scheme $X$ over $k$ is smooth if and only if $X\to\Spec k$ is locally of finite type and regular.
\item  $X$ is irreducible (reduced or connected resp.) if and only if $X_{\bar{k}}$ is irreducible (reduced or connected resp.), where $\bar{k}$ is the separable closure of $k$.

\item If $X$ is geometrically integral over $k$ if and only if $X$ is geometrically reduced and geometrically irreducible over $k$.

\item If $k$ is a prefect field then $X$ is geometrically reduced (normal resp.) if and only if $X_k$ is reduced (normal resp.).

\item Assume that $X$ is a smooth scheme over field $k$. Then $X$ is geometrically regular, geometrically normal, and geometrically reduced over $k$.

\item If $X$ is a proper geometrically normal $k$-scheme the following are equivalent:
\begin{itemize}
    \item $X$ is geometrically connected
    \item $X$ is geometrically integral
    \item $X$ is geometrically irreducible
\end{itemize}
 \end{enumerate}
\end{prop}
\begin{proof}
(1): \cite[\href{https://stacks.math.columbia.edu/tag/0364}{Section 0364}]{stacks-project}
(2): \cite[\href{https://stacks.math.columbia.edu/tag/0366}{Section 0366}]{stacks-project}
(3): \cite[\href{https://stacks.math.columbia.edu/tag/038L}{Section 038L}]{stacks-project}
(4): \cite[\href{https://stacks.math.columbia.edu/tag/038L}{Section 038L}]{stacks-project}.
\end{proof}

\begin{defn}[\cite{katz1985arithmetic}]
\begin{enumerate}
    \item  Let $k$ be a field. A variety over $k$ is a reduced, separated scheme of finite type over $k$.
    \item A smooth curve $C$ of genus $g$ over a scheme $S$ is a proper, smooth morphism $C\to S$ of relative dimension 1 such that all the geometrically connected fibres are curves of genus $g$.
    \item A smooth curve $E$ of genus $1$ over a scheme $S$ is called an elliptic curve over $S$. Equivalently, an elliptic curve over $S$ is a smooth, proper group scheme $E$ over $S$ of relative dimension $1$ with a section $0\colon S\to E$.
\end{enumerate}
\end{defn}
\begin{remark}
For curve $C$ over $S=\Spec k$ where $k$ is perfect, the following are equivalent:
\begin{itemize}
    \item $C\to S$ is smooth.
    \item $C\to S$ is normal.
    \item $\cO_{C,x}$ is a discrete valuation ring for any $x\in C$.
    \item Each point $x$ of $C$ is an effective Cartier divisor.
\end{itemize}
\end{remark}
\begin{remark}
Recall that the local dimension of a scheme $X$ at a point $x\in X$ is $\dim_x X=\inf_U\dim U$ where $\inf$ runs through all open neighborhoods of $x$. We define $\dim X:=\sup_x\dim_x X$.\\ If $X$ is a scheme that is locally of finite type over a field and $x$ is closed, then $\dim_x X=\dim\cO_{X,x}$ where $\dim\cO_{X,x}$ is the Krull dimension of a local ring. Assume that $f\colon X\to S$ is locally of finite type, then $\dim_x X_{f(x)}=\dim\cO_{X_{f(x)},x}+\text{trdeg}_{k(f(x))}k(x)$ (\cite[\href{https://stacks.math.columbia.edu/tag/00P1}{Lemma 00P1}]{stacks-project}).
\end{remark}

\section{Topologies on scheme}
\begin{defn}
Let $X$ be a scheme.
\begin{enumerate}
    \item An \'etale covering (of $X$) is a family of morphisms $\{f_i\colon T_i\to X\}_{i\in I}$ such that each $f_i$ is \'etale and $X=\bigcup_{i\in I}f_i(T_i)$.
    \item An fppf\footnote{It stands for fidèlement plate de présentation finie} covering (of $X$) is a family of morphisms $\{f_i\colon T_i\to X\}_{i\in I}$ such that each $f_i$ is flat, locally of finite presentation, and such that $X=\bigcup_{i\in I}f_i(T_i)$.
    \item An fpqc\footnote{It stands for fidèlement plate et quasi-compacte} covering (of $X$) is a family of morphisms $\{f_i\colon T_i\to X\}_{i\in I}$ such that each $f_i$ is flat and for every affine open $U\subseteq X$ there exist quasi-compact opens $U_i\subseteq T_i$ which are almost all empty except finitely many of them, such that $U=\bigcup_{i\in I} f_i(U_i)$.
\end{enumerate}
\end{defn}
\begin{defn}{\cite[\href{https://stacks.math.columbia.edu/tag/00VH}{Definition 00VH}]{stacks-project}}
A site consists of a category $\cC$ and a set $Cov(\cC)$ of families of morphisms $U=\{\varphi_i\colon U_i\to U\}_{i\in I}$ called coverings, such that
\begin{enumerate}
    \item (isomorphism) If $\varphi\colon V \to U$ is an isomorphism in $\cC$, then $\{\varphi\colon V\to U\}$ is a covering in $Cov(\cC)$.
    \item (locality) If $\{\varphi_i\colon U_i\to U\}_{i\in I}$ is a covering and for all $i\in I$ we are given a covering $\{\psi_{ij}\colon U_{ij}\to U_i\}_{j\in I_i}$, then $\{\varphi_i \circ \psi_{ij}\colon U_{ij}\to U\}_{i\in I,j\in I_i}\in Cov(\cC)$.
    \item (base change) If $\{U_i\to U\}_{i\in I}$ is a covering and $V\to U$ is a morphism in $\cC$, then
    for all $i\in I$ the fibre product $U_i \times_U V$ exists in $\cC$, and $\{U_i \times_U V\to V \}_{i\in I}\in Cov(\cC)$.
\end{enumerate}
\end{defn}
\begin{defn}\label{Appendix: def fppf and etale}
Let $S$ be a scheme.
\begin{enumerate}
    \item The Zariski site of $S$, denoted by $S_{\text{zar}}$, is the site that the underlying category $\cC$ is the category of open immersions $U\hookrightarrow S$ with open immersions over $S$ as morphisms. The covering $Cov(\cC)$ is the set of all open coverings $\{U_i\hookrightarrow S\}_{i\in I}$.
    \item The (big) fppf site of $S$, denoted by $(\text{Sch}/S)_{\text{fppf}}$, is the site that underlying category $\cC$ is the category of schemes over $S$ and the morphisms are flat and locally of finite presentation. The coverings $Cov(\cC)$ is the set of all fppf covering of $S$.
    \item The (big) \'etale site of $S$, denoted by $(\text{Sch}/S)_{\text{\'etale}}$, is the site that underlying category $\cC$ is the category of schemes over $S$ and the morphisms are \'etale morphisms. The coverings $Cov(\cC)$ is the set of all \'etale covering of $S$.
    \item The small fppf (\'etale resp.) site of $S$, denoted by $S_{\text{fppf}}$ ($S_{\text{\'etale}}$ resp.), is the full subcategory of the site $(Sch/S)_{\text{fppf}}$ ($(Sch/S)_{\text{\'etale}}$ resp.) with the same set of coverings consisting of fppf (\'etale resp.) schemes over $S$.
\end{enumerate}
\end{defn}
\begin{defn}\label{def: sheaf on site}
Let $\cC$ be a site. A presheaf $\cF$ of sets (resp. abelian groups, vector spaces, etc.) on $\cC$ is a contravariant functor from
the underlying category $\cC$ to the category of sets (resp. abelian groups, vector spaces, etc.).\\
We say that the presheaf $\cF$ on $\cC$ is a sheaf if for all coverings $\{U_i\to U\}_{i\in I}$ in $Cov(\cC)$, the sequence 
    $$0\to \cF(U)\to \prod_{i\in I}\cF(U_i)\rightrightarrows\prod_{i,j\in I}\cF(U_i\times_U U_j)$$ is exact.
\end{defn}
\begin{defn}
An fppf (\'etale resp.) sheaf on $S$ is a sheaf on the small site $S_{\text{fppf}}$ ($S_{\text{\'etale}}$ resp.). 
\end{defn}
\begin{example}[\cite{conradetale}]
The coverings in the \'etale site $(\Spec k)_{\text{\'etale}}$ are refined form 
$$(\coprod_{j\in J_i} k_{ij}\to k)_{i\in I} $$ where $k_{ij}/k$ are finite separable extensions upto isomorphism.\\
The presheaf $\cF$ on $(\Spec k)_{\text{\'etale}}$ is sheaf if and only if
\begin{enumerate}
    \item $\cF(\coprod U_i)=\prod \cF(U_i)$, for any disjoint union $U_i$ of \'etale schemes over $k$,
    \item $\cF(\Spec k')\to\cF(\Spec k'')$ is injective and $$\cF(\Spec k')= \cF(\Spec k'')^{Gal(k''/k')},$$ for all separable extensions $k''/k'/k$ such that $k''/k$ is Galois.
\end{enumerate}
We have an equivalence of categories
$$\{\text{abelian \'etale sheaves on }\Spec k\}\xrightarrow{\cong} \{\text{discrete }G_k-\text{modules}\}$$
\end{example}
\begin{remark}
    The fpqc is finer than the Zariski, étale, and fppf topologies. Hence any presheaf satisfying the sheaf condition for the fpqc topology will be a sheaf on the Zariski, étale, and fppf sites.
\end{remark}



\section{On local rings}
\textbf{References:} \cite[\S18]{grothendieck1967elements}, \cite{raynaud2006rings}, \cite[I.\S4]{milne2016etale}, \cite[\href{https://stacks.math.columbia.edu/tag/04GE}{Section 04GE}]{stacks-project}, \cite[\href{https://stacks.math.columbia.edu/tag/04GE}{Section 04GE}]{stacks-project}.
\begin{defn}
Let $R$ be a local ring. We say that $R$ is henselian if every finite R-algebra 
decomposes into a finite product of local rings. It is called strictly
henselian if it is henselian and its residue field is separably closed.
\end{defn}
\begin{prop}
Let $R$ be a local ring.
\begin{enumerate}
    \item $R$ is henselian if and only if Hensel's lemma holds for $R[T]$. 
    \item Any complete local ring is henselian.
    \item Let $R$ be a henselian local ring with residue field $k$. The functor $S\mapsto S\times_R k$ induces an equivalence of categories 
    $$\{\text{finite \'etale schemes over }R \} \to \{\text{finite \'etale schemes over }k\}$$
\end{enumerate}
\end{prop}
\begin{proof}
(1), (2): See \cite{milne2016etale}, or \cite{grothendieck1967elements}.\\
(3): see \cite[Proposition I.4.4]{milne2016etale}, or \cite[\href{https://stacks.math.columbia.edu/tag/0A48}{Lemma 0A48}]{stacks-project}. 
\end{proof}
\begin{thm}[Hensel's lemma]
Let $R$ be a complete noetherian ring with maximal ideal $\fm$ and $X$ a scheme over $R$.
\begin{enumerate}
    \item If $X$ is smooth over $\Spec R$, then the natural map $X(R)\to X(R/\fm)$ is surjective.
    \item If $X$ is \'etale over $\Spec R$, then the natural map $X(R)\to X(R/\fm)$ is bijective.
\end{enumerate}
\end{thm}
\begin{proof}
    See \cite[Theorem 3.5.63]{poonen2017rational}
\end{proof}
\begin{remark}
The above theorem implies the Hensel's lemma in algebraic number theory when $R=\Z_p$ and $X=\Z_p[t]/(f)$ where $f$ is a monic polynomial over $\Z_p$ such that its reduction modulo $p$ is separable.
\end{remark}

\section{Tangent space and differentials}
\textbf{References:} \cite[\S 26]{matsumura1970commutative}, \cite[\href{https://stacks.math.columbia.edu/tag/00RM}{Section 00RM}]{stacks-project}, \cite[\href{https://stacks.math.columbia.edu/tag/08RL}{Section 08RL}]{stacks-project}  

Let $\varphi: R\to A$ be a ring homomorphism and $M$ an $A$-module. A derivation or more precisely an $R$-derivation into $M$ is a map $D\colon A\to M$ which is a group homomorphism, annihilates elements of $\varphi(R)$, and satisfies the Leibniz rule: $D(ab)=aD(b)+bD(a)$. We denote the $A$-module of all $R$-derivations into $M$ by $\Der_R(A,M)$.

There exists a universal object $\Omega_{A/R}$, the module of K\"ahler differentials, such that $\Hom_A(\Omega_{A/R},-)\to\Der_R(A,-)$ is an isomorphism of functors of $A$-modules. The map $d\colon A\to \Omega_{A/R}$ corresponding to the identity map $id\in \Hom_A(\Omega_{A/R},\Omega_{A/R})$ is called universal derivation.

Let $X\to S$ be a morphism of schemes and $\cM$ be an $\cO_X$-module. A derivation or more precisely an $S$-derivation into $\cM$ is a map of abelian sheaves $D\colon \cO_X\to \cM$ on $X$ such that for each open subset $U$ of $X$, $D_U\colon \Gamma(U,\cO_X)\to\Gamma(U,\cM)$ is a $\Gamma(U,\cO_S)$-derivation of $\Gamma(U,\cO_X)$ into $\Gamma(U,\cM)$. We denote the $\cO_X$-module of all $S$-derivations into $\cM$ by $\Der_S(\cO_X,\cM)$.

The functor $\cM\to \Der_S(\cO_X,\cM)$ is representable. The module of differentials is the object representing the above functor. It is denoted $\Omega_{X/S}$, and the universal derivation is denoted $d\colon \cO_X\to \Omega_{X/S}$. The elements in $\Der_S(\cO_X,\cO_X)=\Hom_{\cO_X}(\Omega_{X/S},\cO_X)$ are called vector fields on $X$.

For any ring $R$ the ring of dual numbers over $R$ is the $R$-algebra $R[\epsilon]=R[x]/(x^2)$. Let $X\to S$ be a morphism of schemes. Let $x\in X$ with $f(x)=s\in S$. The set of all dotted arrows making the diagram
\begin{center}
\begin{tikzcd}
{\Spec(k(x)[\epsilon])} \arrow[r, dotted] \arrow[d] & X \arrow[d] \\
\Spec(k(s)) \arrow[r]                               & S 
\end{tikzcd}
\end{center}
commute is a $k(x)$-vector space. It is called the tangent space of $X$ over $S$ at $x$ and it is denoted $\Tang_{X/S,x}$. 

\begin{prop}
\begin{enumerate}
    \item Let $R\to A$ be a ring map. Let $J=\ker(A\otimes_R A\to A)$ be the kernel of multiplication. There is a canonical isomorphism of $A$-modules $\Omega_{A/R}\to J/J^2,\, adb\mapsto a\otimes b-ab\otimes 1$.
    \item Let $X\to S$ be a morphism of schemes. Let $x\in X$. There are canonical isomorphisms
    $\Tang\ve_{X/S,x}=\Omega_{X/S,x}\otimes_{\cO_{X,x}}k(x)$ and $\Tang_{X/S,x}=\Hom_{\cO_{X,x}}(\Omega_{X/S,x},k(x))$ as $k(x)$-vector spaces.
    \item Let $X\to S$ be a morphism of schemes. Let $x\in X$ and $s=f(x)$. If $k(x)/k(s)$ is separable algebraic extension, then $$\fm_x/(\fm^2_x+\fm_s\cO_{X,x})=\Omega_{X/S,x}\otimes_{\cO_{X,x}}k(x)=\Tang\ve_{X/S,x}$$ 
    \item Let $f\colon X\to S$ be locally of finite presentation. $f$ is smooth (unramified resp.) at $x$ if and only if the $k(x)$-vector space 
    $$\Omega_{X_s/s,x}\otimes_{X_s,x}k(x)=\Omega_{X/S,x}\otimes_{\cO_{x,x}}k(x)=\Tang\ve_{X/S,x}$$
    is of dimension $\dim_x(X_{f(x)})$ (dimension $0$ resp.). 
\end{enumerate}
\end{prop}
(1): \cite[\href{https://stacks.math.columbia.edu/tag/00RW}{Lemma 00RW}]{stacks-project}\\
(2): \cite[\href{https://stacks.math.columbia.edu/tag/0B2D}{Lemma 0B2D}]{stacks-project}\\
(3): \cite[\href{https://stacks.math.columbia.edu/tag/0B2E}{Lemma 0B2E}]{stacks-project}\\
(4): \cite[\href{https://stacks.math.columbia.edu/tag/01V9}{Lemma 01V9}]{stacks-project} and \cite[\href{https://stacks.math.columbia.edu/tag/02GF}{Lemma 02GF}]{stacks-project}.


\subsection{Derived category and Yoneda extensions} The derived category is a fundamental tool in homological algebra, algebraic geometry, and other areas of mathematics. It provides a framework to study objects like complexes of modules or sheaves, and is particularly useful in dealing with homological properties of these objects. We refer the reader to \cite[\href{https://stacks.math.columbia.edu/tag/05QI}{Chapter 05QI}]{stacks-project} for more details.

Let \( \mathcal{A} \) be an abelian category. The derived category \( D(\mathcal{A}) \) is constructed from the category of chain complexes in \( \mathcal{A} \) by formally inverting quasi-isomorphisms. The derived category contains objects that are complexes of objects in \( \mathcal{A} \), but its morphisms are refined, as homotopic complexes are identified.

Derived categories allow us to define derived functors, such as derived $\Hom$ and tensor product, without needing explicit resolutions. The derived $\Hom$ functor, denoted \( \mathbb{R}\mathrm{Hom}(A, B) \), is an object in the derived category, whose cohomology recovers the classical Ext groups:
  \[
  H^n(\mathbb{R}\mathrm{Hom}(A, B)) = \mathrm{Ext}^n(A, B).
  \]
  
For an abelian category \( \mathcal{A} \), the $\Ext$ groups \( \mathrm{Ext}^n(A, B) \) of $n$-th Yoneda extensions of $A$ by $B$ classify \( n \)-fold extensions of an object \( A \) by an object \( B \). An element of \( \mathrm{Ext}^n(A, B) \) corresponds to an exact sequence of the form
  \[
  0 \to B \to E_n \to E_{n-1} \to \dots \to E_1 \to A \to 0,
  \]
  where \( E_i \) are objects in \( \mathcal{A} \), and the class of this sequence in the derived category gives the corresponding $\Ext$ class. For the definition of equivalence classes of Yoneda extensions, see \cite[\href{https://stacks.math.columbia.edu/tag/06XT}{Definition 06XT}]{stacks-project}.

The group \( \mathrm{Ext}^1(A, B) \) classifies short exact sequences \( 0 \to B \to E \to A \to 0 \). This can be interpreted as equivalence classes of extensions of \( A \) by \( B \), where two extensions are equivalent if they differ by a split sequence.

Higher Ext groups \( \mathrm{Ext}^n(A, B) \) for \( n > 1 \) classify longer exact sequences, or more generally, \( n \)-fold extensions of \( A \) by \( B \). In the derived category, morphisms between objects represent all possible extensions, and the $\Ext$ groups can be computed directly as the cohomology of the derived $\Hom$ complex:
  \[
  \mathrm{Ext}^n(A, B) = H^n(\mathbb{R}\mathrm{Hom}(A, B)).
  \]
We can naturally associate a Yoneda $i$-extension
\[
0\to Y\to E_{i-1}\to\cdots\to Z_0\to X\to 0
\]
of $X$ by $Y$ the element 
\[
Y[i]\leftarrow[Y\to Z_{i-1}\to\cdots\to Z_0]\to X
\]
of $\Hom_{D(\cA)}(X,Y[i]).$ This is an isomorphism if $\cA$ has enough injective or projective objects. See \cite{MR1453167} for further details.


\section{Galois cohomology}\label{sec: Galois cohomology}

    Let $G$ be a group and $A$ a $\Z[G]$-module. We define $A^G=\{a\in A\mid ga=a\}$. Obviously, we have $$A^G=\Hom_{\Z[G]}(\Z,A)$$
    We define the $i$-th group cohomology of $G$ with coefficients in $A$ to be $$H^i(G,A):=\Ext^i_{\Z[G]}(\Z,A)$$
Let $G$ be a topological abelian group. Assume that $A$ is a topological abelian group endowed with a continuous $G$-action. Set $C^i_{cont}(G,A)$ be the group of all continuous maps $G^i\to A$ and $d^{i+1}:C^{i+1}_{cont}(G,A)\to C^i_{cont}(G,A)$ given by 
\begin{align*}
d^{i+1}(f)(g_1,\dots,g_{i+1})=g_1(f(g_2,\dots,g_{i+1}))\\+\sum_{j=1}^i(-1)^jf(g_1,\dots,g_j,g_{j+1},\dots,g_{i+1})+(-1)^{i+1}f(g_1,\dots,g_i)
\end{align*}
and for $i=0$, sends $a\in A$ to the map $g\mapsto g(a)-a$. We define the $i$-th continuous cohomology of $G$ with coefficients in $A$ to be 
$$H^i_{cont}(G,A):=H^i(C^{\bullet}_{cont}(G,A))$$

\begin{defn}
    When $G$ is an absolute Galois group of a field $K$, we say that $H^i_{cont}(G,A)$ is $i$-th Galois cohomology group of $K$ with coefficients in $A$, and it is often denoted by $H^i(K,A)$. 
\end{defn}
\begin{remark}
    For an abstract group $G$ (group $G$ with discrete topology) or profinite group $G$, we have an isomorphism $H^i(G,A)\xrightarrow{\cong} H^i_{cont}(G,A)$.
\end{remark}
\begin{prop}[\cite{serre_galois_1997}]
    Let $A$ be a $G$-module and $H$ a normal subgroup of $G$. Then we have the exact sequence
    \[
    0\to H^1(G/H,A^H)\to H^1(G,A)\to H^1(H,A)\to H^2(G/H,A^H)\to\cdots.
    \]
\end{prop}

\begin{thm}[\cite{tate1967pdivisble}, \cite{sen1980continuous}]
Let $K$ be a finite extension of $\Q_p$. Then
     \[H^i_{cont}(\Gamma_K,\C_p(n))\cong\begin{cases}
        K ,& n=0,i=0,1\\
        0 ,& n=0, i\neq 0,1\\ 
        0 ,&  n\neq 0, i\in\Z.
    \end{cases}\]
\end{thm}
Let $K$ be a nonarchimedean valued field with the norm $|.|:K\to\R_{\geq 0}$. Let Let $V$ be a vector space over $K$. A seminorm on $V$ is $\|.\|:V\to\R_{\geq 0}$ satisfying:
\begin{itemize}
    \item For $a\in K$ and $v\in V$, $\|av\|=|a|\|v\|$,
 \item For $v,w\in V$, $\|v+w\|\leq max\{\|v\|,\|w\|\}$.
\end{itemize}
The function $\|.\|$ is said to be a norm on $V$ if moreover $\|v\|=0$ implies that $v=0$, for any $v\in V$. If $V$ has a norm, it is said to be a normed space. If $V$ is a normed space which is complete under its norm, we say $V$ is a Banach space over $K$.

Any finite dimensional vector space over $K$ is a Banach space, and any two norms on a Banach space are equivalent. A Banach algebra over $K$ is a Banach space over $K$ which is also a commutative $K$-algebra, and
\begin{itemize}
    \item $\|1\|=1$,
    \item for $x,y\in A$, $\|xy\|\leq\|x\|y\|$.
\end{itemize}

\chapter{Groups}\label{Appendix: groups}
\section{Commutative group schemes}
\textbf{References:} \cite{milne2017algebraic}, \cite{waterhouse2012introduction}, \cite{milneLAG}, \cite[\href{https://stacks.math.columbia.edu/tag/022R}{Section 022R}]{stacks-project}, \cite{artin2006schemas}, \cite{oort2006commutative}, \cite{tate1997finite}, \cite{milneAV}, \cite{MumfordAVar}, \cite{poonen2017rational}, \cite{Sc2}, \cite{hase2020group}  

Let $S$ be a scheme. A group scheme $G$ over $S$ is a group object $G\to S$ in the category of $S$-schemes i.e. there are given maps $m\colon G\times_S G\to G$, $i\colon G\to G$, and $e\colon S\to G$ satisfying the commutative diagrams that characterize the group law with multiplication, inversion, and identity. We say that a group scheme is commutative if the multiplication $m$ satisfies the commutative diagram that characterize the commutative multiplication law. Using Yoneda's lemma, equivalently, we can say $G$ is a group scheme over $S$ if for any $S$-scheme $T$, $G(T):=\Hom_S(T,G)$ equipped with a group structure which is functorial in $T$. We can regard group scheme over $S$ as a representable object in the category sheaves on site $(\text{Sch}/S)$. But the category of all representable sheaves on $(\text{Sch}/S)$ is not abelian.

If $G=\Spec A$ is an affine group scheme over a ring $R$, then $A$ is an Hopf-algebra over $R$ with the comultiplication $\mu\colon A\to A\otimes_R A$, counit $\epsilon\colon A\to R$, and coinverse $\iota\colon A\to A$ respectively induced by the multiplication, unit section, and inverse of $G$.
\begin{defn}
    A group scheme of finite type over a field $K$ is said to be an algebraic group over $K$.
\end{defn}
\begin{lemma}{\cite[\href{https://stacks.math.columbia.edu/tag/023Q}{Lemma 023Q}]{stacks-project}}\label{representable functors are fppf,etale sheaf}
Every representable functor on $(\text{Sch}/S)$ is an fpqc (fppf, and \'etale) sheaf.
\end{lemma}

\begin{prop}
Let $G$ be a group scheme over $S$.
\begin{enumerate}
    \item If $S$ is a noetherian integral regular scheme whose irreducible components all have dimension $1$ and $G$ is quasi-compact and separated over $S$, then $G$ is quasi-projective.
    \item If $S=\Spec K$ and $G$ is locally of finite type over $\Spec K$, then $\dim G=\dim_g G$ for all $g\in G$.
    \item The structure morphism $G\to S$ is separated if and only if the identity morphism $e:S\to G$ is a closed immersion.
    \item If $S=\Spec K$, then the multiplication map is open, $G$ is separated. Moreover, $G$ is connected if and only if $G$ is irreducible. In addition, if the characteristic of $K$ is $0$ or $G$ is reduced over $K$, then $G$ is smooth over $K$.
\end{enumerate}
\end{prop}
\begin{proof}
    (1): See \cite[Th\'eor\`eme VIII.2]{raynaud2006rings}\\
    (2): See  \cite[\href{https://stacks.math.columbia.edu/tag/045X}{Lemma 045X}]{stacks-project}\\
    (3): See \cite[\href{https://stacks.math.columbia.edu/tag/047G}{Lemma 047G}]{stacks-project}\\
    (4): See \cite[\href{https://stacks.math.columbia.edu/tag/047J}{Section 047J}]{stacks-project}.
\end{proof}
\begin{defn}
An abelian variety over a field $K$ is a group scheme over $K$ which is also a proper, geometrically integral variety over $K$.
\end{defn}
\begin{prop}[Properties of abelian varieties-\cite{MumfordAVar}]
Let $A$ be an abelian variety over a field $k$.
\begin{enumerate}
    \item $A$ is projective and commutative.
    \item If $n\in k^{\times}$ and $k$ is algebraic closed, then $A(k)[n]\cong (\Z/n\Z)^{2\dim A}$.
    \item If $k$ is algebraic closed of characteristic $p$, then there exists and integer $0\leq f\leq\dim A$ such that $A(k)[p^m]\cong(\Z/p^m\Z)^f$ for all $m\geq 1$.
\end{enumerate}
\end{prop}
\begin{example}\label{examples group schemes in appendix}
\begin{enumerate}
    \item Additive group: The additive group over $\Z$ is the affine group scheme $\G_{a,\Z}=\Spec \Z[t]$ with the natural additive group structure on $\G_a(R)=R$ for each $\Z$-algebra $R$ that is associated with the Hopf-algebra structure given by $$\mu(t)=t\otimes 1 +1\otimes t,\quad \epsilon(t)=0,\quad \iota(t)=-t$$
    The additive group over $S$ is the group scheme $\G_{a,S}:=\G_{a,\Z}\times_{\Spec \Z}S$.
    
    \item Multiplicative group: The multiplicative group over $\Z$ is the affine group scheme $\G_{m,\Z}=\Spec \Z[t,t^{-1}]$ with the natural multiplicative group structure on $\G_m(R)=R^{\times}$ for each $\Z-$algebra $R$ that is associated to the Hopf-algebra structure given by 
    $$\mu(t)=t\otimes t,\quad \epsilon(t)=1,\quad \iota(t)=t^{-1}$$
    The multiplicative group over $S$ is the group scheme $\G_{m,S}:=G_{m,\Z}\times_{\Spec\Z}S$.
    
    \item The $n$-th roots of unity: The group of $n$-th roots of unity over $\Z$, denoted by $\mu_{n,\Z}$, is the closed subgroup scheme $\G_{m,\Z}$ associated to $\Z[t,t^{-1}]\twoheadrightarrow\Z[t]/(t^n-1)$ that has the natural multiplicative group structure on $\mu_n(R)=\{r\in R\mid r^n=1\}$ for each $\Z-$algebra $R$.\\
    The group $n$-th roots of unity over $S$ is the group scheme $\mu_{n,S}:=\mu_{n,\Z}\times_{\Spec\Z}S$.
    
    \item The group $\alpha_p$: The group scheme $\alpha_p$ over $\F_p$ is the closed subgroup scheme $\G_{a,\F_p}$ associated to $\F_p[t]\to\F_p[t]/(t^p)$ that has the natural additive group structure on $\alpha_p(R)=\{r\in R\mid r^p=0\}$ for each $\F_p$-algebra $R$.\\
    If $S$ is a scheme in characteristic $p$. The group $\alpha_{p,S}$ over $S$ is the group scheme $\alpha_{p,S}:=\alpha_{p,\F_p}\times_{\Spec\F_p}S$.
    
    \item Abstract groups: A constant group scheme is an abstract group. Assume that $M$ is an abstract group. Let $\underline{M}_S:=M\times S$, where $M\times S$ denotes the disjoint union of copies of $S$ indexed by $M$. We define the group operation $m:\underline{M}_S\times_S\underline{M}_S\to \underline{M}_S$ as follows: Note that $\underline{M}_S\times_S\underline{M}_S\cong \underline{M\times M}_S$. If $(x_1,x_2)\in M\times M$ then $m$ sends $s_{(x_1,x_2)}$ to $s_{x_1x_2}$ via the identity map. The morphisms $i$, and $e$ are defined analogously. One can see that $\underline{M}_S$ is an $S$-group scheme.\\
    The assignment $M\mapsto \underline{M}_S$ is functorial and $\underline{M}_S$ is commutative if $M$ is a commutative abstract group.
    
    \item Kummer sequence: If $n\in \cO^*_S$ then the sequence
    $$0\to \mu_{n,S}\to \G_{m,S}\xrightarrow{(.)^n}\G_{m,S}\to 0 $$
    is exact as sheaves on the site $(\text{Sch}/S)_{\text{fppf}}$, the small site $S_{\text{\'etale}}$, and the big site $(\text{Sch}/S)_{\text{\'etale}}$. But it is not exact on the Zariski site $S_{\text{zar}}$ (see \cite[\href{https://stacks.math.columbia.edu/tag/03PK}{Section 03PK}]{stacks-project}). 
\end{enumerate}
\end{example}

\begin{lemma}
Let $G$ be a scheme over a locally noetherian base scheme $S$. Then, $G\to S$ is finite locally free if and only if $G\to S$ is finite and flat.
\end{lemma}
\begin{proof}
 See \cite[\href{https://stacks.math.columbia.edu/tag/00NX}{Lemma 00NX}]{stacks-project}.
\end{proof}

\begin{defn}\label{def: finite flat group scheme}
Over a locally noetherian base scheme $S$, a group scheme over $S$ is called finite flat group scheme if $G\to S$ is finite locally free. If $G\to S$ is finite locally free of rank $n$, we say that the group scheme $G$ is of order $n$ (or rank $n$). For an affine group scheme $G=\Spec A$ over an affine base $S=\Spec R$, $G$ is a finite flat group scheme of rank $n$ if $A$ is locally free of rank $n$ as $R$-module.
\end{defn}
\begin{remark}
\begin{enumerate}
    \item The category of commutative group schemes of finite type over field $K$ (algebraic groups) is an abelian category.
 
    \item We always assume that the finite flat group scheme $G$ is commutative.
    
    \item The category of finite flat commutative group schemes over a field $K$ is an abelian category (see \cite[II.6]{demazure2006lectures}). 

    \item The category of finite flat group schemes over a general base scheme $S$ is just a pre-abelian category. However, we always regard the category of finite flat group schemes over $S$, denoted $(\text{ffgrp}/S)$, as representable sheaves on (big) fppf site $(\text{Sch}/S)_{\text{fppf}}$ (see Lemma \ref{representable functors are fppf,etale sheaf}); such category of sheaves is an abelian category with enough injectives (see \cite{poonen2017rational}).\\
    We can go back to the original category of finite flat group schemes over $S$ via fppf descent. Roughly speaking a sheaf is representable if and only if it is representable locally on the fppf site. See \cite[\href{https://stacks.math.columbia.edu/tag/02W5}{Lemma 02W5}]{stacks-project} and also 
    \cite{vistoli2004notes} for more details.

    \item We define a complex of finite flat group schemes over a base scheme $S$ to be exact if it is exact as a complex of fppf sheaves on fppf site of $S$. Over a noetherian ring $R$, a sequence $0\to G'\to G\to G''\to 0$ of finite flat $R$-group schemes is exact if and only if $G\to G''$ is faithfully flat and $G'\to G$ is the kernel of $G\to G''$. It is equivalent to say that $G'\to G$ is a closed immersion and the image is a normal subgroup of $G$, and $G\to G''$ is identified with the cokernel of $G'\to G$.
    
    \item If $\cF$ is a quasi-coherent sheaf of $\cO_S$-modules, we denote by $\underline{\cF}$ the fppf sheaf given by $\underline{\cF}(S')=\Gamma(S',\cF\otimes_{\cO_S}\cO_{S'})$ for any $S'\in (\text{Sch}/S)_{\text{fppf}}$.

\end{enumerate}
\end{remark}
\begin{prop}
Assume that $0\to G'\to G\to G''\to 0$ is an exact sequence of commutative group schemes over $S$ (as fppf sheaves).
\begin{enumerate}
    \item If $G'$ is finite flat $S$-group scheme, then $G\to G''$ is finite and faithfully flat.
    \item If $G'$ and $G''$ are finite flat group schemes over $S$, then so is $G$ i.e. extensions of finite flat group schemes (as fppf sheaves) are finite flat group schemes.
\end{enumerate}
\end{prop}
\begin{proof}
    \cite[Exp. IV]{demazure1970schemas(SGA3I)} and \cite[Prop. 5.2.7]{poonen2017rational}
\end{proof}


\begin{defn}
Let $G$ be a commutative group scheme over $S$.
\begin{enumerate}
    \item $G$ is called a vector group scheme if $G$ is locally isomorphic to $\G^r_a$ in the fpqc topology i.e. for every point $s\in S$ there exists a Zariski open neighbourhood $U$ of $s$ and an fpqc morphism $S'\to U$ such that $G\times_U S'$ is isomorphic to $\G^r_{a,S'}$ for some $r\geq 0$.
    \item $G$ is called a torus if it is locally isomorphic to $\G^r_m$ in fpqc topology. A torus is called quasi-isotrivial if in the above definition one can choose the morphisms $S'\to U$ to be \'etale i.e. $G$ is locally isomorphic to $\G^r_m$ in \'etale topology. It is called isotrivial, if there exists a surjective finite \'etale map $S'\to S$ such that $G' = G\times_S S'$ is isomorphic to $\G^r_{m,S'}$ as a group scheme over $S'$.
    \item $G$ is called a (free) quasi-Galois $S$-module or a $S$-lattice if it is locally isomorphic for \'etale topology to a finite free abelian constant group i.e. $\Z^r$ for some $r\geq 0$. A (free) quasi-Galois $S$-module is called a (free) Galois $S$-module or isotrivial lattice, if there exists a surjective finite \'etale map $S'\to S$ such that $G'=G\times_{S}S'$ is isomorphic to $\underline{\Z^r}_{S'}$ for some $r\geq 0$.
    \item $G$ is called an abelian scheme if $G\to S$ is proper, smooth and all fibres are geometrically connected.
    \item $G$ is called a semi-abelian scheme over $S$ if $G$ is an extension of an abelian scheme by a torus over $S$.
    
\end{enumerate}
\end{defn}
\begin{defn}
Let $f\colon A\to B$ be a homomorphism of abelian schemes or fppf group schemes over $S$. We say that $f$ is an isogeny if $f$ is finite and faithfully flat (surjective as fppf sheaves) with finite flat kernel.
\end{defn}
\begin{prop}
    Abelian group scheme $A$ over $S$ is commutative. If $S$ is normal, then $A$ is projective over $S$. The multiplication by $n$ map, $[n]\colon A\to A$, is an isogeny of degree $n^{2g}$, where $g$ is the relative dimension of $A$ over $S$.
\end{prop}

\section{Cartier dual}\label{sec: Cartier dual in appendix}
Let $V$ be an $R$-module. We define the dual module of $V$, denoted by $V^{\vee}$, to be $V^{\vee}:=\Hom_R(V,R)$.\\
If $G$ is either an $S$-torus, quasi-Galois $S$-module, or an affine commutative finite flat group scheme over $S$, we define Cartier dual of $G$ to be the $S$-group scheme $G^{\vee}:=\ihom(G,\G_m)$ as a representable sheaf on fppf site. For $G=\Spec A$ over $S=\Spec R$, by dualizing comultiplication, counit, and coinverse, $A^{\vee}=\Hom_R(A,R)$ becomes an $R$-Hopf algebra which is also finite flat over $R$ if $A$ is finite flat, and it represents $G^{\vee}$.

For abelian scheme $A$ over $S$, notice that $\ihom(A,\G_m)=0$ since $A$ is proper and $\G_m$ is affine. We define the Cartier dual of $A$ to be $A^{\vee}:=\Pic^0_{A/S}$.\\
Barsotti-Weil formula states that there is a natural isomorphism $$\Ext^1_S(A,\G_m)\to A^{\vee}(S)$$ which is compatible with base change, hence induces an isomorphism $\underline{\Ext}^1(A,\G_m)\cong A^{\vee}$ (see \cite[Chapter III]{oort2006commutative}). This means that $A^{\vee}$ is representable.

\begin{thm}
Let $G$ be one of the previously mentioned group schemes. The Cartier dual $(-)^{\vee}$ is an exact contravariant functor and $(G^{\vee})^{\vee}\cong G$. Moreover, The Cartier dual induces an antiequivalence between the category of $S$-tori and $S$-lattices that restricts to an antiequivalence between the category of isotrivial $S$-tori and isotrivial $S$-lattices.
\end{thm}
\begin{thm}[Duality theorem for abelian schemes]\label{duality theorem for abelian schemes} 
Let $f\colon A\to B$ be an isogeny of abelian schemes over $S$. Then the $\ker (f^{\vee})$ is naturally isomorphic to $(\ker f)^{\vee}$. So, we obtain the exact sequence $$0\to (\ker f)^{\vee}\to B^{\vee}\to A^{\vee}\to 0$$
\end{thm}
\begin{proof}
    See see \cite[Theorem 19.1]{oort2006commutative}
\end{proof}
\begin{prop}{\cite[Exp. X]{artin2006schemas}}\label{Appendix isotrivial lattice torus}
Let $R$ be a henselian local ring with residue field $k$.
\begin{enumerate}
    \item Every $R$-lattice ($R$-torus, resp.) is an isotrivial $R$-lattice ($R$-torus, resp.).
    \item The special fibre functor $G\mapsto G\times_{\Spec R}\Spec k$ induces an equivalence between the category of Galois $R$-modules ($R$-tori reps.) and the category of Galois $k$-modules ($k$-tori resp.).
    \item The functor $G\mapsto G(k^{sep})$ induces an equivalence between the category of Galois $R$-modules (finite \'etale group schemes resp.) and the category of finitely-generated free $\Z$-modules (finite abelian groups resp.) with continuous $G_k$-action.
\end{enumerate}
\end{prop}

\begin{prop}
Let $G$ be either an $S$-torus, an $S$-lattice, or an abelian $S$-scheme.
\begin{enumerate}
    \item The multiplication by $n$ map $[n]\colon G\to G$ is finite and faithfully flat. Its kernel $G[n]$ is a finite flat group scheme over $S$.
    \item If $n$ is coprime to the characteristics of all residue fields of $S$, then $A[n]$ is \'etale over $S$.
\end{enumerate}
\end{prop}
\begin{proof}
\cite{artin2006schemas}, \cite[Theorem 8.2]{milne1986arithmetic}
\end{proof}

\begin{thm}[Structure theory]\label{Structure theory of algebraic groups}
Let $K$ be a prefect field and $G$ an algebraic group over $K$. There is an exact sequence of algebraic groups 
$$0\to L\to G \to A\to 0,$$
where $L$ is a smooth linear algebraic group and $A$ an abelian variety. Assume that $G$ is commutative. Then $L$ is commutative and $L\cong V\times T$ with $V$ a unipotent group and $T$ a torus. If characteristic of $K$ is $0$, then $V$ is a vector group.
\end{thm}
\begin{proof}
See \cite{conrad2002modern}, or \cite{milne2013proof} for exactness of the sequence.\\
By \cite[Ch. III Proposition 12]{serrealgebraicgroupsandclaasfields} we have $L$ is the product of a unipotent group and a torus. By \cite[Ch. VII \S 2.7]{serrealgebraicgroupsandclaasfields}, in characteristic $0$ all unipotents are vector groups.
\end{proof}

From now on, by a finite flat group scheme, we mean an affine commutative finite flat group scheme over an affine base, unless we specify otherwise.
\begin{prop}
Assume that $G$ is a finite flat group scheme over $S=\Spec R$ of order $n$.
    \begin{enumerate}
        \item If the order of $G$ is invertible in $R$, then $G$ is \'etale. In particular, if $R$ is a field of characteristic $0$, then every finite flat group scheme over $R$ is \'etale.
        \item(Deligne) The multiplication by $n$ map $[n]\colon G\to G$ factors through unit section $\Spec R\to G$ i.e. $[n]$ annihilates $G$. 
    \end{enumerate}
\end{prop}

\begin{thm}[The connected-\'etale sequence]\label{connected-etale exact sequence for finite flat}
Let $R$ be a henselian local ring with residue field $k$ and $G$ a finite flat group scheme over $R$.
\begin{enumerate}
    \item $G$ is \'etale (connected resp.) if and only if its special fibre is \'etale (connected resp.).
    \item $G$ is connected if and only if it is a spectrum of a henselian local finite $R$-algebra.
    \item We denote by $G^0$ the connected component of the unit section. It is closed normal subgroup scheme of $G$.
    \item The quotient $G^{\'et}:=G/G^0$ is finite \'etale $R$-group scheme and we have the exact sequence
    $$0\to G^0\to G\to G^{\'et}\to 0 $$
    called the connected-\'etale sequence for $G$. \item Every group homomorphism $G\to H$ of \'etale finite $R$-group schemes factors through $G\to G^{\'et}$.
    \item Every group homomorphism $H\to G$ of connected finite flat $R$-group schemes factors through $G^{0}\to G$.
    \item If $R=k$ is a prefect field, then the connected-\'etale sequence splits canonically.
    \item The functor $(.)^0$ and $(.)^{\'et}$ are exact.
    \item $G$ is connected (\'etale resp.) if and only if $G(\bar{k})=0$ ($G^0=0$ resp.).
    \item An extension of a connected (\'etale resp.) finite flat $R$-group scheme by a connected (\'etale resp.) finite flat $R$-group scheme is connected (\'etale resp.).
\end{enumerate}
See \cite{tate1997finite} , \cite{pink2004finite}, or\cite{stix2009course}.
\end{thm}

\subsection{Lie module}

Let $G$ be a group scheme over $S$. The Lie functor $\underline{\Lie}_{G}$ or more precisely $\underline{\Lie}_{G/S}$ is the sheaf of $\cO_S$-module of left invariant vector fields on $G$ i.e. for any $S'$, $\underline{\Lie}_G(S')$ is the elements in $\Der_S(\cO_G(S'),\cO_G(S'))=\Hom_{\cO_G}(\Omega_{G/S}(S'),\cO_G(S'))$ which are left invariant.

If $S'=\Spec(R)$ is an affine base scheme over $S$, we have the canonical identification
$$\underline{\Lie}_{G}(R)=\ker(G(R[\epsilon])\to G(R))=\Hom_{R}(e^*\Omega_{G/R}(G),R)$$
where $e$ is just the unit section $e\colon S\to G$ and $R[\epsilon]=R[x]/(x^2)$ is the dual number over $R$.

We denote $\underline{\Lie}_G(S)$ by $\Lie(G)$ and it is called the Lie algebra of $G$. $\Lie(G)$ is exactly the tangent space of $G$ at $e$ i.e. $\Lie(G)=\Tang_{G,e}$. Over an affine base $S=\Spec R$, the co-Lie algebra of $G$ is the dual $R$-algebra $\Lie\ve(G):=\Hom_R(\Lie(G),R)$.

Assume that $G=\Spec(A)$ is an affine commutative group scheme over an affine base $S=\Spec(R)$. We have $\Omega_{G/S}=\Omega_{A/R}\cong J/J^2$, where $J$ is the kernel of multiplication map $A\otimes_R A\to A,\, a\otimes b\mapsto ab$. We define the augmentation ideal of $G$ to be the kernel of the counit i.e. $I:=\ker(\epsilon:A\to R)$. We have $A\cong R\oplus I$ as an $R$-module. We can write the following isomorphisms $$J\cong A\otimes_R I,\,\, J/J^2\cong A\otimes_R I/I^2\cong \Omega_{A/R}$$
So, $\Lie(G)=\Hom_R(\Omega_{A/R},R)=\Hom_R(A\otimes_R I/I^2,R)$.

\section{Relative Frobenius}\label{sec: relative Frobenius}
Let $S$ be a scheme in characteristic $p$. $S$ is equipped with the absolute Frobenius morphism, denoted by $\sigma_S:S\to S$, which is identity map on the underlying topology of $X$ with $\cO_S\to\cO_S$ given by $x\mapsto x^p$. The absolute Frobenius map is integral and purely inseparable.
\begin{defn}\label{Frobenius twist}
Let $S$ be a scheme in characteristic $p$ and $f\colon G\to S$ a morphism. We define $G^{(p)}:=G\times_{S,\sigma_S}S$ viewed as a scheme over $S$. Looking at the below diagram, there is a unique morphism $F_{G/S}\colon G\to G^{(p)}$ over $S$ making the diagram commute.

\begin{equation*}
\begin{tikzcd}
G \arrow[rrd, "\sigma_G", bend left] \arrow[rdd, "f", bend right] \arrow[rd, "F_{G/S}", dashed] &                           &                  \\
                                                                                                & G^{(p)} \arrow[r] \arrow[d] & G \arrow[d, "f"] \\
                                                                                                & S \arrow[r, "\sigma_S"]   & S               
\end{tikzcd}
\end{equation*}
The morphism $F_{G/S}$ is called the relative Frobenius of $G/S$ and when the base is known, it is denoted by $F_G$.

Assume that $G=\Spec A$ is affine scheme over affine base scheme $S=\Spec R$. Then the relative Frobenius $F_G\colon G\to G^{(p)}$ is given by $(\sigma_G,f)$ which is induced by $A\otimes_{F_A,R}R\to A,\, a\otimes r\mapsto a^pr$. 

The assignment $X\mapsto X^{(p)}$ is a base change functor for the absolute Frobenius map $\sigma_S\colon S\to S$. Inductively, we can define $G^{(p^n)}:=(G^{(p^{n-1})})^{(p)}$ and $F^n_G=F_{G^{(p^{n-1})}}\circ F^{n-1}_G$ for any $n\geq 2$. The scheme $G^{(p^n)}$ is $n$-th Frobenius twist.
\end{defn}

\begin{prop}
\begin{enumerate}
    \item  The relative Frobenius $F_G:G\to G^{(p)}$ is integral and purely inseparable (\cite[\href{https://stacks.math.columbia.edu/tag/0CCB}{Lemma 0CCB}]{stacks-project}).
\item $\Omega_{G/S}=\Omega_{G/G^{(p)}}$. (\cite[\href{https://stacks.math.columbia.edu/tag/0CCC}{Lemma 0CCC}]{stacks-project}).
\item If $G\to S$ is locally finite, then $F_{G}\colon G\to G^{(p)}$ is finite (\cite[\href{https://stacks.math.columbia.edu/tag/0CCD}{Lemma 0CCD}]{stacks-project}).
\item If $G$ is a finite flat group scheme over $S$, then $G^{(p^{n})}$ is a finite flat group scheme over $S$ and the relative Frobenius map $F^n_G$ is a homomorphism. 
\end{enumerate}
\end{prop}
\begin{defn}
Let $G$ be a finite flat group scheme over field $k$ of characteristic $p$. We define the Verschiebung of $G/k$ to be the map $V_G\colon G^{(p)}\to G$ induces by the dual of the relative Frobenius of $G^{\vee}$ via the natural identification $((G^{\vee})^{(p)})^{\vee}\cong G^{(p)}$.
\end{defn}
\begin{prop}\cite[\S 14]{pink2004finite}\label{frobenius and connectedness and etaleness}\\
Assume that $G$ is a finite flat group scheme over field $k$ of characteristic $p$.
    \begin{enumerate}
        \item $V_G\circ F_G=[p]_G$ and $F_G\circ V_G=[p]_{G^{(p)}}$.
        \item $G$ is connected if and only if $F_G$ is nilpotent.
        \item $G$ is \'etale if and only if $F_G$ is an isomorphism.
        \item $G$ is unipotent if and only if $V$ is topologically nilpotent.
        \item $G$ is multiplicative if and only if $V$ is an isomorphism.
        \item $G$ is bi-infinitesimal if and only if both $F_G$ and $V_G$ are nilpotent.
    \end{enumerate}
\end{prop}
\begin{example}
    \begin{enumerate}
        \item $F_{\mu_p}=0$ and $V_{\mu_p}$ is an isomorphism.
        \item $F_{\alpha_p}=0$ and $V_{\alpha_p}=0$.
        \item $F_{\underline{\Z/p\Z}}$ is an isomorphism and $V_{\underline{\Z/p\Z}}=0$.
    \end{enumerate}
\end{example}

\begin{thm}
\begin{enumerate}
    \item Let $k$ be perfect field of characteristic $p>0$ and $G=\Spec A$ is a connected finite group scheme over $k$. Then there is a $k$-isomorphism 
    $$A\cong k[X_1,\dots,X_n]/(X^{p^{e_1}}_1,\dots,X^{p^{e_n}}_n)   $$
    for some $n\in\N$ and $e_i\in\N$, which are invariant of $G$ up to permutation of $e_i$'s.
    
    \item The order of an affine connected finite flat group scheme over a field of characteristic $p$ is a power of $p$.
    
    \item A finite flat group scheme of order invertible in the base scheme $S=\Spec R$ is \'etale.
    
    \item Let $(R,\fm)$ be a complete noetherian local ring with prefect residue field $k$. Then for a finite flat connected group scheme $G=\Spec A$ over $R$ there is a $k$-isomorphism
    $$A\cong R[[X_1,\dots,X_n]]/(f_1,\dots,f_n)  $$
    for each $1\leq i\leq n$, there exists $e_i\N$ such that $f_i-X^{p^{e_i}_i}_i\in \fm R[X_1,\dots,X_n]$ is a polynomial of degree less than $p^{e_i}$.
    \item Let $R$ be a noetherian domain and $p$ a prime ideal of $R$. Let $\hat{R}$ be the completion of $R$ with respect to the p-adic topology. The functor 
    $$G\mapsto (G_{\hat{R}},G_{R[\frac{1}{p}]},id_{G_{\hat{R}[\frac{1}{p}]}})$$
    is an equivalence of categories from the category of finite flat group schemes over $R$ to the category of the triples $(G,H,\phi)$, where $G$ and $H$ are finite flat group schemes over $\hat{R}$ and $R[\frac{1}{p}]$, respectively, and $\varphi\colon G\times_{\Spec \hat{R}}\Spec \hat{R}[\frac{1}{p}]\to H\times_{\Spec R[\frac{1}{p}]}\Spec \hat{R}[\frac{1}{p}]$ is an isomorphism.
    \item (Mayer-Vietoris exact sequence): Let $G$ and $H$ be p-power order finite flat group schemes over noetherian ring $R$. Then, we have the following exact sequence
    \begin{align*}
    0\to \Hom_R(G,H)\to\Hom_{\hat{R}}(G,H)\times\Hom_{R[\frac{1}{p}]}(G,H)\to \Hom_{\hat{R}[\frac{1}{p}]}(G,H)\\ \xrightarrow{\delta} \Ext^1_R(G,H)\to \Ext^1_{\hat{R}}(G,H)\times \Ext^1_{R[\frac{1}{p}]}(G,H)\to \Ext^1_{\hat{R}[\frac{1}{p}]}(G,H)  
    \end{align*}
    where $\delta$ is defined by $$\delta\alpha=((G\times_R H)_{\hat{R}},(G\times_R H)_{R[\frac{1}{p}]},id_H\circ id_G+\alpha)$$ for any $\alpha\in \Hom_{\hat{R}[\frac{1}{p}]}(G,H)$.
\end{enumerate}
See \cite[\S 14]{pink2004finite}, \cite{stix2009course}, or \cite{Sc2}. See also \cite[\href{https://stacks.math.columbia.edu/tag/032A}{Theorem 032A}]{stacks-project} for (4).
\end{thm}

\end{appendix}
\cleardoublepage
\phantomsection
\addcontentsline{toc}{chapter}{\bibname}
\bibliographystyle{alpha}
\bibliography{References.bib, references}

\end{document}